\documentclass[11pt]{article}

\usepackage[letterpaper,margin=1in]{geometry}
\usepackage{amsmath,amssymb,amsthm,mathtools}
\usepackage{enumitem}
\usepackage{xcolor}
\usepackage{etoolbox}
\usepackage{graphicx}
\usepackage{tikz}
\usetikzlibrary{arrows,automata,positioning}
\usetikzlibrary{arrows.meta,calc}
\usepackage[colorlinks=true,linkcolor=purple,citecolor=blue,urlcolor=blue,linktoc=page]{hyperref}

\newlength{\tocsectionskip}
\newlength{\tocpartskip}
\makeatletter
\patchcmd{\l@section}{\addvspace{1.0em \@plus\p@}}{\addvspace{\tocsectionskip}}{}{\PackageWarning{main}{toc section patch failed}}
\patchcmd{\l@part}{\addvspace{2.25em \@plus\p@}}{\addvspace{\tocpartskip}}{}{\PackageWarning{main}{toc part patch failed}}
\makeatother

\theoremstyle{definition}
\newtheorem{theorem}{Theorem}[section]
\newtheorem{lemma}[theorem]{Lemma}
\newtheorem{proposition}[theorem]{Proposition}
\newtheorem{corollary}[theorem]{Corollary}
\newtheorem{example}[theorem]{Example}
\newtheorem{remark}[theorem]{Remark}
\newtheorem*{lemma*}{Lemma}
\numberwithin{equation}{section}

\newcommand{\RR}{\mathbb{R}}
\newcommand{\QQ}{\mathbb{Q}}
\newcommand{\ZZ}{\mathbb{Z}}
\newcommand{\NN}{\mathbb{N}}
\newcommand{\CC}{\mathbb{C}}
\newcommand{\HH}{\mathbb{H}}
\newcommand{\TT}{\mathbb{T}}
\newcommand{\cM}{\mathcal{M}}
\newcommand{\cT}{\mathcal{T}}
\newcommand{\cD}{\mathcal{D}}
\newcommand{\cR}{\mathcal{R}}
\newcommand{\cK}{\mathcal{K}}
\newcommand{\HGAS}{\mathrm{HGAS}}
\newcommand{\HLS}{\mathrm{HLS}}
\newcommand{\gradQ}{\operatorname{grad}_Q}
\newcommand{\gradS}{\operatorname{grad}_{S_Q}}
\newcommand{\re}{\operatorname{Re}}
\newcommand{\im}{\operatorname{Im}}
\newcommand{\avg}[1]{\langle #1\rangle}
\newcommand{\lem}{\le_{\mathrm m}}
\newcommand{\pg}[1]{p.~#1}

\makeatletter
\newcommand{\paperpart}[1]{%
  \par\addvspace{3.5ex \@plus 1ex \@minus .2ex}%
  \phantomsection\addcontentsline{toc}{part}{#1}%
  \noindent\rule{\textwidth}{0.4pt}\par\nobreak\vspace{1.2ex}%
  {\centering\normalfont\Large\bfseries #1\par}\nobreak\vspace{0.6ex}%
  \noindent\rule{\textwidth}{0.4pt}\par\nobreak\vspace{3ex}%
  \@afterheading}
\makeatother

\title{Global asymptotic stability of homogeneous polynomial vector fields is undecidable}
\author{Ivan O. Shevchenko\thanks{Email: \texttt{i3shevch@uwaterloo.ca}}
\and Jun Liu\thanks{Email: \texttt{j.liu@uwaterloo.ca}}
\and Xinzhi Liu\thanks{Email: \texttt{xzliu@uwaterloo.ca}}}
\date{}
\makeatletter
\renewcommand{\@maketitle}{%
  \newpage
  \null
  \vskip 2em%
  \begin{center}%
  \let \footnote \thanks
    {\LARGE \@title \par}%
    \vskip 1.5em%
    {\large
      \lineskip .5em%
      \begin{tabular}[t]{c}%
        \@author
      \end{tabular}\par}%
    \vspace{1em}%
    {\large University of Waterloo, Department of Applied Mathematics\par}%
  \end{center}%
  \par
  \vskip 1.5em}
\makeatother

\begin{document}

\pagenumbering{roman}
\maketitle
\thispagestyle{empty}

\begin{abstract}
We study whether global asymptotic stability (GAS) of the origin is decidable for homogeneous polynomial vector fields with rational coefficients. Our main technical contribution is showing that, for such vector fields of odd degree $d \geq 3$ in every sufficiently large dimension $n$, the problem of deciding GAS of the origin is many-one complete for the class of recursively enumerable sets (and hence undecidable). Using the tools we have built up, we also improve on some recent work by proving a similar result for Lyapunov stability instead of GAS.

The proof of our main contribution splits naturally into two parts:
\begin{itemize}
    \item In the first part we give, for every fixed dimension $n$ and odd degree $2k+1$, a computable map which sends a homogeneous polynomial vector field of degree $2k+1$ on $\RR^n$ to a homogeneous polynomial vector field of degree three on $\RR^m$, where $m=\binom{n+k-1}{k}$, and preserves both GAS and Lyapunov stability. As a consequence, the problems of deciding whether the origin is GAS or Lyapunov stable for rational-coefficient homogeneous polynomial vector fields of any degree, of any fixed odd degree $d \geq 3$, and of degree three are pairwise many-one equivalent.
    \item In the second part we prove that there is a fixed dimension $n_*$ and a fixed odd degree $d_*$ such that the set of rational-coefficient homogeneous polynomial vector fields on $\RR^{n_*}$ of degree $d_*$ whose origin is GAS is many-one complete. This is done by ``transferring'' many-one completeness between several different problems, starting from the solvability of a Diophantine equation and eventually arriving at our desired GAS decision problem.
\end{itemize}
\end{abstract}

\newpage
\pagestyle{empty}
\setcounter{tocdepth}{2}
{\small\tableofcontents}
\clearpage
\pagestyle{plain}
\pagenumbering{arabic}

\section{Introduction}\label{sec:intro}

Mathematicians working in the fields of dynamical systems and control theory have a long tradition of searching for algorithms to decide whether a given system (e.g., a system of ordinary differential equations) exhibits a certain property. This is not surprising, given that determining such properties is often closely tied to applications in the real world. One such property is that of \textit{global asymptotic stability}, which is an example of a \textit{stability} property of a system. Stability properties describe the long-term behaviour of a system as time goes on, for example, whether all system trajectories settle down to an \textit{equilibrium point} or if there are some trajectories which run off to infinity (the second option does not preclude the existence of an equilibrium point somewhere in the system). An equilibrium point is said to be globally asymptotically stable (GAS) if no matter where a trajectory of the system starts, it approaches the equilibrium, and starting close to the equilibrium guarantees that the trajectory never drifts too far away (see Figure~\ref{fig:gas} on page~\pageref{fig:gas}).

If a stability property such as GAS is shown to be undecidable, researchers in dynamical systems and control theory can stop searching for a general algorithm to decide this property and instead direct their attention to other questions. This is perhaps why V. I. Arnold thought the problem of deciding stability was important enough to include in his 1976 list of ``Problems of present day mathematics'' \cite{Arnold76}, in the special case of systems $\dot{x} = f(x)$ determined by polynomial vector fields $f: \RR^{n} \to \RR^{n}$ with rational coefficients:

\begin{quote}
\textit{Let a vector field be given by polynomials of a fixed degree, with rational coefficients. Does an algorithm exist, allowing to decide, whether the stationary point is stable?}
\end{quote}

\vspace{2px}

\noindent \textbf{Our work.} We will consider the subclass of polynomial vector fields with rational coefficients which are \textit{homogeneous of degree} $d \geq 1$: $f(\lambda x) = \lambda^{d}f(x)$ for all $\lambda>0$. This property necessarily implies that the origin $0 \in \RR^{n}$ is an equilibrium point of the system, and stability of such systems is always considered with respect to the origin. Any undecidability result for this subclass will imply an analogous undecidability result for polynomials without the homogeneity condition (Corollary~\ref{cor:nonhomog}).

The reason we require rational coefficients is that each such function $f$ may then be described by a finite amount of data: the dimension $n$ of the state space $\RR^{n}$, the degree of homogeneity $d$, and the finite list of rational coefficients. We may fix an ``encoding'' of this data such that every rational-coefficient homogeneous polynomial vector field $f$ corresponds to some ``code'' $e \in \NN$. We denote by $\mathrm{HGAS} \subseteq \NN$ the set of all codes $e \in \NN$ which correspond to homogeneous polynomial vector fields $f$ whose origin is GAS, over all state dimensions $n \geq 1$ and degrees of homogeneity $d \geq 1$. For the subset of codes corresponding to functions in a fixed dimension $n$ and of fixed degree $d$, we write $\mathrm{HGAS}(n,d) \subseteq \mathrm{HGAS}$.

After including the problem in his 1976 list, Arnold later conjectured \cite{dCD93} that for some sufficiently high dimension and degree, determining stability is undecidable. For global asymptotic stability of homogeneous vector fields, the main result of this paper proves his conjecture.

\begin{theorem}\label{thm:main}
There is an integer $N\ge1$ such that for every $n\ge N$ and every odd $d\ge3$, the set $\HGAS(n,d)$ is many-one complete (and hence undecidable).
\end{theorem}

When we write \textit{many-one complete}, we will always mean ``many-one complete for the class of recursively enumerable sets.'' Since the cases of even degree $d$ and of degree $d=1$ for any dimension $n \geq 1$ are known to be decidable (see Remark \ref{rem:even} and Lemma \ref{lem:linear}), our Theorem \ref{thm:main} completes the classification of the decidability of homogeneous polynomial vector fields with respect to degree. This theorem, together with some of its most important corollaries and the intermediate results building up to it, has been formalized in Lean~4 (Section~\ref{subsec:formal}).

We posit that Theorem \ref{thm:main} is a ``gateway'' to proving undecidability of many other properties of interest in dynamical systems and control theory, by reducing or relating these other properties to GAS of homogeneous polynomial vector fields. For example, in Corollary~\ref{cor:nonhomog} we show that Theorem \ref{thm:main} implies the analogous undecidability results for both GAS and the stability notion of \textit{local asymptotic stability} of non-homogeneous polynomial vector fields, thereby proving Arnold's conjecture for both of these stability notions for \textit{every} degree $d \geq 3$ (including even degrees). The first two authors have also very recently used Theorem \ref{thm:main} to prove that global controllability (a property of great practical importance in engineering mathematics) is undecidable for polynomial control systems with rational coefficients, and furthermore that the set of rational-coefficient controllable polynomial systems is not even recursively enumerable \cite{liu2026contr}. The same conclusions are also shown to hold for the class of control systems with polynomial drift term of the form $\dot{\xi} = f(\xi) + Bu$. These results address two conjectures posed by Sontag in 1988 \cite{Sontag1988} (one is proven and the other is disproven). The authors believe that, through similar reductions to Theorem \ref{thm:main}, many other undecidability results are now also within reach.

\medskip

\noindent \textbf{Relation to the literature.} Undecidability results are by now fairly common in dynamical systems and control theory, and many of them rely on the observation that sufficiently complicated dynamical systems can simulate Turing machines. Moore showed this for the motion of a single particle in a suitable three-dimensional potential~\cite{Moore1990}, Koiran, Cosnard, and Garzon for iterated piecewise-linear maps in dimension two~\cite{KoiranCosnardGarzon94}, and Koiran and Moore for closed-form analytic maps in dimensions one and two~\cite{KoiranMoore99}, so that practically any question about the long-term behaviour of a single trajectory of such systems is undecidable. It is more difficult, however, to obtain similar results about properties of all trajectories at once. Blondel, Bournez, Koiran, and Tsitsiklis proved, answering a question of Sontag, that GAS is undecidable for discrete-time saturated linear systems and for continuous piecewise affine maps in dimension three~\cite{BBKT2001}, and Blondel and Tsitsiklis proved that boundedness of trajectories is undecidable for discrete-time linear switched systems~\cite{BlondelTsitsiklis2000}. For polynomial ODEs, Gra{\c{c}}a, Buescu, and Campagnolo proved that it is undecidable whether the maximal interval of existence of a solution is bounded~\cite{gracca2008boundedness}. It is important to note that in each of these results either the system is not given by a smooth (let alone polynomial) vector field, or the property in question concerns a single trajectory rather than every trajectory of the system at once (as is the case for GAS).

For continuous-time systems $\dot x = f(x)$, da Costa and Doria showed that stability of an equilibrium point is undecidable when $f$ may be written as an arithmetic combination of elementary functions such as sine and exponential, thus resolving a non-polynomial version of Arnold's conjecture~\cite{dCD93}. Much later, Ahmadi showed that deciding global and local asymptotic stability for polynomial vector fields (a much smaller class than the one considered by da Costa and Doria) is strongly NP-hard already for homogeneous vector fields of degree $d = 3$, and asked explicitly whether this problem is decidable at all \cite{Ahmadi12} (this is the same problem as in Theorem \ref{thm:main}). Ahmadi and El Khadir also proved that whenever the origin of a homogeneous polynomial vector field is GAS, the system admits a \textit{Lyapunov function} of a special rational form \cite{AEK} (a Lyapunov function may be thought of as a ``certificate of stability'' for the system). Combined with Tarski's quantifier elimination over the real numbers \cite{Tarski}, this implies that $\HGAS$ is recursively enumerable (see Proposition \ref{prop:re}), so the many-one completeness in Theorem \ref{thm:main} is the strongest hardness statement one could hope for.

Very recently, Korda \cite{Korda} posted a resolution of Arnold's conjecture for the weaker stability notion of \textit{Lyapunov stability}, obtained with AI systems. Roughly speaking, Lyapunov stability only requires that trajectories starting close to the equilibrium point must remain close to it for all forward time. Denoting by $\mathrm{HLS}$ and $\mathrm{HLS}(n, d)$ the analogous sets of codes to $\mathrm{HGAS}$ and $\mathrm{HGAS}(n, d)$ but for Lyapunov stability instead of GAS, in \cite{Korda} it is proven that there exist (very large) integers $n \geq 1$ and odd $d \geq 3$ such that $\mathrm{HLS}(n,d)$ is undecidable. In the supplementary document \cite{KordaSupp} this result was strengthened to the $n=5$ case, although no attempt was made to lower the degree of homogeneity. Our Theorem~\ref{thm:transfer}, which we use as a tool for proving Theorem~\ref{thm:main}, may be used in conjunction with the results in \cite{Korda,KordaSupp} to greatly improve on the ``degree'' part of this undecidability result: we show in Corollary~\ref{cor:ls-cubic} that there exists an integer $N \geq 1$ such that for every $n \geq N$ and odd $d \geq 3$ (in particular we may take $d=3$), the set $\mathrm{HLS}(n,d)$ is \textit{many-one hard} (that is, every recursively enumerable set many-one reduces to it).

The document \cite{KordaSupp} also shows that the property of GAS is undecidable for non-homogeneous polynomial vector fields of the form $f(x) = F(x) - x$, where $F$ is homogeneous, in some large dimension $n$ and large odd degree $d \geq 3$. This however does not address the homogeneous case, does not say anything about polynomial vector fields of small degree $d$, and cannot be used to conclude similar undecidability results for local asymptotic stability. All of this is done in Theorem \ref{thm:main} and Corollary~\ref{cor:nonhomog}.

\section{Proof architecture and overview}\label{sec:overview}

For our approach we will require many preliminaries from computability theory, ordinary differential equations, and ergodic theory, which are described in Section~\ref{sec:prelim}. We will make heavy use of some alternative characterizations of the properties of GAS and Lyapunov stability, which will be proven in Section~\ref{sec:angular}. The proof of Theorem~\ref{thm:main} is split into two main parts, which correspond to Part I (Sections~\ref{sec:lifting} and~\ref{sec:normalform}) and Part II (Sections~\ref{sec:arith}--\ref{sec:single}) of the paper. Theorem \ref{thm:main} follows immediately from these two parts, and various corollaries are discussed in Part III (Section~\ref{sec:main}).

The goal of Part I is to show that the properties of both GAS and Lyapunov stability for homogeneous polynomial vector fields of odd degree reduce to the same problem for vector fields of degree three, albeit in a higher dimension. The main result of Part I is the following. 

\begin{theorem}\label{thm:transfer}
Suppose there exist integers $n_{*} \geq 1$ and odd $d_{*} = 2k_{*}+1 \geq 3$ such that $\HGAS(n_{*}, d_{*})$ is many-one complete (respectively, undecidable). Set $N =\binom{n_*+k_*-1}{k_*}$. Then for every $n \geq N$ and odd $d \geq 3$, the set $\HGAS(n,d)$ is also many-one complete (respectively, undecidable). The same result holds for $\HLS$ and many-one hardness instead of $\HGAS$ and many-one completeness.
\end{theorem}

Combining the $\HLS$ part of Theorem \ref{thm:transfer} with the work in \cite{Korda,KordaSupp} yields the following result, which we carefully prove in Section~\ref{sec:transfer}.

\begin{corollary}\label{cor:ls-cubic}
There is an integer $N \geq 1$ such that for every $n \geq N$ and every odd $d\ge3$, the set $\HLS(n,d)$ is many-one hard (and hence undecidable).
\end{corollary}

The goal of Part II is to find one fixed dimension and one fixed odd degree for which GAS is many-one complete. This is exactly what we need, together with Theorem \ref{thm:transfer}, to prove Theorem \ref{thm:main}. Precisely, the main result of Part II is the following.

\begin{theorem}\label{thm:fixedpair}
There exist an integer $n_*\ge1$ and an odd integer $d_*\ge3$ such that $\HGAS(n_*,d_*)$ is many-one complete.
\end{theorem}

For the rest of Section \ref{sec:overview}, we explain the main technical and conceptual ideas of the work in Section~\ref{sec:angular} and Parts~I and~II of the paper. We do not give an overview of the consequences in Part~III (the cases where the dimension or the degree is part of the input in Corollary~\ref{cor:variable}, non-homogeneous vector fields and local asymptotic stability in Corollary~\ref{cor:nonhomog}, and Lyapunov stability in Section~\ref{sec:transfer}), since they follow in a fairly direct way from Theorem~\ref{thm:main} and the results of Part~I.

\subsection{Overview of Section~\ref{sec:angular}}

\subsubsection{Equivalent criteria for GAS and Lyapunov stability (Sections~\ref{subsec:polar} and~\ref{subsec:the_two_criteria})}\label{subsec:overview-criteria}

Before proceeding with anything else, it is desirable to rephrase the GAS and Lyapunov stability properties of the system $\dot{y} = P(y)$, where $P:\RR^s\to\RR^s$ is a homogeneous polynomial vector field of degree $D>1$, into something easier to work with for the remainder of the paper. This is accomplished by Lemmas \ref{lem:gascrit} and \ref{lem:lyap}, whose main ideas we describe below.

For any positive definite quadratic form $Q(y) = y^{\top}My$ on $\RR^s$, with $M$ a symmetric positive definite matrix, write $\langle v,w \rangle_{Q} = v^{\top}Mw$ and $\|v\|_Q=Q(v)^{1/2}$, and set $S_{Q} = \{y \in \RR^{s} : Q(y) = 1\}$. This is the surface of an ellipsoid. When trying to determine whether a system is GAS, it is often useful to perform a variant of polar decomposition: we decompose the original system trajectories $y(t)$ solving $\dot{y} = P(y)$ into a \textit{radial part} $r(y(t)) \in \RR_{\geq0}$ indicating a ``distance'' from the origin, and an \textit{angular part} $u(y(t))$ which always lies on the surface $S_{Q}$. Then, $y(t) = r(y(t))\cdot u(y(t))$. 

Now, define functions $a_{P}: \RR^{s} \to \RR$ by $a_{P}(y) = \langle y, P(y) \rangle_{Q} := (y^{\top}M) P(y)$ and $Y_{P}: \RR^{s} \to \RR^{s}$ by $Y_{P}(y) = P(y) - a_{P}(y)y$. These functions are chosen such that, after a clever rescaling of time $t \mapsto \tau(t)$, all functions $u(y(\tau(t)))$ are solutions to $\dot{u} = Y_{P}(u)$. As such, $S_{Q}$ is invariant under the flow $\Phi_{Y_P}$ corresponding to this differential equation. Since we will only need to consider this flow on the set $S_{Q}$, from now on we restrict the domain of $Y_{P}$ from $\RR^{s}$ to $S_{Q}$ and regard $\Phi_{Y_{P}}$ as a flow on $S_{Q}$. Moreover, with this choice of $a_{P}$ and $Y_{P}$, the radial part $r(y(\tau(t)))$ depends \textit{solely} on $a_{P}(u(y(\tau(t))))$ in the following way:
\begin{equation}\label{eq:log_r_y_tau_DE}
    \frac{d}{dt}\log r(y(\tau(t))) = a_{P}(u(y(\tau(t)))).
\end{equation}
The computations used to obtain \eqref{eq:log_r_y_tau_DE} make heavy use of the homogeneity of $P$. To make it more transparent where $Y_{P}$ and $a_{P}$ come from, noting that $\langle u,u \rangle_{Q} = 1$ for $u \in S_{Q}$, we may rearrange $Y_{P}(u) = P(u) - a_{P}(u)u$ to get
\begin{equation}
    P(u) = Y_{P}(u) + \left( \frac{\langle u, P(u) \rangle_{Q}}{\langle u, u \rangle_{Q}} \right) u.
\end{equation}
The first term, $Y_{P}(u)$, is the part of the vector $P(u)$ which is tangent to the ellipsoid $S_{Q}$, since the second term, $a_{P}(u)u$, is precisely the $Q$-orthogonal projection of $P(u)$ onto the line through the vector $u$ (the part pointing either directly inwards or directly outwards from $S_{Q}$); see Figure~\ref{fig:polar} on page~\pageref{fig:polar}. This explains intuitively why $S_{Q}$ is invariant, which is what allowed us to restrict the domain of $Y_{P}$ from $\RR^{s}$ to $S_{Q}$. This also explains, together with \eqref{eq:log_r_y_tau_DE}, why the sign of $a_{P}(u)$ for $u \in S_{Q}$ is an important quantity: it tells us when $y(t)$ is pointing inwards from the surface $S_{Q}$ towards the origin.

However, requiring $a_{P}<0$ on all of $S_{Q}$ would be too strong of a condition to be a characterization of GAS; it is possible for $a_{P}$ to be slightly positive on some parts of $S_{Q}$ but very negative on other parts, so that the overall trend of the system trajectories is still to move inwards towards the origin (see Figure~\ref{fig:netdrift} on page~\pageref{fig:netdrift}). It turns out that the correct quantity to completely characterize GAS is instead the sign of the maximal \emph{mean} value of $a_P$, where the maximum is taken not over all probability distributions on $S_Q$ (which would just give us $\max_{u \in S_Q} a_P(u)$) but only over those that the flow of $\dot{u} = Y_{P}(u)$ on $S_{Q}$ leaves unchanged, i.e., the \emph{invariant probability measures} under the flow $\Phi_{Y_P}$.

\begin{lemma*}[Angular criterion for GAS; Lemma \ref{lem:gascrit} below]
Let $\mathcal{M}(\Phi_{Y_P})$ be the set of $\Phi_{Y_P}$-invariant Borel probability
measures on $S_Q$. Then the origin is GAS for $\dot y = P(y)$ if and only if
\begin{equation}\label{eq:overview_beta}
  \beta(\Phi_{Y_P}, a_P) := \max_{\mu \in \mathcal{M}(\Phi_{Y_P})} \int_{S_Q} a_P(u)\, d\mu(u) < 0.
\end{equation}
\end{lemma*}

Although this characterization is technical, its formulation in terms of $\Phi_{Y_P}$-invariant probability measures as opposed to individual trajectories of the system $\dot{u} = Y_{P}(u)$ on $S_{Q}$ greatly simplifies many proofs throughout Part I and at the end of Part II. We also make use of a variant of this lemma for Lyapunov stability, which is Lemma \ref{lem:lyap} below (originally Lemma 2.2 in \cite{Korda}); instead of $\beta(\Phi_{Y_P},a_{P})<0$, the ``angular criterion'' there is a statement about integrals of $a_{P}$ along the trajectories of $\dot{u} = Y_{P}(u)$ on $S_{Q}$:
\begin{equation}\label{eq:overview_sec_Lyap_characterization}
    M_{P}^{*} := \sup_{T\geq0} \max_{u \in S_{Q}} \int_{0}^{T} a_{P}(\Phi_{Y_P}^{t}(u)) \ dt < \infty.
\end{equation}

Unlike GAS, the property of Lyapunov stability cannot be characterized using integrals of $a_{P}$ with respect to $\Phi_{Y_P}$-invariant probability measures on $S_{Q}$ (Remark \ref{rem:compare}). The fact that we must work with individual trajectories ends up making some of the proofs related to Lyapunov stability in Part I more involved.

We note that both of these criteria are valid for any positive definite quadratic form $Q$. In particular, the choice of $Q$ does not affect whether $\beta(\Phi_{Y_P}, a_{P}) < 0$. In the main parts of the paper, it will be necessary to apply these criteria to different quadratic forms.

\subsubsection{Recursive enumerability of $\HGAS$ (Section~\ref{subsec:re_upper_bound})}

Prior to starting with Parts~I and~II we warm up by describing the easier half of Theorem~\ref{thm:main}, namely, the recursive enumerability of $\HGAS$ (from which recursive enumerability of each set $\HGAS(n,d)$ follows easily). This result, proven in Proposition~\ref{prop:re}, follows with a bit of work from Tarski's quantifier elimination theorem (Theorem~\ref{thm:tarski}) and a theorem of Ahmadi and El Khadir~\cite[Theorem~3.1]{AEK}. Roughly speaking, the latter result says that the origin of a homogeneous polynomial vector field $f: \RR^{n} \to \RR^{n}$ is GAS if and only if $f$ admits a ``certificate of stability'' of a particular algebraic form. This certificate is given by a homogeneous polynomial $p: \RR^{n} \to \RR$ of some even degree $2r+2$ satisfying two explicit polynomial inequalities on the unit sphere, which are determined by $f$ and $r$. As such, if $e$ is the code corresponding to some $f$, then $e \in \HGAS$ if and only if for some $r \in \mathbb{N}$, the following statement $D(e,r)$ is true: there exist real coefficients of a homogeneous polynomial $p$ of degree $2r+2$ which satisfy these inequalities (see~\eqref{eq:Der}). We note that there is no need to include the dimension $n$ corresponding to $f$ in the statement $D$ (which is also the dimension of the domain of $p$), since this information is already ``baked into'' the code $e$.

The statement $D(e,r)$ is a first-order sentence over $\RR$, and hence Tarski's theorem gives an algorithm that decides whether any such statement is true. Given a code $e$ we may thus sequentially decide the truth of $D(e,0), D(e,1), D(e,2), \dots$, and halt as soon as one of them is true. This procedure halts if and only if $e \in \HGAS$, so by definition this set is recursively enumerable.

\subsection{Overview of Part I}\label{subsec:overview1}

\subsubsection{Lifting to degree three (Section~\ref{sec:lifting})}

Section~\ref{sec:lifting} is dedicated to proving the following result, which together with some additional work in Section~\ref{sec:normalform} yields Theorem~\ref{thm:transfer} (the main result of Part I, which is one half of the proof of Theorem~\ref{thm:main}).

\begin{theorem}[Reduction to degree three]\label{thm:cubicization}
Fix $n\ge1$ and $k\ge1$, and put $m=\binom{n+k-1}{k}$. There is a computable map which sends the rational coefficients of a homogeneous polynomial vector field $F:\RR^n\to\RR^n$ of degree $2k+1$ to the rational coefficients of a homogeneous polynomial vector field $H:\RR^m\to\RR^m$ of degree three such that
\begin{align}
0\text{ is GAS for }\dot x=F(x)&\iff 0\text{ is GAS for }\dot y=H(y),\label{eq:mainGAS}\\
0\text{ is Lyapunov stable for }\dot x=F(x)&\iff 0\text{ is Lyapunov stable for }\dot y=H(y).\label{eq:mainLS}
\end{align}
The map can be chosen uniformly in $(n,k)$ (Proposition~\ref{prop:uniform}).
\end{theorem}

The main idea is that we want to ``trade'' the typically smaller dimension $n$ and larger degree of homogeneity $2k+1$ for the typically much larger dimension $m$ and the degree of homogeneity three, and do so in a way which preserves the properties of GAS and Lyapunov stability.

Suppose $x(t)$ solves $\dot{x} = F(x)$, and write $y(t) = \nu_{k}(x(t)) \in \RR^{m}$ for some carefully chosen map $\nu_{k}: \RR^{n} \to \RR^{m}$. Just like how we denote the coordinate indices of $\RR^{n}$ by $i \in \{1, \ldots, n\}$, we will denote the coordinate indices of $\RR^{m}$ by $\alpha \in A_{k}$ for some index set $A_{k}$. Using the chain rule, for each index $\alpha \in A_{k}$ we get
\begin{equation}\label{eq:sum_chain_rule_nu_x_t}
    \frac{d}{dt} y_{\alpha}(t) = \frac{d}{dt}(\nu_{k}(x(t)))_{\alpha} = \sum_{i=1}^{n} \left( \frac{\partial (\nu_{k})_{\alpha}}{\partial x_{i}} (x(t)) \right) F_{i}(x(t)).
\end{equation}
The map $\nu_{k}$ is defined in such a way that the following properties hold:
\begin{enumerate}
    \item Each coordinate $(\nu_{k}(x))_{\alpha}$ is a homogeneous polynomial of degree $k$ in $x$, so the first factor in each term of the sum in \eqref{eq:sum_chain_rule_nu_x_t} is homogeneous of degree $k-1$ in $x(t)$ (or identically zero). Since $F_{i}$ is homogeneous of degree $2k+1$, the right-hand side of \eqref{eq:sum_chain_rule_nu_x_t} is a homogeneous polynomial of degree $3k$ in $x$.
    \item For every homogeneous polynomial $p : \RR^{n} \to \RR$ of degree $3k$ there is a homogeneous polynomial $g_{p} : \RR^{m} \to \RR$ of degree three (which is computable from $p$) such that $p(x) = g_{p}(\nu_{k}(x))$ for all $x \in \RR^{n}$. Applying this to the right-hand side of \eqref{eq:sum_chain_rule_nu_x_t}, we obtain $\frac{d}{dt} y_{\alpha}(t) = G_{\alpha}(\nu_{k}(x(t))) = G_{\alpha}(y(t))$ for some homogeneous degree-three polynomials $G_{\alpha}: \RR^{m} \to \RR$.
\end{enumerate}

Hence, defining the homogeneous polynomial vector field $G: \RR^{m} \to \RR^{m}$ of degree three by $G(y) = (G_{\alpha}(y))_{\alpha \in A_{k}}$, we see that the system $\dot{x} = F(x)$ on $\RR^{n}$ ``lifts'' to the system $\dot{y} = G(y)$ on $\nu_{k}(\RR^{n}) \subseteq \RR^{m}$ with the desired change in the degree of homogeneity. For technical reasons it is more convenient to work with the set $V=\{t\,\nu_k(x): t\in\RR,\ x\in\RR^n\}$. From homogeneity of $\nu_{k}$, we have $V=\nu_k(\RR^n)$ for odd $k$ and $V = \nu_k(\RR^n) \cup (-\nu_k(\RR^n))$ for even $k$. The set $V$ is a cone in $\RR^{m}$; see Figure~\ref{fig:cone} (page~\pageref{fig:cone}). When $k$ is even, the lifted system on $-\nu_{k}(\RR^{n})$ is the mirror image of the lifted system on $\nu_{k}(\RR^{n})$.

It remains to make sure that this lifting procedure preserves the properties of GAS and Lyapunov stability from the original system $\dot{x} = F(x)$ to the system $\dot{y} = G(y)$. Unfortunately, as is, this is not necessarily (globally) true. From the way $\nu_{k}$ is defined, these properties are indeed preserved between trajectories of $\dot{x} = F(x)$ on $\RR^{n}$ and trajectories of $\dot{y} = G(y)$ on $V$; however, we may very easily have undesirable behaviour of trajectories of the system $\dot{y} = G(y)$ on the set $\RR^{m} \setminus V$ which we cannot control. To fix this issue, we append a ``penalty term'' $-\tilde{G}(y)$ to the map $G$ to obtain $H(y) = G(y) - \tilde{G}(y)$, defined (in \eqref{eq:H}) in such a way that the following properties hold:
\begin{enumerate}
    \item $\tilde{G}$ is homogeneous of degree three, so that $H$ is also homogeneous of degree three.
    \item $\tilde{G}(y)=0$ for $y \in V$, so that trajectories $y(t)$ of the system $\dot{y} = H(y)$ starting \textit{inside} the cone $V$ act the same as trajectories of the system $\dot{y} = G(y)$.
    \item Trajectories $y(t)$ of the system $\dot{y} = H(y)$ starting \textit{outside} the cone $V$ satisfy the following. If $y(t)$ is in some sense ``close'' to $V$, it approaches the dynamics of $\dot{y} = H(y)$ contained in $V$. Otherwise if $y(t)$ is not ``close'' to $V$, it tends to either $V$ (in which case it eventually becomes ``close'' to $V$) or to the origin.
    \item The map $\tilde{G}$ (and thus the map $H$) for which the above properties hold must be obtained in a computable way.
\end{enumerate}
All the results in Section~\ref{sec:lifting} up to and including Lemma~\ref{lem:lipschitz} build up to defining $\tilde{G}$ in a careful way such that all of these properties hold.

Having now defined the system $\dot{y} = H(y)$ and stated its properties in relation to the original system $\dot{x} = F(x)$, to prove Theorem \ref{thm:cubicization} it remains to show that GAS and Lyapunov stability properties are equivalent between these two systems. This is done using the stability characterizations in Lemmas \ref{lem:gascrit} and \ref{lem:lyap}, working in parallel with the systems $\dot{y} = H(y)$ on $\RR^{m}$ and $\dot{x} = F(x)$ on $\RR^{n}$:
\begin{itemize}
    \item Define a quadratic form $Q_{H}: \RR^{m} \to \RR$, which is specifically chosen to ``work nicely'' with $\nu_{k}$. Then, $S_{H} := S_{Q_{H}} = \{y \in \RR^{m} : Q_{H}(y) = 1\}$ as before.
    \item Define the quadratic form $Q_{F}: \RR^{n} \to \RR$ by $Q_{F}(x) = |x|^{2}$. Then, $S_{F} := S_{Q_{F}} = \{x \in \RR^{n} : Q_{F}(x) = 1\} = S^{n-1}$ is the $(n-1)$-sphere in $\RR^{n}$.
\end{itemize}
As in Section \ref{subsec:overview-criteria}, using the quadratic forms $Q_{H}$ and $Q_{F}$ we define the functions $a_{H}, Y_{H}, a_{F}, Y_{F}$ and the flows $\Phi_{Y_H}, \Phi_{Y_F}$ of $Y_H, Y_F$ with domains restricted to $S_{H} \subseteq \RR^{m}, S_{F} \subseteq \RR^{n}$, respectively. We would be done if we could show
\begin{equation}
    \beta(\Phi_{Y_H}, a_{H}) < 0 \ \ \text{iff} \ \ \beta(\Phi_{Y_F}, a_{F}) < 0 \ \ \ \ \text{and} \ \ \ \ M^{*}_{H} < \infty \ \ \text{iff} \ \ M^{*}_{F} < \infty.
\end{equation}
It turns out that comparing these quantities associated with the systems $\dot{y} = Y_{H}(y)$ and $\dot{x} = Y_{F}(x)$ is much easier if we also consider an intermediate system as a ``bridge'' between the two (see Figure~\ref{fig:bridge} on page~\pageref{fig:bridge}). To this end, define $\Sigma = V \cap S_{H} = \{\pm \nu_{k}(u) : u \in S_{F}\} \subseteq S_{H}$. As shown in Lemma \ref{lem:conedyn}, the set $\Sigma$ is invariant under the flow $\Phi_{Y_H}$; this allows us to separately consider the restricted flow $\Phi_{Y_H}|_{\Sigma}$. Using ergodic-theoretic arguments (\ref{E:decomp} together with Lemma \ref{lem:dichotomy}), we can show that
\begin{equation}\label{eq:Sigma_to_SH_int}
    \max_{\mu \in \cM(\Phi_{Y_H}), \mu(\Sigma)=1} \int_{\Sigma} a_{H} \ d\mu < 0 \ \ \text{iff} \ \ \max_{\mu \in \cM(\Phi_{Y_H})} \int_{S_{H}} a_{H} \ d\mu < 0.
\end{equation}
The second maximum above is equal to $\beta(\Phi_{Y_H}, a_{H})$. Lemma \ref{lem:conedyn} also shows that the flow $\Phi_{Y_H}|_{\Sigma}$ is essentially a lifted version of the flow $\Phi_{Y_F}$. This close relationship allows Lemma \ref{lem:coneavg} to conclude that
\begin{equation}\label{eq:Sigma_to_SF_int}
    \max_{\mu \in \cM(\Phi_{Y_H}), \mu(\Sigma)=1} \int_{\Sigma} a_{H} \ d\mu = k \cdot \max_{\mu \in \cM(\Phi_{Y_F})} \int_{S_{F}} a_{F} \ d\mu.
\end{equation}
The right-hand side above is equal to $k \cdot \beta(\Phi_{Y_F}, a_{F})$. Combining \eqref{eq:Sigma_to_SH_int} and \eqref{eq:Sigma_to_SF_int}, we conclude that $\beta(\Phi_{Y_H}, a_{H}) < 0$ if and only if $\beta(\Phi_{Y_F}, a_{F}) < 0$.

Showing that $M^{*}_{H} < \infty$ if and only if $M^{*}_{F} < \infty$ is also done by using $\Sigma$ as a ``bridge'' between the two systems, but instead of working with invariant measures we work directly with the system trajectories. Since trajectories of $\dot{y} = Y_{H}(y)$ on $\Sigma$ are precisely the lifts of trajectories of $\dot{x} = Y_{F}(x)$ on $S_{F}$, the only problem that may arise is from the trajectories of $\dot{y} = Y_{H}(y)$ on $S_{H} \setminus \Sigma$. This problem is addressed by property 3 of $H$, which allows us to carefully control these trajectories on $S_{H} \setminus \Sigma$ which could otherwise lead to having $M^{*}_{H} = \infty$ even when $M^{*}_{F} < \infty$.

\subsubsection{From degree three to everything else (Section~\ref{sec:normalform})}

Although we have established Theorem~\ref{thm:cubicization}, which is the ``hard part'' of the proof of Theorem~\ref{thm:transfer}, there is still a bit more work to be done. Notably, while Theorem~\ref{thm:cubicization} allows us to go from any odd degree $d \geq 3$ and dimension $n$ to degree three and some dimension $m \geq n$, what Theorem~\ref{thm:cubicization} does not address is how to instead go from degree three back up to \textit{every} odd degree $d \geq 3$, and also how to go from dimension $m$ up to \textit{every} dimension $n \geq m$.

This missing piece is addressed by ``padding'' the system with decoupled variables which do not interfere with the properties of GAS or Lyapunov stability of the rest of the system. This is done in Lemma \ref{lem:padding} by going from the system $\dot{x} = H(x)$ on $\RR^{m}$ to the system $(\dot{x}, \dot{z}) = \widehat{H}(x,z)$ on $\RR^{n}$ (for any $n \geq m$), where $\widehat H(x,z) = \bigl(|x|^2+|z|^2\bigr)^{j} \bigl(H(x),\,-z_1^3,\dots,-z_{n-m}^3\bigr)$ is homogeneous of degree $3+2j$ with $(x,z)\in\mathbb{R}^m\times \mathbb{R}^{n-m}$ and $j \in \NN$. It is straightforward to verify (Lemma \ref{lem:padding}), using the quantities $\beta$ and $M^{*}$ from Lemmas \ref{lem:gascrit} and \ref{lem:lyap}, that GAS and Lyapunov stability of the original system are preserved in the ``padded'' system.

Putting the two parts together proves Theorem~\ref{thm:transfer}. Given a homogeneous polynomial vector field $F$ on $\RR^{n_*}$ of odd degree $d_*=2k_*+1$, from Theorem~\ref{thm:cubicization} we obtain a homogeneous polynomial vector field $H$ of degree three on $\RR^m$, where $m=\binom{n_*+k_*-1}{k_*}$ is the number $N$ in Theorem~\ref{thm:transfer}. Then for any $n \geq N$ and $j \in \NN$, the padding procedure described above converts $H$ into a homogeneous polynomial vector field $\widehat H$ of degree $d=3+2j$ on $\RR^n$. Both of these parts preserve the properties of GAS and Lyapunov stability, and hence the computable map from the code corresponding to $F: \RR^{n_{*}} \to \RR^{n_{*}}$ to the code corresponding to $\widehat{H}: \RR^{n} \to \RR^{n}$ is precisely a many-one reduction from $\HGAS(n_*,d_*)$ to $\HGAS(n,d)$, and from $\HLS(n_*,d_*)$ to $\HLS(n,d)$. Since many-one reductions preserve many-one hardness and undecidability (Section~\ref{subsec:computability}), and since $\HGAS(n,d)$ is recursively enumerable by Proposition~\ref{prop:re}, this yields Theorem~\ref{thm:transfer}.

\subsection{Overview of Part II}\label{subsec:overview2}

Denote by $K\subseteq\NN$ the \textit{halting set} of a universal Turing machine \ref{C:K}, which is the classical example of a many-one complete (and hence undecidable) set. Theorem~\ref{thm:fixedpair} is proven by constructing a computable map $e\mapsto F_e$ which sends every $e\in\NN$ to a corresponding rational-coefficient homogeneous polynomial vector field $F_e : \RR^{n_{*}} \to \RR^{n_{*}}$ of fixed dimension $n_{*}$ and odd degree $d_* \geq 3$ such that
\begin{equation}\label{eq:overview2}
e\in K\iff\text{the origin is GAS for }\dot x=F_e(x).
\end{equation}
Such a mapping is a many-one reduction from $K$ to $\HGAS(n_*,d_*)$, so since every recursively enumerable set many-one reduces to $K$, this shows that $\HGAS(n_*,d_*)$ is many-one hard. Combined with Proposition~\ref{prop:re}, we conclude that $\HGAS(n_*,d_*)$ is many-one complete. This equivalence \eqref{eq:overview2} is concluded from the following chain of equivalences, in which every object is computable from $e$ (each section of Part II proves one equivalence in the chain):

\begin{align*}
e\in K&\iff\text{an explicit function $\mathfrak D_e$ has a nonzero constant Fourier coefficient}&&\text{(Section~\ref{sec:arith})}\\
&\iff\text{no trajectory of a polynomial system satisfies certain conditions}&&\text{(Section~\ref{sec:orbits})}\\
&\iff\text{an explicit quantity $\beta_e$ of the same kind as in \eqref{eq:overview_beta} is negative}&&\text{(Section~\ref{sec:compact})}\\
&\iff\text{the origin is GAS for $\dot x=F_e(x)$}&&\text{(Section~\ref{sec:single})}.
\end{align*}

\subsubsection{From the halting problem to a Fourier coefficient on a torus (Section~\ref{sec:arith})}

In Section~\ref{sec:arith} we reduce the question of whether $e\in K$ to the question of whether an explicit function $\mathfrak D_e$ on a torus has a nonzero constant Fourier coefficient. This will be done in two steps, namely, by first reducing the halting problem to the solvability of a Diophantine equation and then reducing this to the particular constant Fourier coefficient being nonzero.

Going from the halting problem to the solvability of a Diophantine equation is a classical result, and in fact this reduction is precisely the negative solution of Hilbert's tenth problem by Matiyasevich~\cite{Matiyasevich70}, who built on the work of Davis, Putnam and Robinson~\cite{DPR61}. The so-called ``MRDP theorem'' (restated in Theorem~\ref{thm:mrdp}) proves that for every recursively enumerable set $S\subseteq\NN$ (in particular, we may take $S=K$) there is a polynomial $P(t,y_1,\dots,y_k)$ with integer coefficients such that $e\in S$ if and only if $P(e,y)=0$ for some $y\in\NN^k$. It is necessary for our purposes to rewrite the statement ``$P(e,y)=0$ for some $y \in \NN^{k}$'' in a particular form, which we do in Lemma~\ref{lem:template}. Namely, we show there exist fixed integers $q,r \geq 1$, integer matrices $A=(A_{lj}), B=(B_{lj}) \in \ZZ^{r\times q}$, and $c_0,c_1\in\ZZ^r$, such that with $c(e)=c_0+e\,c_1$, for every $e \in \NN$ the statement above is equivalent to the system
\begin{equation}\label{eq:overview_template}
\sum_{j=1}^q\bigl(A_{lj}n_j^2+B_{lj}n_j\bigr)=c_l(e),\qquad l=1,\dots,r,
\end{equation}
having a solution $n\in\ZZ^q$. The fact that $e$ appears only on the right-hand side is what will eventually allow us to keep the dimension and degree of $F_e$ independent of $e$. It is worth noting that we never need the actual numerical values of $A,B,c_0,c_1$.

Unsurprisingly, going from the above statement to a nonzero Fourier coefficient requires some Fourier analysis. With $q,r$ as in \eqref{eq:overview_template}, we will be working on the $(q+r)$-torus $\TT^{q+r} = (\RR / 2\pi \ZZ)^{q+r}$ (a product of $q+r$ circles), whose points are represented as tuples $\varphi = (\varphi_{1}, \ldots, \varphi_{q+r})$. Roughly speaking, each element $\varphi_{l}$ of the tuple $\varphi$ is going to be ``responsible'' for a separate condition related to \eqref{eq:overview_template}; we need the $r$ individual equations to be satisfied, and we will also need $q$ additional conditions on potential solutions $n \in \ZZ^{q}$ of this system of equations (one condition for each $n_{j}$).

For $e \in \NN$, consider a (not yet defined) function $\mathfrak{D}_{e}: \TT^{q+r} \to \CC$. Suppose we want to define such a function to ``encode'' a solution to the system \eqref{eq:overview_template}, if one exists. Assuming $\mathfrak{D}_{e}$ is a ``nice enough'' function, it can be written as a Fourier series
\begin{equation}\label{eq:detector_fourier_series}
    \mathfrak{D}_{e}(\varphi) = \sum_{k \in \ZZ^{q+r}} (\hat{\mathfrak{D}}_{e})_{k} e^{ik\cdot\varphi} := \sum_{k \in \ZZ^{q+r}} (\hat{\mathfrak{D}}_{e})_{k} \prod_{l=1}^{q+r} \exp(ik_{l}\varphi_{l}),
\end{equation}
where each $(\hat{\mathfrak{D}}_{e})_{k} \in \CC$ and the ``constant Fourier coefficient'' of $\mathfrak{D}_{e}$ is the value $(\hat{\mathfrak{D}}_{e})_{(0, \ldots, 0)}$. In \eqref{eq:detector}, the function $\mathfrak{D}_{e}$ is defined in such a way that in addition to \eqref{eq:detector_fourier_series}, we also have (for some positive constants $w_{n,s}$, where $(n,s) \in \ZZ^{q+q}$, with $\sum_{n,s} w_{n,s} < \infty$) that $\mathfrak{D}_{e}(\varphi)$ is equal to
\begin{equation}\label{eq:detector_convenient_series}
    \sum_{(n,s) \in \ZZ^{q+q}} w_{n,s} \left( \prod_{j=1}^{q} \exp \bigg( 4i[n_{j}^{2} - s_{j}] \varphi_{j} \bigg) \ \cdot \ \prod_{l=1}^{r}  \exp \bigg( 4i \Big[ \sum_{j=1}^{q}(A_{lj}s_{j} + B_{lj}n_{j}) - c_{l}(e) \Big] \varphi_{q+l} \bigg)  \right).
\end{equation}
The factor of $4$ in the exponents appears for technical reasons and plays no role here. As an aside, note that $\frac{1}{2\pi} \int_{0}^{2\pi} e^{im \xi} d \xi$ equals $1$ if $m=0$ and equals $0$ otherwise. Hence, if we integrate $\mathfrak{D}_{e}(\varphi)$ term-by-term over $\TT^{q+r}$ and divide by $(2\pi)^{q+r}$, from \eqref{eq:detector_fourier_series} and \eqref{eq:detector_convenient_series} we obtain
\begin{equation}
    \frac{1}{(2\pi)^{q+r}} \int_{\TT^{q+r}} \mathfrak{D}_{e}(\varphi) \ d\varphi_{1} \cdots d\varphi_{q+r} = (\hat{\mathfrak{D}}_{e})_{(0, \ldots, 0)} = \sum_{\substack{n \in \ZZ^{q} \text{ solving } \eqref{eq:overview_template}}} w_{(n_{1}, \ldots, n_{q}),(n_{1}^{2}, \ldots, n_{q}^{2})}.
\end{equation}
The last equality is true because all the terms in the exponentials of \eqref{eq:detector_convenient_series} must be equal to zero (these are the ``$m$'' in our convenient integration fact stated above); thus any corresponding ``index'' $(n,s) \in \ZZ^{q+q}$ must satisfy $n_{j}^{2} = s_{j}$ for all $j$ and $\sum_{j=1}^{q} (A_{lj}s_{j} + B_{lj}n_{j}) = c_{l}(e)$ for all $l$, which together are equivalent to $n=(n_{1}, \ldots, n_{q})$ solving \eqref{eq:overview_template}. Therefore, since all $w_{n,s}$ are positive real numbers, the constant Fourier coefficient $(\hat{\mathfrak{D}}_{e})_{(0, \ldots, 0)}$ of $\mathfrak D_e$ is nonzero if and only if there exists some $n \in \ZZ^{q}$ solving the system \eqref{eq:overview_template}  (Lemma~\ref{lem:constterm}), which is true if and only if $e \in K$.

\subsubsection{From a Fourier coefficient to trajectories of a polynomial system (Section~\ref{sec:orbits})}

In Section~\ref{sec:orbits} we relate the constant Fourier coefficient of $\mathfrak D_e$ (a function $\TT^{q+r} \to \CC$) to a statement about the trajectories of an explicit polynomial system $\dot{x} = X_{e}(x)$ with rational coefficients on $\RR^{n_{s}}$ (Section~\ref{sec:the_source_system}), for some explicitly chosen $n_{s} > q+r$ independent of $e$. We want to construct this system $\dot{x} = X_{e}(x)$ so that some of its trajectories $x(t)$ correspond to $\mathfrak{D}_{e}(\varphi(t))$, where $\varphi(t)$ are trajectories of some ``simple'' system on $\TT^{q+r}$, in such a way that we may easily recover the constant Fourier coefficient $(\hat{\mathfrak{D}}_{e})_{(0, \ldots, 0)}$ from $\mathfrak{D}_{e}(\varphi(t))$. More precisely, the trajectories of this ``simple'' system are just linear functions of the form $\varphi(t) = \varphi(0)+ \omega t$ for fixed $\omega \in \RR^{q+r}$, and the correspondence between specific trajectories $x(t)$ and $\mathfrak{D}_{e}(\varphi(t))$ is
\begin{equation}\label{eq:D_e_obs}
    \mathfrak{D}^{\mathrm{obs}}_e(x(t)) = R_0\, \mathfrak{D}_{e}(\varphi(t)) = R_0\, \mathfrak{D}_{e}(\varphi(0) + \omega t),
\end{equation}
where $\mathfrak{D}^{\mathrm{obs}}_e : \RR^{n_{s}} \to \CC$ is a polynomial function (defined in \eqref{eq:detobs}) that can be directly obtained from $\mathfrak{D}_{e}$, and $R_{0} \in \CC$ is a nonzero constant. We emphasize that \eqref{eq:D_e_obs} does not hold for \textit{all} possible trajectories $x(t)$ of $\dot{x} = X_{e}(x)$, but only for some subset of the trajectories.

Once we establish this correspondence, the following convenient fact about Fourier series allows us to ``extract'' $(\hat{\mathfrak D}_e)_{(0,\ldots,0)}$ from the right-hand side of \eqref{eq:D_e_obs}. With $\varphi = \varphi(t)$ in \eqref{eq:detector_fourier_series}, each term of this Fourier series becomes $(\hat{\mathfrak D}_e)_k\,e^{ik\cdot\varphi(0)}e^{i(k\cdot\omega)t}$, and if $k \cdot \omega \neq 0$ then $\lim_{T \to \infty} \frac{1}{T} \int_{0}^{T} e^{i(k\cdot\omega)t} \ dt = 0$ (otherwise, for $k \cdot \omega = 0$ this limit equals $1$). Thus if we can guarantee that $k \cdot \omega = 0$ if and only if $k = (0, \ldots, 0)$ (by, for example, taking $\omega_{l} = \pm \sqrt{p_{l}}$ for distinct primes $p_{1}, \ldots, p_{q+r}$; see Lemma~\ref{lem:radicals} and Corollary~\ref{cor:irrational}), then, since the Fourier series of $\mathfrak D_e$ converges
absolutely (Lemma~\ref{lem:constterm}), we may integrate term-by-term (Lemma~\ref{lem:avg}) to obtain
\begin{equation}\label{eq:overview_time_average}
\lim_{T \to \infty} \frac{1}{T} \int_{0}^{T} \mathfrak{D}^{\mathrm{obs}}_e(x(t)) \ dt = \lim_{T\to\infty}\frac1T\int_0^T R_{0} \, \mathfrak D_e(\varphi(0)+\omega t)\,dt = R_{0} \, (\hat{\mathfrak D}_e)_{(0,\ldots,0)}
\end{equation}
for any trajectory $x(t)$ satisfying \eqref{eq:D_e_obs}. This is precisely how we relate $(\hat{\mathfrak D}_e)_{(0,\ldots,0)}$ to the trajectories of this system $\dot{x} = X_{e}(x)$. There are two main questions left to answer about this construction: how exactly is the system $\dot{x} = X_{e}(x)$ constructed, and what is the subset of trajectories $x(t)$ for which \eqref{eq:D_e_obs} holds?

The system $\dot{x} = X_{e}(x)$ is constructed by working backwards from \eqref{eq:D_e_obs}. Since we want to somehow relate $x(t)$ to $\mathfrak{D}_{e}(\varphi(t))$, this system may be thought of as a ``circuit'' whose job is to compute the actual value of $\mathfrak{D}_{e}(\varphi(t))$ in real time. The formal definition of $\mathfrak{D}_{e}$ in \eqref{eq:detector} is as a product of various factors, most of which are not polynomials (many of these factors are infinite sums known as ``theta functions''). In the (very high dimensional) system $\dot{x} = X_{e}(x)$, separate components of $x$ are allocated towards computing and storing the factors of $\mathfrak{D}_{e}$, the value of $\varphi(t) \in \TT^{q+r}$, and any other miscellaneous functions or constants needed to write down $\mathfrak{D}_{e}(\varphi(t))$. Then, the role of the function $\mathfrak{D}^{\mathrm{obs}}_{e}$ is to ``assemble'' these factors (evaluated at $\varphi(t)$) to obtain $\mathfrak{D}_{e}(\varphi(t))$, up to some constant $R_{0}$. Recalling that $X_{e} : \RR^{n_{s}} \to \RR^{n_{s}}$ is also required to be a polynomial function for our purposes, this whole scheme is possible because the factors of $\mathfrak{D}_{e}$ and all auxiliary quantities that need to be computed (including the irrational values $\omega_{1}, \ldots, \omega_{q+r}$, which cannot appear as coefficients of $X_{e}$; see Lemma~\ref{lem:controller}) can be obtained using polynomial differential equations with rational coefficients (Sections~\ref{sec:modular} and~\ref{sec:jacobi}; see Figure~\ref{fig:circuit} on page~\pageref{fig:circuit}).

From the circuit analogy, it seems reasonable that \eqref{eq:D_e_obs} would only hold for some trajectories $x(t)$ of $\dot{x} = X_{e}(x)$ and not all of them. Since different trajectories $x(t)$ (which we obtain by initializing the system at different points $x(0) \in \RR^{n_{s}}$) correspond to different computations on the same circuit, only some subset of these actually perform the intended computations; for this reason, we call such $x(t)$ the \textit{intended trajectories}. The way that the intended trajectories are singled out from all of the trajectories is by Theorem~\ref{thm:selector}, which states that trajectories which are \textit{recurrent} (they keep returning close to where they started), do not ``grow too fast'', and pass a finite number of ``tests'' behave like intended trajectories. These tests are encoded into a single polynomial function $h_{e}: \RR^{n_{s}} \to \RR^{n_{h}}$ in such a way that all of the tests are ``passed'' if
\begin{equation}\label{eq:avg_of_h_e_is_zero}
    \lim_{T \to \infty} \frac{1}{T} \int_{0}^{T} h_{e}(x(t)) \ dt = 0,
\end{equation}
together with an additional technical condition which we do not worry about here. We will also include the real and imaginary parts of $\mathfrak{D}^{\mathrm{obs}}_{e}(\cdot)$ among the components of $h_{e}$. This way, \eqref{eq:avg_of_h_e_is_zero} holding for a recurrent trajectory $x(t)$ which does not ``grow too fast'' implies that (a) all the tests have passed and (b) $(\hat{\mathfrak D}_e)_{(0,\ldots,0)} = 0$ by Theorem~\ref{thm:selector} and \eqref{eq:overview_time_average}, and therefore $e \notin K$ by the work of Section~\ref{sec:arith}. Conversely, if \eqref{eq:avg_of_h_e_is_zero} fails for every recurrent trajectory which does not grow too fast, it fails in particular for the intended trajectories, hence $(\hat{\mathfrak D}_e)_{(0,\ldots,0)} \neq 0$ (Lemma~\ref{lem:graph}) and so $e \in K$.

\subsubsection{From trajectories to the sign of a quantity $\beta_e$ (Section~\ref{sec:compact})}

In Section~\ref{sec:orbits} we concluded with a statement about individual trajectories of $\dot x=X_e(x)$ on the (non-compact) space $\RR^{n_s}$. In Section~\ref{sec:compact} we convert this into a statement about a new polynomial system on a compact set together with a polynomial ``reward'' function $\cR_{e}$ on this set, from which we define a quantity $\beta_e$ (the same kind of quantity as in Lemma~\ref{lem:gascrit}) that is negative if $e\in K$ and positive if $e\notin K$ (Proposition~\ref{prop:signgap}). This new polynomial system is obtained by adding to $\dot{x} = X_{e}(x)$ a new variable $z \in \RR^{n_h}$ satisfying $\dot{z} = \kappa h_{e}(x)$ for a small constant $\kappa>0$ (chosen in \eqref{eq:kappaeps}), where $h_{e}$ is from \eqref{eq:avg_of_h_e_is_zero}. The purpose of the $z$ variable is to keep a ``running total'' of $h_{e}$ in such a way that if the limit in \eqref{eq:avg_of_h_e_is_zero} exists but is nonzero, then $z(t)$ tends to infinity as $t \to \infty$.

To end up with a system on a compact set (in contrast to the system $(\dot{x}, \dot{z}) = (X_{e}(x), \kappa h_{e}(x))$ on the non-compact set $\RR^{n_s} \times \RR^{n_h}$) we map $\RR^{n_s}$ onto the ``upper hemisphere'' of the sphere $S^{n_{s}} \subseteq \RR^{n_s+1}$, with infinity corresponding to the ``equator'' (Lemma~\ref{lem:central}), and $\RR^{n_h}$ onto the sphere $S^{n_{h}} \subseteq \RR^{n_h+1}$ minus its ``south pole'', which corresponds to infinity (see Figure~\ref{fig:compactification} on page~\pageref{fig:compactification}). After a change of time (which is necessary for technical reasons), under these mappings the system  $(\dot{x}, \dot{z}) = (X_{e}(x), \kappa h_{e}(x))$ becomes a system on the compact set $S^{n_{s}} \times S^{n_{h}}$ (namely, \eqref{eq:compactsource}) with coordinates $(y,w) \in S^{n_{s}}$ and $(u,u_{0}) \in S^{n_{h}}$. The reward function $\cR_e : S^{n_{s}} \times S^{n_{h}} \to \RR$ is then defined to be negative both on the equator of $S^{n_{s}}$ and on the south pole of $S^{n_{h}}$, and positive on a region of $S^{n_{s}} \times S^{n_{h}}$ corresponding to the variable $x \in \RR^{n_s}$ being bounded and the variable $z \in \RR^{n_h}$ being close to $0$ (see \eqref{eq:reward}).

Denoting by $\Psi_{e}$ the flow of this new polynomial system on $S^{n_{s}} \times S^{n_{h}}$, we define $\beta_{e} = \beta(\Psi_{e}, \cR_{e})$ as in \eqref{eq:beta}; we can do this precisely because the new system is defined on a compact set. From this definition we see that $\beta_{e}$ depends only on the $\Psi_{e}$-invariant measures on $S^{n_{s}} \times S^{n_{h}}$, and an argument from ergodic theory shows that we may in fact only consider such measures which are \textit{ergodic} (by \ref{E:decomp}). In turn, such ergodic invariant measures on $S^{n_{s}} \times S^{n_{h}}$ are either entirely supported on $\{\text{equator of } S^{n_s} \} \times S^{n_{h}}$, or they are supported (almost everywhere) on trajectories whose $(y(t),w(t))$ component corresponds to a trajectory $x(t)$ of $\dot{x} = X_{e}(x)$ which is recurrent and does not grow too fast (indeed, this is exactly why these two conditions appear in Section~\ref{sec:orbits}; see the proof of Lemma~\ref{lem:selectedorbit}).

Now, suppose $e \in K$. Ergodic measures of the first type described above are supported on $\{\text{equator of } S^{n_s} \} \times S^{n_{h}}$, where $\cR_e < 0$. For an ergodic measure $\mu$ of the second type, consider a trajectory $(y(t), w(t), u(t), u_{0}(t))$ starting at a ``typical'' point where $\mu$ is supported, so that the corresponding $x(t)$ is recurrent and does not grow too fast. Then the limit in \eqref{eq:avg_of_h_e_is_zero} exists (by \ref{E:birkhoff}) and since $e \in K$, \eqref{eq:avg_of_h_e_is_zero} does not hold (by Lemma~\ref{lem:selectedorbit} and Corollary~\ref{cor:yes}). Hence the corresponding $z(t)$ tends to infinity, and so $\mu$ ends up being concentrated at the south pole on the $S^{n_h}$-component of the full system (Lemma~\ref{lem:accumulator}), where $\cR_e < 0$ as well. Thus $\beta_e < 0$ in this case (Proposition~\ref{prop:signgap}). On the other hand, suppose $e \notin K$. Then along every trajectory of the new system whose corresponding $x(t)$ trajectory is intended and whose corresponding $z(t)$ trajectory starts close to $0$, we have that $x(t)$ stays bounded and $z(t)$ remains close to $0$ (Lemma~\ref{lem:bounds}). By construction of $\cR_e$, this gives $\beta_e > 0$ as desired (Proposition~\ref{prop:signgap}; see Figure~\ref{fig:signgap} on page~\pageref{fig:signgap}).

\subsubsection{From the sign of $\beta_e$ to a single $(n_{*},d_{*})$ (Section~\ref{sec:single})}\label{subsec:overview-single}

In Section~\ref{sec:single} we go from the flow $\Psi_{e}$ of the polynomial system on $S^{n_s} \times S^{n_h}$ and the reward function $\cR_e$ (which together define $\beta_e$) to a homogeneous polynomial vector field $F_e : \RR^{n_*} \to \RR^{n_*}$ of odd degree $d_* \geq 3$ whose origin is GAS exactly when $\beta_e<0$ (Theorem~\ref{thm:wrapper}). Recall from Section~\ref{subsec:overview-criteria} that a homogeneous polynomial vector field $P$ may be split into a part $Y_{P}$ tangent to a sphere (whose flow on that sphere is $\Phi_{Y_P}$) and a part $a_{P}(u) \cdot u$ normal to it (where the maximum of the integrals of $a_{P}$ over invariant measures is $\beta(\Phi_{Y_P}, a_{P})$ in \eqref{eq:beta}). The main idea of Section~\ref{sec:single} is to perform this procedure in reverse: if $Y$ is a homogeneous polynomial vector field of degree $E$ which is tangent to a sphere (that is, $v \cdot Y(v) = 0$ for all $v$) and $a$ is a homogeneous polynomial of degree $E-1$, then the vector field $F(v)=Y(v)+a(v)\,v$ is homogeneous of degree $E$ and its tangent part and radial function on the sphere are exactly $Y$ and $a$. Thus by Lemma~\ref{lem:gascrit}, the origin of $\dot{v} = F(v)$ is GAS exactly when $\beta(\Phi_Y, a) < 0$, where $Y$ is restricted to the appropriate sphere.

For this strategy to work, we need a flow on a single sphere rather than the flow $\Psi_{e}$ on the product $S^{n_s} \times S^{n_h}$ of two spheres. This issue is addressed by a neat geometric construction which embeds the product $S^{n_s} \times S^{n_h}$ into the sphere $\sqrt{3} \, S^{n_s+n_h+2} \subseteq \RR^{n_s + n_h + 3}$ of radius $\sqrt{3}$. The main idea is that by adding another coordinate $v_{0}$ to $S^{n_s} \times S^{n_h} \subseteq \RR^{n_s + 1} \times \RR^{n_h + 1}$, we may identify the set $Z =  \{ (v_{s}, v_{h}, v_{0}) \in S^{n_s} \times S^{n_h} \times \RR : v_{0} = \pm 1\}$ with two copies of $S^{n_s} \times S^{n_h}$ (one at $v_{0} = 1$ and another at $v_{0} = -1$). The set $Z$ is contained in $\sqrt{3} \, S^{n_s+n_h + 2}$, since $|(v_s, v_h, v_{0})|^{2} = |v_{s}|^2 + |v_{h}|^2 + |v_{0}|^2 = 1 + 1 + 1 = (\sqrt{3})^2$ for every point $(v_s, v_h, v_{0}) \in Z$. In Figure \ref{fig:wrapper} (page~\pageref{fig:wrapper}) we illustrate the analogous construction for a single circle in place of $S^{n_s} \times S^{n_h}$, where the larger sphere is $\sqrt{2} \, S^{2} \subseteq \RR^{3}$. The original flow $\Psi_{e}$ and reward function $\cR_e$ are then ``copied'' from $S^{n_s} \times S^{n_h}$ onto each of $S^{n_s} \times S^{n_h} \times \{1\}$ and $S^{n_s} \times S^{n_h} \times \{-1\}$ (see the computation preceding \eqref{eq:wrapperavg}), and elsewhere on $\sqrt{3} \, S^{n_s+n_h + 2}$ we extend $\Psi_e$ and $\cR_e$ to a flow $\widehat{\Psi}_e$ and a function $\widehat{\cR}_e$ in such a way that only trajectories on $Z$ may contribute to making $\beta(\widehat{\Psi}_e, \widehat{\cR}_e)$ nonnegative (see \eqref{eq:wrapperavg}). Lastly, using a ``homogenization'' trick we replace the vector field generating $\widehat{\Psi}_e$ and the function $\widehat{\cR}_e$ by homogeneous polynomials, of some odd degree $E \geq 3$ and of degree $E-1$ respectively, which agree with the original ones on $\sqrt{3} \, S^{n_s+n_h + 2}$ and hence do not change $\beta(\widehat{\Psi}_e, \widehat{\cR}_e)$. Combining them as in the previous paragraph yields a homogeneous polynomial vector field $F_e : \RR^{n_*} \to \RR^{n_*}$ of odd degree $d_* = E$, where $n_* = (n_s + 1) + (n_h + 1) + 1$, whose origin is GAS precisely when $\beta(\widehat{\Psi}_e, \widehat{\cR}_e) < 0$, which is true if and only if $\beta_{e} = \beta(\Psi_{e}, \cR_{e}) < 0$ (by \eqref{eq:wrapper} and \eqref{eq:wrapperavg}).

Finally, in Section~\ref{sec:fixed} we check that the correspondence $e\mapsto F_e$ is computable and verify that different values of $e$ change only the coefficients of $F_{e}$, so that $n_*$ and $d_*$ remain the same. Together with the work from the other sections in Part II, this completes the proof of Theorem~\ref{thm:fixedpair}. It is then easy to combine Theorem~\ref{thm:fixedpair} with Theorem~\ref{thm:transfer} to obtain Theorem~\ref{thm:main}, whose anticlimactically short proof concludes Section~\ref{sec:fixed}.

\addtocontents{toc}{\protect\setcounter{tocdepth}{1}}
\section{Preliminaries}\label{sec:prelim}

In this section, we recall the background that we use from computability theory (Section~\ref{subsec:computability}), real algebraic geometry (Section~\ref{subsec:tarski}), ordinary differential equations (Sections~\ref{subsec:ode}--\ref{subsec:homog}) and ergodic theory (Section~\ref{subsec:erg}). Throughout, we denote the natural numbers by $\NN=\{0,1,2,\dots\}$, the Euclidean norm by $|\cdot|$, and the sphere in $\RR^{n}$ by $S^{n-1}=\{x\in\RR^n:|x|=1\}$.

For vector fields we use the following letters: $f$ denotes an arbitrary polynomial vector field; $P$ a homogeneous polynomial vector field in the general criteria of Section~\ref{sec:angular}; $F$ a homogeneous polynomial vector field whose stability is in question, such as the input of Section~\ref{sec:lifting}, the output of Section~\ref{sec:single} and $F_e$; $G$ and $H$ the vector fields obtained from $F$ in Section~\ref{sec:lifting}; $Y_P$ (and similarly $Y_F$, $Y_H$, \dots) the angular part of $P$ (Section~\ref{subsec:polar}), and $Y$ a vector field tangent to a sphere in Section~\ref{sec:single}; and $X$, $X_e$ and $\widehat X_e$ the polynomial systems of Part~II. The flow of a vector field $f$ is denoted by $\Phi_f$ (Section~\ref{subsec:ode}).

\subsection{Computability}\label{subsec:computability}
We will use Turing machines as our preferred model of computation, in the standard form of \cite[Definition~3.3]{Sipser}. A Turing machine receives its input as a string of symbols on its tape, and halts when it enters its accept state or its reject state, as described in the paragraph after \cite[Definition~3.3]{Sipser}. We take binary strings as inputs, and we identify $\NN$ with the set of binary strings through the computable bijection which sends $x$ to the binary representation of $x+1$ with its leading digit $1$ removed. In this way, every subset of $\NN$ is a language, that is, a set of strings, in the sense of \cite[Section~0.2]{Sipser}. We code finite objects such as integers, rational numbers, finite lists of rational numbers, polynomials with rational coefficients, formulas and Turing machines by natural numbers in a fixed standard way, as in \cite[Section~3.3]{Sipser}, where it is also noted that the choice of a reasonable coding does not matter, because a Turing machine can translate one such coding into another. The elementary operations on them (arithmetic, composition and differentiation of polynomials, evaluation of a polynomial at a rational point, and so on) are computable. A set $S\subseteq\NN$ is \emph{decidable} if there is a Turing machine which halts on every input $x$, accepting $x$ if $x\in S$ and rejecting $x$ if $x\notin S$ \cite[Definition~3.6]{Sipser}, and \emph{recursively enumerable} (r.e.) if there is a Turing machine which halts on input $x$ if and only if $x\in S$. A set is r.e.\ if and only if it is \emph{Turing-recognizable}, that is, the set of inputs accepted by some Turing machine \cite[Definition~3.5]{Sipser}; as noted in a footnote to that definition, Turing-recognizable sets are called recursively enumerable in other textbooks. Indeed, a Turing machine which halts exactly on the elements of $S$ accepts exactly the elements of $S$ once it is modified to accept whenever it would reject. Conversely, a Turing machine which accepts exactly the elements of $S$ halts exactly on the elements of $S$ once it is modified to run forever whenever it would reject, as in \cite[Example~5.24]{Sipser}. A function $r:\NN\to\NN$ is \emph{total computable} if some Turing machine halts on every input $x$ with just $r(x)$ on its tape \cite[Definition~5.17]{Sipser}. For $A,B\subseteq\NN$ we write $A\lem B$, and say that $A$ \emph{many-one reduces} to $B$, if there is a total computable map $r$ with $x\in A\iff r(x)\in B$ for all $x$; in \cite[Definition~5.20]{Sipser} this is called \emph{mapping reducibility}, and a footnote at the beginning of Section~5.3 there notes the name many-one reducibility. A set $B$ is \emph{many-one hard} if $A\lem B$ for every r.e.\ set $A$, and \emph{many-one complete} if it is r.e.\ and many-one hard.

We use the following standard facts.
\begin{enumerate}[label=(C\arabic*)]
\item\label{C:K} For $m\in\NN$, let $M_m$ be the Turing machine with code $m$ if $m$ is the code of a Turing machine, and otherwise a fixed Turing machine which halts on no input; thus every Turing machine is $M_m$ for some $m$. Fix a computable bijection $\langle\cdot,\cdot\rangle:\NN^2\to\NN$. There is a Turing machine $\mathsf U$ which, on input $\langle m,y\rangle$, simulates $M_m$ on input $y$, so that $\mathsf U$ halts on input $\langle m,y\rangle$ if and only if $M_m$ halts on input $y$. This is the machine $U$ of \cite[Section~4.2]{Sipser}, modified to run forever when $m$ is not the code of a Turing machine. We call $\mathsf U$ a \emph{universal Turing machine}. Its halting set $K=\{e\in\NN:\mathsf U\text{ halts on input }e\}$ is many-one complete and undecidable. Indeed, $K$ is r.e.\ by definition. If $A$ is r.e., then $A$ is the set of inputs on which some $M_m$ halts, so the total computable map $y\mapsto\langle m,y\rangle$ is a many-one reduction from $A$ to $K$. Finally, $K$ is a coding of the halting problem, which is undecidable \cite[Theorem~5.1]{Sipser}. In our notation, the diagonal argument of \cite[proof of Theorem~4.11]{Sipser} reads as follows. If some Turing machine decided $K$, then there would be a Turing machine $D$ which, on input $y$, halts if $\langle y,y\rangle\notin K$ and runs forever if $\langle y,y\rangle\in K$. Writing $D=M_d$, we would get that $D$ halts on input $d$ if and only if $\langle d,d\rangle\notin K$, that is, by the definition of $K$, if and only if $M_d=D$ does not halt on input $d$, which is a contradiction.
\item\label{C:post} A set $S$ is decidable if and only if both $S$ and $\NN\setminus S$ are r.e.\ \cite[Theorem~4.22]{Sipser}, since the r.e.\ sets are exactly the Turing-recognizable sets.
\item\label{C:trans} The relation $\lem$ is transitive \cite[Exercise~5.6 and its solution]{Sipser}. If $A\lem B$ and $B$ is decidable (respectively r.e.), then $A$ is decidable (respectively r.e.) \cite[Theorems~5.22 and~5.28]{Sipser}. If $A\lem B$ and $A$ is many-one hard, then $B$ is many-one hard, by transitivity.
\end{enumerate}
Consequently, if $S$ is many-one hard, then $\NN\setminus S$ is not r.e., and $S$ is undecidable. Indeed, $K\lem S$ because $K$ is r.e., and the same map reduces $\NN\setminus K$ to $\NN\setminus S$. If $\NN\setminus S$ were r.e., then $\NN\setminus K$ would be r.e.\ by \ref{C:trans}; since $K$ is r.e., $K$ would then be decidable by \ref{C:post}, contradicting \ref{C:K}. Finally, if $S$ were decidable, then $\NN\setminus S$ would be r.e.\ by \ref{C:post}.

\subsection{Real quantifier elimination}\label{subsec:tarski}
We define a \emph{formula} to be a statement built from atoms $P(z_1,\dots,z_k)=0$ and $P(z_1,\dots,z_k)>0$, with $P\in\QQ[z_1,\dots,z_k]$, by means of $\neg,\wedge,\vee$ and the quantifiers $\forall z_i$, $\exists z_i$, where all variables range over $\RR$. A \emph{sentence} is a formula without free variables. We use the abbreviations $P\ge0$, $P<0$, $\Rightarrow$ and finite conjunctions in the usual way.

\begin{theorem}[Tarski]\label{thm:tarski}
There is a Turing machine which, given a sentence, halts and outputs $1$ if the sentence is true and $0$ if it is false.
\end{theorem}

This combines Theorems~31 and~37 of \cite{Tarski}; see also \cite[Theorem~2.77 and Algorithm~14.3]{BPR}. Algorithm~14.3 of \cite{BPR} decides sentences in prenex form, and every sentence is equivalent to a sentence in prenex form \cite[Section~2.3, \pg{59}]{BPR}, which is obtained by the usual computable rewriting. We will apply these results with integer coefficients; a sentence with rational coefficients is equivalent to the one obtained by multiplying every atom by a positive common denominator.

\subsection{Ordinary differential equations}\label{subsec:ode}
All vector fields in this paper are polynomial maps $\RR^{n}\to\RR^{n}$, or restrictions of such maps to compact sets, hence locally Lipschitz. For such a vector field $f: \RR^{n} \to \RR^{n}$ and every $x_0\in\RR^n$ there is a unique maximal solution of $\dot x=f(x)$, $x(0)=x_0$, defined on an open interval $(t_-,t_+)\ni0$, and if $t_+<\infty$ then $|x(t)|$ is unbounded on $[0,t_+)$ \cite[Theorems~2.2 and~2.13 and Corollary~2.16]{Teschl}; solutions depend continuously on initial conditions \cite[Theorem~2.8]{Teschl}. We use Gronwall's inequality in the following standard differential form: if $w$ is $C^1$ on an interval $[0,T]$ and $w'\le-\lambda w$ there for a constant $\lambda\in\RR$, then $w(t)\le w(0)e^{-\lambda t}$ for $0\le t\le T$. A set $\cK\subseteq\RR^n$ is \emph{invariant} under $f$ if every maximal solution starting in $\cK$ stays in $\cK$ on its whole interval of existence, and \emph{forward invariant} if it stays in $\cK$ for all $t\ge0$ in this interval. If $\cK$ is compact and invariant then every solution starting in $\cK$ is bounded, so by the facts above, applied to $f$ and to $-f$, it exists for all $t\in\RR$. The map sending $(t,x)\in\RR\times\cK$ to the value at time $t$ of the solution starting at $x$ is then continuous, by continuous dependence on initial conditions and because this value changes by at most $|t-s|\max_\cK|f|$ when $t$ is replaced by $s$. Thus the restriction of the flow of $f$ to $\cK$ is a continuous flow defined for all $t\in\RR$. We write $\Phi_f$ for the flow of $f$ and $\Phi_f^t$ for its time-$t$ map, so that $\Phi_f^t(x_0)=x(t)$ for the maximal solution $x$ with $x(0)=x_0$ and every $t$ in its interval of existence. If $\cK$ is compact and invariant under $f$, we write $\Phi_f|_\cK$ for the continuous flow on $\cK$ just described. Alternatively, we may restrict the domain of $f$ to $\cK$; the flow $\Phi_f$ of the restricted vector field is then this flow on $\cK$. Whenever a flow is regarded as a flow on a compact set, for instance in Section~\ref{subsec:erg}, either the restriction is written explicitly or the domain of the vector field has been restricted to that set, and we say so where this is done.

\subsection{Stability notions and decision problems}\label{subsec:stability}
Throughout, a \emph{polynomial vector field} (without the restriction that the coefficients must be rational numbers) is a map $f=(f_1,\dots,f_n):\RR^n\to\RR^n$ with $f_i\in\RR[x_1,\dots,x_n]$, and we always assume $f(0)=0$, so that the origin is an equilibrium point of the system
\begin{equation}\label{eq:intro-ode}
\dot x=f(x),\qquad x\in\RR^n.
\end{equation}
By Section~\ref{subsec:ode}, every initial state $x(0)$ determines a unique maximal solution of \eqref{eq:intro-ode}. We use the following three stability notions for the origin.
\begin{itemize}
\item The origin is \emph{Lyapunov stable} if for every $\varepsilon>0$ there is $\delta>0$ such that every solution of \eqref{eq:intro-ode} with $|x(0)|<\delta$ exists for all $t\ge0$ and satisfies $|x(t)|<\varepsilon$ for all $t\ge0$.
\item The origin is \emph{locally asymptotically stable} if it is Lyapunov stable and there is $\delta_0>0$ such that every solution of \eqref{eq:intro-ode} with $|x(0)|<\delta_0$ satisfies $x(t)\to0$ as $t\to\infty$.
\item The origin is \emph{globally asymptotically stable} (GAS) if it is Lyapunov stable, and every solution of \eqref{eq:intro-ode} exists for all $t\ge0$ and satisfies $x(t)\to0$ as $t\to\infty$.
\end{itemize}

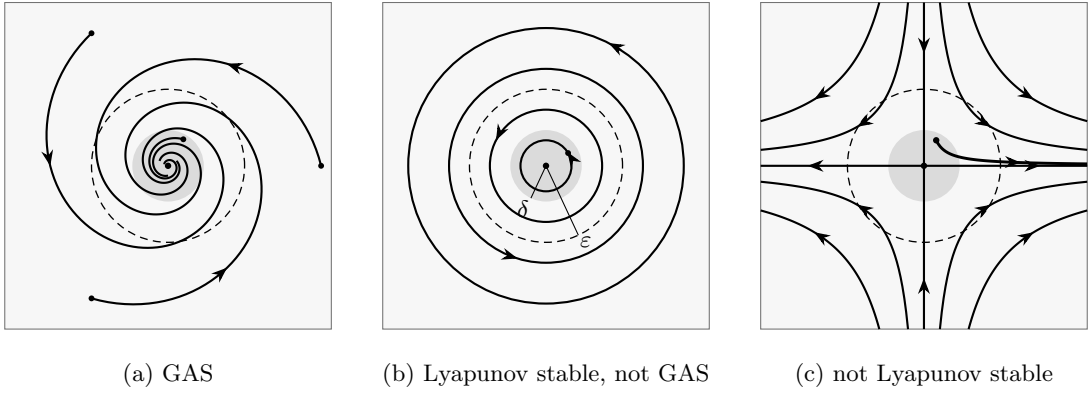
\begin{figure}[t]
\centering
\begin{tikzpicture}[
  >={Stealth[length=2.2mm]},
  lab/.style={font=\footnotesize, inner sep=1pt},
  traj/.style={line width=0.8pt, line join=round},
  arr/.style={-{Stealth[length=2.1mm]}, line width=0.8pt},
  delta/.style={fill=black!14},
  eps/.style={line width=0.5pt, densely dashed},
]
  \foreach \ox in {0, 5.0, 10.0} {
    \begin{scope}[xshift=\ox cm, scale=1.35]
      \fill[black!3] (-1.6,-1.6) rectangle (1.6,1.6);
      \draw[line width=0.4pt, black!55] (-1.6,-1.6) rectangle (1.6,1.6);
      \fill[delta] (0,0) circle (0.35);
      \draw[eps] (0,0) circle (0.75);
    \end{scope}
  }
  \begin{scope}[scale=1.35]
    \foreach \p in {0, 120, 240} {
      \draw[traj, domain=0:11, samples=200, variable=\t]
        plot ({1.5*exp(-\t/4)*cos(\t r + \p)}, {1.5*exp(-\t/4)*sin(\t r + \p)});
      \draw[arr] ({1.5*exp(-0.95/4)*cos(0.95 r + \p)}, {1.5*exp(-0.95/4)*sin(0.95 r + \p)}) -- ({1.5*exp(-1.05/4)*cos(1.05 r + \p)}, {1.5*exp(-1.05/4)*sin(1.05 r + \p)});
    }
    \draw[traj, domain=0:8, samples=120, variable=\t] plot ({0.3*exp(-\t/4)*cos(\t r + 60)}, {0.3*exp(-\t/4)*sin(\t r + 60)});
    \foreach \p in {0, 120, 240} \fill (\p:1.5) circle (1.1pt/1.35);
    \fill (60:0.3) circle (1.1pt/1.35);
    \fill (0,0) circle (1.2pt/1.35);
    \node[lab, align=center] at (0,-2.05) {(a) GAS};
  \end{scope}
  \begin{scope}[xshift=5cm, scale=1.35]
    \foreach \r/\a in {0.25/30, 0.55/150, 0.95/250, 1.35/60} {
      \draw[traj] (0,0) circle (\r);
      \draw[arr] (\a-3:\r) arc[start angle=\a-3, end angle=\a+3, radius=\r];
    }
    \fill (0,0) circle (1.2pt/1.35);
    \fill (30:0.25) circle (1.1pt/1.35);
    \draw[line width=0.4pt] (0,0) -- (-115:0.35) node[lab, anchor=north east, inner sep=0.5pt] {$\delta$};
    \draw[line width=0.4pt] (0,0) -- (-65:0.75) node[lab, anchor=north west, inner sep=0.5pt] {$\varepsilon$};
    \node[lab, align=center] at (0,-2.05) {(b) Lyapunov stable, not GAS};
  \end{scope}
  \begin{scope}[xshift=10cm, scale=1.35]
    \begin{scope}
      \clip (-1.6,-1.6) rectangle (1.6,1.6);
      \foreach \c in {0.25, 0.7} {
        \foreach \sx/\sy in {1/1, -1/1, 1/-1, -1/-1} {
          \draw[traj, domain={\c/1.6}:1.6, samples=80, variable=\s] plot ({\sx*\s}, {\sy*\c/\s});
        }
      }
      \draw[traj] (-1.6,0) -- (1.6,0) (0,-1.6) -- (0,1.6);
      \draw[traj, line width=1.1pt, domain=0.12:1.6, samples=80, variable=\s] plot (\s, {0.03/\s});
    \end{scope}
    \fill (0.12,0.25) circle (1.2pt/1.35);
    \draw[arr, line width=1.1pt] (0.9,{0.03/0.9}) -- (0.95,{0.03/0.95});
    \foreach \c in {0.25, 0.7} {
      \foreach \sx/\sy in {1/1, -1/1, 1/-1, -1/-1} {
        \pgfmathsetmacro{\xa}{sqrt(\c)*1.25}
        \draw[arr] ({\sx*\xa}, {\sy*\c/\xa}) -- ({\sx*(\xa+0.03)}, {\sy*\c/(\xa+0.03)});
      }
    }
    \foreach \sx in {1,-1} { \draw[arr] (\sx*1.1,0) -- (\sx*1.14,0); \draw[arr] (0,\sx*1.14) -- (0,\sx*1.1); }
    \fill (0,0) circle (1.2pt/1.35);
    \node[lab, align=center] at (0,-2.05) {(c) not Lyapunov stable};
  \end{scope}
\end{tikzpicture}
\caption{The stability notions above, for three homogeneous polynomial vector fields of degree three on $\RR^2$. In
each panel the shaded disc is $\{|x|<\delta\}$ and the dashed circle is $\{|x|=\varepsilon\}$, as in the
definition of Lyapunov stability, and the arrows point forward in time. (a)
$f(x)=|x|^2\bigl(-\tfrac14x_1-x_2,\ x_1-\tfrac14x_2\bigr)$: every solution spirals into the origin, and the
solutions which start in the disc stay inside the circle, so the origin is GAS. (b)
$f(x)=|x|^2(-x_2,x_1)$: the solutions are circles around the origin, so the solutions which start in the
disc stay inside the circle, but no solution other than $x(t)=0$ approaches the origin; the origin is
Lyapunov stable but not GAS. (c) $f(x)=|x|^2(x_1,-x_2)$: the solutions move along the curves
$x_1x_2=\mathrm{const}$, and every solution which starts off the vertical axis, however close to the
origin (thick curve), leaves the circle and escapes to infinity (in fact, escapes to infinity in finite time); the origin is not Lyapunov
stable, and hence not GAS. Since these vector fields are homogeneous, each picture looks the same at
every scale, by the scaling identity \eqref{eq:scaling} below.}
\label{fig:gas}
\end{figure}

These are the notions called stable, asymptotically stable and globally asymptotically stable in \cite[Definition~4.1 and Section~4.1]{Khalil}. There, the existence of the solutions for all $t\ge0$ is assumed implicitly; since it does not always hold for polynomial vector fields, we include it in the definitions. Clearly GAS implies local asymptotic stability, which implies Lyapunov stability. Homogeneous vector fields of degree $d\ge1$ satisfy $f(0)=0$ automatically; this assumption matters for the vector fields which are not homogeneous, such as those of Corollary~\ref{cor:nonhomog}.

We use the codes of homogeneous polynomial vector fields described in the introduction, with the following conventions. A code is \emph{valid} if it encodes two integers $n,d\ge1$ and the rational coefficients of a polynomial vector field $f$ on $\RR^n$ each of whose coordinates, after duplicate monomials have been combined, is homogeneous of degree $d$ (a zero coordinate is regarded as homogeneous of every degree). Whether a code is valid is decidable, and the dimension $n$ and the degree $d$ of a valid code can be computed from it. We regard codes which are not valid as ``negative instances'' of every decision problem in this paper; thus $\HGAS$ and $\HLS$ consist of valid codes only. Besides the sets $\HGAS(n,d)$ and $\HLS(n,d)$ of the introduction, in which both the dimension and the degree are fixed, we write $\HGAS(\cdot,d)$ and $\HLS(\cdot,d)$ for the sets of codes in $\HGAS$, respectively $\HLS$, which have degree $d$ and any dimension.

\subsection{Homogeneous polynomial vector fields}\label{subsec:homog}
Let $f: \RR^{n} \to \RR^{n}$ be homogeneous of degree $d\ge1$, so that $f(\lambda x)=\lambda^df(x)$ for $\lambda>0$, and let $\Phi_f$ be its flow (Section~\ref{subsec:ode}), so that $t\mapsto\Phi_f^t(x)$ is the solution of \eqref{eq:intro-ode} with initial value $x$. Uniqueness of solutions gives the scaling identity
\begin{equation}\label{eq:scaling}
\Phi_f^t(\lambda x)=\lambda\,\Phi_f^{\lambda^{d-1}t}(x),\qquad\lambda>0,
\end{equation}
valid whenever either side is defined. Consequently local asymptotic stability and GAS are equivalent for homogeneous vector fields; we keep the global formulation because forward completeness is part of the angular criterion below. The scaling identity also makes Lyapunov stability a uniform statement.

\begin{lemma}[{\cite[Lemma~1.5]{Korda}}]\label{lem:scaling}
For a homogeneous polynomial vector field $f$, the origin is Lyapunov stable if and only if there is $C<\infty$ such that every solution of \eqref{eq:intro-ode} exists for all $t\ge0$ and satisfies $\|x(t)\|\le C\|x(0)\|$ for all $t\ge0$.
\end{lemma}

In \cite{Korda}, the lemma is stated for the Euclidean norm; since all norms on $\RR^n$ are equivalent, it holds for every norm $\|\cdot\|$.

\begin{remark}\label{rem:even}
No homogeneous polynomial vector field $f\ne0$ of even degree $d$ has an asymptotically stable origin. To see why, recall first that having an even degree implies $f(-x)=f(x)$. If $x(t)=\Phi_f^t(x_0)\to0$ for some $x_0\ne0$, then for every $T>0$ the path $y(s)=-x(T-s)$, $0\le s\le T$, satisfies $\dot y(s)=\dot x(T-s)=f(x(T-s))=f(-y(s))=f(y(s))$ and $y(0)=-x(T)$, so $\Phi_f^T(-x(T))=y(T)=-x_0$ by uniqueness. As $T\to\infty$ the initial points $-x(T)$ tend to zero while their images at time $T$ have norm $|x_0|$, contradicting Lyapunov stability. The zero vector field is not GAS either. This obstruction concerns asymptotic stability only: the homogeneous polynomial vector field $(-x_2x_3,\,x_1x_3,\,0)$ of degree two on $\RR^3$ rotates each plane $\{x_3=c\}$ with angular speed $c$ and preserves $|x|$, so its origin is Lyapunov stable.

In dimension one every homogeneous polynomial vector field is $\dot x=ax^d$. Its origin is GAS exactly when $d$ is odd and $a<0$, and Lyapunov stable exactly when $a=0$, or $d$ is odd and $a<0$. Hence the one-dimensional problems are decidable, uniformly in $d$.
\end{remark}

\subsection{Invariant measures and maximal averages}\label{subsec:erg}
In Section~\ref{sec:angular} we will express the property of GAS through the invariant probability measures of a flow on a compact set; in this subsection, we collect what we will use about such measures. We use the basic notions of measure theory, such as $\sigma$-algebras, Borel sets, measures, integrals and Fubini's theorem, without definition; the interested reader can find them in \cite[Chapters~1 and~2]{Folland}, where Fubini's theorem is Theorem~2.37. For the Borel probability measures on a compact metric space, and the weak-star topology on them, see \cite[Appendix~B, Section~B.5, pp.~422--423]{EW}. Let $\cK$ be a nonempty compact metric space and $\Phi:\RR\times\cK\to\cK$ a continuous flow; we write $\Phi^t=\Phi(t,\cdot)$ as a shorthand. A Borel probability measure $\mu$ on $\cK$ is \emph{invariant} if $\mu((\Phi^t)^{-1}(E))=\mu(E)$ for all Borel sets $E$ and all $t\in\RR$, and \emph{ergodic} if moreover $\mu(E)\in\{0,1\}$ for every Borel set $E$ with $\mu((\Phi^t)^{-1}(E)\,\triangle\,E)=0$ for all $t$; these are the definitions of \cite[Chapter~8, \S8.1, Definition~8.1 and \pg{232}]{EW} for the continuous action $(t,x)\mapsto\Phi^tx$ of the group $\RR$. We write $\cM(\Phi)$ for the set of invariant Borel probability measures. We use the following standard facts.
\begin{enumerate}[label=(E\arabic*)]
\item\label{E:compact} $\cM(\Phi)$ is nonempty, convex, and compact in the weak-star topology. It is nonempty by \cite[Chapter~8, Theorem~8.11, \pg{253}]{EW}. It is the intersection over $t\in\RR$ of the sets of $\Phi^t$-invariant measures, which are closed and convex \cite[Chapter~4, \pg{98}]{EW}, in the weak-star compact set of Borel probability measures on $\cK$ \cite[Appendix~B, Theorem~B.11, pp.~422--423]{EW}.
\item\label{E:decomp} (Ergodic decomposition.) For every $\mu\in\cM(\Phi)$ there are Borel probability measures $\mu_y$, defined for $y$ in a set of full $\mu$-measure and depending measurably on $y$, such that $\mu_y\in\cM(\Phi)$ is ergodic for $\mu$-almost every $y$ and $\int a\,d\mu=\int\bigl(\int a\,d\mu_y\bigr)d\mu(y)$ for every $a\in C(\cK)$ \cite[Chapter~8, Theorem~8.20, pp.~266--267, with Chapter~5, Theorem~5.14, \pg{135}]{EW}. In particular, if $\int a\,d\nu\le c$ for every ergodic $\nu\in\cM(\Phi)$, then $\int a\,d\mu\le c$ for every $\mu\in\cM(\Phi)$.
\item\label{E:birkhoff} (Birkhoff's theorem.) If $\mu$ is ergodic and $a\in L^1(\mu)$, then for $\mu$-almost every $x$, $\frac1T\int_0^Ta(\Phi^tx)\,dt\to\int a\,d\mu$ as $T\to+\infty$ and as $T\to-\infty$ \cite[Chapter~1, \S2, Theorem~1, pp.~11--12, and \pg{14}]{CFS}. (There the limits of $\frac1T\int_0^Ta(\Phi^{\pm t}x)\,dt$ as $T\to+\infty$ are shown to exist and to equal $\int a\,d\mu$ almost everywhere when $\mu$ is ergodic in the sense that every Borel set $E$ with $(\Phi^t)^{-1}(E)=E$ for all $t$ has measure $0$ or $1$; this is implied by our definition.)
\item\label{E:poincare} (Poincar\'e recurrence.) If $\mu\in\cM(\Phi)$, then $\mu$-almost every $x$ is forward and backward recurrent: there are $t_k\to+\infty$ and $s_k\to-\infty$ with $\Phi^{t_k}x\to x$ and $\Phi^{s_k}x\to x$. Indeed $\mu$ is invariant under the maps $\Phi^1$ and $\Phi^{-1}$, and $\cK$ has a countable basis of open sets, so by \cite[Theorem~1.2.4]{VO}, applied to $\Phi^1$ and to $\Phi^{-1}$, $\mu$-almost every $x$ has $\Phi^{n_j}x\to x$ and $\Phi^{-m_j}x\to x$ for some integers $n_j,m_j\to\infty$; compare the remark on flows in \cite[\S1.2.1]{VO}.
\item\label{E:deriv} If $\cK\subseteq\RR^n$ is a compact set invariant under a $C^1$ vector field $f$ and $\mu\in\cM(\Phi_f|_\cK)$, then $\int_\cK Dg(x)f(x)\,d\mu(x)=0$ for every $C^1$ function $g$ defined near $\cK$, since $\frac{d}{dt}\big|_{t=0}\int g\circ\Phi_f^t\,d\mu=0$.
\item\label{E:factor} If $\pi:\cK\to\cK'$ is continuous, $\Phi'$ is a continuous flow on $\cK'$ and $\pi\circ\Phi^t={\Phi'}^t\circ\pi$, then $\pi_*\mu\in\cM(\Phi')$ for $\mu\in\cM(\Phi)$, and $\pi_*\mu$ is ergodic if $\mu$ is.
\end{enumerate}

For $a\in C(\cK)$, put
\begin{equation}\label{eq:beta}
\beta(\Phi,a)=\max_{\mu\in\cM(\Phi)}\int_\cK a\,d\mu.
\end{equation}
By \ref{E:compact} and weak-star continuity of $\mu\mapsto\int a\,d\mu$, this maximum exists. The following lemma is a continuous-time form of a standard identity of ergodic optimization \cite[Proposition~2.2]{Jenkinson}; closely related statements for flows and semiflows are in~\cite{Schreiber98} and~\cite[Lemma~4.1]{GarayHofbauer}.

\begin{lemma}\label{lem:blocks}
Let $a\in C(\cK)$ and $M(T)=\max_{x\in\cK}\int_0^Ta(\Phi^tx)\,dt$ for $T>0$. Then
\begin{equation}\label{eq:blocks-limit}
\beta(\Phi,a)=\lim_{T\to\infty}\frac{M(T)}{T}=\inf_{T>0}\frac{M(T)}{T}.
\end{equation}
In particular, $\beta(\Phi,a)<0$ if and only if there are $T,c>0$ such that
\begin{equation}\label{eq:blocks}
\int_0^Ta(\Phi^tx)\,dt\le-c\qquad\text{for every }x\in\cK.
\end{equation}
\end{lemma}

\begin{proof}
The flow property gives $\int_0^{S+T}a(\Phi^tx)\,dt=\int_0^Sa(\Phi^tx)\,dt+\int_0^Ta(\Phi^t\Phi^Sx)\,dt$, so $M(S+T)\le M(S)+M(T)$. The maximum defining $M(T)$ exists because $x\mapsto\int_0^Ta(\Phi^tx)\,dt$ is continuous on the compact set $\cK$, the integrand being continuous on $[0,T]\times\cK$. Put $M(0)=0$. Changing $T$ by $h$ changes each of these integrals by at most $|h|\max|a|$, so $M$ is Lipschitz continuous. It is also subadditive on $[0,\infty)$, and $M(T)\ge-T\max|a|$, so by Fekete's lemma \cite[Chapter~16, Theorem~16.2.9 and (16.2.10), pp.~463--464, and the remark on \pg{465} for functions on $[0,\infty)$]{Kuczma} the limit $\lim_{T\to\infty}M(T)/T$ exists, equals $\inf_{T>0}M(T)/T$, and is finite. If $\mu\in\cM(\Phi)$ and $T>0$, then invariance and Fubini's theorem give $\int a\,d\mu=\frac1T\int\int_0^Ta(\Phi^tx)\,dt\,d\mu(x)\le M(T)/T$; hence $\beta(\Phi,a)\le\inf M(T)/T$. Conversely, choose $T_j\to\infty$ and $x_j\in\cK$ with $M(T_j)=\int_0^{T_j}a(\Phi^tx_j)\,dt$, and let $\mu_j$ be the probability measure $g\mapsto\frac1{T_j}\int_0^{T_j}g(\Phi^tx_j)\,dt$. By weak-star compactness a subsequence converges to a probability measure $\mu$. For $s\in\RR$ and $g\in C(\cK)$ we have $\bigl|\int g\circ\Phi^s\,d\mu_j-\int g\,d\mu_j\bigr|\le2|s|\max|g|/T_j\to0$, so $\int g\circ\Phi^s\,d\mu=\int g\,d\mu$ for all $g\in C(\cK)$, and hence $\mu\in\cM(\Phi)$ by \cite[Appendix~B, Theorem~B.11(2)]{EW}. Moreover $\int a\,d\mu=\lim_jM(T_j)/T_j=\lim_{T\to\infty}M(T)/T$, so $\beta(\Phi,a)\ge\lim_{T\to\infty}M(T)/T$. Together with $\beta(\Phi,a)\le\inf_{T>0}M(T)/T$ and Fekete's lemma, this proves \eqref{eq:blocks-limit}. If $\beta(\Phi,a)<0$, some $T$ has $M(T)<0$, which is \eqref{eq:blocks} with $c=-M(T)$. Conversely \eqref{eq:blocks} gives $M(T)/T\le-c/T<0$.
\end{proof}

\addtocontents{toc}{\protect\setcounter{tocdepth}{2}}
\section{Stability criteria and recursive enumerability}\label{sec:angular}

In this section we prove the two criteria of Section~\ref{subsec:overview-criteria}, which express the properties of GAS and Lyapunov stability of the origin for a homogeneous polynomial vector field through the ``angular flow'' on an ellipsoid (Lemmas~\ref{lem:gascrit} and~\ref{lem:lyap}), and we prove that $\HGAS$ is recursively enumerable (Proposition~\ref{prop:re}), as outlined in Section~\ref{sec:overview}. In Section~\ref{subsec:polar} we carry out the polar decomposition, in Section~\ref{subsec:the_two_criteria} we prove the two criteria, and in Section~\ref{subsec:re_upper_bound} we prove Proposition~\ref{prop:re}.

\subsection{Polar coordinates}\label{subsec:polar}
We now carry out the polar decomposition outlined in Section~\ref{subsec:overview-criteria}, where we ``measure'' the distance of any point to the origin with a general positive definite quadratic form $Q$ instead of just the Euclidean norm (corresponding to the choice $Q(y) = y^{\top}y$). By homogeneity, after a change of time the direction of a solution follows a flow on the ellipsoid $S_Q=\{Q=1\}$ which does not involve the radius, and the logarithm of the radius is the integral of a function along this flow \eqref{eq:polarsol}. Questions about all solutions in $\RR^s$ thus become questions about a single flow (the ``angular flow'') on a compact set.

Let $Q(y)=y^{\top}My$ with $M$ a symmetric positive definite $s\times s$ matrix, write $\langle v,w\rangle_Q=v^{\top}Mw$ and $\|v\|_Q=Q(v)^{1/2}$, which is the norm of the inner product $\langle\cdot,\cdot\rangle_Q$, and put $S_Q=\{y\in\RR^s:Q(y)=1\}$, a compact set. For $y\ne0$ we write
\[
r(y)=\|y\|_Q>0,\qquad u(y)=\frac{y}{r(y)}\in S_Q,\qquad\text{so that }y=r(y)\,u(y).
\]
For a homogeneous polynomial vector field $P:\RR^s\to\RR^s$ of degree $D>1$ define the polynomial maps
\begin{equation}\label{eq:aY}
a_P(y)=\langle y,P(y)\rangle_Q,\qquad Y_P(y)=P(y)-a_P(y)\,y .
\end{equation}
For $u\in S_Q$, the vector $a_P(u)u$ is the $Q$-orthogonal projection of $P(u)$ onto the line through $u$, and $Y_P(u)$ is the remaining part, which is tangent to $S_Q$ (Figure~\ref{fig:polar}). More precisely, $\langle y,Y_P(y)\rangle_Q=a_P(y)(1-Q(y))$ for all $y$. Hence along a solution of $\dot y=Y_P(y)$ the function $w=Q(y)-1$ satisfies $\dot w=2\langle y,Y_P(y)\rangle_Q=-2a_P(y)w$; if $w$ vanishes at one time, it vanishes identically by uniqueness of solutions of this linear equation, so $S_Q$ is invariant under $Y_P$. Since $S_Q$ is compact, the restriction to $S_Q$ of the flow of $Y_P$ is a continuous flow defined for all times (Section~\ref{subsec:ode}). Henceforth we restrict the domain of $Y_P$ to $S_Q$, so that $\Phi_{Y_P}$ is this flow on $S_Q$. We call it the \emph{angular flow} $\Phi_{Y_P}$, and we call $a_P$ the \emph{radial rate}. Put
\begin{equation}\label{eq:cocycle}
A_P(u,T)=\int_0^Ta_P(\Phi_{Y_P}^tu)\,dt,\qquad M_P(T)=\max_{u\in S_Q}A_P(u,T),\qquad M_P^*=\sup_{T\ge0}M_P(T)\in[0,\infty],
\end{equation}
for $u\in S_Q$ and $T\ge0$; the maximum exists because $u\mapsto A_P(u,T)$ is continuous on the compact set $S_Q$. Splitting the integral at $T_1$ and using the flow property $\Phi_{Y_P}^{T_1+t}u=\Phi_{Y_P}^t(\Phi_{Y_P}^{T_1}u)$ gives, for $T_1,T_2\ge0$,
\begin{equation}\label{eq:cocycleid}
A_P(u,T_1+T_2)=A_P(u,T_1)+A_P(\Phi_{Y_P}^{T_1}u,T_2).
\end{equation}

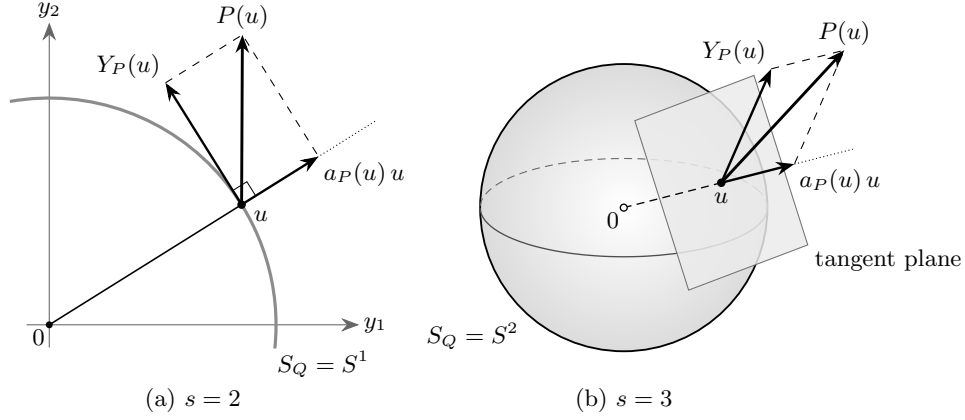
\begin{figure}[t]
\centering
\begin{tikzpicture}[
  >={Stealth[length=2.4mm]},
  vec/.style={->, line width=0.9pt},
  lab/.style={font=\footnotesize, inner sep=1.5pt},
]
  \def\R{3}
  \def\ang{32}
  \draw[->, line width=0.4pt, black!60] (-0.3,0) -- (4.1,0) node[lab, anchor=west, black] {$y_1$};
  \draw[->, line width=0.4pt, black!60] (0,-0.3) -- (0,4.0) node[lab, anchor=south, black] {$y_2$};
  \draw[line width=1.2pt, black!45] (100:\R) arc[start angle=100, end angle=-6, radius=\R];
  \node[lab, anchor=west] at (-10:\R) {$S_Q=S^1$};
  \coordinate (O) at (0,0);
  \coordinate (u) at (\ang:\R);
  \coordinate (a) at ($(u)+(\ang:1.2)$);
  \coordinate (Y) at ($(u)+(\ang+90:1.9)$);
  \coordinate (P) at ($(a)+(\ang+90:1.9)$);
  \draw[line width=0.6pt] (O) -- (u);
  \draw[densely dotted, line width=0.4pt] (u) -- ($(u)!1.75!(a)$);
  \draw[dashed, line width=0.4pt] (Y) -- (P) -- (a);
  \draw[line width=0.3pt] ($(u)+(\ang:0.22)$) -- ($(u)+(\ang:0.22)+(\ang+90:0.22)$) -- ($(u)+(\ang+90:0.22)$);
  \draw[vec] (u) -- (Y) node[lab, anchor=south east] {$Y_P(u)$};
  \draw[vec] (u) -- (a) node[lab, anchor=north west] {$a_P(u)\,u$};
  \draw[vec, line width=1.1pt] (u) -- (P) node[lab, anchor=south] {$P(u)$};
  \fill (u) circle (1.6pt);
  \node[lab, anchor=north west] at ($(u)+(0.1,-0.02)$) {$u$};
  \fill (O) circle (1.3pt) node[lab, anchor=north east] {$0$};
  \node[lab] at (1.9,-1.0) {(a) $s=2$};
\shade[inner color=white, outer color=black!14] (7.600,1.550) circle (1.900);
\draw[line width=0.7pt] (7.600,1.550) circle (1.900);
\draw[line width=0.4pt, densely dashed, black!60] (9.385,1.772) -- (9.368,1.788) -- (9.349,1.804) -- (9.329,1.819) -- (9.308,1.835) -- (9.285,1.850) -- (9.262,1.865) -- (9.237,1.880) -- (9.211,1.894) -- (9.184,1.909) -- (9.156,1.923) -- (9.127,1.937) -- (9.097,1.950) -- (9.066,1.963) -- (9.034,1.976) -- (9.001,1.989) -- (8.967,2.001) -- (8.932,2.013) -- (8.896,2.025) -- (8.859,2.037) -- (8.821,2.048) -- (8.783,2.059) -- (8.743,2.069) -- (8.703,2.079) -- (8.662,2.089) -- (8.621,2.098) -- (8.579,2.107) -- (8.536,2.116) -- (8.492,2.124) -- (8.448,2.132) -- (8.403,2.139) -- (8.358,2.146) -- (8.312,2.153) -- (8.265,2.159) -- (8.219,2.164) -- (8.171,2.170) -- (8.124,2.175) -- (8.076,2.179) -- (8.027,2.183) -- (7.979,2.187) -- (7.930,2.190) -- (7.881,2.193) -- (7.832,2.195) -- (7.782,2.197) -- (7.733,2.198) -- (7.683,2.199) -- (7.633,2.200) -- (7.583,2.200) -- (7.534,2.199) -- (7.484,2.199) -- (7.434,2.197) -- (7.385,2.196) -- (7.336,2.194) -- (7.286,2.191) -- (7.237,2.188) -- (7.189,2.184) -- (7.140,2.181) -- (7.092,2.176) -- (7.044,2.171) -- (6.997,2.166) -- (6.950,2.161) -- (6.904,2.155) -- (6.858,2.148) -- (6.812,2.141) -- (6.767,2.134) -- (6.723,2.126) -- (6.679,2.118) -- (6.636,2.110) -- (6.593,2.101) -- (6.551,2.092) -- (6.510,2.082) -- (6.470,2.072) -- (6.430,2.062) -- (6.391,2.051) -- (6.353,2.040) -- (6.316,2.029) -- (6.280,2.017) -- (6.245,2.005) -- (6.210,1.993) -- (6.177,1.981) -- (6.145,1.968) -- (6.113,1.955) -- (6.083,1.941) -- (6.053,1.927) -- (6.025,1.913) -- (5.998,1.899) -- (5.971,1.885) -- (5.946,1.870) -- (5.922,1.855) -- (5.900,1.840) -- (5.878,1.825) -- (5.858,1.809) -- (5.838,1.793) -- (5.820,1.778) -- (5.804,1.762) -- (5.788,1.745) -- (5.774,1.729) -- (5.761,1.713) -- (5.749,1.696) -- (5.738,1.680) -- (5.729,1.663) -- (5.721,1.646) -- (5.714,1.629) -- (5.709,1.612) -- (5.705,1.595) -- (5.702,1.578) -- (5.700,1.561) -- (5.700,1.544);
\draw[line width=0.5pt, black!70] (5.700,1.544) -- (5.701,1.527) -- (5.704,1.510) -- (5.707,1.493) -- (5.712,1.476) -- (5.718,1.460) -- (5.726,1.443) -- (5.735,1.426) -- (5.745,1.409) -- (5.756,1.393) -- (5.769,1.376) -- (5.783,1.360) -- (5.798,1.344) -- (5.815,1.328) -- (5.832,1.312) -- (5.851,1.296) -- (5.871,1.281) -- (5.892,1.265) -- (5.915,1.250) -- (5.938,1.235) -- (5.963,1.220) -- (5.989,1.206) -- (6.016,1.191) -- (6.044,1.177) -- (6.073,1.163) -- (6.103,1.150) -- (6.134,1.137) -- (6.166,1.124) -- (6.199,1.111) -- (6.233,1.099) -- (6.268,1.087) -- (6.304,1.075) -- (6.341,1.063) -- (6.379,1.052) -- (6.417,1.041) -- (6.457,1.031) -- (6.497,1.021) -- (6.538,1.011) -- (6.579,1.002) -- (6.621,0.993) -- (6.664,0.984) -- (6.708,0.976) -- (6.752,0.968) -- (6.797,0.961) -- (6.842,0.954) -- (6.888,0.947) -- (6.935,0.941) -- (6.981,0.936) -- (7.029,0.930) -- (7.076,0.925) -- (7.124,0.921) -- (7.173,0.917) -- (7.221,0.913) -- (7.270,0.910) -- (7.319,0.907) -- (7.368,0.905) -- (7.418,0.903) -- (7.467,0.902) -- (7.517,0.901) -- (7.567,0.900) -- (7.617,0.900) -- (7.666,0.901) -- (7.716,0.901) -- (7.766,0.903) -- (7.815,0.904) -- (7.864,0.906) -- (7.914,0.909) -- (7.963,0.912) -- (8.011,0.916) -- (8.060,0.919) -- (8.108,0.924) -- (8.156,0.929) -- (8.203,0.934) -- (8.250,0.939) -- (8.296,0.945) -- (8.342,0.952) -- (8.388,0.959) -- (8.433,0.966) -- (8.477,0.974) -- (8.521,0.982) -- (8.564,0.990) -- (8.607,0.999) -- (8.649,1.008) -- (8.690,1.018) -- (8.730,1.028) -- (8.770,1.038) -- (8.809,1.049) -- (8.847,1.060) -- (8.884,1.071) -- (8.920,1.083) -- (8.955,1.095) -- (8.990,1.107) -- (9.023,1.119) -- (9.055,1.132) -- (9.087,1.145) -- (9.117,1.159) -- (9.147,1.173) -- (9.175,1.187) -- (9.202,1.201) -- (9.229,1.215) -- (9.254,1.230) -- (9.278,1.245) -- (9.300,1.260) -- (9.322,1.275) -- (9.342,1.291) -- (9.362,1.307) -- (9.380,1.322) -- (9.396,1.338) -- (9.412,1.355) -- (9.426,1.371) -- (9.439,1.387) -- (9.451,1.404) -- (9.462,1.420) -- (9.471,1.437) -- (9.479,1.454) -- (9.486,1.471) -- (9.491,1.488) -- (9.495,1.505) -- (9.498,1.522) -- (9.500,1.539) -- (9.500,1.556);
\draw[line width=0.4pt, densely dashed, black!60] (9.500,1.556) -- (9.499,1.573) -- (9.496,1.590) -- (9.493,1.607) -- (9.488,1.624) -- (9.482,1.640) -- (9.474,1.657) -- (9.465,1.674) -- (9.455,1.691) -- (9.444,1.707) -- (9.431,1.724) -- (9.417,1.740) -- (9.402,1.756) -- (9.385,1.772);
\filldraw[fill=black!8, fill opacity=0.85, draw=black!55, line width=0.4pt] (8.458,0.459) -- (10.034,1.057) -- (9.322,3.299) -- (7.746,2.700) -- cycle;
\draw[line width=0.5pt, densely dashed] (7.600,1.550) -- (8.890,1.879);
\draw[densely dotted, line width=0.4pt] (9.857,2.126) -- (10.631,2.323);
\draw[dashed, line width=0.4pt] (9.551,3.386) -- (10.518,3.633) -- (9.857,2.126);
\coordinate (Ou) at (7.600,1.550);
\coordinate (Uu) at (8.890,1.879);
\coordinate (Au) at (9.857,2.126);
\coordinate (Yu) at (9.551,3.386);
\coordinate (Pu) at (10.518,3.633);
\coordinate (TPc) at (10.034,1.057);
\coordinate (TPd) at (9.322,3.299);
  \draw[vec] (Uu) -- (Yu) node[lab, anchor=south east] {$Y_P(u)$};
  \draw[vec] (Uu) -- (Au) node[lab, anchor=north west] {$a_P(u)\,u$};
  \draw[vec, line width=1.1pt] (Uu) -- (Pu) node[lab, anchor=south] {$P(u)$};
  \fill (Uu) circle (1.6pt);
  \node[lab, anchor=north] at ([yshift=-2pt]Uu) {$u$};
  \filldraw[fill=white, draw=black, line width=0.5pt] (Ou) circle (1.3pt);
  \node[lab, anchor=north east] at (Ou) {$0$};
  \node[lab, anchor=north west] at ([xshift=1pt]TPc) {tangent plane};
  \node[lab] at (5.6,-0.15) {$S_Q=S^2$};
  \node[lab] at (7.6,-1.0) {(b) $s=3$};
\end{tikzpicture}
\caption{The splitting \eqref{eq:aY} of $P(u)$ at a point $u\in S_Q$, for $Q(y)=|y|^2$: (a) in dimension
$s=2$, where $S_Q$ is the unit circle $S^1$, and (b) in dimension $s=3$, where $S_Q$ is the unit sphere
$S^2$. In both cases the radial part $a_P(u)\,u$ is the projection of $P(u)$ onto the line through $0$ and
$u$ (dotted), and the remaining part $Y_P(u)$ is tangent to $S_Q$: in (a) it points along the circle, and
in (b) it lies in the tangent plane of $S_Q$ at $u$, in which it may point in any direction. The same
holds in every dimension $s$, with the tangent space of $S_Q$ at $u$ in place of the tangent line or
plane. The angular flow $\Phi_{Y_P}$ is the flow of $Y_P$, whose domain is restricted to $S_Q$, and by \eqref{eq:polar} the radial rate
$a_P(u)$ is the rate at which $\log r$ changes in angular time; here $a_P(u)>0$, so $P(u)$ points out of
$S_Q$. For a general positive definite $Q$, the sphere is replaced by the ellipsoid $S_Q$ and the
projection is $Q$-orthogonal.}
\label{fig:polar}
\end{figure}

Now let $y$ be the maximal solution of $\dot y=P(y)$ with $y(0)\ne0$, and let $[0,t_+)$ be its maximal forward interval of existence. Since $P(0)=0$, the zero function is also a solution, so uniqueness gives $y(t)\ne0$ for all $t$, and $r(y(t))$ and $u(y(t))$ are defined. Homogeneity gives $P(y)=r(y)^DP(u(y))$, hence $\langle y,P(y)\rangle_Q=r(y)^{D+1}a_P(u(y))$. Differentiating $r(y(t))^2=\langle y(t),y(t)\rangle_Q$ and $u(y(t))=y(t)/r(y(t))$ therefore gives, for $0\le t<t_+$,
\begin{equation}\label{eq:polar-t}
\frac{d}{dt}r(y(t))=r(y(t))^D\,a_P\bigl(u(y(t))\bigr),\qquad\frac{d}{dt}u(y(t))=r(y(t))^{D-1}\,Y_P\bigl(u(y(t))\bigr).
\end{equation}
Up to the factor $r(y(t))^{D-1}$, which only changes the speed, the direction thus follows $Y_P$, and $\log r$ changes at the rate $a_P$. We remove this factor by a change of time. The \emph{angular time}
\[
\sigma(t)=\int_0^tr(y(\xi))^{D-1}\,d\xi,\qquad0\le t<t_+,
\]
is a strictly increasing $C^1$ function with $\sigma(0)=0$. It maps $[0,t_+)$ onto an interval $[0,\sigma_+)$ with $0<\sigma_+\le\infty$, and its inverse $\tau:[0,\sigma_+)\to[0,t_+)$ is $C^1$, with $\tau(0)=0$ and $\tau'(t)=r(y(\tau(t)))^{1-D}$. Thus $y(\tau(t))$ is the point of the solution at which the angular time equals $t$. By \eqref{eq:polar-t} and the chain rule,
\begin{equation}\label{eq:polar}
\frac{d}{dt}u\bigl(y(\tau(t))\bigr)=Y_P\Bigl(u\bigl(y(\tau(t))\bigr)\Bigr),\qquad\frac{d}{dt}\log r\bigl(y(\tau(t))\bigr)=a_P\Bigl(u\bigl(y(\tau(t))\bigr)\Bigr),\qquad0\le t<\sigma_+,
\end{equation}
the second of which is \eqref{eq:log_r_y_tau_DE}. The first identity says that $t\mapsto u(y(\tau(t)))$ is a solution of $\dot u=Y_P(u)$ with values in $S_Q$, so by uniqueness of solutions it is the angular trajectory of $u(y(0))$. Integrating the second identity then gives
\begin{equation}\label{eq:polarsol}
u\bigl(y(\tau(t))\bigr)=\Phi_{Y_P}^t\bigl(u(y(0))\bigr),\qquad r\bigl(y(\tau(t))\bigr)=r(y(0))\exp A_P\bigl(u(y(0)),t\bigr),\qquad0\le t<\sigma_+.
\end{equation}
Since $\tau$ maps $[0,\sigma_+)$ onto $[0,t_+)$, every point $y(s)$ of the solution, $0\le s<t_+$, is of the form $y(\tau(t))$ with $t=\sigma(s)$, so \eqref{eq:polarsol} describes the whole forward solution. Finally, the angular time is unbounded along forward complete solutions:
\begin{equation}\label{eq:tauinfty}
\text{if }t_+=\infty,\text{ then }\sigma_+=\infty.
\end{equation}
Indeed, put $\bar a=\max_{S_Q}|a_P|$, so that $|A_P(u,t)|\le\bar at$. If $\sigma_+<\infty$, then \eqref{eq:polarsol} gives $r(y(\tau(t)))\ge r_{\min}:=r(y(0))e^{-\bar a\sigma_+}>0$ for $0\le t<\sigma_+$, hence $\tau'(t)\le r_{\min}^{1-D}$ because $D>1$, and $\tau(t)\le r_{\min}^{1-D}\sigma_+$ for all $t$. As $\tau$ maps $[0,\sigma_+)$ onto $[0,t_+)$, this gives $t_+<\infty$.

\subsection{Equivalent criteria for GAS and Lyapunov stability}\label{subsec:the_two_criteria}
The following two criteria express the stability of the origin through the angular flow. The property of GAS translates to the largest average of the radial rate with respect to an invariant probability measure of the angular flow being negative (Lemma~\ref{lem:gascrit}); this criterion is the link between GAS and ergodic theory which we use throughout the paper (Figure~\ref{fig:netdrift} shows an example in which $a_P$ takes both signs). The property of Lyapunov stability translates to the integrals of the radial rate along the angular flow being bounded above (Lemma~\ref{lem:lyap}).

Criteria of the type of Lemma~\ref{lem:gascrit} are classical in other settings. For linear skew-product flows, the boundary points of the dynamical (Sacker--Sell) spectrum are Lyapunov exponents with respect to ergodic invariant measures~\cite{JPS87}; see also~\cite[Chapter~5]{ColoniusKliemann}, where the flow induced on the projective bundle plays the role of our angular flow. In the theory of permanence of ecological differential equations, a compact invariant set is a repeller, or an attractor, if the time averages of a suitable growth rate along the trajectories in this set~\cite{Hofbauer81}, or equivalently its averages with respect to the invariant measures supported on it~\cite[Theorems~2.2 and~4.4 and Remark~2.3]{GarayHofbauer}, have a suitable sign. Lemma~\ref{lem:gascrit} is a statement of this type for the radial rate along the angular flow of a homogeneous polynomial vector field.

\begin{lemma}[Angular criterion for GAS]\label{lem:gascrit}
The origin is GAS for $\dot y=P(y)$ if and only if $\beta(\Phi_{Y_P},a_P)<0$.
\end{lemma}

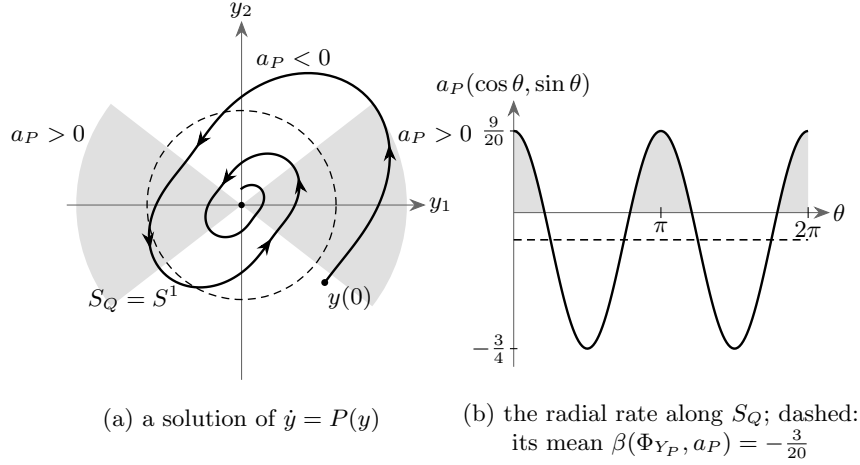
\begin{figure}[t]
\centering
\begin{tikzpicture}[
  >={Stealth[length=2.2mm]},
  lab/.style={font=\footnotesize, inner sep=1pt},
  traj/.style={line width=0.9pt, line join=round},
  arr/.style={-{Stealth[length=2.2mm]}, line width=0.9pt},
  sector/.style={fill=black!12},
]
  \def\tz{37.76}  
  \begin{scope}[scale=1.25]
    \fill[sector] (0,0) -- (-\tz:1.75) arc[start angle=-\tz, end angle=\tz, radius=1.75] -- cycle;
    \fill[sector] (0,0) -- (180-\tz:1.75) arc[start angle=180-\tz, end angle=180+\tz, radius=1.75] -- cycle;
    \draw[->, line width=0.4pt, black!60] (-1.85,0) -- (1.95,0) node[lab, anchor=west, black] {$y_1$};
    \draw[->, line width=0.4pt, black!60] (0,-1.85) -- (0,1.95) node[lab, anchor=south, black] {$y_2$};
    \draw[line width=0.5pt, densely dashed] (0,0) circle (1);
    \draw[traj, domain=-0.75:14.2, samples=400, variable=\t]
      plot ({1.45*exp(-0.15*\t+0.3*sin(2*\t r))*cos(\t r)}, {1.45*exp(-0.15*\t+0.3*sin(2*\t r))*sin(\t r)});
    \foreach \t in {0.35, 2.2, 3.55, 5.4, 6.7, 8.6}
      \draw[arr] ({1.45*exp(-0.15*(\t-0.04)+0.3*sin(2*(\t-0.04) r))*cos((\t-0.04) r)}, {1.45*exp(-0.15*(\t-0.04)+0.3*sin(2*(\t-0.04) r))*sin((\t-0.04) r)})
              -- ({1.45*exp(-0.15*(\t+0.04)+0.3*sin(2*(\t+0.04) r))*cos((\t+0.04) r)}, {1.45*exp(-0.15*(\t+0.04)+0.3*sin(2*(\t+0.04) r))*sin((\t+0.04) r)});
    \fill (0,0) circle (1.2pt/1.25);
    \fill ({1.45*exp(-0.15*(-0.75)+0.3*sin(-1.5 r))*cos(-0.75 r)}, {1.45*exp(-0.15*(-0.75)+0.3*sin(-1.5 r))*sin(-0.75 r)}) circle (1.3pt/1.25) node[lab, anchor=north west] {$y(0)$};
    \node[lab, anchor=west] at (1.62,0.75) {$a_P>0$};
    \node[lab, anchor=east] at (-1.62,0.75) {$a_P>0$};
    \node[lab] at (0.55,1.55) {$a_P<0$};
    \node[lab, anchor=north east] at (-0.62,-0.82) {$S_Q=S^1$};
    \node[lab, align=center] at (0,-2.4) {(a) a solution of $\dot y=P(y)$\\ \phantom{x}};
  \end{scope}
  \begin{scope}[xshift=3.6cm, yshift=-0.1cm, x=0.62cm, y=2.4cm]
    \fill[sector] (0,0) -- plot[domain=0:0.659, samples=30] (\x, {-0.15+0.6*cos(2*\x r)}) -- cycle;
    \fill[sector] plot[domain=pi-0.659:pi+0.659, samples=40] (\x, {-0.15+0.6*cos(2*\x r)}) -- cycle;
    \fill[sector] plot[domain=2*pi-0.659:2*pi, samples=20] (\x, {-0.15+0.6*cos(2*\x r)}) -- (2*pi,0) -- cycle;
    \draw[->, line width=0.4pt, black!60] (0,0) -- (6.75,0) node[lab, anchor=west, black] {$\theta$};
    \draw[->, line width=0.4pt, black!60] (0,-0.85) -- (0,0.62) node[lab, anchor=south, black] {$a_P(\cos\theta,\sin\theta)$};
    \draw[line width=0.9pt, domain=0:2*pi, samples=160] plot (\x, {-0.15+0.6*cos(2*\x r)});
    \draw[densely dashed, line width=0.6pt] (0,-0.15) -- (2*pi,-0.15);
    \foreach \x/\l in {pi/\pi, 2*pi/2\pi} \draw[line width=0.4pt] (\x,0.03) -- (\x,-0.03) node[lab, below] {$\l$};
    \foreach \y/\l in {0.45/\tfrac{9}{20}, -0.75/-\tfrac34} \draw[line width=0.4pt] (0.08,\y) -- (-0.08,\y) node[lab, left] {$\l$};
    \node[lab, align=center] at (3.14,-1.2) {(b) the radial rate along $S_Q$; dashed:\\ its mean $\beta(\Phi_{Y_P},a_P)=-\tfrac{3}{20}$};
  \end{scope}
\end{tikzpicture}
\caption{A homogeneous polynomial vector field whose radial rate takes both signs but whose origin is GAS:
$P(y)=|y|^2(-y_2,y_1)+\bigl(-\tfrac{3}{20}|y|^2+\tfrac35(y_1^2-y_2^2)\bigr)\,y$ on $\RR^2$, with
$Q(y)=|y|^2$. Its radial rate is $a_P(u)=-\tfrac{3}{20}+\tfrac35\cos2\theta$ at
$u=(\cos\theta,\sin\theta)\in S^1$, and its angular vector field is $Y_P(u)=(-u_2,u_1)$, so the angular
flow $\Phi_{Y_P}$ rotates $S_Q=S^1$ at unit speed. (a) Since $a_P(y)$ has the sign of $a_P(y/|y|)$, solutions
move away from the origin while they cross the shaded sectors, where $a_P>0$, and towards it elsewhere;
over each full turn, $r$ is multiplied by $e^{-3\pi/10}<1$. (b) The only invariant probability measure of
the rotation $\Phi_{Y_P}$ is the uniform measure on $S^1$, so $\beta(\Phi_{Y_P},a_P)$ is the mean of $a_P$ over the
circle, $-\tfrac{3}{20}<0$, although $\max_{S_Q}a_P=\tfrac{9}{20}>0$. By Lemma~\ref{lem:gascrit}, the
origin is GAS.}
\label{fig:netdrift}
\end{figure}

\begin{proof}
Suppose $\beta(\Phi_{Y_P},a_P)<0$, and let $T,c>0$ be given by Lemma~\ref{lem:blocks}, so that $A_P(u,T)\le-c$ for every $u\in S_Q$. Let $u\in S_Q$ and $t\ge0$, and write $t=kT+q$ with $k\in\NN$ and $0\le q<T$. Applying \eqref{eq:cocycleid} $k$ times gives
\[
A_P(u,t)=\sum_{j=0}^{k-1}A_P\bigl(\Phi_{Y_P}^{jT}u,T\bigr)+A_P\bigl(\Phi_{Y_P}^{kT}u,q\bigr)\le-kc+\bar aT\le\bar aT+c-\frac{ct}{T},
\]
because each term of the sum is at most $-c$, the last term is at most $\bar aq\le\bar aT$, with $\bar a=\max_{S_Q}|a_P|$ as in the proof of \eqref{eq:tauinfty}, and $k>t/T-1$. Now let $y$ be a solution with $y(0)\ne0$, in the notation of \eqref{eq:polarsol}. By \eqref{eq:polarsol} and the bound just proved,
\begin{equation}\label{eq:gasbound}
\|y(\tau(t))\|_Q\le\|y(0)\|_Qe^{\bar aT+c}e^{-ct/T}\le e^{\bar aT+c}\|y(0)\|_Q,\qquad0\le t<\sigma_+.
\end{equation}
Since every point of the forward solution is of the form $y(\tau(t))$, this gives $\|y(s)\|_Q\le e^{\bar aT+c}\|y(0)\|_Q$ for $0\le s<t_+$. In particular $y$ is bounded, so $t_+=\infty$ by Section~\ref{subsec:ode}. As the zero solution is global as well, Lemma~\ref{lem:scaling}, applied with the norm $\|\cdot\|_Q$ and $C=e^{\bar aT+c}$, shows that the origin is Lyapunov stable. Finally $\sigma_+=\infty$ by \eqref{eq:tauinfty}, so $\sigma(s)\to\infty$ as $s\to\infty$, and the first bound in \eqref{eq:gasbound} with $t=\sigma(s)$ gives $\|y(s)\|_Q\le\|y(0)\|_Qe^{\bar aT+c}e^{-c\sigma(s)/T}\to0$. Hence the origin is GAS.

Conversely, suppose the origin is GAS.
Rosier's homogeneous converse theorem \cite[Theorem~2]{Rosier} supplies a $C^1$ function $V$ on $\RR^s$, positive away from $0$ and homogeneous of some degree $\ell>0$, with $\nabla V(y)\cdot P(y)<0$ for $y\ne0$. By compactness of $S_Q$ there is $\varepsilon>0$ with $b(u):=\nabla V(u)\cdot P(u)/V(u)\le-\varepsilon$ on $S_Q$. Euler's identity $\nabla V(u)\cdot u=\ell V(u)$ \cite[Section~9.3, Exercise~8]{Apostol} gives
\begin{equation}\label{eq:euler}
\nabla(\log V)(u)\cdot Y_P(u)=b(u)-\ell\,a_P(u),\qquad u\in S_Q.
\end{equation}
By \ref{E:deriv} the left side has integral zero against every $\mu\in\cM(\Phi_{Y_P})$, so $\ell\int a_P\,d\mu=\int b\,d\mu\le-\varepsilon$, and $\beta(\Phi_{Y_P},a_P)\le-\varepsilon/\ell<0$. This proves the converse implication.
\end{proof}

\begin{lemma}[Angular criterion for Lyapunov stability {\cite[Lemma~2.2]{Korda}}]\label{lem:lyap}
The origin is Lyapunov stable for $\dot y=P(y)$ if and only if $M_P^*<\infty$. In that case every solution is defined for all $t\ge0$ and satisfies $\|y(t)\|_Q\le e^{M_P^*}\|y(0)\|_Q$.
\end{lemma}

\begin{proof}
We run the argument of \cite[Lemma~2.2]{Korda} with the norm $\|\cdot\|_Q$. Let $y$ be a solution with $y(0)\ne0$, in the notation of \eqref{eq:polarsol}. Suppose first that $M_P^*<\infty$. By \eqref{eq:polarsol} and the definition of $M_P^*$ in \eqref{eq:cocycle}, $\|y(\tau(t))\|_Q\le e^{M_P^*}\|y(0)\|_Q$ for $0\le t<\sigma_+$, that is, $\|y(s)\|_Q\le e^{M_P^*}\|y(0)\|_Q$ for $0\le s<t_+$. Hence $y$ is bounded, so $t_+=\infty$ (Section~\ref{subsec:ode}). This proves the last assertion of the lemma and, since the zero solution is trivial, Lemma~\ref{lem:scaling} applied with the norm $\|\cdot\|_Q$ and $C=e^{M_P^*}$ shows that the origin is Lyapunov stable. Conversely, suppose that the origin is Lyapunov stable, let $C$ be given by Lemma~\ref{lem:scaling} for the norm $\|\cdot\|_Q$, and let $y(0)=u_0\in S_Q$. Then $t_+=\infty$, so $\sigma_+=\infty$ by \eqref{eq:tauinfty}, and \eqref{eq:polarsol} gives $A_P(u_0,t)=\log\|y(\tau(t))\|_Q\le\log C$ for every $t\ge0$. As $u_0\in S_Q$ and $t\ge0$ were arbitrary, $M_P^*\le\log C<\infty$.
\end{proof}

\begin{remark}\label{rem:compare}
By Lemma~\ref{lem:blocks}, $\beta(\Phi_{Y_P},a_P)=\lim_{T\to\infty}M_P(T)/T$, so both criteria concern the same function $M_P$: the origin is GAS if and only if $M_P(T)<0$ for some $T>0$, and Lyapunov stable if and only if $M_P$ is bounded above. In particular Lyapunov stability implies $\beta(\Phi_{Y_P},a_P)\le0$. Invariant averages cannot detect Lyapunov stability, however. The homogeneous polynomial vector fields $P_1(y)=|y|^2(-y_2,y_1)$ and $P_2(y)=|y|^2(y_2,0)$ of degree three on $\RR^2$ both have $\beta=0$ with respect to $Q(y)=|y|^2$: for $P_1$ because $a_{P_1}\equiv0$, and for $P_2$ because the only zeros of its angular vector field are $\pm(1,0)$, where $a_{P_2}=0$, and every other angular trajectory is a heteroclinic arc, which carries no invariant probability measure. The origin is Lyapunov stable for $P_1$, which preserves $|y|$, but not for $P_2$, whose solutions with $y_2(0)\ne0$ blow up in finite time. This is why the Lyapunov half of Theorem~\ref{thm:cubicization} needs a statement about individual trajectories (Proposition~\ref{prop:shadow}).
\end{remark}

\subsection{\texorpdfstring{Recursive enumerability of $\HGAS$}{Recursive enumerability of HGAS}}\label{subsec:re_upper_bound}

We now prove that $\HGAS$ is recursively enumerable, following the outline in Section~\ref{sec:overview}: we combine the converse Lyapunov theorem of Ahmadi and El~Khadir~\cite[Theorem~3.1]{AEK} with Tarski's theorem (Theorem~\ref{thm:tarski}). To begin, put $\rho(x)=|x|^2$. For $r\in\NN$ and a homogeneous polynomial $p$ of degree $2r+2$ on $\RR^n$, define
\begin{equation}\label{eq:Vp}
V(x)=\frac{p(x)}{\rho(x)^r},\qquad x\ne0.
\end{equation}
For a homogeneous polynomial vector field $f$ of degree $d$, differentiating \eqref{eq:Vp} gives $\nabla V=\rho^{-r}\nabla p-r\,p\,\rho^{-r-1}\nabla\rho$ with $\nabla\rho(x)=2x$, and hence
\begin{equation}\label{eq:Vdot}
-\nabla V(x)\cdot f(x)=\frac{q_{f,p,r}(x)}{\rho(x)^{r+1}},\qquad q_{f,p,r}=-\rho\,\nabla p\cdot f+2r\,p\,(x\cdot f).
\end{equation}
For a valid input $e$ coding $(n,d,f)$ and $r\in\NN$, let $D(e,r)$ be the sentence
\begin{equation}\label{eq:Der}
\exists(\text{coefficients of }p)\ \forall x\ \ \rho(x)=1\implies\bigl(p(x)>0\wedge q_{f,p,r}(x)>0\bigr),
\end{equation}
in which $p$ ranges over homogeneous polynomials of degree $2r+2$ in $n$ variables, written with indeterminate real coefficients. All constants in $D(e,r)$ are rational, and $D(e,r)$ is computable from $(e,r)$.

\begin{proposition}[Recursive enumerability of GAS]\label{prop:re}
$\HGAS$ is recursively enumerable. Consequently, so are $\HGAS(n,d)$ and $\HGAS(\cdot,d)$ for all fixed $n,d$.
\end{proposition}

\begin{proof}
Consider the following procedure. Given a code $e$, first decide whether it is valid; if it is not, run forever. For a valid code, decide $D(e,0),D(e,1),\dots$ in turn by Theorem~\ref{thm:tarski}, and halt as soon as one of them is true; if none is true, the search runs forever. Thus the procedure halts on $e$ if and only if $e$ is valid and $D(e,r)$ holds for some $r$. Since invalid codes are not in $\HGAS$, it remains to show that a valid code $e$ lies in $\HGAS$ if and only if $D(e,r)$ holds for some $r$; then the procedure halts exactly on $\HGAS$.

Suppose that $D(e,r)$ holds, fix $p$ as in \eqref{eq:Der}, and define $V$ by \eqref{eq:Vp}. The function $V$ is homogeneous of degree two, and by \eqref{eq:Vdot} the function $-\nabla V\cdot f$ is homogeneous of degree $d+1$ on $\RR^n\setminus\{0\}$; on $S^{n-1}$, where $\rho=1$, they equal $p$ and $q_{f,p,r}$, which are positive there. Taking for $m_0$ and $M_0$ the minimum and the maximum of $p$ on $S^{n-1}$, and for $\mu_0$ the minimum of $q_{f,p,r}$ on $S^{n-1}$, we obtain $m_0|x|^2\le V(x)\le M_0|x|^2$ and $-\nabla V(x)\cdot f(x)\ge\mu_0|x|^{d+1}$ for $x\ne0$. Let $x(t)$ be a solution with $x(0)\ne0$, on its maximal interval $[0,t_+)$. It never vanishes, since $f(0)=0$, so $\frac{d}{dt}V(x(t))=\nabla V(x(t))\cdot f(x(t))\le-\mu_0|x(t)|^{d+1}\le0$. Hence $m_0|x(t)|^2\le V(x(t))\le V(x(0))\le M_0|x(0)|^2$, so $x$ is bounded, $t_+=\infty$ (Section~\ref{subsec:ode}), and $|x(t)|\le(M_0/m_0)^{1/2}|x(0)|$ for all $t\ge0$; by Lemma~\ref{lem:scaling}, the origin is Lyapunov stable. Integrating the bound on $\frac{d}{dt}V(x(t))$ gives $\mu_0\int_0^\infty|x(t)|^{d+1}dt\le V(x(0))<\infty$. Since $x$ is bounded, $\dot x=f(x)$ is bounded, so $x$ is Lipschitz and $|x(t)|^{d+1}$ is uniformly continuous on $[0,\infty)$; by Barbalat's lemma \cite[Lemma~8.2]{Khalil}, $|x(t)|^{d+1}\to0$, that is, $x(t)\to0$. Thus the origin is GAS.

Conversely, if the origin is GAS, then by \cite[Theorem~3.1]{AEK} there are $r\in\NN$ and a positive definite homogeneous polynomial $p$ of degree $2r+2$ (with real coefficients) such that \eqref{eq:Vp} is a strict Lyapunov function, that is, $\nabla V(x)\cdot f(x)<0$ for $x\ne0$. Then $p>0$ and, by \eqref{eq:Vdot}, $q_{f,p,r}=-\rho^{r+1}\nabla V\cdot f>0$ on $S^{n-1}$, so the coefficients of $p$ witness $D(e,r)$. This proves that $\HGAS$ is r.e. The statements for fixed $n,d$ follow by \ref{C:trans}, since $\HGAS(n,d)$ and $\HGAS(\cdot,d)$ are preimages of $\HGAS$ under total computable maps of codes, which send a code of dimension $n$ and degree $d$ (respectively, of degree $d$) to itself, and every other code to an invalid code (Section~\ref{subsec:stability}).
\end{proof}

\paperpart{Part I. A single \texorpdfstring{$(n_*,d_*)$}{(n*, d*)} implies many \texorpdfstring{$(n,d)$}{(n, d)}}
\section{Lifting to degree three}\label{sec:lifting}

In this section we prove Theorem~\ref{thm:cubicization}. Lifting the original system via a special map $\nu_{k}$, called the \textit{Veronese map}, lowers the degree of a polynomial system on an invariant set, but this alone is useless for GAS or Lyapunov stability, since trajectories off the invariant set are not controlled. We therefore combine this exact lift of degree three on the invariant set, the so-called \textit{Veronese cone}, with a vector field of degree three that vanishes on the cone and pulls nearby points towards it, the ``penalty term'' of Section~\ref{subsec:overview1}. We then apply the angular criteria of Section~\ref{sec:angular} with a non-Euclidean quadratic form adapted to the Veronese map, and we use the set of unit vectors on the cone as a ``bridge'' between the angular flows of the original and the lifted vector field, as explained in Section~\ref{subsec:overview1}.

Despite its name, which comes from algebraic geometry, the Veronese map is elementary. It sends $x\in\RR^n$ to the vector $(x^\alpha)_{|\alpha|=k}$ of all monomials of degree $k$ in $x_1,\dots,x_n$, for example $(x_1,x_2)\mapsto(x_1^2,x_1x_2,x_2^2)$, and the Veronese cone consists of the real multiples of such vectors. It produces vector fields of degree three for a simple reason: along a solution of $\dot x=F(x)$ with $F$ of degree $2k+1$, the derivative of the vector $y$ of these monomials has degree $3k$ in $x$, and every monomial of degree $3k$ is a product of three monomials of degree $k$, so $\dot y$ is a homogeneous polynomial of degree three in $y$. We use only the elementary facts proved below.

Lowering the degree of a polynomial system by introducing monomials as new variables is a classical idea. By adding new variables, every polynomial system can be recast as a system whose nonlinearities are at most quadratic~\cite{Kerner}, and every solution of a polynomial system is a projection of a solution of a system of degree at most two whose new variables are monomials~\cite[Theorem~1]{CPSW}. Such ``quadratizations'' have been studied with additional requirements: with the smallest number of new variables~\cite{BychkovPogudin}; keeping every equilibrium at which all eigenvalues of the Jacobian matrix have negative real parts with this property~\cite[Theorem~1]{CaiPogudin}; and, for homogeneous systems, using all monomials of a fixed degree as new variables, in which case Mencinger relates the stability of the origin of a planar homogeneous system of degree three to that of its quadratization~\cite[Theorem~3.14]{Mencinger}. These constructions reproduce the original dynamics on the image of the lifting map, but they do not control the trajectories which start off this image, whereas GAS and Lyapunov stability concern every initial state. Controlling these trajectories is the role of the penalty term $-\tilde G$ in \eqref{eq:H} below, and most of the work in this section.

In Section~\ref{sec:veronese} we collect the facts about the Veronese map that we use. In Section~\ref{sec:lift} we construct the vector field $H$ of degree three and show that it pulls nearby points towards the cone. In Sections~\ref{subsec:dichotomy} and~\ref{subsec:shadow} we compare the angular flows of $H$ and $F$, for GAS and for Lyapunov stability respectively, and in Section~\ref{subsec:proofcubic} we prove Theorem~\ref{thm:cubicization}. Table~\ref{tab:const5} lists the constants of this section, what they depend on, and which of them are computed by the reduction.

\begin{table}[t]
\centering
\caption{The constants of Section~\ref{sec:lifting}, for a homogeneous polynomial vector field $F$ on $\RR^n$ of degree $2k+1$. The entries marked ``computed'' are computed by the reduction of Theorem~\ref{thm:cubicization}, and $\lambda$ is computable from them; $c_1$ and $C_*$ are only shown to exist, and they are used only in proofs. Constants which appear inside a single proof are not listed.}
\label{tab:const5}
\vspace{10px}
\footnotesize
\renewcommand{\arraystretch}{1.15}
\begin{tabular}{@{}llll@{}}
\hline
Constant & Defined in & Depends on & Role\\
\hline
$m=\binom{n+k-1}{k}$ & Section~\ref{sec:veronese} & $n,k$; computed & dimension of the lifted space $E\cong\RR^m$\\
$c_1$ & Lemma~\ref{lem:dist} & $n,k$; exists & $W\ge c_1d^2$ on $S_Q$\\
$\eta,c$ (rational) & Lemma~\ref{lem:gradineq} & $n,k$; computed & $\|\gradS W\|_Q^2\ge cW$ where $W\le\eta$\\
$C_G,B_G$ (rational) & Lemma~\ref{lem:tangency} & $F$; computed & $|Y_GW|\le C_GW$ where $W\le\eta$; $|a_G|\le B_G$ on $S_Q$\\
$L_G$ (rational) & Lemma~\ref{lem:lipschitz} & $F$; computed & Lipschitz constant of $Y_G$ and $a_G$\\
$K$ (rational) & \eqref{eq:Kchoice} & $F$; computed & weight of the penalty term $-K\gradQ W$\\
$\lambda=Kc-C_G$ & Lemma~\ref{lem:tube} & $F$; computable & decay rate of $W$ near $\Sigma$; $\lambda\ge2L_G+1$\\
$C_*$ & Proposition~\ref{prop:shadow} & $F$; exists & finite-horizon shadowing\\
\hline
\end{tabular}
\end{table}

\subsection{\texorpdfstring{The Veronese map $\nu_{k}$ and friends}{The Veronese map nu\_k and friends}}\label{sec:veronese}

Fix $n\ge1$ and $k\ge1$, and let
\[
A_k=\{\alpha\in\NN^n:|\alpha|=k\},\qquad E=\RR^{A_k},\qquad m=\#A_k=\binom{n+k-1}{k}.
\]
Coordinates of $y\in E$ are written $y_\alpha$. The Veronese map is
\begin{equation}\label{eq:nu}
\nu_k:\RR^n\to E,\qquad\nu_k(x)_\alpha=x^\alpha=x_1^{\alpha_1}\cdots x_n^{\alpha_n}.
\end{equation}
Define the positive definite quadratic form with rational coefficients
\begin{equation}\label{eq:Q}
Q(y)=\sum_{\alpha\in A_k}\binom{k}{\alpha}y_\alpha^2,\qquad\binom{k}{\alpha}=\frac{k!}{\alpha_1!\cdots\alpha_n!}.
\end{equation}
The multinomial theorem \cite[Theorem~5.4.1]{Brualdi}, applied to $x_1^2,\dots,x_n^2$, gives
\begin{equation}\label{eq:Qnu}
Q(\nu_k(x))=\Bigl(\sum_{i=1}^nx_i^2\Bigr)^k=|x|^{2k}.
\end{equation}
Consider the binomials of degree two
\begin{equation}\label{eq:binom}
y_\alpha y_\beta-y_\gamma y_\delta,\qquad\alpha+\beta=\gamma+\delta,\quad\alpha,\beta,\gamma,\delta\in A_k.
\end{equation}
The value of such a binomial at $\nu_k(x)$ is $x^{\alpha+\beta}-x^{\gamma+\delta}=0$. Let $q_1,\dots,q_s$ be the finite list of all nonzero polynomials of the form \eqref{eq:binom}, taking one representative from each pair $\{q,-q\}$. (If $n=1$ or $k=1$ there are none; then $s=0$ and all sums over $j$ below are empty.) Since each $q_j$ is homogeneous of degree two,
\begin{equation}\label{eq:E1}
q_j\bigl(t\,\nu_k(x)\bigr)=t^2q_j\bigl(\nu_k(x)\bigr)=0\qquad t\in\RR,\ x\in\RR^n,\ j=1,\dots,s.
\end{equation}
Set
\begin{equation}\label{eq:W}
W(y)=\sum_{j=1}^sq_j(y)^2,\qquad V=\{y\in E:W(y)=0\},\qquad\Sigma=V\cap S_Q,
\end{equation}
where $S_Q=\{y\in E:Q(y)=1\}$. The form $Q$ is the quadratic form $Q_H$ of Section~\ref{subsec:overview1}, and $S_Q$ is the ellipsoid $S_H$ there; to lighten the notation, in this section we write $Q$ and $S_Q$ instead of $Q_H$ and $S_H$, and we identify $E$ with $\RR^m$ by fixing an ordering of $A_k$. All of these objects have integer coefficients and are computable from $(n,k)$. The polynomials \eqref{eq:binom} are the classical equations of the Veronese cone \cite{CEMR,Sturmfels}; we use only the elementary consequences of their definition proved in Lemmas~\ref{lem:local}, \ref{lem:cone} and~\ref{lem:dist}.

\begin{example}\label{ex:binary}
For $n=k=2$ we have $\nu_2(x_1,x_2)=(x_1^2,x_1x_2,x_2^2)$, $Q=y_{20}^2+2y_{11}^2+y_{02}^2$, a single binomial $q=y_{20}y_{02}-y_{11}^2$, so that $V=\{y_{20}y_{02}=y_{11}^2\}$ is a double cone (Figure~\ref{fig:cone}), $W=q^2$, and the $Q$-gradient \eqref{eq:gradQ} defined below is $\gradQ W=M^{-1}\nabla W=(2qy_{02},\,-2qy_{11},\,2qy_{20})$, where $M=\mathrm{diag}(1,2,1)$ is the matrix of $Q$; its components are homogeneous of degree three. In general the binomials \eqref{eq:binom} have degree two, so $W$ has degree four and $\gradQ W$ has degree three, the degree of the vector field constructed in Section~\ref{sec:lift}.
\end{example}

\input{fig_veronese_cone}

The next three lemmas describe $V$ and show that $W$ controls the distance to it. By Lemma~\ref{lem:local}, a point of $V$ with $y_{ke_i}\ne0$ is determined by its $n$ coordinates indexed by the set $B_i$ defined below, and every $y$ with $y_{ke_i}\ne0$ lies within $Q$-distance $\kappa_i(y)W(y)^{1/2}$ of the point of $V$ with the same such coordinates, where $\kappa_i$ is continuous. Lemma~\ref{lem:cone} deduces that the nonzero points of $V$ are the real multiples of the vectors $\nu_k(u)$, and Lemma~\ref{lem:dist} shows that on $S_Q$, $W$ is at least a constant multiple of the squared distance to $\Sigma$.

For $1\le i\le n$ let $e_i$ be the $i$-th standard unit vector, regarded as a multi-index, put $U_i=\{y\in E:y_{ke_i}\ne0\}$, and let $B_i\subseteq A_k$ be the set of the $n$ multi-indices $ke_i$ and $(k-1)e_i+e_j$, $j\ne i$. For $\alpha\in A_k$ and $i\ne j$, the polynomials
\begin{align}
&y_\alpha^2-y_{\alpha+e_i-e_j}\,y_{\alpha-e_i+e_j}&&\text{if }\alpha_i\ge1\text{ and }\alpha_j\ge1,\label{eq:E2}\\
&y_{ke_i}\,y_\alpha-y_{(k-1)e_i+e_j}\,y_{\alpha+e_i-e_j}&&\text{if }\alpha_j\ge1,\label{eq:E3}
\end{align}
are of the form \eqref{eq:binom}, since all multi-indices occurring in them lie in $A_k$ and $2\alpha=(\alpha+e_i-e_j)+(\alpha-e_i+e_j)$, $ke_i+\alpha=((k-1)e_i+e_j)+(\alpha+e_i-e_j)$. Hence each of them is $0$ or $\pm q_j$ for some $j$, so it vanishes on $V$ and its square is at most $W$. We write $|v|$ for the Euclidean norm on $E$; since $1\le\binom{k}{\alpha}\le k!$, we have $|v|\le\|v\|_Q\le\sqrt{k!}\,|v|$.

\begin{lemma}\label{lem:local}
\begin{enumerate}[label=(\alph*)]
\item $V\setminus\{0\}\subseteq U_1\cup\dots\cup U_n$.
\item Let $1\le i\le n$. For $y\in U_i$ define $\xi=\xi^{(i)}(y)\in\RR^n$ by $\xi_i=1$ and $\xi_j=y_{(k-1)e_i+e_j}/y_{ke_i}$ for $j\ne i$, and put $\pi_i(y)=y_{ke_i}\,\nu_k(\xi)$. Then $\pi_i(y)\in V$, the coordinates of $\pi_i(y)$ indexed by $B_i$ agree with those of $y$, and there is a continuous function $\kappa_i:U_i\to[0,\infty)$ with
\begin{equation}\label{eq:E4}
\|y-\pi_i(y)\|_Q\le\kappa_i(y)\,W(y)^{1/2},\qquad y\in U_i.
\end{equation}
In particular every $y\in V\cap U_i$ satisfies $y=\pi_i(y)$.
\end{enumerate}
\end{lemma}

\begin{proof}
(a) Let $y\in V\setminus\{0\}$. Since $y\ne0$ and $k\ge1$, there are $i$ and $\alpha\in A_k$ with $y_\alpha\ne0$ and $\alpha_i\ge1$; among all such pairs choose one with $\alpha_i$ maximal. If $\alpha\ne ke_i$, pick $j\ne i$ with $\alpha_j\ge1$. Since \eqref{eq:E2} vanishes at $y$, $y_{\alpha+e_i-e_j}y_{\alpha-e_i+e_j}=y_\alpha^2\ne0$, so $y_{\alpha+e_i-e_j}\ne0$, contradicting the maximality of $\alpha_i$. Hence $\alpha=ke_i$ and $y\in U_i$.

(b) Let $N_i=A_k\setminus B_i$, so $\#N_i=m-n$, and write $y_N=(y_\alpha)_{\alpha\in N_i}$. Each $\alpha\in N_i$ has $\alpha_i\le k-2$, so we may fix $j_\alpha\ne i$ with $\alpha_{j_\alpha}\ge1$ and set $\alpha'=\alpha+e_i-e_{j_\alpha}\in A_k$. The polynomial $r_\alpha=y_{ke_i}y_\alpha-y_{(k-1)e_i+e_{j_\alpha}}y_{\alpha'}$ is of type \eqref{eq:E3}, so it vanishes on $V$ and $r_\alpha^2\le W$. As a function of $y_N$, $r_\alpha$ is affine: its coefficients $y_{ke_i}$ and $y_{(k-1)e_i+e_{j_\alpha}}$ are coordinates indexed by $B_i$, and $\alpha'$ lies either in $B_i$ or in $N_i$. Hence the vector $r=(r_\alpha)_{\alpha\in N_i}$ has the form $r(y)=L_i(y)\,y_N-\ell_i(y)$, where the square matrix $L_i(y)$ and the vector $\ell_i(y)$ depend only on the coordinates indexed by $B_i$, all diagonal entries of $L_i(y)$ equal $y_{ke_i}$, and its only other nonzero entries are $-y_{(k-1)e_i+e_{j_\alpha}}$ in the positions $(\alpha,\alpha')$ with $\alpha'\in N_i$. If $N_i$ is ordered by nonincreasing $\alpha_i$, these positions lie below the diagonal, since $\alpha'_i=\alpha_i+1$. Thus $\det L_i(y)=y_{ke_i}^{m-n}$, and by the adjugate formula for the inverse \cite[(6.2.7)]{Meyer} the function $\kappa_i(y):=\sqrt{k!(m-n)}\,\|L_i(y)^{-1}\|$ (operator norm) is continuous on $U_i$.

Let $y\in U_i$ and $\pi=\pi_i(y)$. Since $\pi_{ke_i}=y_{ke_i}\xi_i^k=y_{ke_i}$ and $\pi_{(k-1)e_i+e_j}=y_{ke_i}\xi_j=y_{(k-1)e_i+e_j}$, the point $\pi$ agrees with $y$ in the coordinates indexed by $B_i$, and $\pi\in V$ by \eqref{eq:E1}. Hence $0=r(\pi)=L_i(y)\pi_N-\ell_i(y)$, so $r(y)=L_i(y)(y_N-\pi_N)$. As $y-\pi$ vanishes in the coordinates indexed by $B_i$ and $|r(y)|^2\le(m-n)W(y)$,
\[
\|y-\pi\|_Q\le\sqrt{k!}\,|y_N-\pi_N|\le\sqrt{k!}\,\|L_i(y)^{-1}\|\,|r(y)|\le\kappa_i(y)\,W(y)^{1/2}.
\]
If $y\in V$, then $W(y)=0$ and $y=\pi$. (If $N_i=\emptyset$, then $y=\pi$ and $\kappa_i=0$.)
\end{proof}

\begin{lemma}\label{lem:cone}
The nonzero points of $V$ are exactly the points $t\,\nu_k(u)$ with $t\in\RR\setminus\{0\}$ and $u\in S^{n-1}$. Consequently $\Sigma=\{\sigma\nu_k(u):\sigma\in\{-1,1\},\ u\in S^{n-1}\}$, and $\Sigma$ is a compact subset of $S_Q$.
\end{lemma}

\begin{proof}
If $t\ne0$ and $u\in S^{n-1}$, then $t\nu_k(u)\in V$ by \eqref{eq:E1}, and $Q(t\nu_k(u))=t^2$ by \eqref{eq:Qnu}, so $t\nu_k(u)\ne0$. Conversely, if $y\in V\setminus\{0\}$, Lemma~\ref{lem:local}(a) gives $i$ with $y\in U_i$, and Lemma~\ref{lem:local}(b) gives $y=\pi_i(y)=c\,\nu_k(\xi)$, where $c=y_{ke_i}\ne0$ and $\xi\ne0$; as $\nu_k(sx)=s^k\nu_k(x)$, this is $t\nu_k(u)$ with $t=c|\xi|^k$ and $u=\xi/|\xi|$. Since $Q(t\nu_k(u))=t^2$, such a point lies on $S_Q$ if and only if $t=\pm1$, and $\Sigma$ is compact as a continuous image of $\{-1,1\}\times S^{n-1}$.
\end{proof}

Since $\nu_k(-u)=(-1)^k\nu_k(u)$, we have $\Sigma=\nu_k(S^{n-1})$ for odd $k$, and $\Sigma=\nu_k(S^{n-1})\cup(-\nu_k(S^{n-1}))$ for even $k$.

\begin{lemma}\label{lem:dist}
For $y\in S_Q$ let $d(y)=\min_{p\in\Sigma}\|y-p\|_Q$. There is $c_1>0$ such that $W(y)\ge c_1\,d(y)^2$ for all $y\in S_Q$.
\end{lemma}

\begin{proof}
The minimum exists since $\Sigma$ is compact and nonempty, and $d\le2$ on $S_Q$. Put $\delta=\frac12n^{-k/2}$. Each $p=\sigma\nu_k(u)\in\Sigma$ has $|p_{ke_i}|=|u_i|^k\ge2\delta$ for every $i$ with $|u_i|\ge n^{-1/2}$, and such $i$ exist since $|u|=1$. Hence the compact set $K_0=\{y\in S_Q:\max_i|y_{ke_i}|\le\delta\}$ does not meet $\Sigma$, which is the zero set of $W$ on $S_Q$; so if $K_0\ne\emptyset$, then $w_0=\min_{K_0}W>0$ and, since $d\le2$, $W\ge w_0\ge\frac{w_0}4d^2$ on $K_0$. Every other $y\in S_Q$ lies in one of the compact sets $K_i=\{y\in S_Q:|y_{ke_i}|\ge\delta\}\subseteq U_i$; let $C_i>0$ be an upper bound for $\kappa_i$ on $K_i$, which exists because $\kappa_i$ is continuous on $U_i$ by Lemma~\ref{lem:local}(b) and $K_i$ is compact. For $y\in K_i$, the point $\pi=\pi_i(y)$ is nonzero, $p=\pi/\|\pi\|_Q$ lies in $V\cap S_Q=\Sigma$ by \eqref{eq:E1}, and $\bigl|\|\pi\|_Q-1\bigr|\le\|y-\pi\|_Q$ because $\|y\|_Q=1$. By \eqref{eq:E4},
\[
d(y)\le\|y-p\|_Q\le\|y-\pi\|_Q+\bigl|\|\pi\|_Q-1\bigr|\le2\|y-\pi\|_Q\le2C_i\,W(y)^{1/2}.
\]
Thus $c_1=\min\{w_0/4,(4C_1^2)^{-1},\dots,(4C_n^2)^{-1}\}$ works (omit $w_0/4$ if $K_0=\emptyset$).
\end{proof}

The vector field $H$ constructed in Section~\ref{sec:lift} is the lift of $F$ plus the term $-K\gradQ W$, with $\gradQ W$ defined below and a constant $K>0$; this term decreases $W$. To show that it makes $W$ decrease at a rate proportional to $W$ near $\Sigma$, we need a lower bound, in terms of $W$, for the gradient of $W$ along $S_Q$; this is Lemma~\ref{lem:gradineq}. It is a gradient inequality of {\L}ojasiewicz type~\cite{Lojasiewicz63} (see, e.g.,~\cite[(0.2)]{ColdingMinicozzi}), with exponent $\frac12$, for the restriction of $W$ to $S_Q$ near its zero set $\Sigma$; we need it with constants which can be computed, and we prove it directly.

Write $M$ for the diagonal matrix of $Q$, and let
\begin{equation}\label{eq:gradQ}
\gradQ W=M^{-1}\nabla W
\end{equation}
be the $Q$-gradient \cite[Chapter~13, pp.~342--343]{Lee}, so that $\langle\gradQ W(y),v\rangle_Q=DW(y)v$. Since $W$ is homogeneous of degree four, $\langle\gradQ W(y),y\rangle_Q=4W(y)$, so on $S_Q$ the component of $\gradQ W$ tangent to $S_Q$ is
\begin{equation}\label{eq:gradS}
\gradS W(y)=\gradQ W(y)-4W(y)\,y,
\end{equation}
and
\begin{equation}\label{eq:Dy}
g(y):=\|\gradS W(y)\|_Q^2=\|\gradQ W(y)\|_Q^2-16W(y)^2,\qquad y\in S_Q.
\end{equation}
The vector field which we construct in Section~\ref{sec:lift} uses the ambient gradient \eqref{eq:gradQ}, whose components are homogeneous of degree three. The tangential gradient \eqref{eq:gradS} contains a term of degree five, and we use it only for estimates on $S_Q$.

\begin{lemma}\label{lem:gradineq}
There are computable positive rationals $\eta,c$ such that
\begin{equation}\label{eq:gradineq}
Q(y)=1,\ W(y)\le\eta\implies g(y)\ge c\,W(y).
\end{equation}
\end{lemma}

\begin{proof}
Let $c_1$ and $d$ be as in Lemma~\ref{lem:dist}, let $L=\max\{|q_j(h)|:1\le j\le s,\ Q(h)=1\}$ (and $L=0$ if $s=0$), so that $|q_j(h)|\le L\|h\|_Q^2$ for $h\in E$, and put $\rho=\min\{\frac12,\sqrt{c_1}/(2(L\sqrt s+1))\}$. We first show
\begin{equation}\label{eq:E5}
y\in S_Q,\ d(y)\le\rho\implies g(y)\ge\tfrac{c_1}{4}W(y).
\end{equation}
If $y\in\Sigma$, then $W(y)=0$ and \eqref{eq:E5} is trivial. Otherwise choose $p\in\Sigma$ with $\|y-p\|_Q=d:=d(y)>0$ and put $h=y-p$. Every homogeneous polynomial $q$ of degree two on $E$ satisfies $q(a+b)=q(a)+Dq(a)b+q(b)$, as one sees by writing $q(y)=y^TAy$ with $A$ symmetric. With $a=y$, $b=-h$ and $q_j(p)=0$ this gives $Dq_j(y)h=q_j(y)+q_j(h)$. By Lemma~\ref{lem:dist} we have $d\le(W(y)/c_1)^{1/2}$ and $d\le\rho$, thus by the Cauchy--Schwarz inequality,
\[
DW(y)h=2W(y)+2\sum_jq_j(y)q_j(h)\ge2W(y)-2L\sqrt s\,d^2W(y)^{1/2}\ge2\Bigl(1-\frac{L\sqrt s\,d}{\sqrt{c_1}}\Bigr)W(y)\ge W(y).
\]
Let $w=h-\langle y,h\rangle_Qy$ be the $Q$-orthogonal projection of $h$ onto the $Q$-orthogonal complement of $y$, so that $\langle y,w\rangle_Q=0$ and $\|w\|_Q\le\|h\|_Q=d$. As $Q(y)=Q(p)=1$, we have $\langle y,h\rangle_Q=1-\langle y,p\rangle_Q=\frac12\|y-p\|_Q^2=\frac12d^2$, and Euler's identity $DW(y)y=4W(y)$ gives $DW(y)w=DW(y)h-2d^2W(y)\ge(1-2d^2)W(y)\ge\frac12W(y)$, the last step because $d\le\rho\le\frac12$. On the other hand, by \eqref{eq:gradS}, $\langle y,w\rangle_Q=0$, and the Cauchy--Schwarz inequality, $DW(y)w=\langle\gradQ W(y),w\rangle_Q=\langle\gradS W(y),w\rangle_Q\le g(y)^{1/2}d$. Hence $g(y)^{1/2}d\ge\frac12W(y)$, and, since $0<d^2\le W(y)/c_1$ by Lemma~\ref{lem:dist}, $g(y)\ge W(y)^2/(4d^2)\ge\frac{c_1}{4}W(y)$, proving \eqref{eq:E5}.

The set $K_\rho=\{y\in S_Q:d(y)\ge\rho\}$ is compact and disjoint from $\Sigma$, so there is $\eta_0>0$ with $W\ge2\eta_0$ on $K_\rho$ (any $\eta_0>0$ if $K_\rho=\emptyset$). If $y\in S_Q$ and $W(y)\le\eta_0$, then $d(y)<\rho$ and \eqref{eq:E5} applies. Thus \eqref{eq:gradineq} holds for $(\eta_0,c_1/4)$, and hence for every rational pair $(\eta,c)$ with $0<\eta\le\eta_0$ and $0<c\le c_1/4$, because decreasing $\eta$ strengthens the hypothesis and decreasing $c$ weakens the conclusion. To compute such a pair, enumerate pairs of positive rationals and decide for each the sentence
\begin{equation}\label{eq:QEgrad}
\forall y\ \ \bigl(Q(y)=1\wedge W(y)\le\eta\bigr)\implies g(y)\ge c\,W(y)
\end{equation}
by Theorem~\ref{thm:tarski}, stopping at the first true one; by what we have shown, the search terminates.
\end{proof}

\subsection{The lifted vector field and attraction to the cone}\label{sec:lift}

We keep the notation of Section~\ref{sec:veronese}. Let $F:\RR^n\to\RR^n$ be a rational-coefficient homogeneous polynomial vector field of degree $2k+1$. For $\alpha\in A_k$, the $\alpha$-coordinate of $D\nu_k(x)F(x)$ is a homogeneous polynomial of degree $(k-1)+(2k+1)=3k$ in $x$. Every monomial of degree $3k$ factors effectively into three monomials of degree $k$: list its $3k$ variable indices in increasing order and cut the list into three blocks of length $k$. Replacing each factor $x^\beta$ by $y_\beta$ gives a map $G:E\to E$ with rational coefficients, whose components are homogeneous of degree three, such that
\begin{equation}\label{eq:Gnu}
G(\nu_k(x))=D\nu_k(x)F(x),\qquad x\in\RR^n.
\end{equation}
We compute the coefficients of $G$ from those of $F$ by a finite procedure.

On $S^{n-1}$ we use the Euclidean form, that is, the form $Q_F(x)=|x|^2$ of Section~\ref{subsec:overview1}, for which $S_F=S^{n-1}$, and write $a_F(u)=u\cdot F(u)$, $Y_F(u)=F(u)-a_F(u)u$; on $E$ we use $Q$ and define $a_G,Y_G$ by \eqref{eq:aY}, regarded as polynomial maps on $E$. Differentiating \eqref{eq:Qnu} and using Euler's identity gives, for $|u|=1$,
\begin{equation}\label{eq:aGnu}
a_G(\nu_k(u))=\langle\nu_k(u),D\nu_k(u)F(u)\rangle_Q=k\,u\cdot F(u)=k\,a_F(u),\qquad Y_G(\nu_k(u))=D\nu_k(u)Y_F(u).
\end{equation}
Since $G$ is homogeneous of degree three, $G(-y)=-G(y)$; so $a_G$ is even and $Y_G$ is odd. Moreover $Y_G$ is tangent to $\Sigma$ in the following sense:
\begin{equation}\label{eq:E7}
Dq_j(p)\,Y_G(p)=0\qquad p\in\Sigma,\ j=1,\dots,s.
\end{equation}
Indeed, for $u\in S^{n-1}$, differentiating $q_j(\nu_k(x))=0$ at $x=u$ in the direction $Y_F(u)$ and using \eqref{eq:aGnu} gives $Dq_j(\nu_k(u))Y_G(\nu_k(u))=0$; since $Dq_j(-y)=-Dq_j(y)$ and $Y_G$ is odd, the left side of \eqref{eq:E7} takes the same value at $-p$ as at $p$, and \eqref{eq:E7} follows from Lemma~\ref{lem:cone}.

The next two lemmas provide, together with the constants $\eta$ and $c$ of Lemma~\ref{lem:gradineq}, the computable constants from which we choose $K$ in \eqref{eq:Kchoice} below. The first of them uses \eqref{eq:E7}: since $q_j$ and $Dq_j\,Y_G$ both vanish on $\Sigma$, the derivative of $W=\sum_jq_j^2$ along $Y_G$ is bounded near $\Sigma$ by a multiple of $W$.

\begin{lemma}\label{lem:tangency}
For the rational $\eta$ of Lemma~\ref{lem:gradineq} one can compute rationals $C_G,B_G>0$ such that
\begin{align}
Q(y)=1,\ W(y)\le\eta&\implies|Y_GW(y)|\le C_G\,W(y),\label{eq:tangency}\\
Q(y)=1&\implies|a_G(y)|\le B_G,\label{eq:AF}
\end{align}
where $Y_GW(y)=DW(y)Y_G(y)$.
\end{lemma}

\begin{proof}
Let $c_1$ and $d$ be as in Lemma~\ref{lem:dist}, and put $\psi_j(y)=Dq_j(y)Y_G(y)$. Since $q_j$ and $\psi_j$ are polynomials, their derivatives are bounded on the open convex set $\{Q<2\}$, so the mean value inequality \cite[Theorem~9.19]{Rudin} on this set and the equivalence of the norms $|\cdot|$ and $\|\cdot\|_Q$ give $L'>0$ with $|q_j(y)-q_j(y')|\le L'\|y-y'\|_Q$ and $|\psi_j(y)-\psi_j(y')|\le L'\|y-y'\|_Q$ for $y,y'\in\{Q\le1\}$. For $y\in S_Q$ choose $p\in\Sigma$ with $\|y-p\|_Q=d(y)$. As $q_j(p)=0$ and $\psi_j(p)=0$ by \eqref{eq:E7}, the identity $Y_GW=2\sum_jq_j\psi_j$ and Lemma~\ref{lem:dist} give
\[
|Y_GW(y)|\le2sL'^2d(y)^2\le\frac{2sL'^2}{c_1}W(y),\qquad y\in S_Q.
\]
Thus \eqref{eq:tangency} holds for some real constant, and then for every larger rational. To compute one, enumerate positive rationals $C$ and decide by Theorem~\ref{thm:tarski} the sentence
\begin{equation}\label{eq:QEtan}
\forall y\ \ \bigl(Q(y)=1\wedge W(y)\le\eta\bigr)\implies\bigl(Y_GW(y)\bigr)^2\le C^2W(y)^2;
\end{equation}
the search terminates. Similarly, $a_G$ is bounded on the compact set $S_Q$, and a rational $B_G$ is found by deciding $\forall y\,(Q(y)=1\Rightarrow a_G(y)^2\le B^2)$ for $B=1,2,\dots$.
\end{proof}

\begin{lemma}\label{lem:lipschitz}
One can compute a rational $L_G>0$ such that
\begin{equation}\label{eq:LG}
\|Y_G(y)-Y_G(y')\|_Q\le L_G\|y-y'\|_Q,\qquad|a_G(y)-a_G(y')|\le L_G\|y-y'\|_Q
\end{equation}
whenever $Q(y)\le1$ and $Q(y')\le1$. Moreover the flow $\Phi_{Y_G}|_{S_Q}$ is defined for all $t\in\RR$.
\end{lemma}

\begin{proof}
$Y_G$ and $a_G$ are polynomial maps, so, as in the proof of Lemma~\ref{lem:tangency}, the mean value inequality on $\{Q<2\}$ gives a real constant with \eqref{eq:LG}, and then every larger rational works; a rational one is found by enumerating rationals $L$ and deciding
\[
\begin{aligned}
\forall y,y'\ \ \bigl(Q(y)\le1\wedge Q(y')\le1\bigr)\implies\Bigl(&Q\bigl(Y_G(y)-Y_G(y')\bigr)\le L^2Q(y-y')\\
&\wedge\bigl(a_G(y)-a_G(y')\bigr)^2\le L^2Q(y-y')\Bigr).
\end{aligned}
\]
By Section~\ref{sec:angular}, the compact set $S_Q$ is invariant under $Y_G$, so $\Phi_{Y_G}|_{S_Q}$ is a flow on $S_Q$ (Section~\ref{subsec:ode}).
\end{proof}

Choose a positive rational $K$ (depending on $F$) such that
\begin{equation}\label{eq:Kchoice}
Kc\ge C_G+2L_G+1,\qquad4K\eta>B_G+1,
\end{equation}
and define
\begin{equation}\label{eq:H}
H(y)=G(y)-\tilde G(y),\qquad\tilde G(y)=K\gradQ W(y).
\end{equation}
Because $W$ is homogeneous of degree four, the components of $\gradQ W$ are homogeneous of degree three, so $H$ is a homogeneous polynomial vector field of degree three with rational coefficients. No test of whether $F$ is GAS or Lyapunov stable enters the choice of $K$. Since $DW(y)\gradQ W(y)=\|\gradQ W(y)\|_Q^2$, the term $-K\gradQ W$ decreases $W$, that is, it pulls points towards the cone $V=\{W=0\}$ (see Figure~\ref{fig:cone} for an example). By the first condition in \eqref{eq:Kchoice}, $W$ decays along the angular flow of $H$ near $\Sigma$ at the rate $\lambda=Kc-C_G\ge2L_G+1$ (Lemma~\ref{lem:tube}); the GAS part of the proof uses only $\lambda>0$, and the lower bound $2L_G+1$ is used only in Proposition~\ref{prop:shadow}. By the second condition, $a_H<-1$ at the points of $S_Q$ where $W>\eta$ (Lemma~\ref{lem:tube}). On $V$ all $q_j$ vanish, hence $\gradQ W=2\sum_jq_j\gradQ q_j=0$ and $H=G$ there. Thus $-\tilde G=-K\gradQ W$ is the ``penalty term'' of Section~\ref{subsec:overview1}: properties~1, 2 and~4 listed there hold by the remarks in this paragraph (the choice of $K$ is effective, see Section~\ref{subsec:proofcubic}), and we make property~3 precise in Lemma~\ref{lem:tube} and Proposition~\ref{prop:shadow} below. Adding a term which makes an invariant set attracting is also the idea behind the stabilization of constraints and integrals of motion in the numerical integration of differential equations~\cite{Baumgarte}.
Since $\langle y,\gradQ W(y)\rangle_Q=4W(y)$ by Euler's identity, the definitions \eqref{eq:aY} and \eqref{eq:gradS} give
\begin{equation}\label{eq:aH}
a_H=a_G-4KW,\qquad Y_H=Y_G-K\gradS W\ \text{ on }S_Q,\qquad Y_HW=Y_GW-Kg,
\end{equation}
where the last identity holds on $S_Q$ because $DW(y)v=\langle\gradQ W(y),v\rangle_Q$ for $v\in E$ and, for $y\in S_Q$, the vector $\gradS W(y)$ is $Q$-orthogonal to $y$, so that $\langle\gradQ W(y),\gradS W(y)\rangle_Q$ equals $\|\gradS W(y)\|_Q^2=g(y)$.

As in Section~\ref{sec:angular}, we henceforth restrict the domains of $Y_F$ and $Y_H$ to $S^{n-1}$ and $S_Q$, respectively. The angular flows of Section~\ref{sec:angular} for $F$, with the Euclidean form on $\RR^n$, and for $H$, with $Q$ on $E$, are then $\Phi_{Y_F}$ and $\Phi_{Y_H}$. We do not restrict the domain of $Y_G$, although we do so for $Y_F$ and $Y_H$: Lemmas~\ref{lem:tangency} and~\ref{lem:lipschitz} use $Y_G$ on the whole ball $\{Q\le1\}$, and we only need that $S_Q$ is invariant under its flow, so the flow of $Y_G$ on $S_Q$ is written $\Phi_{Y_G}|_{S_Q}$. These flows are defined for all $t\in\RR$ (Section~\ref{subsec:ode}); their time $t$ is the angular time of \eqref{eq:polar}, not the time of $\dot x=F(x)$. As in \eqref{eq:cocycle}, $A_F(u,\tau)=\int_0^\tau a_F(\Phi_{Y_F}^tu)\,dt$ and $A_H(y,\tau)=\int_0^\tau a_H(\Phi_{Y_H}^ty)\,dt$, with $M_F,M_F^*$ and $M_H,M_H^*$ defined accordingly.

\begin{lemma}\label{lem:conedyn}
For $u\in S^{n-1}$, $\sigma\in\{-1,1\}$ and $t\in\RR$,
\begin{equation}\label{eq:E8}
\Phi_{Y_H}^t(\sigma\nu_k(u))=\Phi_{Y_G}^t(\sigma\nu_k(u))=\sigma\nu_k(\Phi_{Y_F}^t(u)),\qquad a_H(\sigma\nu_k(u))=k\,a_F(u).
\end{equation}
Consequently $\Phi_{Y_H}^t(\Sigma)=\Sigma$ for all $t\in\RR$, and
\begin{equation}\label{eq:AHcone}
A_H(\sigma\nu_k(u),\tau)=k\,A_F(u,\tau),\qquad\max_{p\in\Sigma}A_H(p,\tau)=k\,M_F(\tau),\qquad\tau\ge0.
\end{equation}
\end{lemma}

\begin{proof}
On $V$ we have $W=0$ and $H=G$, so $a_H=a_G$ and $Y_H=Y_G$ on $\Sigma$ by \eqref{eq:aH}. For $u\in S^{n-1}$ and $\sigma=\pm1$, the curve $c(t)=\sigma\nu_k(\Phi_{Y_F}^t(u))$ lies in $\Sigma$ and, by \eqref{eq:aGnu} and the oddness of $Y_G$, satisfies $c'(t)=\sigma Y_G(\nu_k(\Phi_{Y_F}^tu))=Y_G(c(t))=Y_H(c(t))$. Since $c(0)=\sigma\nu_k(u)$ and $Y_G,Y_H$ are polynomial, hence locally Lipschitz, uniqueness of solutions (Section~\ref{subsec:ode}) gives the first identity in \eqref{eq:E8}, and \eqref{eq:aGnu} with the evenness of $a_G$ gives the second. By Lemma~\ref{lem:cone} every point of $\Sigma$ has the form $\sigma\nu_k(u)$, so $\Phi_{Y_H}^t(\Sigma)=\Sigma$. Integrating the second identity in \eqref{eq:E8} along the first gives \eqref{eq:AHcone}.
\end{proof}

\begin{lemma}\label{lem:tube}
Let $\lambda=Kc-C_G$, so that $\lambda\ge2L_G+1$ by \eqref{eq:Kchoice}, and let $\cT=\{y\in S_Q:W(y)\le\eta\}$. Then $a_H<-1$ on $S_Q\setminus\cT$. Moreover $\cT$ is forward invariant under $\Phi_{Y_H}$, and $W(\Phi_{Y_H}^ty)\le W(y)e^{-\lambda t}$ for $y\in\cT$ and $t\ge0$.
\end{lemma}

\begin{proof}
On $\cT$, Lemmas~\ref{lem:gradineq} and~\ref{lem:tangency} and \eqref{eq:aH} give
\begin{equation}\label{eq:tubedecay}
Y_HW\le C_GW-KcW=-\lambda W.
\end{equation}
Hence the function $w(t)=W(\Phi_{Y_H}^ty)$ satisfies $w'\le-\lambda w$ as long as $w(t)\le\eta$; this implies that $\cT$ is forward invariant and, by Gronwall's inequality, that $w(t)\le w(0)e^{-\lambda t}$. If $y\in S_Q\setminus\cT$, then $W(y)>\eta$, and \eqref{eq:AF}, \eqref{eq:Kchoice} and \eqref{eq:aH} give $a_H(y)\le B_G-4K\eta<-1$.
\end{proof}

\subsection{Invariant measures on and off the cone}\label{subsec:dichotomy}

By Lemma~\ref{lem:gascrit}, the GAS part of Theorem~\ref{thm:cubicization} amounts to showing that $\beta(\Phi_{Y_H},a_H)<0$ if and only if $\beta(\Phi_{Y_F},a_F)<0$. The next lemma shows that every ergodic invariant measure of the angular flow of $H$ either lives on $\Sigma$, where this flow is a copy of the angular flow of $F$ (Lemma~\ref{lem:conedyn}), or gives $a_H$ an average of at most $-1$; Lemma~\ref{lem:coneavg} turns this into a comparison of $\beta(\Phi_{Y_H},a_H)$ with $\beta(\Phi_{Y_F},a_F)$ (Figure~\ref{fig:bridge}). Together, these two lemmas make precise the two steps \eqref{eq:Sigma_to_SH_int} and \eqref{eq:Sigma_to_SF_int} of the outline in Section~\ref{subsec:overview1}, in which $\Sigma$ serves as a ``bridge'' between the two angular flows.

\input{fig_veronese_bridge}

\begin{lemma}\label{lem:dichotomy}
Let $\mu$ be an ergodic invariant probability measure for $\Phi_{Y_H}$. Then either $\mu(\Sigma)=1$, or $\int a_H\,d\mu\le-1$.
\end{lemma}

\begin{proof}
The set $\cT$ is closed, and for $t\ge0$ it satisfies $\Phi_{Y_H}^t(\cT)\subseteq\cT$ by Lemma~\ref{lem:tube}, equivalently $\cT\subseteq(\Phi_{Y_H}^t)^{-1}(\cT)$. The sets $\Phi_{Y_H}^t(\cT)=(\Phi_{Y_H}^{-t})^{-1}(\cT)$ and $(\Phi_{Y_H}^t)^{-1}(\cT)$ have measure $\mu(\cT)$ by invariance of $\mu$, so $\mu(\cT\setminus\Phi_{Y_H}^t(\cT))=0$ and $\mu((\Phi_{Y_H}^t)^{-1}(\cT)\setminus\cT)=0$. These two sets are the symmetric differences $(\Phi_{Y_H}^s)^{-1}(\cT)\,\triangle\,\cT$ for $s=-t$ and for $s=t$. Hence $\mu((\Phi_{Y_H}^s)^{-1}(\cT)\,\triangle\,\cT)=0$ for every $s\in\RR$, and $\mu(\cT)\in\{0,1\}$ by ergodicity (Section~\ref{subsec:erg}). If $\mu(\cT)=1$, then \ref{E:deriv} and \eqref{eq:tubedecay} give $0=\int Y_HW\,d\mu\le-\lambda\int W\,d\mu$, so $W=0$ $\mu$-almost everywhere and $\mu(\Sigma)=1$. If $\mu(\cT)=0$, then $a_H<-1$ $\mu$-almost everywhere by Lemma~\ref{lem:tube}.
\end{proof}

\begin{lemma}\label{lem:coneavg}
Let $\beta_F=\beta(\Phi_{Y_F},a_F)$ and $\beta_H=\beta(\Phi_{Y_H},a_H)$. The largest $a_H$-average among measures in $\cM(\Phi_{Y_H})$ supported on $\Sigma$ is $k\beta_F$. Consequently
\begin{equation}\label{eq:sandwich}
k\beta_F\le\beta_H\le\max\{k\beta_F,-1\}.
\end{equation}
\end{lemma}

\begin{proof}
By Lemma~\ref{lem:conedyn}, $\Sigma$ is a compact invariant set and $\max_{p\in\Sigma}A_H(p,T)=kM_F(T)$ for every $T>0$. The measures in $\cM(\Phi_{Y_H})$ supported on $\Sigma$ are exactly the invariant probability measures of the restricted flow $\Phi_{Y_H}|_\Sigma$. Dividing by $T$ and letting $T\to\infty$, Lemma~\ref{lem:blocks} applied to $(\Phi_{Y_H}|_\Sigma,a_H|_\Sigma)$ and to $(\Phi_{Y_F},a_F)$ shows that the largest such average is $k\beta_F$. This gives the lower bound in \eqref{eq:sandwich}, and \ref{E:decomp} with Lemma~\ref{lem:dichotomy} gives the upper bound.
\end{proof}

The upper bound in \eqref{eq:sandwich} need not be an equality, since there need not be an invariant measure off the cone with average exactly $-1$; the two bounds in \eqref{eq:sandwich} suffice for the sign equivalence.

\subsection{Finite-horizon shadowing}\label{subsec:shadow}

For Lyapunov stability we need a statement about individual trajectories: by Lemma~\ref{lem:lyap} we must bound the integrals $A_H(y,\tau)$ uniformly in $y\in S_Q$ and $\tau\ge0$, and by Remark~\ref{rem:compare} invariant measures do not give such a bound. By Lemma~\ref{lem:tube}, an angular trajectory of $H$ in $\cT$ is an approximate solution of $\dot y=Y_G(y)$, with an error that decays like $e^{-\lambda t/2}$, whereas two solutions of $\dot y=Y_G(y)$ separate at most at the rate $L_G<\lambda/2$. The next proposition deduces that on each finite time interval $[0,\tau]$ such a trajectory stays close to a trajectory on $\Sigma$, and that its radial integral exceeds that of the trajectory on $\Sigma$ by at most a constant. This is how we use property~3 of the penalty term (Section~\ref{subsec:overview1}) for Lyapunov stability. Proposition~\ref{prop:shadow} is a quantitative, finite-horizon analogue of the classical fact that the trajectories near a normally attracting invariant manifold have an ``asymptotic phase'', that is, each of them is asymptotic to a trajectory on the manifold (see~\cite[Theorem~5]{Fenichel74}, which is stated for diffeomorphisms, and~\cite{HPS}); since we need explicit constants only on finite time intervals, we prove it directly.

\begin{proposition}[Finite-horizon shadowing]\label{prop:shadow}
Let $\lambda$ and $\cT$ be as in Lemma~\ref{lem:tube}. There is $C_*>0$ such that for every $y_0\in\cT$ and every $\tau\ge0$ there is $q\in\Sigma$ with
\begin{equation}\label{eq:shadow}
\|\Phi_{Y_H}^t(y_0)-\Phi_{Y_G}^t(q)\|_Q\le C_*e^{-\lambda t/2},\quad0\le t\le\tau,\qquad\text{and}\qquad A_H(y_0,\tau)\le A_H(q,\tau)+C_*.
\end{equation}
\end{proposition}

\begin{proof}
\emph{Step 1: a bound for $\gradS W$.} Since $\langle\gradQ W(y),y\rangle_Q=4W(y)$ and $Q(y)=1$, the vector $\gradS W(y)$ is the $Q$-orthogonal projection of $\gradQ W(y)$ onto the $Q$-orthogonal complement of $y$. Hence, with $\Gamma=\max_j\max_{S_Q}\|\gradQ q_j\|_Q$ (taking $\Gamma=0$ if $s=0$) and the Cauchy--Schwarz inequality, for $y\in S_Q$,
\[
\|\gradS W(y)\|_Q\le\|\gradQ W(y)\|_Q=\Bigl\|2\sum_jq_j(y)\gradQ q_j(y)\Bigr\|_Q\le2\Gamma\sqrt s\,W(y)^{1/2}=:C_WW(y)^{1/2}.
\]

\emph{Step 2: the trajectory is an approximate trajectory of $Y_G$.} Let $y(t)=\Phi_{Y_H}^t(y_0)$ for $t\ge0$. By Lemma~\ref{lem:tube}, $W(y(t))\le\eta e^{-\lambda t}$. By \eqref{eq:aH} and Step~1, with $\theta_1=KC_W\eta^{1/2}$,
\begin{equation}\label{eq:sh1}
y'(t)=Y_G(y(t))+\epsilon(t),\qquad\|\epsilon(t)\|_Q\le\theta_1e^{-\lambda t/2},
\end{equation}
where $\epsilon(t)=-K\gradS W(y(t))$.

\emph{Step 3: the comparison trajectory.} Fix $\tau\ge0$, and let $d$ and $c_1$ be as in Lemma~\ref{lem:dist}. There is $p\in\Sigma$ with $\|y(\tau)-p\|_Q=d(y(\tau))\le(W(y(\tau))/c_1)^{1/2}\le\theta_0e^{-\lambda\tau/2}$, where $\theta_0=(\eta/c_1)^{1/2}$. Put $z(t)=\Phi_{Y_G}^{t-\tau}(p)$ for $t\in\RR$ and $q=z(0)$. By Lemma~\ref{lem:conedyn}, $z(t)\in\Sigma$ for all $t$; in particular $q\in\Sigma$, and $\Phi_{Y_G}^t(q)=z(t)$.

\emph{Step 4: the error decays.} Let $e(t)=y(t)-z(t)$ and $r(t)=\|e(t)\|_Q$ for $0\le t\le\tau$. Since $z'(t)=Y_G(z(t))$, \eqref{eq:sh1} gives $e'(t)=Y_G(y(t))-Y_G(z(t))+\epsilon(t)$; as both $y(t)$ and $z(t)$ lie in $S_Q$, \eqref{eq:LG} and the bound in \eqref{eq:sh1} give $\|e'(t)\|_Q\le L_Gr(t)+\theta_1e^{-\lambda t/2}$. Writing $e(t)=e(s)-\int_t^se'(\sigma)\,d\sigma$, we obtain
\begin{equation}\label{eq:sh2}
r(t)\le r(s)+\int_t^s\bigl(L_Gr(\sigma)+\theta_1e^{-\lambda\sigma/2}\bigr)\,d\sigma,\qquad0\le t\le s\le\tau.
\end{equation}
Let $R_0=2\theta_0+2\theta_1+1$ and $R(t)=R_0e^{-\lambda t/2}$. Since $\lambda\ge2L_G+1$, we have $\lambda/2-L_G\ge\frac12$, so $(\lambda/2-L_G)R(\sigma)\ge\frac{R_0}2e^{-\lambda\sigma/2}>\theta_1e^{-\lambda\sigma/2}$, that is,
\begin{equation}\label{eq:sh3}
L_GR(\sigma)+\theta_1e^{-\lambda\sigma/2}<\tfrac\lambda2R(\sigma)=-R'(\sigma),\qquad\sigma\in\RR.
\end{equation}
We claim that $r\le R$ on $[0,\tau]$. Otherwise the closed set $\{t\in[0,\tau]:r(t)\ge R(t)\}$ is nonempty; let $t_1$ be its largest element. Since $r(\tau)\le\theta_0e^{-\lambda\tau/2}<R(\tau)$, we have $t_1<\tau$ and $r<R$ on $(t_1,\tau]$. Then \eqref{eq:sh2} with $t=t_1$ and $s=\tau$, together with $r(\tau)<R(\tau)$ and $r\le R$ on $(t_1,\tau]$, and then \eqref{eq:sh3} on the interval $[t_1,\tau]$, which has positive length, give
\[
r(t_1)\le R(\tau)+\int_{t_1}^\tau\bigl(L_GR(\sigma)+\theta_1e^{-\lambda\sigma/2}\bigr)\,d\sigma<R(\tau)-\int_{t_1}^\tau R'(\sigma)\,d\sigma=R(t_1),
\]
which contradicts $r(t_1)\ge R(t_1)$. Hence $\|\Phi_{Y_H}^t(y_0)-\Phi_{Y_G}^t(q)\|_Q=r(t)\le R_0e^{-\lambda t/2}$ for $0\le t\le\tau$.

\emph{Step 5: the radial rates.} On $\Sigma$ we have $a_H=a_G$ and $\Phi_{Y_H}^t=\Phi_{Y_G}^t$ (Lemma~\ref{lem:conedyn}), so $A_H(q,\tau)=\int_0^\tau a_G(z(t))\,dt$. By \eqref{eq:aH}, $a_H\le a_G$ on $S_Q$. Hence
\[
A_H(y_0,\tau)-A_H(q,\tau)\le\int_0^\tau\bigl|a_G(y(t))-a_G(z(t))\bigr|\,dt\le L_GR_0\int_0^\infty e^{-\lambda t/2}\,dt=\frac{2L_GR_0}{\lambda}\le2L_GR_0,
\]
where the first inequality uses $a_H(y(t))\le a_G(y(t))$, the second uses \eqref{eq:LG} and the bound $\|y(t)-z(t)\|_Q\le R_0e^{-\lambda t/2}$ of Step~4, and the last uses $\lambda\ge1$. Thus $C_*=(1+2L_G)R_0$ works for both inequalities in \eqref{eq:shadow}; it depends neither on $y_0$ nor on $\tau$. It involves the constant $c_1$ of Lemma~\ref{lem:dist}, which is only known to exist; this does not matter, since $C_*$ does not enter the construction of $H$.
\end{proof}

The point $q$ in Proposition~\ref{prop:shadow} depends on $\tau$; this suffices for the proof of \eqref{eq:mainLS} below.

\subsection{\texorpdfstring{Proof of Theorem~\ref*{thm:cubicization}}{Proof of the reduction to degree three}}\label{subsec:proofcubic}

\begin{proof}[Proof of Theorem~\ref{thm:cubicization}]
If $k=1$, take $H=F$. For $k\ge2$, construct $Q,W,G,K,H$ by \eqref{eq:Q}--\eqref{eq:H}. Every step is effective: the binomials \eqref{eq:binom} are a finite explicit list, $G$ is obtained by expanding $D\nu_k(x)F(x)$ and factoring monomials, the rationals $\eta,c,C_G,B_G,L_G$ are found by the terminating searches in Lemmas~\ref{lem:gradineq}, \ref{lem:tangency} and~\ref{lem:lipschitz}, and $K$ is chosen as the least positive integer satisfying the rational inequalities \eqref{eq:Kchoice}.

\emph{GAS.} The inequalities \eqref{eq:sandwich} give $\beta_H<0\iff\beta_F<0$: if $\beta_F<0$, both entries of the upper bound are negative, and if $\beta_F\ge0$, the lower bound gives $\beta_H\ge0$. Lemma~\ref{lem:gascrit}, applied with the Euclidean form to $F$ and with $Q$ to $H$, proves \eqref{eq:mainGAS}.

\emph{Lyapunov stability.} By \eqref{eq:AHcone}, $kM_F(\tau)=\max_{p\in\Sigma}A_H(p,\tau)\le M_H(\tau)$ for all $\tau\ge0$, so $M_F^*\le M_H^*/k$. Conversely, let $y\in S_Q$ and $\tau\ge0$, and let $t_0=\inf\{t\ge0:\Phi_{Y_H}^t(y)\in\cT\}\in[0,\infty]$. By Lemma~\ref{lem:tube}, $a_H(\Phi_{Y_H}^ty)<-1$ for $0\le t<t_0$. If $\tau\le t_0$, then $A_H(y,\tau)\le-\tau\le0$. If $\tau>t_0$, then $\Phi_{Y_H}^{t_0}(y)\in\cT$ because $\cT$ is closed and $t\mapsto\Phi_{Y_H}^t(y)$ is continuous, and Proposition~\ref{prop:shadow}, applied to $y_0=\Phi_{Y_H}^{t_0}(y)$ with the horizon $\tau-t_0$, gives $q\in\Sigma$ with $A_H(\Phi_{Y_H}^{t_0}y,\tau-t_0)\le A_H(q,\tau-t_0)+C_*$; by Lemma~\ref{lem:cone}, $q=\sigma\nu_k(u)$ with $\sigma\in\{-1,1\}$ and $u\in S^{n-1}$. Splitting the integral as in \eqref{eq:cocycleid} gives $A_H(y,\tau)=A_H(y,t_0)+A_H(\Phi_{Y_H}^{t_0}y,\tau-t_0)$, where $A_H(y,t_0)\le-t_0$ because $a_H(\Phi_{Y_H}^ty)<-1$ for $0\le t<t_0$. With \eqref{eq:AHcone} this gives
\begin{equation}\label{eq:mainest}
A_H(y,\tau)\le-t_0+k\,A_F(u,\tau-t_0)+C_*,\qquad\tau>t_0.
\end{equation}
Since $A_F(u,\tau-t_0)\le M_F^*$ and $t_0\ge0$, the cases $\tau\le t_0$ and $\tau>t_0$ together show that $M_F^*<\infty$ implies $M_H^*\le\max\{0,kM_F^*+C_*\}<\infty$. Thus $M_F^*<\infty$ if and only if $M_H^*<\infty$, and Lemma~\ref{lem:lyap}, applied with the Euclidean form to $F$ and with $Q$ to $H$, proves \eqref{eq:mainLS}.
\end{proof}

\section{From degree three to everything else}\label{sec:normalform}

Theorem~\ref{thm:cubicization} lowers odd degrees to three in a fixed dimension. In this section we construct the ``padding'' which raises the degree and the dimension (Lemma~\ref{lem:padding}), and we prove Theorem~\ref{thm:transfer}, as outlined in Section~\ref{subsec:overview1}.

\begin{lemma}[Degree raising and dimension padding]\label{lem:padding}
Let $H:\RR^m\to\RR^m$ be a homogeneous polynomial vector field of degree three, let $j\ge0$, and let $n\ge m$. For $(x,z)\in\RR^m\times\RR^{n-m}$, define
\begin{equation}\label{eq:Hhat}
\widehat H(x,z)=\bigl(|x|^2+|z|^2\bigr)^j\bigl(H(x),-z_1^3,\dots,-z_{n-m}^3\bigr).
\end{equation}
Then $\widehat H$ is a homogeneous polynomial vector field of degree $3+2j$. Furthermore, its origin is GAS if and only if the origin of $H$ is GAS, and Lyapunov stable if and only if the origin of $H$ is Lyapunov stable.
\end{lemma}

\begin{proof}
Let $H_0(x,z)=(H(x),-z_1^3,\dots,-z_{n-m}^3)$. The system $(\dot x,\dot z)=H_0(x,z)$ is the product of $\dot x=H(x)$ with the independent scalar equations $\dot z_i=-z_i^3$, whose solutions exist for all $t\ge0$, satisfy $|z_i(t)|\le|z_i(0)|$, and tend to $0$. Hence the origin of $H_0$ is GAS if and only if that of $H$ is; for the reverse implication, the invariant subspace $\{z=0\}$ reproduces every trajectory of $H$. In the same way, a bound $|x(t)|\le C|x(0)|$ for $H$ gives a bound for $H_0$ with respect to the norm $|x|+|z|$ and conversely, so by Lemma~\ref{lem:scaling} the origin of $H_0$ is Lyapunov stable if and only if that of $H$ is. On the Euclidean unit sphere the factor $(|x|^2+|z|^2)^j$ equals one, so $\widehat H$ and $H_0$ have the same angular vector field and the same radial rate, and therefore the same quantities $\beta$ and $M^*$. Lemmas~\ref{lem:gascrit} and~\ref{lem:lyap} show that $\widehat H$ and $H_0$ have the same answer for GAS and for Lyapunov stability.
\end{proof}

\begin{proof}[Proof of Theorem~\ref{thm:transfer}]
Let $n_*$, $d_*=2k_*+1$ and $N=\binom{n_*+k_*-1}{k_*}$ be as in Theorem~\ref{thm:transfer}, and let $n\ge N$ and $d=3+2j$ with $j\ge0$. Consider the map which sends a valid code of dimension $n_*$ and degree $d_*$, describing a vector field $F$, to the code of the vector field $\widehat H$ of Lemma~\ref{lem:padding}, applied with $N$ in place of $m$ and with the vector field $H$ of degree three on $\RR^N$ given by Theorem~\ref{thm:cubicization}; every other code is sent to the code of $x\mapsto|x|^{d-1}x$ on $\RR^n$. This map is total and computable, since validity, dimension and degree of a code can be computed (Section~\ref{subsec:stability}), and $H$ and $\widehat H$ are computed from $F$ by finite procedures. By \eqref{eq:mainGAS}, \eqref{eq:mainLS} and Lemma~\ref{lem:padding}, the origin is GAS, respectively Lyapunov stable, for $F$ if and only if it is for $\widehat H$. The vector field $x\mapsto|x|^{d-1}x$ is neither GAS nor Lyapunov stable, since $v=|x|^2$ satisfies $\dot v=2v^{(d+1)/2}$ with $(d+1)/2\ge2$, so that nonzero solutions blow up in finite time. Hence the map is a many-one reduction from $\HGAS(n_*,d_*)$ to $\HGAS(n,d)$, and from $\HLS(n_*,d_*)$ to $\HLS(n,d)$. Undecidability and many-one hardness pass from the source to the target of a many-one reduction by \ref{C:trans}. Hence, if $\HGAS(n_*,d_*)$ is many-one complete, then $\HGAS(n,d)$ is many-one hard, and since it is r.e.\ by Proposition~\ref{prop:re}, it is many-one complete. In the same way, if $\HGAS(n_*,d_*)$ is undecidable, then so is $\HGAS(n,d)$, and if $\HLS(n_*,d_*)$ is many-one hard, respectively undecidable, then so is $\HLS(n,d)$.
\end{proof}

\paperpart{Part II. Existence of a single \texorpdfstring{$(n_*,d_*)$}{(n*, d*)}}
\section{From the halting problem to a Fourier coefficient on a torus}\label{sec:arith}

In Part II we prove Theorem~\ref{thm:fixedpair}. Since $\HGAS(n_*,d_*)$ is r.e.\ for every pair $(n_*,d_*)$ by Proposition~\ref{prop:re}, we must construct one fixed pair $(n_*,d_*)$ and a computable map $e\mapsto F_e$, from $\NN$ to rational-coefficient homogeneous polynomial vector fields on $\RR^{n_*}$ of degree $d_*$, such that $e\in K$ if and only if the origin of $\dot x=F_e(x)$ is GAS, where $K$ is the halting set of \ref{C:K}. Indeed, such a map shows $K\lem\HGAS(n_*,d_*)$, and since every r.e.\ set many-one reduces to $K$ by \ref{C:K}, transitivity \ref{C:trans} then shows that $\HGAS(n_*,d_*)$ is many-one hard. As explained in Section~\ref{subsec:overview2}, we obtain the equivalence \eqref{eq:overview2} from a chain of four equivalences, in which every object is computable from $e$, and we prove one of them in each of Sections~\ref{sec:arith}--\ref{sec:single} (Lemma~\ref{lem:constterm}, Corollary~\ref{cor:yes}, Proposition~\ref{prop:signgap} and Theorem~\ref{thm:wrapper}). In this section we prove the first one: $e\in K$ if and only if an explicit function $\mathfrak D_e$ on a torus has a nonzero constant Fourier coefficient. In Section~\ref{subsec:diophantine} we reduce membership in $K$ to the solvability of a fixed Diophantine system, and in Section~\ref{subsec:theta_constant_coefficient} we encode this solvability in the constant Fourier coefficient of $\mathfrak D_e$ (Lemma~\ref{lem:constterm}).

\subsection{A fixed Diophantine equation}\label{subsec:diophantine}
We first replace halting by the solvability of a Diophantine equation, using the theorem of Matiyasevich, Robinson, Davis and Putnam.

\begin{theorem}[MRDP]\label{thm:mrdp}
For every r.e.\ set $S\subseteq\NN$ there are $k\ge1$ and a polynomial $P\in\ZZ[t,y_1,\dots,y_k]$ such that for every $e\in\NN$, $e\in S$ if and only if $P(e,y)=0$ for some $y\in\NN^k$.
\end{theorem}

In \cite[Section~5.5]{Matiyasevich93} it is proved that every set which is semidecidable by a Turing machine (that is, r.e.) is Diophantine in the sense of \cite[(1.4.2)]{Matiyasevich93}, with unknowns ranging over $\NN$; this is Theorem~\ref{thm:mrdp}. See also the original papers \cite{DPR61,Matiyasevich70}. The function $\mathfrak D_e$ which we construct in Section~\ref{subsec:theta_constant_coefficient} sees each unknown only through its value and its square, so we bring the equation into the following form, without products of different unknowns or higher powers. For integer matrices $A,B\in\ZZ^{r\times q}$ and $c\in\ZZ^r$, write $n^{\circ2}=(n_1^2,\dots,n_q^2)$ and consider systems of equations of the form
\begin{equation}\label{eq:dql}
An^{\circ2}+Bn=c,\qquad n\in\ZZ^q,
\end{equation}
in which every equation is an integer linear combination of the unknowns and of their squares. We write $A_j,B_j\in\ZZ^r$ for the $j$-th columns of $A,B$. In the following lemma the input $e$ enters only through the right-hand side, and affinely; this is what will keep the dimension and the degree of the final vector field independent of $e$. We use the nonzero columns of $B$ in Section~\ref{sec:clocks} (Corollary~\ref{cor:irrational}).

\begin{lemma}[A fixed Diophantine equation]\label{lem:template}
There are integers $q,r\ge1$, matrices $A_*,B_*\in\ZZ^{r\times q}$ and vectors $c_0,c_1\in\ZZ^r$ such that every column of $B_*$ is nonzero and, with $c(e)=c_0+e\,c_1$,
\begin{equation}\label{eq:template}
e\in K\iff\exists n\in\ZZ^q:\ A_*n^{\circ2}+B_*n=c(e),\qquad e\in\NN.
\end{equation}
\end{lemma}

In the proof we compute $P$ by additions and multiplications, express each multiplication through squares by the polarization identity $2g'g''=(g'+g'')^2-g'^2-g''^2$, and use Lagrange's four-square theorem to let the unknowns range over $\ZZ$ instead of $\NN$.

\begin{proof}
Apply Theorem~\ref{thm:mrdp} to $S=K$ and fix the resulting polynomial $P(t,y)$. Write $P=P^+-P^-$, where $P^\pm$ have nonnegative coefficients, and fix two finite lists of instructions which compute $P^+(t,y)+1$ and $P^-(t,y)+1$, respectively, where each instruction produces either the constant $1$ or the sum or the product of two inputs or earlier results. Such lists exist because $P^\pm+1$ are nonzero polynomials with nonnegative integer coefficients; the summand $1$ is added because the zero polynomial cannot be produced by such instructions. Introduce an integer variable for $t$, for each $y_i$, and for the output of each instruction, and write down the following equations: $t=e$; $g=1$ for each instruction producing the constant $1$; $g=g'+g''$ for each addition; and for each multiplication $g=g'g''$ a new variable $w$ with the two equations
\begin{equation}\label{eq:polarization}
w=g'+g'',\qquad w^2-g'^2-g''^2-2g=0;
\end{equation}
finally, the equation equating the two outputs, and, for each $i$, four new variables $a_{i1},\dots,a_{i4}$ and the equation $y_i=a_{i1}^2+a_{i2}^2+a_{i3}^2+a_{i4}^2$. Every equation is an integer linear combination of the variables and their squares, and only the right-hand side of the equation $t=e$ depends on $e$, affinely. Thus the system has the form \eqref{eq:dql} with fixed $A,B$ and right-hand side $c_0+ec_1$.

If $P(e,y)=0$ with $y\in\NN^k$, assign to every variable its value in the computation, choose the $a_{i\cdot}$ by Lagrange's four-square theorem \cite[Theorem~369]{HardyWright} (and $a_{i\cdot}=0$ if $y_i=0$), and put $w=g'+g''$; this solves the system. Conversely, in an integer solution every $y_i$ is a sum of four squares, hence in $\NN$; the two equations \eqref{eq:polarization} give $2g=(g'+g'')^2-g'^2-g''^2=2g'g''$, so every instruction is computed correctly, and $P^+(e,y)+1=P^-(e,y)+1$. Hence the system is solvable if and only if $e\in K$.

Finally, to make every column of $B$ nonzero, add for every variable $x_j$ a duplicate variable $x_j'$ and the equation $x_j-x_j'=0$, leaving the old equations unchanged. If $x$ solves the old system, $(x,x)$ solves the new one, and conversely the first half of a solution of the new system solves the old one. In the new equation for $x_j$ the column of $x_j$ has entry $+1$ and that of $x_j'$ has entry $-1$, so every column of the new linear coefficient matrix is nonzero. The new system still has the form \eqref{eq:dql}, with the right-hand side $c_0+ec_1$ extended by zeros, and it is solvable exactly when the old one is; this proves \eqref{eq:template}.
\end{proof}

We fix $q,r,A_*,B_*,c_0,c_1$ once and for all, and write $A=A_*$, $B=B_*$ from now on. These are fixed integers, integer matrices and integer vectors. We obtain them from a fixed universal machine and a fixed MRDP polynomial, and a computer program may contain them; we do not need to know their numerical values.

\subsection{Theta series and the constant Fourier coefficient}\label{subsec:theta_constant_coefficient}

We now encode the solvability of \eqref{eq:template} in one function $\mathfrak D_e$ on a torus. It is a product of series over all candidate values of the unknowns, in which each term is a positive weight times a character recording the residual $An^{\circ2}+Bn-c(e)$; averaging over the torus then keeps only the solutions, and their weights cannot cancel (Lemma~\ref{lem:constterm}). The idea of representing a recursively enumerable set through the Fourier coefficients of a function built from theta functions, in such a way that all these coefficients are nonnegative and a coefficient is positive exactly when the corresponding statement holds, goes back to Buser and Scarpellini~\cite[Definition~1 and Proposition~3.4]{BuserScarpellini} (see also~\cite{BuserScarpellini16}). We follow the same idea; our function $\mathfrak D_e$ is adapted to the form \eqref{eq:template}, in which the input $e$ enters only through the right-hand side, and to its realization by a polynomial system in Section~\ref{sec:orbits}. Fix
\begin{equation}\label{eq:lambda}
r_0=\tfrac14,\qquad\lambda=r_0^4=\tfrac1{256},\qquad\eta=\tfrac12,
\end{equation}
and for complex numbers $Q,V,W$ of modulus one define the theta series $\Theta_\lambda$, which runs over all integers $n$ and records $n^2$ in the power of $Q$ and $n$ in the power of $V$, and the rational function $P_\eta$, which runs over an auxiliary integer $s$ that averaging will force to equal $n^2$:
\begin{equation}\label{eq:ThetaP}
\Theta_\lambda(Q,V)=\sum_{n\in\ZZ}\lambda^{n^2}Q^{n^2}V^n,\qquad P_\eta(W)=\sum_{s\in\ZZ}\eta^{|s|}W^s=\frac{1-\eta^2}{(1-\eta W)(1-\eta W^{-1})}.
\end{equation}
Both series converge absolutely and uniformly. We use a torus $\TT^{q+r}=(\RR/2\pi\ZZ)^{q+r}$ with coordinates $\varphi=(\varphi_1,\dots,\varphi_{q+r})$; we call $\varphi_1,\dots,\varphi_q$ the ``private phases'' (one for each variable of \eqref{eq:template}) and $\varphi_{q+1},\dots,\varphi_{q+r}$ the ``equation phases'' (one for each equation); each of them is ``responsible'' for one of the conditions which a solution must satisfy, as in Section~\ref{subsec:overview2}. Put
\begin{equation}\label{eq:characters}
Q_j=e^{4i\varphi_j},\qquad\chi_l=e^{4i\varphi_{q+l}},\qquad V_j=\prod_{l=1}^r\chi_l^{B_{lj}},\qquad S_j=\prod_{l=1}^r\chi_l^{A_{lj}},
\end{equation}
for $1\le j\le q$ and $1\le l\le r$. In the function
\begin{equation}\label{eq:detector}
\mathfrak D_e(\varphi)=\prod_{l=1}^r\chi_l^{-c_l(e)}\ \prod_{j=1}^q\Theta_\lambda(Q_j,V_j)\,P_\eta(Q_j^{-1}S_j)
\end{equation}
the factor $P_\eta(Q_j^{-1}S_j)$ transfers the square $n_j^2$ from the private phase of the $j$-th unknown to the equation phases, and $\prod_l\chi_l^{-c_l(e)}$ subtracts the right-hand side.
The fourth powers in \eqref{eq:characters} will allow us to take fourth roots of $Q_j$ and square roots of $V_j$ as monomials in $e^{i\varphi}$ (Section~\ref{sec:jacobi}); they do not affect the following computation.

\begin{lemma}[Positive constant Fourier coefficient]\label{lem:constterm}
The Fourier series of $\mathfrak D_e$ converges absolutely, and its constant Fourier coefficient is
\begin{equation}\label{eq:constterm}
(\hat{\mathfrak D}_e)_{(0,\dots,0)}=\frac1{(2\pi)^{q+r}}\int_{\TT^{q+r}}\mathfrak D_e(\varphi)\,d\varphi=\sum_{\substack{n\in\ZZ^q\\An^{\circ2}+Bn=c(e)}}(\lambda\eta)^{|n|^2}.
\end{equation}
In particular $(\hat{\mathfrak D}_e)_{(0,\dots,0)}$ is a positive real number if $e\in K$, and it is zero if $e\notin K$.
\end{lemma}

\begin{proof}
Expanding the $j$-th factor gives $\sum_{n_j,s_j\in\ZZ}\lambda^{n_j^2}\eta^{|s_j|}Q_j^{n_j^2-s_j}V_j^{n_j}S_j^{s_j}$, a series of characters with absolutely summable coefficients. Multiplying these $q$ series and $\prod_l\chi_l^{-c_l(e)}$ writes $\mathfrak D_e$ as a series over $(n,s)\in\ZZ^q\times\ZZ^q$ of characters with coefficients of modulus $\prod_j\lambda^{n_j^2}\eta^{|s_j|}$, which are summable. Hence this series, and the Fourier series of $\mathfrak D_e$ obtained from it by collecting equal characters, converge absolutely and uniformly on $\TT^{q+r}$, and the series may be integrated term by term. The character $Q_j^{n_j^2-s_j}$ involves only $\varphi_j$, so a term survives integration over the private phases only if $s_j=n_j^2$ for all $j$. The remaining character is $\prod_l\chi_l^{\sum_j(A_{lj}n_j^2+B_{lj}n_j)-c_l(e)}$, which survives integration over the equation phases only if $An^{\circ2}+Bn=c(e)$. The surviving terms are $\prod_j\lambda^{n_j^2}\eta^{n_j^2}=(\lambda\eta)^{|n|^2}$. They are positive, so there is no cancellation, and the sum is nonempty exactly when \eqref{eq:template} has a solution, that is, when $e\in K$ by Lemma~\ref{lem:template}.
\end{proof}

\begin{example}\label{ex:detector}
For the single equation $n^2+n=2$, with one private phase and one equation phase, the function $\mathfrak D_e$ is
\[
\chi^{-2}\Theta_\lambda(Q,\chi)P_\eta(Q^{-1}\chi)=\sum_{n,s\in\ZZ}\lambda^{n^2}\eta^{|s|}Q^{n^2-s}\chi^{n+s-2}.
\]
Integration over the private phase keeps only the terms with $s=n^2$, and integration over the equation phase then keeps only those with $n^2+n=2$. The two solutions $n=1$ and $n=-2$ contribute $(\lambda\eta)+(\lambda\eta)^4>0$.
\end{example}

In Section~\ref{sec:orbits} we realize every factor of $\mathfrak D_e$ as a coordinate, or a polynomial in coordinates, of a polynomial differential system with rational coefficients, driven by a linear flow on the torus. The rational factors $P_\eta$ are easy (Section~\ref{sec:source}); we treat the theta factors in Sections~\ref{sec:modular} and~\ref{sec:jacobi}. The factor $\prod_l\chi_l^{-c_l(e)}$ has a degree depending on $e$, and we replace it by one oscillator whose coefficients, but not its degree, depend on $e$ (Section~\ref{subsec:oscillator}).

\section{From a Fourier coefficient to trajectories of a polynomial system}\label{sec:orbits}

In this section we prove the second equivalence in the chain of Section~\ref{subsec:overview2}. We realize the function $\mathfrak D_e$ of Section~\ref{sec:arith} by a polynomial system $\dot x=X_e(x)$ with rational coefficients, together with a polynomial map $h_e$ whose coordinates are the ``tests'' of Section~\ref{subsec:overview2}, and we show that $e\in K$ if and only if no trajectory of $\dot x=X_e(x)$ which is recurrent and does not ``grow too fast'' passes all the tests (Corollary~\ref{cor:yes}). Along a linear flow on the torus whose frequencies are linearly independent over $\QQ$, the time average of $\mathfrak D_e$ equals its constant Fourier coefficient (Lemma~\ref{lem:avg}). In Section~\ref{sec:clocks} we produce such frequencies, square roots of distinct primes, as values of state variables at equilibria of a polynomial system. The factors of $\mathfrak D_e$ are theta series and simple rational functions; in Sections~\ref{sec:modular} and~\ref{sec:jacobi} we produce the theta series as outputs of polynomial systems for the theta constants and for the Jacobi theta functions. In Section~\ref{sec:source} we assemble these blocks into the system $\dot x=X_e(x)$ and define $h_e$ (Section~\ref{sec:the_source_system}), we show that every trajectory which is recurrent, does not grow too fast and passes the tests follows a linear flow on the torus along which the polynomial $\mathfrak D^{\mathrm{obs}}_e$ reproduces $\mathfrak D_e$ (Theorem~\ref{thm:selector} in Section~\ref{subsec:pass_tests}), and we describe the intended trajectories (Section~\ref{subsec:intended}). Consequently no such trajectory exists if $e\in K$, while the intended trajectories are such trajectories if $e\notin K$ (Corollary~\ref{cor:yes}). Figure~\ref{fig:circuit} shows how the blocks are connected.

Polynomial differential equations are known to simulate Turing machines~\cite{GCB}, and solutions of polynomial initial value problems of polynomial length even characterize polynomial-time computation~\cite{BGP17}. Buser and Scarpellini~\cite[Section~4]{BuserScarpellini} obtain Jacobi theta functions from algebraic differential equations with suitable initial data, and deduce undecidable statements about the Fourier coefficients of functions built from such solutions. In all of these results the relevant solutions start from specified initial states. Since GAS concerns every initial state, the main work in this section is to show that every trajectory which is recurrent, does not grow too fast and passes the tests behaves like an intended trajectory (Theorem~\ref{thm:selector}).
\begin{figure}[!t]
\centering
\definecolor{cB}{HTML}{E69F00}\definecolor{cC}{HTML}{009E73}%
\definecolor{cD}{HTML}{D55E00}\definecolor{cE}{HTML}{CC79A7}%
\begin{tikzpicture}[
  >={Stealth[length=2.2mm]},
  box/.style={draw=black!70, rounded corners=2pt, align=center, font=\footnotesize,
    inner sep=3pt, line width=0.5pt},
  clock/.style={box, fill=cB!22},
  gen/.style={box, fill=cC!18},
  poly/.style={box, fill=cE!22},
  drive/.style={->, line width=0.8pt},
  read/.style={->, line width=0.7pt, black!45},
  check/.style={->, line width=0.8pt, densely dashed},
  lab/.style={font=\scriptsize, inner sep=1.5pt, align=center},
  sw/.style={minimum width=0.45cm, minimum height=0.28cm, inner sep=0pt},
]
  \node[clock, text width=2.3cm] (clocks) at (1.25,1.95)
    {\textbf{clocks}\\ (Section~\ref{sec:clocks})\\[2pt]
     {\scriptsize irrational speeds;\\ one per unknown\\ and one per equation}};

  \node[gen, text width=4.6cm] (osc) at (9.6,4.75)
    {\textbf{oscillator} (Section~\ref{subsec:oscillator})\\[2pt]
     {\scriptsize makes the right-hand-side factor}\\
     {\scriptsize\textcolor{cD}{the only part which depends on $e$}}};

  \draw[black!55, dashed, rounded corners=4pt, line width=0.5pt] (4.0,-0.1) rectangle (11.4,3.85);
  \node[font=\footnotesize, black!70, anchor=south east, inner sep=3pt] at (11.4,-0.1) {repeated for each unknown};
  \node[gen, text width=3.0cm] (mod) at (5.8,2.98)
    {\textbf{modular block}\\ (Section~\ref{sec:modular})\\[2pt] {\scriptsize makes theta constants}};
  \node[gen, text width=3.0cm] (jac) at (5.8,0.9)
    {\textbf{Jacobi block}\\ (Section~\ref{sec:jacobi})\\[2pt] {\scriptsize makes a theta series}};
  \node[gen, text width=3.0cm] (rec) at (9.6,2.98)
    {\textbf{reciprocals}\\ (Section~\ref{sec:the_source_system})\\[2pt] {\scriptsize make a rational factor}};
  \draw[check] (mod.south) -- node[lab, right] {theta constants\\ as coefficients} (jac.north);

  \draw[drive] (clocks.north) |- node[lab, above, pos=0.7] {speeds of the equation phases} (osc.west);
  \draw[drive] (clocks.east) -- node[lab, above] {speeds and} node[lab, below] {phases} (4.0,1.95);

  \node[poly, text width=3.0cm] (out) at (14.4,1.5)
    {\textbf{output}\\[2pt] {\scriptsize product of all factors;\\ equals $\mathfrak D_e$ at the\\ phases of the clocks}};
  \draw[read] (osc.east) -| node[lab, above, pos=0.3, black!60] {right-hand-side factor} (15.1,0 |- out.north);
  \draw[read] (rec.east) -| node[lab, above, pos=0.35, black!60] {rational factor} (13.7,0 |- out.north);
  \draw[read] (jac.east) -- node[lab, above, pos=0.55, black!60] {theta series} (out.west |- jac.east);
  \node[poly, text width=15.6cm] (tests) at (8.05,-1.3)
    {\textbf{tests}: the algebraic relations and prescribed time averages of all coordinates, and the output;\\
     a trajectory passes the tests if every test has time average $0$ and the algebraic relations hold at all times};
  \draw[read] (clocks.south) -- (clocks.south |- tests.north);
  \draw[read] (7.6,-0.1) -- (7.6,0 |- tests.north);
  \draw[read] (out.south) -- (out.south |- tests.north);

  \begin{scope}[yshift=-2.55cm, font=\scriptsize, leg/.style={anchor=base west, inner sep=0pt}]
    \node[clock, sw] at (0.25,0) {};   \node[leg] at (0.55,-0.08) {clocks};
    \node[gen, sw] at (2.45,0) {};     \node[leg] at (2.75,-0.08) {blocks (both: groups of coordinates of $x$)};
    \node[poly, sw] at (9.15,0) {};    \node[leg] at (9.45,-0.08) {polynomials in $x$, evaluated along a trajectory};
    \draw[drive] (0,-0.45) -- (0.5,-0.45);  \node[leg] at (0.6,-0.53) {appears in the differential equations of};
    \draw[read] (5.95,-0.45) -- (6.45,-0.45); \node[leg] at (6.55,-0.53) {is used in};
    \draw[check] (8.9,-0.45) -- (9.4,-0.45); \node[leg] at (9.5,-0.53) {must agree with (checked by the tests)};
  \end{scope}
\end{tikzpicture}
\caption{The polynomial system $\dot x=X_e(x)$ of Section~\ref{sec:source} as a ``circuit'' (Section~\ref{subsec:overview2}), which computes the function $\mathfrak D_e$ of Section~\ref{sec:arith} along the phases of the clocks. The orange and green boxes are groups of coordinates of $x$; the blocks in the dashed frame are repeated for each of the $q$ unknowns (the Jacobi block includes the auxiliary coordinate $G_j$, and the reciprocals are $G_j^+$ and $G_j^-$). The purple boxes are the polynomials $\mathfrak D^{\mathrm{obs}}_e$ and $h_e$ in $x$. The clocks set the speeds of all other blocks, and their phases also enter the equations of the reciprocals; an arrow which ends at the dashed frame applies to every block inside it. Hence, along the intended trajectories, every block moves in step with the phases $\varphi(0)+\omega t$, and the output equals $\mathfrak D_e(\varphi(0)+\omega t)$ (Lemma~\ref{lem:graph}). The dashed arrow is not built into the differential equations: that the coefficients of the Jacobi block are the theta constants made by the modular block is one of the tests, which are listed in Table~\ref{tab:source}. In Section~\ref{sec:compact} the running total of the tests becomes a new coordinate.}
\label{fig:circuit}
\end{figure}
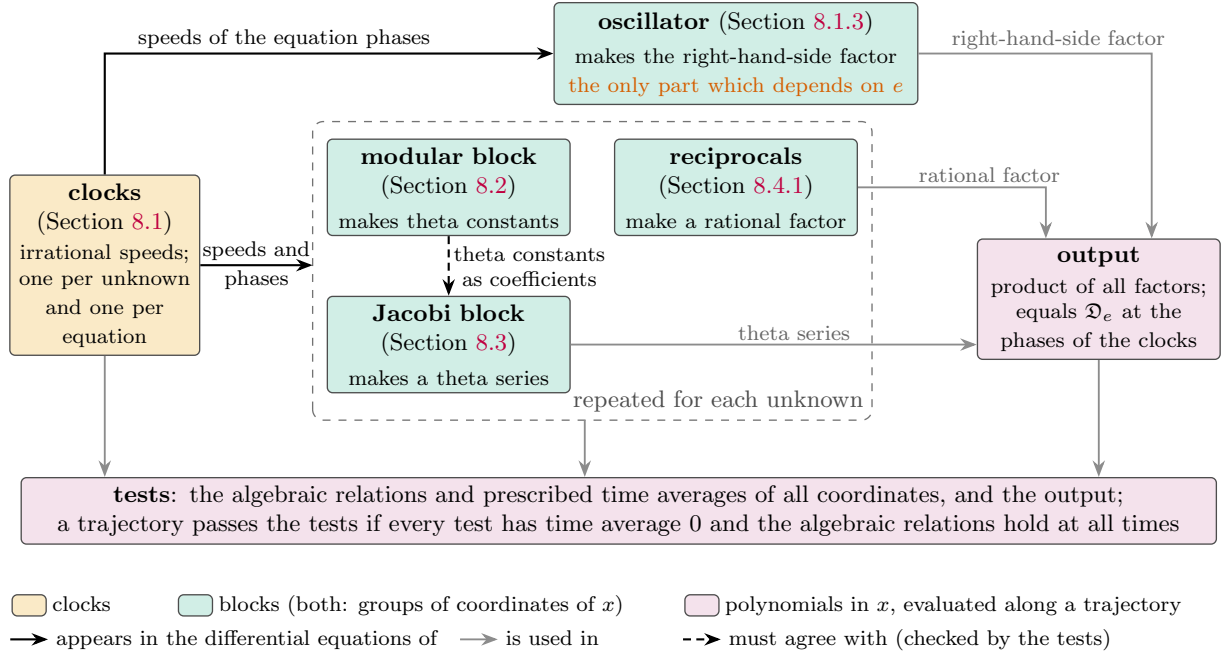

\subsection{Algebraic clocks}\label{sec:clocks}

We will drive the phases $\varphi$ of Section~\ref{sec:arith} by a linear flow $\dot\varphi=\omega$ whose frequencies are square roots of distinct primes. These square roots are irrational, so they cannot be coefficients of the vector field; we will instead obtain them as values of state variables at equilibria of a polynomial system with rational coefficients.

\subsubsection{Controllers and phase circles}
Each phase $\varphi_\ell$ of the torus of Section~\ref{sec:arith} must rotate with the irrational frequency $\sqrt{p_\ell}$, for a prime $p_\ell$. We produce this frequency with a controller, a planar system with rational coefficients whose only recurrent states on the unit circle are equilibria at which the polynomial $p_\ell r_\ell$ equals $\pm\sqrt{p_\ell}$ (Lemma~\ref{lem:controller}), and a phase circle, which then rotates with this frequency.

Let $p_1<p_2<\dots<p_{q+r}$ be the first $q+r$ prime numbers. For each $\ell\in\{1,\dots,q+r\}$ we use a \emph{controller} $(r_\ell,s_\ell)\in\RR^2$ and a \emph{phase circle} $\zeta_\ell\in\CC\cong\RR^2$, with the equations
\begin{equation}\label{eq:controller}
\dot r_\ell=-s_\ell(p_\ell r_\ell^2-1),\qquad\dot s_\ell=r_\ell(p_\ell r_\ell^2-1),\qquad\dot\zeta_\ell=i\,p_\ell r_\ell\,\zeta_\ell.
\end{equation}
We write $\omega_\ell=p_\ell r_\ell$; this is a polynomial in the state.

\begin{lemma}\label{lem:controller}
The functions $r_\ell^2+s_\ell^2$ and $|\zeta_\ell|^2$ are constant along solutions of \eqref{eq:controller}. Let $(r_\ell(t),s_\ell(t))$ be a solution with $r_\ell^2+s_\ell^2=1$ which is forward recurrent, that is, which returns arbitrarily close to $(r_\ell(0),s_\ell(0))$ at arbitrarily large times. Then it is one of the four equilibria $r_\ell=\epsilon_\ell p_\ell^{-1/2}$, $s_\ell=\pm(1-p_\ell^{-1})^{1/2}$, $\epsilon_\ell\in\{-1,1\}$, and then $\omega_\ell=\epsilon_\ell\sqrt{p_\ell}$ and $\zeta_\ell(t)=\zeta_\ell(0)e^{i\omega_\ell t}$.
\end{lemma}

\begin{proof}
The first two components of \eqref{eq:controller} are $(p_\ell r_\ell^2-1)$ times the rotation vector field $(-s_\ell,r_\ell)$, so $r_\ell^2+s_\ell^2$ is constant; and $\frac d{dt}|\zeta_\ell|^2=2\re(\bar\zeta_\ell\dot\zeta_\ell)=0$. On the unit circle write $(r_\ell,s_\ell)=(\cos\theta,\sin\theta)$; then $\dot\theta=p_\ell\cos^2\theta-1$. Since $p_\ell\ge2$, the right side has exactly four zeros on the circle, and on each open arc between consecutive zeros it has a constant sign. A solution starting on such an arc is therefore strictly monotone in $\theta$ and converges to an endpoint of the arc, so it does not return near its initial point. Hence a forward recurrent solution is an equilibrium, $p_\ell r_\ell^2=1$, and the last assertion follows from $\dot\zeta_\ell=i\omega_\ell\zeta_\ell$.
\end{proof}

\subsubsection{Independence of the frequencies}\label{subsec:freqindep}

We use square roots of distinct primes because no nontrivial integer combination of them vanishes \cite[Theorem~2]{Besicovitch}. Along the linear flow $\dot\varphi=\omega$ this makes the time average of every function with an absolutely convergent Fourier series equal to its constant Fourier coefficient (Lemma~\ref{lem:avg} below), which is how we will detect the constant Fourier coefficient \eqref{eq:constterm}. The following lemma also bounds these combinations from below; we use this bound in Section~\ref{sec:compact}.

\begin{lemma}\label{lem:radicals}
Let $p_1,\dots,p_d$ be distinct primes, $P=\max_ap_a$, $\epsilon\in\{-1,1\}^d$ and $k\in\ZZ^d\setminus\{0\}$. Then
\begin{equation}\label{eq:smalldiv}
\Bigl|\sum_{a=1}^dk_a\epsilon_a\sqrt{p_a}\Bigr|\ge\bigl(\sqrt P\,|k|_1\bigr)^{1-2^d}>0,\qquad|k|_1=\sum_a|k_a|.
\end{equation}
\end{lemma}

\begin{proof}
By \cite[Theorem~2]{Besicovitch}, applied to the positive square roots of the distinct primes $p_1,\dots,p_d$ and to polynomials of degree one, the numbers $\sqrt{p_1},\dots,\sqrt{p_d}$ are linearly independent over $\QQ$, so $L(\delta):=\sum_ak_a\delta_a\sqrt{p_a}\ne0$ for every $\delta\in\{-1,1\}^d$. Consider $\Pi=\prod_{\delta\in\{-1,1\}^d}L(\delta)$ as a polynomial with integer coefficients in the indeterminates $u_a$, after replacing $\sqrt{p_a}$ by $u_a$. Changing the sign of one $u_a$ permutes the factors, so $\Pi$ is unchanged; hence every monomial of $\Pi$ has even degree in every $u_a$, and $\Pi$ is a polynomial with integer coefficients in $u_a^2=p_a$. Thus $\Pi$ is a nonzero integer, and $|\Pi|\ge1$. Every factor satisfies $|L(\delta)|\le\sqrt P|k|_1$, and \eqref{eq:smalldiv} follows by dividing $|\Pi|\ge1$ by the $2^d-1$ factors with $\delta\ne\epsilon$.
\end{proof}

For the matrix $B$ of \eqref{eq:template} we define the frequencies
\begin{equation}\label{eq:freqs}
\Omega_j=4\omega_j,\qquad N_j=4\sum_{l=1}^rB_{lj}\,\omega_{q+l},\qquad j=1,\dots,q,
\end{equation}
which are polynomials in the controller coordinates, since $\omega_\ell=p_\ell r_\ell$. Fix signs $\epsilon\in\{-1,1\}^{q+r}$. When every controller is at an equilibrium of Lemma~\ref{lem:controller} with $\omega_\ell=\epsilon_\ell\sqrt{p_\ell}$, these polynomials take constant values; in the rest of Section~\ref{subsec:freqindep}, $\omega=(\epsilon_\ell\sqrt{p_\ell})_{\ell}$, $\Omega_j$ and $N_j$ denote these values. Along the linear flow $\varphi(t)=\varphi(0)+\omega t$ the characters \eqref{eq:characters} satisfy $Q_j(t)=Q_j(0)e^{i\Omega_jt}$ and $V_j(t)=V_j(0)e^{iN_jt}$.

\begin{corollary}\label{cor:irrational}
For every $j$, $\Omega_j\ne0$ and $N_j/\Omega_j\notin\QQ$. Moreover $k\cdot\omega\ne0$ for every $k\in\ZZ^{q+r}\setminus\{0\}$.
\end{corollary}

\begin{proof}
The last assertion is Lemma~\ref{lem:radicals}, and $\Omega_j=4\epsilon_j\sqrt{p_j}\ne0$. If $N_j/\Omega_j=a/b$ with integers $a,b$, $b\ne0$, then $0=bN_j-a\Omega_j=k\cdot\omega$ with $k_j=-4a$, $k_{q+l}=4bB_{lj}$, and all other entries zero. By Lemma~\ref{lem:radicals} this $k$ is zero; so $a=0$, and then $bB_j=0$ with $B_j\ne0$ (Lemma~\ref{lem:template}) gives $b=0$, a contradiction.
\end{proof}

The nonzero columns of $B$ are used through Corollary~\ref{cor:irrational} and in Lemma~\ref{lem:graph}(c); see Remark~\ref{rem:zeroratio}.

We use the following elementary averaging fact, a form of the Kronecker--Weyl equidistribution theorem, repeatedly. When the entries of $\omega$ below are rationally independent, it is a special case of the theorem on averages for conditionally periodic motions \cite[\S51]{ArnoldMMCM}, which holds for every continuous function on the torus; we need it under the weaker nonresonance condition of the lemma. A function $g$ on $\TT^{d}$ has an \emph{absolutely convergent Fourier series} if $g(\varphi)=\sum_{k\in\ZZ^d}\hat g_ke^{ik\cdot\varphi}$ with $\sum_k|\hat g_k|<\infty$.

\begin{lemma}\label{lem:avg}
Let $g$ have an absolutely convergent Fourier series on $\TT^d$, let $\varphi_0\in\TT^d$, and let $\omega\in\RR^d$ satisfy $k\cdot\omega\ne0$ whenever $\hat g_k\ne0$ and $k\ne0$. Then $\frac1T\int_0^Tg(\varphi_0+\omega t)\,dt\to\hat g_0$ as $T\to\infty$.
\end{lemma}

\begin{proof}
As in the first step of H.~Weyl's proof of the theorem on averages, as presented in \cite[\S51, Lemma~1]{ArnoldMMCM}, for each $k$ the average of $\hat g_ke^{ik\cdot(\varphi_0+\omega t)}$ over $[0,T]$ is $\hat g_k$ if $k=0$, and is bounded by $2|\hat g_k|/(T|k\cdot\omega|)\to0$ otherwise. Since $\sum_k|\hat g_k|<\infty$ and every character has modulus one, the Fourier series converges uniformly, so the average of $g$ over $[0,T]$ is the sum of the averages of the individual terms. These are bounded by $|\hat g_k|$ for every $T$, so dominated convergence for series gives the claim.
\end{proof}

\subsubsection{The right-hand-side oscillator}\label{subsec:oscillator}
The phase circles carry the characters of Section~\ref{sec:arith}: when $|\zeta_\ell|=1$ we write $\zeta_\ell=e^{i\varphi_\ell}$, so that \eqref{eq:characters} becomes $Q_j=\zeta_j^4$ and $\chi_l=\zeta_{q+l}^4$, with $\zeta_\ell^{-1}=\bar\zeta_\ell$ for negative powers (we use these identifications in Section~\ref{sec:source}). The factor $\prod_l\chi_l^{-c_l(e)}$ in \eqref{eq:detector} is then a monomial of degree $4|c(e)|_1$ in the phase circles and their conjugates, and this degree grows with $e$. We replace it by one additional state $R\in\CC$ with the equation
\begin{equation}\label{eq:oscillator}
\dot R=-4i\Bigl(\sum_{l=1}^rc_l(e)\,\omega_{q+l}\Bigr)R,\qquad\omega_{q+l}=p_{q+l}r_{q+l},
\end{equation}
which is bilinear in the state; only its coefficients depend on $e$. If the controllers are at equilibria and $|\zeta_\ell|=1$, then by Lemma~\ref{lem:controller} each $\omega_{q+l}=\epsilon_{q+l}\sqrt{p_{q+l}}$ is constant and $\dot\chi_l=4i\omega_{q+l}\chi_l$, so $\chi_l(t)=\chi_l(0)e^{4i\omega_{q+l}t}$, and the product rule gives
\[
\frac d{dt}\Bigl(R\prod_l\chi_l^{c_l(e)}\Bigr)=\Bigl(-4i\sum_lc_l(e)\,\omega_{q+l}+4i\sum_lc_l(e)\,\omega_{q+l}\Bigr)R\prod_l\chi_l^{c_l(e)}=0.
\]
Hence
\begin{equation}\label{eq:oscgraph}
R(t)=R_0\prod_{l=1}^r\chi_l(t)^{-c_l(e)},\qquad R_0=R(0)\prod_l\chi_l(0)^{c_l(e)}.
\end{equation}

\subsection{The theta-constant system}\label{sec:modular}

The theta series $\Theta_\lambda(Q,V)=\sum_n(\lambda Q)^{n^2}V^n$ of \eqref{eq:ThetaP} depends on $Q$ only through $\lambda Q$. In Section~\ref{sec:jacobi} we generate it, as a function of the phases of $Q$ and $V$, by a polynomial system whose coefficients are the theta constants $\vartheta_2,\vartheta_3,\vartheta_4$ and the series $E_2$ (recalled below) at the point $\tau$ of the upper half-plane with $e^{\pi i\tau}=\lambda Q$. In this subsection we generate these coefficients by a polynomial system with rational coefficients, the modular block. Along its intended solution $Q=e^{i\alpha}$ rotates, $\tau$ moves along a horizontal line, and each of its coordinates is an absolutely convergent Fourier series in the phase $\alpha$. Since GAS concerns every initial state, we also determine all other solutions, and we show that the only ones which are forward recurrent, satisfy a bound on time averages of a power of their norm, and have two prescribed time averages are the intended ones (Proposition~\ref{prop:modselect}).

\subsubsection{Theta constants}
In Section~\ref{sec:jacobi} we will generate the theta series $\Theta_\lambda(Q,V)$ by a differential system whose coefficients are theta constants $\vartheta_k(\tau)$ at the point $\tau\in\HH$ with $e^{\pi i\tau}=\lambda Q$. Here we collect what we need about theta constants: the polynomial differential equations \ref{th:ode}, which define the modular block in the next subsection, and the transformation laws and the modular function $\Lambda$, which describe all of its other solutions.

For $\tau$ in the upper half-plane $\HH$ put
\[
\vartheta_2(\tau)=\sum_{n\in\ZZ}e^{\pi i(n+\frac12)^2\tau},\qquad\vartheta_3(\tau)=\sum_{n\in\ZZ}e^{\pi in^2\tau},\qquad\vartheta_4(\tau)=\sum_{n\in\ZZ}(-1)^ne^{\pi in^2\tau},
\]
\[
E_2(\tau)=1-24\sum_{n\ge1}\sigma_1(n)e^{2\pi in\tau},
\]
where $\sigma_1(n)$ is the sum of the divisors of $n$. We use the following classical facts. For $\gamma=\bigl(\begin{smallmatrix}a&b\\c&d\end{smallmatrix}\bigr)$ in $SL_2(\RR)$ or $SL_2(\CC)$ we write $\gamma\tau=(a\tau+b)/(c\tau+d)$ and $j(\gamma,\tau)=c\tau+d$; recall $j(\gamma_1\gamma_2,\tau)=j(\gamma_1,\gamma_2\tau)j(\gamma_2,\tau)$ and $\im(\gamma\tau)=\im\tau/|j(\gamma,\tau)|^2$ for $\gamma\in SL_2(\RR)$.
\begin{enumerate}[label=($\vartheta$\arabic*)]
\item\label{th:nozero} $\vartheta_2,\vartheta_3,\vartheta_4$ and $E_2$ are holomorphic on $\HH$, since their series converge locally uniformly (note $\sigma_1(n)\le n^2$), and $\vartheta_2,\vartheta_3,\vartheta_4$ have no zeros in $\HH$. Indeed $\vartheta_k(\tau)$ is the value at $z=0$ of the theta function $\vartheta_k(z|\tau)$ of \cite{WW}, and the zeros of $\vartheta_2(z|\tau)$, $\vartheta_3(z|\tau)$, $\vartheta_4(z|\tau)$ are the points congruent modulo $\pi$ and $\pi\tau$ to $\frac12\pi$, $\frac12\pi+\frac12\pi\tau$, $\frac12\pi\tau$, respectively \cite[\S21.12]{WW}; none of these is congruent to $0$.
\item\label{th:jacobi} $\vartheta_3^4=\vartheta_2^4+\vartheta_4^4$ on $\HH$ \cite[Corollary~21.2.1]{WW}.
\item\label{th:transf} $(\vartheta_2^4,\vartheta_3^4,\vartheta_4^4)(\tau+1)=(-\vartheta_2^4,\vartheta_4^4,\vartheta_3^4)(\tau)$ and $(\vartheta_2^4,\vartheta_3^4,\vartheta_4^4)(-1/\tau)=-\tau^2(\vartheta_4^4,\vartheta_3^4,\vartheta_2^4)(\tau)$. These are the fourth powers of $\vartheta_2(\tau+1)=e^{\pi i/4}\vartheta_2(\tau)$, $\vartheta_3(\tau+1)=\vartheta_4(\tau)$, $\vartheta_4(\tau+1)=\vartheta_3(\tau)$ \cite[(9.3.2) at $z=0$]{Lawden} and of $(-i\tau)^{1/2}(\vartheta_2,\vartheta_3,\vartheta_4)(\tau)=(\vartheta_4,\vartheta_3,\vartheta_2)(-1/\tau)$, where $|\arg(-i\tau)|<\frac12\pi$ \cite[\S21.51]{WW}; note $((-i\tau)^{1/2})^4=-\tau^2$. The group $SL_2(\ZZ)$ is generated by $\bigl(\begin{smallmatrix}1&1\\0&1\end{smallmatrix}\bigr)$ and $\bigl(\begin{smallmatrix}0&-1\\1&0\end{smallmatrix}\bigr)$ \cite[\S1.2, Remark~2]{Zagier123}.
\item\label{th:E2} $E_2(\gamma\tau)=j(\gamma,\tau)^2E_2(\tau)-\frac{6ic}{\pi}j(\gamma,\tau)$ for $\gamma\in SL_2(\ZZ)$. This is \cite[Proposition~6, (19)]{Zagier123} multiplied by $6/\pi^2$: there $E_2=\frac6{\pi^2}G_2=1-24\sum_{n\ge1}\sigma_1(n)e^{2\pi in\tau}$ \cite[(17)]{Zagier123}, which is our $E_2$.
\item\label{th:lambda} Let $\Gamma(2)=\{\gamma\in SL_2(\ZZ):\gamma\equiv I\bmod 2\}$, and let $\bar\Gamma(2)$ be its image in $PSL_2(\RR)$. The function $\Lambda=\vartheta_2^4/\vartheta_3^4$ is holomorphic on $\HH$ with values in $\CC\setminus\{0,1\}$, it takes every value in $\CC\setminus\{0,1\}$, and $\Lambda(\tau)=\Lambda(\tau')$ holds if and only if $\tau'=\gamma\tau$ for some $\gamma\in\Gamma(2)$ \cite[Chapter~7, \S3.4, (24), and Theorems~7 and~8; Chapter~8, \S3.1]{Ahlfors}. Here Ahlfors's modular function $\lambda(\tau)=(e_3-e_2)/(e_1-e_2)$, formed with the periods $1$ and $\tau$, equals $\Lambda$ by \cite[(6.9.2) and (6.9.3)]{Lawden}, since Ahlfors's $e_1,e_2,e_3$ are Lawden's $e_1,e_3,e_2$. Consequently $\Lambda$ is a covering map from $\HH$ onto $\CC\setminus\{0,1\}$, whose group of deck transformations is $\bar\Gamma(2)$, and $\bar\Gamma(2)$ acts freely on $\HH$ (see below). Moreover $\Gamma(2)$ is a \emph{lattice} in $SL_2(\RR)$, that is, a discrete subgroup such that $\Gamma(2)\backslash SL_2(\RR)$ carries a finite measure invariant under right multiplication by $SL_2(\RR)$: indeed $SL_2(\ZZ)$ is a lattice \cite[Chapter~9, Proposition~9.18, \pg{307}; Chapter~11, \S11.2, pp.~357--358]{EW}, $\Gamma(2)$ has index $6$ in $SL_2(\ZZ)$ \cite[Chapter~11, Example~11.10, \pg{358}]{EW}, and a subgroup of finite index in a lattice is a lattice \cite[Chapter~11, \S11.2.1, \pg{358}]{EW}. This is used only in Step~2 of the proof of Proposition~\ref{prop:modselect}.
\item\label{th:ode} With $D=\frac1{\pi i}\frac{d}{d\tau}$,
\[
\begin{gathered}
D\vartheta_2=\tfrac1{12}(E_2+\vartheta_3^4+\vartheta_4^4)\vartheta_2,\qquad D\vartheta_3=\tfrac1{12}(E_2+\vartheta_2^4-\vartheta_4^4)\vartheta_3,\\
D\vartheta_4=\tfrac1{12}(E_2-\vartheta_2^4-\vartheta_3^4)\vartheta_4,
\end{gathered}
\]
and $DE_2=\frac1{12}(2E_2^2-\vartheta_2^8-\vartheta_3^8-\vartheta_4^8)$ \cite[Proposition~1, (8)]{BLS}.
\end{enumerate}

The covering property in \ref{th:lambda} follows from the other assertions there. If $\gamma=\bigl(\begin{smallmatrix}a&b\\c&d\end{smallmatrix}\bigr)\in\Gamma(2)$, $\gamma\ne\pm I$, fixes a point of $\HH$, then $c\ne0$ and $(a+d)^2<4$, because the fixed points are the roots of $cz^2+(d-a)z-b$, whose discriminant is $(a+d)^2-4$. But $a$ and $d$ are odd and $4\mid bc$, so $ad\equiv1\bmod4$, hence $a\equiv d\bmod 4$ and $a+d\equiv2\bmod4$, a contradiction. So $\bar\Gamma(2)$ acts freely on $\HH$. Since $\bar\Gamma(2)$ is discrete, for every compact $P\subseteq\HH$ only finitely many $\gamma\in\bar\Gamma(2)$ satisfy $\gamma P\cap P\ne\emptyset$ \cite[Chapter~11, Lemma~11.2, \pg{348}]{EW}. Taking for $P$ a closed disc around $\tau\in\HH$ and shrinking it, we find a neighbourhood $U$ of $\tau$ with $\gamma U\cap U=\emptyset$ for all $\gamma\in\bar\Gamma(2)\setminus\{1\}$. Hence the quotient map $p:\HH\to\bar\Gamma(2)\backslash\HH$ is a covering map whose group of deck transformations is $\bar\Gamma(2)$ \cite[Proposition~1.40]{Hatcher}. By \ref{th:lambda}, $\Lambda=\phi\circ p$ for a continuous bijection $\phi:\bar\Gamma(2)\backslash\HH\to\CC\setminus\{0,1\}$, and $\phi$ is open because the nonconstant holomorphic function $\Lambda$ is open. So $\phi$ is a homeomorphism, and $\Lambda$ is a covering map with the same deck transformations as $p$.

In \ref{th:ode} we have rewritten \cite[(8)]{BLS} in terms of $E_2=12\eta/\pi^2$ (their $\eta$ is the Weierstrass series, \cite[(6)]{BLS}). Brezhnev, Lyakhovich and Sharapov also explain that the classical Darboux--Halphen, Chazy and Ramanujan equations are differential consequences or reductions of their systems; for instance, the logarithmic derivatives of the theta constants satisfy the Darboux--Halphen system~\cite[(1)--(2)]{BLS}. The next lemma says that every $\gamma\in SL_2(\ZZ)$ acts on the fourth powers of the theta constants as the two generators do in \ref{th:transf}: up to the factor $j(\gamma,\tau)^2$, they are permuted and their signs may change. It is used in Step~3 of the proof of Proposition~\ref{prop:modselect}, and its consequence, the invariance of the function $m$, in Step~2 of that proof.

\begin{lemma}\label{lem:thetagamma}
For every $\gamma\in SL_2(\ZZ)$ there is a $3\times3$ matrix $\sigma_\gamma$, with exactly one entry $\pm1$ in every row and every column and zeros elsewhere, such that
\[
(\vartheta_2^4,\vartheta_3^4,\vartheta_4^4)^T(\gamma\tau)=j(\gamma,\tau)^2\,\sigma_\gamma\,(\vartheta_2^4,\vartheta_3^4,\vartheta_4^4)^T(\tau),\qquad\tau\in\HH.
\]
Consequently the function $m(\tau)=\im(\tau)\max_{k=2,3,4}|\vartheta_k(\tau)|^4$ satisfies $m(\gamma\tau)=m(\tau)$ for $\gamma\in SL_2(\ZZ)$.
\end{lemma}

\begin{proof}
Write $\vec\vartheta=(\vartheta_2^4,\vartheta_3^4,\vartheta_4^4)^T$. For the two generators of $SL_2(\ZZ)$ the claim is \ref{th:transf}, with $j=1$ and $j=\tau$, respectively. If it holds for $\gamma_1$ and $\gamma_2$, then $\vec\vartheta(\gamma_1\gamma_2\tau)=j(\gamma_1,\gamma_2\tau)^2\sigma_{\gamma_1}\vec\vartheta(\gamma_2\tau)=j(\gamma_1\gamma_2,\tau)^2\sigma_{\gamma_1}\sigma_{\gamma_2}\vec\vartheta(\tau)$ by the identity $j(\gamma_1\gamma_2,\tau)=j(\gamma_1,\gamma_2\tau)j(\gamma_2,\tau)$. If it holds for $\gamma$, then applying it at $\gamma^{-1}\tau$ and using $j(\gamma^{-1},\tau)=j(\gamma,\gamma^{-1}\tau)^{-1}$ (the same identity with $\gamma_1=\gamma$, $\gamma_2=\gamma^{-1}$) gives it for $\gamma^{-1}$, with $\sigma_{\gamma^{-1}}=\sigma_\gamma^{-1}$. Products and inverses of signed permutation matrices are signed permutation matrices, so induction on the length of a word in the generators and their inverses proves the first claim. For the second, the first claim gives $\{|\vartheta_k(\gamma\tau)|^4\}_{k=2,3,4}=\{|j(\gamma,\tau)|^2|\vartheta_k(\tau)|^4\}_{k=2,3,4}$, and $\im(\gamma\tau)=\im\tau/|j(\gamma,\tau)|^2$.
\end{proof}

\subsubsection{The modular block}
We now turn \ref{th:ode} into a polynomial system, the \emph{modular block}. Its intended solution is $(\vartheta_2^4,\vartheta_4^4,E_2,\vartheta_3)$ at a point $\tau(t)$ which moves parallel to the real axis with speed $\Omega/\pi$ (Lemma~\ref{lem:modcanon}). Since GAS concerns every initial state, we also determine all other solutions on the invariant set $\{Y^4=L+C\}$ (see after Lemma~\ref{lem:modcanon}). Those with $LC(L+C)\ne0$ are given by the formula \eqref{eq:modfamily}, which evaluates the intended solution at $M\tau$ for some $M\in SL_2(\CC)$, multiplied by powers of $j(M,\tau)$ and with an additional term in the third component (Lemma~\ref{lem:modgeneral} and Proposition~\ref{prop:modcomplete}); those with $LC(L+C)=0$ are constant if they are recurrent (Lemma~\ref{lem:moddiv}). For a solution defined for all real times, $M$ maps the real line into $\HH$, and Lemma~\ref{lem:modshape} describes the two possible shapes of its image.

Define the polynomial map $\cD:\CC^4\to\CC^4$, $\cD(L,C,E,Y)=(\cD_L,\cD_C,\cD_E,\cD_Y)$, by
\begin{equation}\label{eq:modsys}
\begin{gathered}
\cD_L=\tfrac13L(E+L+2C),\qquad\cD_C=\tfrac13C(E-2L-C),\\
\cD_E=\tfrac16(E^2-L^2-LC-C^2),\qquad\cD_Y=\tfrac1{12}(E+L-C)Y.
\end{gathered}
\end{equation}
For a real number $\Omega\ne0$ the \emph{modular block} with frequency $\Omega$ is the system
\begin{equation}\label{eq:modblock}
\dot{\mathbf m}=i\Omega\,\cD(\mathbf m),\qquad\mathbf m=(L,C,E,Y)\in\CC^4,
\end{equation}
which after separating real and imaginary parts is a polynomial system on $\RR^8$ with rational coefficients, multiplied by $\Omega$. In the full construction $\Omega$ will be the polynomial $4\omega_j=4p_jr_j$ in a controller coordinate.

\begin{lemma}\label{lem:modcanon}
The map $\mathbf m^{\mathrm{can}}=(\vartheta_2^4,\vartheta_4^4,E_2,\vartheta_3)$ satisfies $D\mathbf m^{\mathrm{can}}=\cD(\mathbf m^{\mathrm{can}})$ and $Y^4=L+C$ on $\HH$. Consequently, for every $\tau_0\in\HH$, $t\mapsto\mathbf m^{\mathrm{can}}(\tau_0+\Omega t/\pi)$ solves \eqref{eq:modblock}.
\end{lemma}

\begin{proof}
By \ref{th:ode} and \ref{th:jacobi}, $D\vartheta_2^4=4\vartheta_2^4D\log\vartheta_2=\frac13\vartheta_2^4(E_2+\vartheta_2^4+2\vartheta_4^4)$, and similarly for the other components; for $\cD_E$ use $\vartheta_2^8+\vartheta_3^8+\vartheta_4^8=2(\vartheta_2^8+\vartheta_2^4\vartheta_4^4+\vartheta_4^8)$, which follows from \ref{th:jacobi}. The identity $Y^4=L+C$ is \ref{th:jacobi}. For the last assertion note $\frac{d}{dt}\mathbf m^{\mathrm{can}}(\tau_0+\Omega t/\pi)=\frac\Omega\pi\cdot\pi i\,D\mathbf m^{\mathrm{can}}$.
\end{proof}

We also record that $\{Y^4=L+C\}$ is invariant: $\frac{d}{dt}(Y^4-L-C)=\frac{i\Omega}3(E+L-C)(Y^4-L-C)$, by a direct computation. All solutions considered below lie on this set.

Part (a) of the next lemma says that composing the intended solution with any Möbius transformation $M\in SL_2(\CC)$, and correcting by powers of $j(M,\tau)=c\tau+d$ and by a term in $E$, gives again a solution of \eqref{eq:modblock}. It is the analogue for our system of the classical invariance of the Darboux--Halphen system under Möbius transformations of the independent variable~\cite[\pg{3165}]{ACH}. Part (b) says that every initial state with $LC(L+C)\ne0$ is the value of such a solution; there the ratio $L/(L+C)$ determines the point $M\tau^*$ through the modular function $\Lambda$, and the remaining freedom in $M$ and $\rho$ fits $L+C$, $E$ and $Y$.

\begin{lemma}\label{lem:modgeneral}
\begin{enumerate}[label=(\alph*)]
\item Let $M\in SL_2(\CC)$, let $U\subseteq\CC$ be open with $M(U)\subseteq\HH$, write $j(\tau)=j(M,\tau)$, and let $\rho$ be holomorphic on $U$ with $\rho^4=j^{-2}$. Then
\begin{equation}\label{eq:modfamily}
\mathbf m_{M,\rho}(\tau)=\Bigl(\frac{\vartheta_2^4(M\tau)}{j(\tau)^2},\ \frac{\vartheta_4^4(M\tau)}{j(\tau)^2},\ \frac{E_2(M\tau)}{j(\tau)^2}+\frac{6ic}{\pi j(\tau)},\ \rho(\tau)\vartheta_3(M\tau)\Bigr),
\end{equation}
where $c$ is the lower left entry of $M$, satisfies $D\mathbf m_{M,\rho}=\cD(\mathbf m_{M,\rho})$ and $Y^4=L+C$ on $U$.
\item Let $\mathbf m^*=(L^*,C^*,E^*,Y^*)$ satisfy $Y^{*4}=L^*+C^*$ and $L^*C^*(L^*+C^*)\ne0$, let $\tau^*\in\CC$, and let $T^*\in\HH$ with $\Lambda(T^*)=L^*/(L^*+C^*)$. Then there are $M,U,\rho$ as in (a), with $\tau^*\in U$, such that $M\tau^*=T^*$ and $\mathbf m_{M,\rho}(\tau^*)=\mathbf m^*$.
\end{enumerate}
\end{lemma}

\begin{proof}
(a) Write $T=M\tau$, so $dT/d\tau=j^{-2}$ and $dj/d\tau=c$, and let $\vartheta_k,E_2$ denote values at $T$. Then $D_\tau L=(\pi i)^{-1}\frac{d}{d\tau}(\vartheta_2^4j^{-2})=j^{-4}(D\vartheta_2^4)+\frac{2ic}\pi j^{-3}\vartheta_2^4$, while by Lemma~\ref{lem:modcanon}
\[
\cD_L(\mathbf m_{M,\rho})=\tfrac13\vartheta_2^4j^{-2}\Bigl(\frac{E_2+\vartheta_2^4+2\vartheta_4^4}{j^2}+\frac{6ic}{\pi j}\Bigr)=j^{-4}(D\vartheta_2^4)+\frac{2ic}{\pi}j^{-3}\vartheta_2^4.
\]
The components $C$, $E$ and $Y$ are checked in the same way, using $\rho'/\rho=-c/(2j)$ for $Y$. The identity $Y^4=L+C$ follows from \ref{th:jacobi}.

(b) Put $v^*=(L^*+C^*)/\vartheta_3^4(T^*)\ne0$ and choose $j^*$ with $(j^*)^{-2}=v^*$. Define $c=\pi j^*(E^*-v^*E_2(T^*))/(6i)$, $d=j^*-c\tau^*$, $a=(1+cT^*j^*)/j^*$ and $b=T^*j^*-a\tau^*$. Then $ad-bc=1$, $c\tau^*+d=j^*$ and $a\tau^*+b=T^*j^*$, so $M=\bigl(\begin{smallmatrix}a&b\\c&d\end{smallmatrix}\bigr)$ satisfies $M\tau^*=T^*$ and $j(M,\tau^*)=j^*$. By continuity there is a disc $U\ni\tau^*$ with $M(U)\subseteq\HH$ and $j(M,\cdot)\ne0$ on $U$; let $\rho$ be the holomorphic fourth root of $j(M,\cdot)^{-2}$ on the disc $U$ with $\rho(\tau^*)=Y^*/\vartheta_3(T^*)$, which exists since $j(M,\cdot)$ has no zero on $U$ and $(Y^*/\vartheta_3(T^*))^4=v^*=(j^*)^{-2}$. Now $L=v^*\vartheta_2^4(T^*)=(L^*+C^*)\Lambda(T^*)=L^*$, $C=v^*\vartheta_4^4(T^*)=(L^*+C^*)(1-\Lambda(T^*))=C^*$ by \ref{th:jacobi}, $E=v^*E_2(T^*)+6ic/(\pi j^*)=E^*$, and $Y=Y^*$ at $\tau^*$.
\end{proof}

Lemma~\ref{lem:modgeneral}(b) does not apply when $L$, $C$ or $L+C$ vanishes, because $\Lambda$ does not take the values $0$ and $1$. The next lemma shows that these initial states do not matter for us: on them the system reduces to a scalar equation whose nonconstant solutions decay to zero, so a recurrent solution there is an equilibrium. We call a solution $\mathbf m:\RR\to\CC^4$ \emph{forward recurrent} if $\mathbf m(t_k)\to\mathbf m(0)$ for some sequence $t_k\to\infty$.

\begin{lemma}\label{lem:moddiv}
Each of the sets $\{L=0\}$, $\{C=0\}$, $\{L+C=0\}$ is invariant under \eqref{eq:modblock}. If $\mathbf m:\RR\to\CC^4$ is a solution of \eqref{eq:modblock} on $\{Y^4=L+C\}$ with $L(0)C(0)(L(0)+C(0))=0$ which is forward recurrent, then $\mathbf m$ is constant.
\end{lemma}

\begin{proof}
We have $\cD_L=L\cdot\frac13(E+L+2C)$, $\cD_C=C\cdot\frac13(E-2L-C)$ and $\cD_L+\cD_C=(L+C)\cdot\frac13(E+L-C)$, so along a solution of \eqref{eq:modblock} each $z\in\{L,C,L+C\}$ satisfies a scalar linear equation $\dot z=g(t)z$ with continuous $g$; hence $z(t)=z(0)\exp\int_0^tg$, and $z$ either vanishes identically or never. On $\{L=0\}$ the function $u=E-C$ satisfies $\dot u=\frac{i\Omega}6u^2$; on $\{C=0\}$ the function $u=E+L$ satisfies the same equation; on $\{L+C=0\}$ we have $Y=0$, $C=-L$, and $u=E-L$ satisfies the same equation. A solution of $\dot u=\frac{i\Omega}6u^2$ defined on all of $\RR$ is $u(t)=u(0)/(1-i\Omega u(0)t/6)$, and if $u(0)\ne0$ then $u(t)\to0$ as $t\to\infty$, so $\mathbf m$ is not forward recurrent. Hence $u\equiv0$, and then \eqref{eq:modsys} shows that $\dot{\mathbf m}=0$ in each of the three cases.
\end{proof}

The next proposition shows that a solution defined for all real times with $LC(L+C)\ne0$ is given by the formula \eqref{eq:modfamily} with a single matrix $M$ on the whole real line. Lemma~\ref{lem:modgeneral}(b) gives such a formula near each time, with a matrix that might depend on the time; lifting the path $L/(L+C)$ through the covering map $\Lambda$ shows that it does not.

\begin{proposition}\label{prop:modcomplete}
Let $\mathbf m:\RR\to\CC^4$ be a solution of \eqref{eq:modblock} with $Y^4=L+C$ and $L(0)C(0)(L(0)+C(0))\ne0$. Then there is $M\in SL_2(\CC)$ with $M(\RR)\subseteq\HH$ such that, for every $t\in\RR$, $\mathbf m(t)=\mathbf m_{M,\rho}(\Omega t/\pi)$, where $\rho$ is the continuous function on $\RR$ given by $\rho(\Omega t/\pi)=Y(t)/\vartheta_3(M(\Omega t/\pi))$, which satisfies $\rho^4=j(M,\cdot)^{-2}$; here $\mathbf m_{M,\rho}$ denotes the right-hand side of \eqref{eq:modfamily} on $\RR$, with this continuous $\rho$.
\end{proposition}

\begin{proof}
By Lemma~\ref{lem:moddiv}, $L$, $C$ and $L+C$ never vanish, so $\ell(t)=L(t)/(L(t)+C(t))$ is a continuous path in $\CC\setminus\{0,1\}$. By \ref{th:lambda} and the lifting criterion for covering maps \cite[Proposition~1.33]{Hatcher}, which applies since $\RR$ is simply connected and locally path-connected, there is a continuous $T:\RR\to\HH$ with $\Lambda\circ T=\ell$. Fix $t^*\in\RR$ and apply Lemma~\ref{lem:modgeneral}(b) with $\tau^*=\Omega t^*/\pi$, $\mathbf m^*=\mathbf m(t^*)$ and $T^*=T(t^*)$; this gives $(M_{t^*},\rho_{t^*})$. By Lemma~\ref{lem:modgeneral}(a), $t\mapsto\mathbf m_{M_{t^*},\rho_{t^*}}(\Omega t/\pi)$ solves \eqref{eq:modblock} near $t^*$, and it agrees with $\mathbf m$ at $t^*$; by uniqueness the two agree near $t^*$. Hence $\Lambda(M_{t^*}(\Omega t/\pi))=\ell(t)$ near $t^*$, so $t\mapsto M_{t^*}(\Omega t/\pi)$ and $T$ are continuous lifts of $\ell$ which agree at $t^*$, and therefore agree near $t^*$ by uniqueness of lifts \cite[Proposition~1.34]{Hatcher}. Two Möbius maps which agree on an interval are equal as maps (their matrices agree up to sign), so $t^*\mapsto M_{t^*}$, as a Möbius map, is locally constant; since $\RR$ is connected, all $M_{t^*}$ define one Möbius map $M$, and $T(t)=M(\Omega t/\pi)$ for all $t$. The components $L,C,E$ of \eqref{eq:modfamily} do not change if the matrix of $M$ is replaced by its negative, and $\rho(\Omega t/\pi):=Y(t)/\vartheta_3(T(t))$ is continuous by \ref{th:nozero}, with $\rho^4=j(M,\cdot)^{-2}$ along $\Omega t/\pi$.
\end{proof}

By Proposition~\ref{prop:modcomplete}, a solution defined for all times is determined by a Möbius map $M$ with $M(\RR)\subseteq\HH$. The next lemma distinguishes two cases. If $M(\infty)\in\HH$, the solution tends to zero (case (a)), so it is not recurrent. Otherwise $M(\RR)$ is a \emph{horocycle}: a circle in $\HH$ tangent to the real axis at $M(\infty)$ (without this point), or a horizontal line if $M(\infty)=\infty$ (case (b)); we then write the solution in terms of the matrices $h_\tau=g\,u_{s\tau}\in SL_2(\RR)$, the form used in Step~2 of the proof of Proposition~\ref{prop:modselect}.

\begin{lemma}\label{lem:modshape}
Let $M\in SL_2(\CC)$ with $M(\RR)\subseteq\HH$. Then exactly one of the following holds.
\begin{enumerate}[label=(\alph*)]
\item $M(\infty)\in\HH$. Then $\mathbf m_{M,\rho}(\tau)\to0$ as $\tau\to\pm\infty$ in $\RR$, for every continuous $\rho$ with $\rho^4=j(M,\cdot)^{-2}$.
\item There are $g\in SL_2(\RR)$ and $s\in\RR\setminus\{0\}$ with $M\tau=g(i+s\tau)$ for all $\tau\in\RR$; then $M(\infty)=g(\infty)\in\RR\cup\{\infty\}$.
\end{enumerate}
In case (b), put $u_x=\bigl(\begin{smallmatrix}1&x\\0&1\end{smallmatrix}\bigr)$ and $h_\tau=g\,u_{s\tau}\in SL_2(\RR)$. Then $M\tau=h_\tau i$, and the components of $\mathbf m_{M,\rho}(\tau)$ satisfy
\begin{equation}\label{eq:modframe}
L=s\,\frac{\vartheta_2^4(h_\tau i)}{j(h_\tau,i)^2},\qquad C=s\,\frac{\vartheta_4^4(h_\tau i)}{j(h_\tau,i)^2},\qquad\max\{|L|,|C|,|L+C|\}=|s|\,m(h_\tau i).
\end{equation}
\end{lemma}

\begin{proof}
(a) Here the lower left entry $c$ of $M$ is nonzero, since otherwise $M(\infty)=\infty$. As $\tau\to\pm\infty$, $M\tau\to M(\infty)\in\HH$, so the theta constants and $E_2$ at $M\tau$ stay bounded, since they are continuous on $\HH$, while $|j(M,\tau)|=|c\tau+d|\to\infty$ and $|\rho(\tau)|=|j(M,\tau)|^{-1/2}$; hence every component of \eqref{eq:modfamily} tends to zero.

(b) If $M(\infty)\notin\HH$, then $x^*=M(\infty)$, a limit of points of $M(\RR)\subseteq\HH$, lies in $\RR\cup\{\infty\}$. Choose $g_0\in SL_2(\RR)$ with $g_0(\infty)=x^*$. Then $N=g_0^{-1}M$ fixes $\infty$, so $N\tau=\kappa\tau+b$ with $\kappa\ne0$, and $N(\RR)=g_0^{-1}(M(\RR))\subseteq\HH$. A line $\{\kappa\tau+b:\tau\in\RR\}$ lies in $\HH$ only if $\kappa\in\RR$ and $\im b>0$. Writing $b=x_1+iy_1$ and $h_1=\bigl(\begin{smallmatrix}\sqrt{y_1}&x_1/\sqrt{y_1}\\0&1/\sqrt{y_1}\end{smallmatrix}\bigr)$, we get $N\tau=h_1(i+s\tau)$ with $s=\kappa/y_1$, and $g=g_0h_1$ works.

For \eqref{eq:modframe}, $M$ is the composition of $g$ with $\tau\mapsto i+s\tau$, whose matrix is $\bigl(\begin{smallmatrix}s^{1/2}&is^{-1/2}\\0&s^{-1/2}\end{smallmatrix}\bigr)$ for a choice of $s^{1/2}\in\CC$. Hence, if $g=\bigl(\begin{smallmatrix}a&b\\c&d\end{smallmatrix}\bigr)$, then $j(M,\tau)=s^{-1/2}(c(i+s\tau)+d)=s^{-1/2}j(h_\tau,i)$, because $h_\tau$ has lower row $(c,cs\tau+d)$. This gives the formulas for $L$ and $C$, and $L+C=Y^4=s\vartheta_3^4(h_\tau i)/j(h_\tau,i)^2$. Since $|j(h_\tau,i)|^{-2}=\im(h_\tau i)$, the last identity follows.
\end{proof}

\subsubsection{Selection by two averages}
Along the intended solution, $\avg{Y}=1$ and $\avg{Ye^{-i\alpha}}=2\lambda$; these are the first two Fourier coefficients of $\vartheta_3$ in the phase $\alpha$. Proposition~\ref{prop:modselect} shows that these two averages, together with forward recurrence and a bound on the time averages of $|\mathbf m|^\mu$, single out the intended solution among all solutions of the modular block on $\{Y^4=L+C\}$. In Section~\ref{sec:the_source_system} they become two of the ``moments'', that is, two of the tests encoded in $h_e$.

For a function $f$ defined along a trajectory, we write $\avg{f}=\lim_{T\to\infty}\frac1T\int_0^Tf(t)\,dt$ when the limit exists. Let $\tau_\lambda=i\,\frac{\log256}{\pi}$, so that $e^{\pi i\tau_\lambda}=\lambda$.

\begin{proposition}\label{prop:modselect}
Let $\Omega\ne0$, let $\alpha(t)=\alpha_0+\Omega t$ with $\alpha_0\in\RR$, and let $\mathbf m:\RR\to\CC^4$ be a solution of \eqref{eq:modblock} with $Y^4=L+C$ such that
\begin{enumerate}[label=(\roman*)]
\item $\mathbf m$ is forward recurrent;
\item $\frac1T\int_0^T|\mathbf m(t)|^\mu\,dt$ is bounded for $T\ge1$, for some $\mu\ge1$;
\item $\avg{Y}=1$ and $\avg{Y e^{-i\alpha}}=2\lambda$.
\end{enumerate}
Put $\tau(t)=\tau_\lambda+\alpha(t)/\pi$. Then for all $t\in\RR$,
\begin{equation}\label{eq:modselect}
\mathbf m(t)=\mathbf m^{\mathrm{can}}(\tau(t)),\quad\text{that is,}\quad(L,C,E,Y)(t)=(\vartheta_2^4,\vartheta_4^4,E_2,\vartheta_3)(\tau(t)).
\end{equation}
\end{proposition}

\begin{proof}
The proof has three steps. In Step~1 we use the results of the previous subsection to write $\mathbf m$ in the form \eqref{eq:modfamily}, with $M$ as in case (b) of Lemma~\ref{lem:modshape}, so that $M$ maps $\RR$ into a horocycle tangent to the real axis at $x^*=M(\infty)$. In Step~2 we show that $x^*$ is rational or $\infty$, using hypothesis (ii) and the theorem of Dani and Smillie. In Step~3 we apply an element of $SL_2(\ZZ)$ which moves $x^*$ to $\infty$; afterwards $\mathbf m$ is given by theta constants along a horizontal line, and the averages (iii) determine the line and the order of the components.

\emph{Step 1 (reduction to case (b) of Lemma~\ref{lem:modshape}).} If $L(0)C(0)(L(0)+C(0))=0$, then $\mathbf m$ is constant by Lemma~\ref{lem:moddiv}, and $\avg{Ye^{-i\alpha}}=Y\avg{e^{-i\alpha}}=0\ne2\lambda$. Hence Proposition~\ref{prop:modcomplete} applies and gives $M$ with $\mathbf m(t)=\mathbf m_{M,\rho}(\Omega t/\pi)$. In case (a) of Lemma~\ref{lem:modshape}, $\mathbf m(t)\to0$ as $t\to\infty$ while $\mathbf m(0)\ne0$, contradicting (i). So we are in case (b), with $g\in SL_2(\RR)$, $s\ne0$, and base point $x^*=g(\infty)\in\RR\cup\{\infty\}$.

\emph{Step 2: $x^*$ is rational or $\infty$.} We first explain the idea. By Lemma~\ref{lem:modshape}, $M(\Omega t/\pi)=g_ti$ with $g_t=g\,u_{s\Omega t/\pi}$, so, modulo $\Gamma(2)$, the matrices $g_t$ run along an orbit of the horocycle flow defined below. By the theorem of Dani and Smillie, such an orbit is either periodic, or its time averages converge to integrals over the quotient $\Gamma(2)\backslash SL_2(\RR)$. If $x^*$ were irrational the orbit would not be periodic, and we show that the function on the quotient which bounds $|\mathbf m|^\mu$ from below has infinite integral, because $m(\tau)\ge\im\tau/2$ high up in $\HH$ (there $\vartheta_3$ is close to $1$, while the factor $\im\tau$ in $m$ grows); so the averages in (ii) would be unbounded.

Suppose now that $x^*$ is irrational. Let $X=\Gamma(2)\backslash SL_2(\RR)$ be the space of cosets $\Gamma(2)h$, with its Haar probability measure $m_X$: the probability measure induced by Haar measure on $SL_2(\RR)$, which exists because $\Gamma(2)$ is a lattice and is invariant under right multiplication by $SL_2(\RR)$ \cite[Chapter~9, Proposition~9.20, \pg{312}]{EW}. Let $h(r)x=x\,u_{-r}$ ($r\in\RR$, $x\in X$) be the horocycle flow of \cite[Chapter~11, \S11.5, \pg{378}]{EW}; the matrix $\bigl(\begin{smallmatrix}1&-r\\0&1\end{smallmatrix}\bigr)$ used there is $u_{-r}$. Since $\Gamma(2)$ is a lattice in $SL_2(\RR)$ by \ref{th:lambda}, the theorem of Dani and Smillie \cite{DaniSmillie}, in the form \cite[Chapter~11, Theorem~11.35, \pg{399}]{EW} (see also \cite{HubbardMiller}), gives: if $x\in X$ is not periodic, that is, $h(r)x\ne x$ for all $r>0$, then $\frac1T\int_0^T\phi(h(r)x)\,dr\to\int\phi\,dm_X$ as $T\to\infty$, for every compactly supported continuous $\phi$ on $X$. Let $J=\mathrm{diag}(1,-1)$. The map $\iota:X\to X$, $\iota(\Gamma(2)h)=\Gamma(2)JhJ^{-1}$, is well defined because $J\Gamma(2)J^{-1}=\Gamma(2)$; moreover $\iota\circ\iota$ is the identity, and $\iota(x\,u_r)=\iota(x)\,u_{-r}$ because $Ju_rJ^{-1}=u_{-r}$. By Lemma~\ref{lem:modshape}, the points $g_t=g\,u_{s\Omega t/\pi}$ satisfy $|s|\,m(g_ti)\le2|\mathbf m(t)|$. Put $\nu=|s\Omega|/\pi>0$, put $g'=g$ if $s\Omega<0$ and $g'=JgJ^{-1}$ if $s\Omega>0$, and let $x_0=\Gamma(2)g'$. Then $\Gamma(2)g_t=h(\nu t)x_0$ if $s\Omega<0$, and $\Gamma(2)g_t=\iota(h(\nu t)x_0)$ if $s\Omega>0$. The point $x_0$ is not periodic: otherwise $\gamma g'=g'u_{-r}$ for some $\gamma\in\Gamma(2)$ and $r>0$, so $\gamma=g'u_{-r}g'^{-1}\ne\pm I$ is parabolic with fixed point $g'(\infty)$, which is $x^*$ or $-x^*$ and hence irrational, but the fixed point of a parabolic $\bigl(\begin{smallmatrix}a&b\\c&d\end{smallmatrix}\bigr)\in SL_2(\ZZ)$ is $\infty$ or $(a-d)/(2c)\in\QQ$. The function $f(h)=m(hi)^\mu$ is continuous and, by Lemma~\ref{lem:thetagamma}, invariant under $SL_2(\ZZ)$ on the left, so it defines a continuous function on $X$. It satisfies $f\circ\iota=f$, because $JhJ^{-1}i=-\overline{hi}$ and $|\vartheta_k(-\bar\tau)|=|\vartheta_k(\tau)|$; hence $f(g_t)=f(h(\nu t)x_0)$ in both cases. Choose compactly supported continuous $0\le f_n\le f$ increasing to $f$. Then for every $n$
\[
\Bigl(\frac2{|s|}\Bigr)^\mu\liminf_{T\to\infty}\frac1T\int_0^T|\mathbf m(t)|^\mu\,dt\ge\liminf_{T\to\infty}\frac1T\int_0^Tf(h(\nu t)x_0)\,dt\ge\int f_n\,dm_X,
\]
where the first inequality holds because $f(h(\nu t)x_0)=f(g_t)=m(g_ti)^\mu\le(2/|s|)^\mu|\mathbf m(t)|^\mu$, and the second because $f\ge f_n$ and, after the change of variables $r=\nu t$, $\frac1T\int_0^Tf_n(h(\nu t)x_0)\,dt=\frac1{\nu T}\int_0^{\nu T}f_n(h(r)x_0)\,dr\to\int f_n\,dm_X$ by the theorem of Dani and Smillie. By monotone convergence $\int f_n\,dm_X\to\int f\,dm_X$, so the averages in (ii) have $\liminf$ at least $(|s|/2)^\mu\int f\,dm_X$. We claim that $\int f\,dm_X=\infty$, which contradicts (ii). Since $-I\in\Gamma(2)$, $X$ is also the quotient of $PSL_2(\RR)$ by $\bar\Gamma(2)$; let $\varpi:PSL_2(\RR)\to X$ be the quotient map. In the coordinates $hi=x+iy$ and the angle $\theta$ of the image under $h$ of the upward unit tangent vector at $i$, Haar measure on $PSL_2(\RR)$ is $dx\,dy\,d\theta/y^2$, and $m_X(\varpi(B'))$ is a constant multiple of $\int_{B'}dx\,dy\,d\theta/y^2$ for every Borel set $B'\subseteq PSL_2(\RR)$ on which $\varpi$ is injective \cite[Chapter~9, Lemma~9.16, \pg{306}, and Propositions~9.19 and~9.20, pp.~311--313; Chapter~11, \pg{357}]{EW}. Let $\mathcal C=\{x+iy:|x|<1,\ y>2\}$ and $B=\{h\in PSL_2(\RR):hi\in\mathcal C\}$. Then $\varpi$ is injective on $B$: if $\gamma\in\Gamma(2)$ has lower left entry $c\ne0$, then $|c|\ge2$ and $\im(\gamma\tau)\le1/(c^2\im\tau)<2$ for $\tau\in\mathcal C$, while the elements with $c=0$ are $\pm u_{2k}$, $k\in\ZZ$, and $u_{2k}(\mathcal C)$ meets $\mathcal C$ only for $k=0$. On $\mathcal C$ we have $|\vartheta_3(\tau)-1|\le2\sum_{n\ge1}e^{-2\pi n^2}<\frac1{10}$, because $|e^{\pi in^2\tau}|=e^{-\pi n^2y}$, so $m(\tau)\ge y|\vartheta_3(\tau)|^4\ge(\frac9{10})^4y\ge\frac y2$, and $f(h)\ge(y/2)^\mu$ for $h\in B$. Hence, with constants $\delta_1,\delta_2>0$ ($\delta_2$ comes from the integration over $x\in(-1,1)$ and over $\theta$),
\[
\int_Xf\,dm_X\ge\int_{\varpi(B)}f\,dm_X=\delta_1\int_Bf\,\frac{dx\,dy\,d\theta}{y^2}\ge\delta_2\int_2^\infty\Bigl(\frac y2\Bigr)^\mu\frac{dy}{y^2}=\infty,
\]
since $\mu\ge1$.

\emph{Step 3: the case $x^*\in\QQ\cup\{\infty\}$.} Now $x^*\in\QQ\cup\{\infty\}$, so there is $\gamma\in SL_2(\ZZ)$ with $\gamma(x^*)=\infty$. Then $\tilde M=\gamma M$ fixes $\infty$ and $\tilde M(\RR)\subseteq\HH$, so, as in the proof of Lemma~\ref{lem:modshape}, $\tilde M\tau=\kappa\tau+b$ with $\kappa\in\RR\setminus\{0\}$ and $\im b>0$. Put $\tilde T(t)=\tilde M(\Omega t/\pi)=\kappa\Omega t/\pi+b$. The matrix of $\tilde M$ is $\bigl(\begin{smallmatrix}\kappa^{1/2}&b\kappa^{-1/2}\\0&\kappa^{-1/2}\end{smallmatrix}\bigr)$, so $j(M,\tau)=\pm\kappa^{-1/2}j(\gamma^{-1},\tilde M\tau)$, and the lower left entry of $M$ is $\pm\kappa^{1/2}c'$, where $c'$ is that of $\gamma^{-1}$. Lemma~\ref{lem:thetagamma} and \ref{th:E2} applied to $\gamma^{-1}$ then turn \eqref{eq:modfamily} into
\[
(L,C,L+C)(t)=\kappa\,(\Xi_2,\Xi_4,\Xi_3)(\tilde T(t)),\qquad E(t)=\kappa E_2(\tilde T(t)),
\]
where $(\Xi_2,\Xi_3,\Xi_4)^T=\sigma_{\gamma^{-1}}(\vartheta_2^4,\vartheta_3^4,\vartheta_4^4)^T$; here the terms of $E$ involving $c'$ cancel: by \ref{th:E2} for $\gamma^{-1}$ at $\tilde T=\tilde T(t)$, the first term of $E$ in \eqref{eq:modfamily} at $\Omega t/\pi$ is $E_2(\gamma^{-1}\tilde T)/j(M,\Omega t/\pi)^2=\kappa E_2(\tilde T)-6ic'\kappa/(\pi j(\gamma^{-1},\tilde T))$, and the second is $6ic_M/(\pi j(M,\Omega t/\pi))=6ic'\kappa/(\pi j(\gamma^{-1},\tilde T))$, where $c_M=\pm\kappa^{1/2}c'$ is the lower left entry of $M$ and the two signs $\pm$ above agree. Moreover $\Xi_3=\Xi_2+\Xi_4$ by \ref{th:jacobi} at the points $\gamma^{-1}\tilde T(t)$. Since $\sigma_{\gamma^{-1}}$ is a signed permutation, $\Xi_3=\epsilon\vartheta_y^4$ for some $y\in\{2,3,4\}$ and $\epsilon=\pm1$. As $Y^4=L+C$, the continuous function $Y(t)/\vartheta_y(\tilde T(t))$ has constant fourth power $\epsilon\kappa$, so it is a constant $\tilde\rho$, and $Y(t)=\tilde\rho\,\vartheta_y(\tilde T(t))$.

Write $e^{\pi i\tilde T(t)}=e^{\pi ib}e^{i\kappa\Omega t}$, with $|e^{\pi ib}|<1$. The series for $\vartheta_y(\tilde T(t))$ converges absolutely and uniformly in $t$, so the averages of $Y$ and of $Ye^{-i\alpha}$ may be taken term by term, as in the proof of Lemma~\ref{lem:avg}: a term with nonzero frequency has average zero, and the averages of the terms are dominated by their absolutely summable coefficients. If $y=2$, every term $e^{\pi i(n+1/2)^2\tilde T(t)}$ has nonzero frequency $(n+\frac12)^2\kappa\Omega$, so $\avg{Y}=0$, contradicting (iii). Hence $y\in\{3,4\}$, $Y(t)=\tilde\rho\bigl(1+2\sum_{n\ge1}(\pm1)^ne^{\pi in^2b}e^{in^2\kappa\Omega t}\bigr)$ with the sign $+$ for $y=3$ and $-$ for $y=4$, and $\avg{Y}=\tilde\rho$. Thus $\tilde\rho=1$ and $\epsilon\kappa=1$, so $\kappa=\pm1$. Next, $Y(t)e^{-i\alpha(t)}$ has the terms $e^{-i\alpha_0}e^{-i\Omega t}$ and $2(\pm1)^ne^{\pi in^2b}e^{-i\alpha_0}e^{i(n^2\kappa-1)\Omega t}$, $n\ge1$, and only frequencies zero survive averaging. Since $\avg{Ye^{-i\alpha}}=2\lambda\ne0$, some $n\ge1$ has $n^2\kappa=1$; so $\kappa=1$, $n=1$, and $2(\pm1)e^{\pi ib}e^{-i\alpha_0}=2\lambda$. Hence $|e^{\pi ib}|=\lambda$, that is $\im b=\im\tau_\lambda$, and $e^{\pi i\tilde T(t)}=\pm\lambda e^{i\alpha(t)}=\pm e^{\pi i\tau(t)}$. Thus $\tilde T(t)-\tau(t)$ is a constant integer, even if $y=3$ and odd if $y=4$. In both cases $\vartheta_y(\tilde T(t))=\vartheta_3(\tau(t))$ and $E_2(\tilde T(t))=E_2(\tau(t))$, because $\vartheta_3(\tau+2)=\vartheta_3(\tau)$, $\vartheta_4(\tau+1)=\vartheta_3(\tau)$ and $E_2(\tau+1)=E_2(\tau)$. By \ref{th:transf}, $\Xi_2(\tilde T(t))$ and $\Xi_4(\tilde T(t))$ are, in some order, $\epsilon_2\vartheta_2^4(\tau(t))$ and $\epsilon_4\vartheta_4^4(\tau(t))$ with $\epsilon_2,\epsilon_4\in\{\pm1\}$. Since $L+C=Y^4=\vartheta_2^4(\tau)+\vartheta_4^4(\tau)$, we get $(\epsilon_2-1)\vartheta_2^4(\tau(t))=(1-\epsilon_4)\vartheta_4^4(\tau(t))$; the left side has average zero (all its frequencies are nonzero) and the right side has average $1-\epsilon_4$, so $\epsilon_4=1$, and then $\epsilon_2=1$ by \ref{th:nozero}. If $(L,C)(t)=(\vartheta_4^4,\vartheta_2^4)(\tau(t))$, then \eqref{eq:modblock} and Lemma~\ref{lem:modcanon} would give $\vartheta_4^4(E_2+\vartheta_4^4+2\vartheta_2^4)=\vartheta_4^4(E_2-2\vartheta_2^4-\vartheta_4^4)$ along $\tau(t)$, that is, after division by $2\vartheta_4^4\ne0$ (\ref{th:nozero}), $2\vartheta_2^4+\vartheta_4^4=0$, which is false since the left side has average one along $\tau(t)$. Hence $(L,C)(t)=(\vartheta_2^4,\vartheta_4^4)(\tau(t))$. Together with $E(t)=E_2(\tau(t))$ and $Y(t)=\vartheta_3(\tau(t))$, shown above, this proves \eqref{eq:modselect}.
\end{proof}

Proposition~\ref{prop:modselect} determines the theta constants, but the next subsection also needs square roots of $L$ and $C$: $\vartheta_3^2=Y^2$ is determined, while $\vartheta_2^2$ and $\vartheta_4^2$ are determined by $L$ and $C$ only up to sign, and the sign of $\vartheta_2^2$ also depends on $\tau$ modulo $4$, not only on $e^{\pi i\tau}$. We fix these signs by two further averages in Section~\ref{sec:source}.

\subsection{The Jacobi theta system}\label{sec:jacobi}

In this subsection we realize $\Theta_\lambda(Q,V)$, as a function of the phases of $Q$ and $V$, as a polynomial function, the \emph{output}, of the solution of a polynomial system with rational coefficients, the Jacobi block. Its coefficients are the squares of the theta constants and $E_2/12$, which in the full construction are supplied by the modular block of Section~\ref{sec:modular}. Along the intended solution the coordinates are Jacobi theta functions evaluated at a point that moves linearly in time. As in Section~\ref{sec:modular}, GAS concerns every initial state, so we determine all solutions with the intended coefficients (Proposition~\ref{prop:jaccont}), and we show that a forward recurrent one whose output has three prescribed time averages has output exactly $\Theta_\lambda$ (Proposition~\ref{prop:threemoments}).

\subsubsection{Jacobi theta functions}
The theta series $\Theta_\lambda(Q,V)$ is the Jacobi theta function $\theta_3(z|\tau)$ with $e^{\pi i\tau}=\lambda Q$ and $e^{2\pi iz}=V$, and by \ref{J:half} it is the exponential factor $e^{\pi iz+\pi i\tau/4}$ times $\theta_1$ at the shifted point $z+\frac{1+\tau}2$. Here we fix the normalization of the Jacobi theta functions and collect the classical facts we use; the identities \ref{J:brezhnev} will define the Jacobi block in the next subsection.

For $z\in\CC$ and $\tau\in\HH$ we use Jacobi's theta functions in the normalization with period one in $z$:
\[
\theta_1(z|\tau)=-i\sum_{n\in\ZZ}(-1)^ne^{\pi i(n+\frac12)^2\tau+(2n+1)\pi iz},\qquad\theta_2(z|\tau)=\sum_{n\in\ZZ}e^{\pi i(n+\frac12)^2\tau+(2n+1)\pi iz},
\]
\[
\theta_3(z|\tau)=\sum_{n\in\ZZ}e^{\pi in^2\tau+2\pi inz},\qquad\theta_4(z|\tau)=\sum_{n\in\ZZ}(-1)^ne^{\pi in^2\tau+2\pi inz}.
\]
These are entire in $z$ and holomorphic in $(z,\tau)$, $\theta_k(0|\tau)=\vartheta_k(\tau)$ for $k=2,3,4$, and $\theta_1(0|\tau)=0$. In the notation of \cite[\S21.11]{WW} and of \cite[(1.2.11)--(1.2.14)]{Lawden}, which agree, our $\theta_k(z|\tau)$ is their $\vartheta_k(\pi z|\tau)$; the facts below are the cited ones with $z$ replaced by $\pi z$. We write $\theta_k'=\partial_z\theta_k$, and put $\psi_k=\vartheta_k(\tau)^2$ ($k=2,3,4$) and $\varkappa =E_2(\tau)/12+(\psi_3^2+\psi_4^2)/12$. We use the following facts.
\begin{enumerate}[label=(J\arabic*)]
\item\label{J:period} $\theta_1(z+1)=-\theta_1(z)$, $\theta_2(z+1)=-\theta_2(z)$, $\theta_3(z+1)=\theta_3(z)$, $\theta_4(z+1)=\theta_4(z)$, and $\theta_k(z+\tau)=\epsilon_ke^{-\pi i\tau-2\pi iz}\theta_k(z)$ with $(\epsilon_1,\epsilon_2,\epsilon_3,\epsilon_4)=(-1,1,1,-1)$ \cite[(1.3.2)--(1.3.5)]{Lawden}. Moreover $|\theta_k(z|\tau+2)|=|\theta_k(z|\tau)|$, since replacing $\tau$ by $\tau+1$ multiplies $\theta_1$ and $\theta_2$ by $e^{\pi i/4}$ and interchanges $\theta_3$ and $\theta_4$ \cite[(9.3.2)]{Lawden}.
\item\label{J:zeros} The zeros of $\theta_1(\cdot|\tau)$ are the points of $\ZZ+\tau\ZZ$, those of $\theta_2$, $\theta_3$, $\theta_4$ are the points of $\frac12+\ZZ+\tau\ZZ$, $\frac12+\frac\tau2+\ZZ+\tau\ZZ$, $\frac\tau2+\ZZ+\tau\ZZ$, respectively, and all are simple \cite[\S21.12]{WW}. In particular the four functions have no common zero.
\item\label{J:parity} $\theta_1$ is odd in $z$, and $\theta_2,\theta_3,\theta_4$ are even \cite[\S21.11]{WW}.
\item\label{J:heat} $4\pi i\,\partial_\tau\theta_k=\partial_z^2\theta_k$. This is the equation $\frac{\pi i}4\partial_z^2y+\partial_\tau y=0$ satisfied by $y=\vartheta_k(z|\tau)$ \cite[\S21.4]{WW}, with $z$ replaced by $\pi z$.
\item\label{J:half} $\theta_3(z|\tau)=e^{\pi iz+\pi i\tau/4}\,\theta_1\bigl(z+\tfrac{1+\tau}2\big|\tau\bigr)$, as follows directly from the series; see \cite[(1.3.8)]{Lawden}.
\item\label{J:quad} $\psi_2\theta_4(z)^2-\psi_4\theta_2(z)^2-\psi_3\theta_1(z)^2=0$ and $\psi_2\theta_3(z)^2-\psi_3\theta_2(z)^2-\psi_4\theta_1(z)^2=0$; these are the first and the third of the four relations between the squares of the theta functions in \cite[\S21.2]{WW}.
\item\label{J:brezhnev} With $(k,\nu,\mu)$ running over $(2,3,4),(3,4,2),(4,2,3)$,
\[
\frac1\pi\Bigl(\frac{\theta_k}{\theta_1}\Bigr)'=-\psi_k\frac{\theta_\nu\theta_\mu}{\theta_1^2},\qquad\frac1{\pi^2}\theta_1''=\frac{(\theta_1')^2}{\pi^2\theta_1}-\psi_3\psi_4\frac{\theta_2^2}{\theta_1}-4\varkappa \theta_1
\]
\cite[Theorem~4.1, (27)]{Brezhnev}; in \cite{Brezhnev} the constant $\eta$ equals $\pi^2E_2/12$.
\end{enumerate}

\subsubsection{The Jacobi block}
We now turn \ref{J:brezhnev} and the heat equation \ref{J:heat} into a polynomial system, the \emph{Jacobi block}, whose coefficients $\psi_k$ and $H$ are the theta constants produced by the modular block. Its intended solution consists of $\theta_1,\dots,\theta_4$, $\theta_1'/\pi$ and $1/\theta_1$ at $(w(t)|\tau(t))$, where $\tau(t)$ moves with the phase of $Q$ and $w(t)$ also with the phase of $V$, as in \eqref{eq:tauw}; the coordinates $\Delta$ and $U$ stand for $\theta_1'/\pi$ and $1/\theta_1$, and they make the equations \ref{J:brezhnev}, which involve $\theta_1'$ and division by $\theta_1$, polynomial. As for the modular block, we determine all other solutions with the intended coefficients on which the relations \eqref{eq:R12} and $UT_1=1$ hold: they are the solutions of Lemma~\ref{lem:bloch}, in which the argument $w$ of the theta functions is shifted to $w+a_0\tau+b_0$ and the result is multiplied by an explicit exponential factor (Proposition~\ref{prop:jaccont}).

The Jacobi block has the ten complex coordinates
\[
\mathbf J=(\psi_2,\psi_3,\psi_4,H,T_1,T_2,T_3,T_4,\Delta,U)\in\CC^{10}.
\]
Put $\varkappa =H+(\psi_3^2+\psi_4^2)/12$. We use the polynomial vector fields $\mathcal Z$ and $\mathcal J$ on $\CC^{10}$ given by $\mathcal Z\psi_k=\mathcal ZH=0$,
\begin{equation}\label{eq:Zfield}
\begin{gathered}
\mathcal ZT_1=\Delta,\qquad\mathcal ZT_k=U(\Delta T_k-\psi_kT_\nu T_\mu),\\
\mathcal Z\Delta=U(\Delta^2-\psi_3\psi_4T_2^2)-4\varkappa T_1,\qquad\mathcal ZU=-U^2\Delta,
\end{gathered}
\end{equation}
and
\begin{equation}\label{eq:Phicoef}
\begin{gathered}
\mathcal J \psi_2=2\Bigl(H+\frac{\psi_3^2+\psi_4^2}{12}\Bigr)\psi_2,\qquad\mathcal J \psi_3=2\Bigl(H+\frac{\psi_2^2-\psi_4^2}{12}\Bigr)\psi_3,\\
\mathcal J \psi_4=2\Bigl(H-\frac{\psi_2^2+\psi_3^2}{12}\Bigr)\psi_4,\qquad\mathcal J H=2H^2-\frac{\psi_2^4+\psi_3^4+\psi_4^4}{144},
\end{gathered}
\end{equation}
\begin{equation}\label{eq:Phijac}
\begin{aligned}
\mathcal J T_1&=-\tfrac14\Delta^2U+\tfrac14\psi_3\psi_4T_2^2U+\varkappa T_1,\\
\mathcal J T_k&=-\tfrac14\Delta^2U^2T_k+\tfrac12\psi_k\Delta T_\nu T_\mu U^2+\tfrac14\bigl(\psi_3\psi_4T_2^2-\psi_k\psi_\mu T_\nu^2-\psi_k\psi_\nu T_\mu^2\bigr)U^2T_k+\varkappa T_k,\\
\mathcal J\Delta&=-\tfrac14\Delta^3U^2+3\bigl(\tfrac14\psi_3\psi_4T_2^2U^2+\varkappa \bigr)\Delta-\tfrac12\psi_2\psi_3\psi_4T_2T_3T_4U^2,\qquad\mathcal J U=-U^2\,\mathcal J T_1,
\end{aligned}
\end{equation}
with $(k,\nu,\mu)$ as in \ref{J:brezhnev}. All coefficients are rational. These are the equations of \cite[Theorem~4.1 and Theorem~4.2]{Brezhnev} with the powers of $\pi$ removed; the relations below are all that we use. For real numbers $\Omega\ne0$ and $N$, the \emph{Jacobi block} with frequencies $(\Omega,N)$ is the system
\begin{equation}\label{eq:jacblock}
\dot{\mathbf J}=i\Omega\,\mathcal J(\mathbf J)+\tfrac12(\Omega+N)\,\mathcal Z(\mathbf J).
\end{equation}
In the full construction $\Omega=\Omega_j$ and $N=N_j$ are the polynomials \eqref{eq:freqs} in the controller coordinates.

Let $\alpha(t)=\alpha_0+\Omega t$ and $\beta(t)=\beta_0+Nt$ with $\alpha_0,\beta_0\in\RR$, and put
\begin{equation}\label{eq:tauw}
\tau(t)=\tau_\lambda+\frac{\alpha(t)}\pi,\qquad z(t)=\frac{\beta(t)}{2\pi},\qquad w(t)=z(t)+\frac{1+\tau(t)}2.
\end{equation}
For a function $f(w,\tau)$ we have $\frac{d}{dt}f(w(t),\tau(t))=i\Omega\cdot\frac1{\pi i}\partial_\tau f+\frac12(\Omega+N)\cdot\frac1\pi\partial_wf$. This explains the form of \eqref{eq:jacblock}: $\mathcal Z$ is $\pi^{-1}\partial_w$ and $\mathcal J$ is $(\pi i)^{-1}\partial_\tau$ on the theta functions below.

\begin{lemma}\label{lem:bloch}
Let $a_0,b_0,C_0\in\CC$, $C_0\ne0$, and define
\begin{equation}\label{eq:blochfamily}
T_k(w,\tau)=C_0\,e^{\pi ia_0(2w+a_0\tau)}\,\theta_k(w+a_0\tau+b_0|\tau),\qquad k=1,\dots,4,
\end{equation}
$\Delta=\pi^{-1}\partial_wT_1$, $U=1/T_1$, $\psi_k=\vartheta_k(\tau)^2$ and $H=E_2(\tau)/12$. Then, wherever $T_1\ne0$, the vector $\mathbf J(w,\tau)$ satisfies $\pi^{-1}\partial_w\mathbf J=\mathcal Z(\mathbf J)$ and $(\pi i)^{-1}\partial_\tau\mathbf J=\mathcal J(\mathbf J)$, and
\begin{equation}\label{eq:R12}
R_1:=\psi_2T_4^2-\psi_4T_2^2-\psi_3T_1^2=0,\qquad R_2:=\psi_2T_3^2-\psi_3T_2^2-\psi_4T_1^2=0.
\end{equation}
Consequently, along \eqref{eq:tauw}, $t\mapsto\mathbf J(w(t),\tau(t))$ solves \eqref{eq:jacblock} wherever $T_1\ne0$.
\end{lemma}

\begin{proof}
Write $\xi=w+a_0\tau+b_0$ and $S(w,\tau)=C_0e^{\pi ia_0(2w+a_0\tau)}$, so $T_k=S\,\theta_k(\xi)$ and $\pi^{-1}\partial_wS=2ia_0S$. The relations \eqref{eq:R12} are \ref{J:quad} at $\xi$, multiplied by $S^2$. Since $T_\nu T_\mu/T_1^2=\theta_\nu\theta_\mu/\theta_1^2$ at $\xi$, \ref{J:brezhnev} gives $\pi^{-1}\partial_w(T_k/T_1)=-\psi_kT_\nu T_\mu/T_1^2$, which is equivalent to the second equation in \eqref{eq:Zfield}. Next, $\pi^{-1}\partial_wT_1=S(2ia_0\theta_1+\pi^{-1}\theta_1')$ and $\pi^{-2}\partial_w^2T_1=S(-4a_0^2\theta_1+4ia_0\pi^{-1}\theta_1'+\pi^{-2}\theta_1'')$, so $\pi^{-2}\partial_w^2T_1-(\pi^{-1}\partial_wT_1)^2/T_1=S(\pi^{-2}\theta_1''-\pi^{-2}\theta_1'^2/\theta_1)=-\psi_3\psi_4T_2^2/T_1-4\varkappa T_1$ by \ref{J:brezhnev}; this is the equation for $\mathcal Z\Delta$. The equations for $\mathcal ZT_1$ and $\mathcal ZU$ are definitions, and $\mathcal Z\psi_k=\mathcal ZH=0$ is clear.

For $\mathcal J$: the coefficient equations \eqref{eq:Phicoef} are the identities $D\vartheta_k^2=2\vartheta_k^2D\log\vartheta_k$ and \ref{th:ode}, rewritten with \ref{th:jacobi}. The function $f(w,\tau)=e^{2\pi iaw+\pi ia^2\tau}g(w+a\tau+b,\tau)$ satisfies $4\pi i\partial_\tau f=\partial_w^2f$ whenever $g$ does: with $\partial_1,\partial_2$ the partial derivatives of $g$, evaluated at $(w+a\tau+b,\tau)$, the two sides are $e^{2\pi iaw+\pi ia^2\tau}$ times $-4\pi^2a^2g+4\pi ia\,\partial_1g+4\pi i\,\partial_2g$ and $-4\pi^2a^2g+4\pi ia\,\partial_1g+\partial_1^2g$, respectively. Hence, by \ref{J:heat} with $a=a_0$, each $T_k$ satisfies $4\pi i\partial_\tau T_k=\partial_w^2T_k$, that is, $(\pi i)^{-1}\partial_\tau T_k=-\frac14\pi^{-2}\partial_w^2T_k$, and so does $\Delta=\pi^{-1}\partial_wT_1$, since $\partial_w$ commutes with $\partial_\tau$ and $\partial_w^2$. Since $\pi^{-1}\partial_w\mathbf J=\mathcal Z(\mathbf J)$, we have $\pi^{-1}\partial_w(P(\mathbf J))=(\mathcal ZP)(\mathbf J)$ for every polynomial $P$, so $\pi^{-2}\partial_w^2T_k=(\mathcal Z\mathcal ZT_k)(\mathbf J)$ and $\pi^{-2}\partial_w^2\Delta=(\mathcal Z\mathcal Z\Delta)(\mathbf J)$. Computing these with \eqref{eq:Zfield} and $UT_1=1$ (for instance $\mathcal J T_1=-\frac14\mathcal Z\Delta$) gives exactly \eqref{eq:Phijac}. The last assertion is the chain rule.
\end{proof}

To show that every solution with the intended coefficients is one of the solutions of Lemma~\ref{lem:bloch}, we first match its value at one time. The next lemma does this for $(T_1,\dots,T_4)$: every solution of the two relations \eqref{eq:R12} with $T_1\ne0$ is a multiple of $(\theta_1,\dots,\theta_4)$ at some point.

\begin{lemma}\label{lem:thetafiber}
Let $\tau\in\HH$, $\psi_k=\vartheta_k(\tau)^2$, and let $X_1,\dots,X_4\in\CC$ with $X_1\ne0$ satisfy
\[
\psi_2X_4^2-\psi_4X_2^2-\psi_3X_1^2=0,\qquad \psi_2X_3^2-\psi_3X_2^2-\psi_4X_1^2=0.
\]
Then there are $\xi\in\CC$ and $S\in\CC\setminus\{0\}$ with $\theta_1(\xi|\tau)\ne0$ and $X_k=S\,\theta_k(\xi|\tau)$ for $k=1,\dots,4$.
\end{lemma}

\begin{proof}
Put $f_k=\theta_k/\theta_1$ and $x_k=X_k/X_1$ for $k=2,3,4$. By \ref{J:period}, $f_4$ is a meromorphic function with periods $2$ and $\tau$, and by \ref{J:zeros} its poles in the period parallelogram spanned by $2$ and $\tau$ are the two simple zeros $0$ and $1$ of $\theta_1$. A nonconstant elliptic function takes every complex value as often, counted with multiplicity, as it has poles in a period parallelogram \cite[Chapter~7, Theorem~5 and the remark after it]{Ahlfors}; hence there is $\xi_0$, not a pole, with $f_4(\xi_0)=x_4$. Dividing \ref{J:quad} by $\theta_1^2$ gives $\psi_2f_4^2-\psi_4f_2^2=\psi_3$ and $\psi_2f_3^2-\psi_3f_2^2=\psi_4$; comparing with the hypotheses, and using $\psi_k\ne0$ (\ref{th:nozero}), we get $x_2=\pm f_2(\xi_0)$ and then $x_3=\pm f_3(\xi_0)$. By \ref{J:period} and \ref{J:parity}, the map $\xi\mapsto1-\xi$ changes $(f_2,f_3,f_4)$ into $(-f_2,f_3,f_4)$, and $\xi\mapsto\xi+\tau$ changes it into $(-f_2,-f_3,f_4)$. Hence $(f_2,f_3,f_4)$ takes the values $(\pm f_2,\pm f_3,f_4)(\xi_0)$, with all four choices of signs, at the four points $\xi_0$, $1-\xi_0$, $\xi_0+\tau$ and $1-\xi_0-\tau$, so one of them is a point $\xi$ with $(f_2,f_3,f_4)(\xi)=(x_2,x_3,x_4)$. The point $\xi$ is not a pole of $f_4$, because the set $\ZZ+\tau\ZZ$ of poles is invariant under both maps. Since $\theta_1$ and $\theta_4$ have no common zero, $\theta_1(\xi)\ne0$, and $S=X_1/\theta_1(\xi)$ works.
\end{proof}

The next proposition is the analogue of Proposition~\ref{prop:modcomplete} for the Jacobi block: if the coefficients are the intended ones and the relations hold, a solution defined for all times is a shifted theta solution of Lemma~\ref{lem:bloch}, with one triple $(a_0,b_0,C_0)$ for all times. Lemma~\ref{lem:thetafiber} matches the value at $t=0$, and uniqueness of solutions continues the identity as long as $T_1\ne0$, which holds for all times because $UT_1=1$.

\begin{proposition}\label{prop:jaccont}
Let $\mathbf J:\RR\to\CC^{10}$ be a solution of \eqref{eq:jacblock} such that, for all $t\in\RR$, $\psi_k(t)=\vartheta_k(\tau(t))^2$ for $k=2,3,4$, $H(t)=E_2(\tau(t))/12$, $R_1=R_2=0$ and $UT_1=1$. Then there are $a_0,b_0,C_0\in\CC$, $C_0\ne0$, such that $\mathbf J(t)$ is the vector of Lemma~\ref{lem:bloch} evaluated at $(w(t),\tau(t))$, for all $t\in\RR$.
\end{proposition}

\begin{proof}
By Lemma~\ref{lem:thetafiber} at $\tau(0)$ there are $\xi$ and $S\ne0$ with $T_k(0)=S\theta_k(\xi|\tau(0))$. Choose $a_0$ with $2ia_0=\Delta(0)/T_1(0)-\theta_1'(\xi)/(\pi\theta_1(\xi))$, then $b_0=\xi-w(0)-a_0\tau(0)$ and $C_0=Se^{-\pi ia_0(2w(0)+a_0\tau(0))}$. The vector $\tilde{\mathbf J}(t)$ of Lemma~\ref{lem:bloch} along \eqref{eq:tauw} then agrees with $\mathbf J(0)$ at $t=0$: the coordinates $\psi_k$ and $H$ by hypothesis, the $T_k$ by construction, $\Delta$ because, by the computation of $\pi^{-1}\partial_wT_1$ in the proof of Lemma~\ref{lem:bloch} and the choice of $a_0$, $\tilde\Delta(0)=S\theta_1(\xi)\bigl(2ia_0+\theta_1'(\xi)/(\pi\theta_1(\xi))\bigr)=T_1(0)\cdot\Delta(0)/T_1(0)$, and $U$ because $U(0)=1/T_1(0)$. Let $I$ be the largest open interval containing $0$ on which $\tilde T_1\ne0$ and $\tilde{\mathbf J}=\mathbf J$. It is nonempty: $\tilde T_1(0)=T_1(0)\ne0$, so $\tilde T_1\ne0$ near $0$, there $\tilde{\mathbf J}$ solves \eqref{eq:jacblock} by Lemma~\ref{lem:bloch}, and two solutions of the polynomial system \eqref{eq:jacblock} with the same value at $0$ agree near $0$. If $I$ had a finite endpoint $t^*$, then $\tilde T_1,\dots,\tilde T_4,\tilde\Delta$ are continuous at $t^*$ (they are holomorphic on $\CC\times\HH$), so they agree with $\mathbf J$ at $t^*$; in particular $\tilde T_1(t^*)=T_1(t^*)\ne0$ because $U(t^*)T_1(t^*)=1$, hence $\tilde U=1/\tilde T_1$ is also continuous at $t^*$ and agrees with $U$ there. By uniqueness the two solutions agree on a neighbourhood of $t^*$, contradicting maximality. Hence $I=\RR$.
\end{proof}

\subsubsection{The output}
The output of the Jacobi block is $F=GT_1$, where $G$ is a fourth root of $\lambda QV^2$, an auxiliary factor which turns $\theta_1$ into $\theta_3$ by \ref{J:half}; along the intended solution, this gives $F=\Theta_\lambda(Q,V)$. We compute $F$ along every solution of Lemma~\ref{lem:bloch} and show that three time averages of $F$, its Fourier coefficients for $n=0$ and $n=\pm1$, single out the intended solution (Proposition~\ref{prop:threemoments}).

Let $G(t)$ be a continuous function with $G(t)^4=\lambda e^{i\alpha(t)+2i\beta(t)}$. Then $G(t)e^{-i(\alpha(t)/4+\beta(t)/2)}$ is continuous with fourth power $\lambda=r_0^4$, so it is a constant $r_0\zeta_G$ with $\zeta_G^4=1$, and $G(t)=\zeta_Ge^{\pi iz(t)+\pi i\tau(t)/4}$ by \eqref{eq:tauw}. If $\mathbf J$ is as in Proposition~\ref{prop:jaccont}, then \ref{J:half}, applied at $z+a_0\tau+b_0$, gives
\[
\begin{aligned}
G\,T_1&=\zeta_GC_0e^{\pi ia_0(2w+a_0\tau)}e^{\pi iz+\pi i\tau/4}\theta_1\bigl(z+a_0\tau+b_0+\tfrac{1+\tau}2\big|\tau\bigr)\\
&=\zeta_GC_0e^{\pi ia_0(2w+a_0\tau)}e^{-\pi i(a_0\tau+b_0)}\theta_3(z+a_0\tau+b_0|\tau).
\end{aligned}
\]
Expanding $\theta_3$ and using $2w=2z+1+\tau$, $e^{\pi i\tau(t)}=\lambda e^{i\alpha(t)}$ and $e^{2\pi iz(t)}=e^{i\beta(t)}$, we obtain, with $C'=\zeta_GC_0e^{\pi ia_0-\pi ib_0}$ and $\lambda^x:=e^{-\pi x\im\tau_\lambda}$ for $x\in\CC$,
\begin{equation}\label{eq:Fbloch}
F(t):=G(t)T_1(t)=C'\sum_{n\in\ZZ}\lambda^{(n+a_0)^2}e^{2\pi inb_0}\,e^{i(n+a_0)^2\alpha(t)+i(n+a_0)\beta(t)}.
\end{equation}
For $a_0=b_0=0$ and $C'=1$ this is $\Theta_\lambda(e^{i\alpha(t)},e^{i\beta(t)})$.

The output \eqref{eq:Fbloch} is a sum of oscillating exponentials only if $a_0$ is real. The next lemma excludes nonreal $a_0$ for recurrent solutions: then the point $w+a_0\tau+b_0$ at which the theta functions are evaluated moves away from the real axis linearly in time, and the quasi-periodicity \ref{J:period} makes the theta functions grow like $e^{ct^2}$.

\begin{lemma}\label{lem:nonreal}
If $\im a_0\ne0$ in Proposition~\ref{prop:jaccont}, then $|(T_1,\dots,T_4)(t)|\to\infty$ as $t\to\infty$.
\end{lemma}

\begin{proof}
Put $y_\lambda=\im\tau_\lambda=\im\tau(t)$ and $\xi(t)=w(t)+a_0\tau(t)+b_0$, an affine function of $t$ with $\im\xi(t)=\im(a_0)\Omega t/\pi+O(1)$. The prefactor in \eqref{eq:blochfamily} is the exponential of an affine function of $t$, so its logarithm has modulus $O(1+|t|)$. Write $\xi(t)=\xi_r(t)+n(t)\tau(t)$ with $n(t)\in\ZZ$ and $|\im\xi_r(t)|\le y_\lambda/2$; then $|n(t)|y_\lambda\ge|\im\xi(t)|-y_\lambda/2$, so $|n(t)|\ge c|t|-C$ with $c=|\im a_0|\,|\Omega|/(\pi y_\lambda)>0$ and a constant $C$. Iterating \ref{J:period} gives $|\theta_k(\xi_r+n\tau|\tau)|=e^{\pi n^2y_\lambda+2\pi n\im\xi_r}|\theta_k(\xi_r|\tau)|\ge e^{\pi y_\lambda(n^2-|n|)}|\theta_k(\xi_r|\tau)|$. By \ref{J:period} the function $(\xi,\tau)\mapsto\max_k|\theta_k(\xi|\tau)|$ is invariant under $\xi\mapsto\xi+1$ and $\tau\mapsto\tau+2$, so on $\{|\im\xi|\le y_\lambda/2,\ \im\tau=y_\lambda\}$ it is bounded below by its minimum over a compact set, which is positive by \ref{J:zeros}. Since each $|T_k(t)|$ is the modulus of the prefactor times $|\theta_k(\xi_r(t)+n(t)\tau(t)|\tau(t))|$, and $x\mapsto x^2-x$ is increasing for $x\ge1$, for large $t$ we get $\log|(T_1,\dots,T_4)(t)|\ge\pi y_\lambda\bigl((c|t|-C)^2-(c|t|-C)\bigr)-O(1+|t|)\ge\pi y_\lambda c^2t^2-O(1+|t|)\to\infty$.
\end{proof}

For real $a_0$, \eqref{eq:Fbloch} writes $F$ as an absolutely convergent sum of exponentials $e^{i\nu_nt}$ with frequencies $\nu_n=\Omega x^2+Nx$, $x=n+a_0$, and the time average of $Fe^{-i\mu t}$ picks out the coefficients with $\nu_n=\mu$. The next proposition shows that three such averages, which for $\Theta_\lambda$ equal $1$, $\lambda$ and $\lambda$, force $F=\Theta_\lambda$; the hypothesis $N/\Omega\notin\QQ$ ensures that each of them picks out a single coefficient.

\begin{proposition}\label{prop:threemoments}
In the situation of Proposition~\ref{prop:jaccont}, suppose that $\mathbf J$ is forward recurrent and that $N/\Omega\notin\QQ$, let $G$ be a continuous function with $G(t)^4=\lambda e^{i\alpha(t)+2i\beta(t)}$, and let $F=GT_1$ as in \eqref{eq:Fbloch}. If
\begin{equation}\label{eq:threemoments}
\avg{F}=1,\qquad\avg{Fe^{-i(\alpha+\beta)}}=\lambda,\qquad\avg{Fe^{-i(\alpha-\beta)}}=\lambda,
\end{equation}
then $F(t)=\Theta_\lambda(e^{i\alpha(t)},e^{i\beta(t)})$ for all $t\in\RR$.
\end{proposition}

\begin{proof}
By Lemma~\ref{lem:nonreal} and recurrence, $a_0$ is real. Then \eqref{eq:Fbloch} is a trigonometric series $F(t)=\sum_nd_ne^{i\nu_nt}$, where, with $x=n+a_0$ and $\varrho=N/\Omega$, the frequency is $\nu_n=\Omega x^2+Nx=\Omega x(x+\varrho)$ and $d_n=C'\lambda^{x^2}e^{2\pi inb_0}e^{ix^2\alpha_0+ix\beta_0}$. It converges absolutely, since $|d_n|=|C'|e^{-\pi y_\lambda x^2-2\pi n\im b_0}$ with $y_\lambda=\im\tau_\lambda>0$. For real $\mu$ at most two $n$ have $\nu_n=\mu$, and averaging term by term as in the proof of Lemma~\ref{lem:avg} (a term with $\nu_n\ne\mu$ contributes $0$ in the limit, and the contributions are dominated by the summable $|d_n|$) gives $\avg{Fe^{-i\mu t}}=\sum_{n:\nu_n=\mu}d_n$. The frequencies relevant for \eqref{eq:threemoments} are $\mu=0$, $\Omega+N$ and $\Omega-N$, and
\[
\nu=0\iff x\in\{0,-\varrho\},\qquad\nu=\Omega+N\iff x\in\{1,-1-\varrho\},\qquad\nu=\Omega-N\iff x\in\{-1,1-\varrho\}.
\]
Since $\varrho\notin\QQ$, the set $a_0+\ZZ$ contains at most one element of each of these pairs, and since $\avg F=1\ne0$ it contains $0$ or $-\varrho$.

We treat the two cases in turn. If $a_0\in\ZZ$, replacing $n$ by $n-a_0$ in \eqref{eq:Fbloch} gives the same formula with $a_0=0$ and with $C'$ multiplied by $e^{-2\pi ia_0b_0}$; so we may assume $a_0=0$. The three averages give $d_0=C'=1$ and $d_1=\lambda e^{2\pi ib_0}e^{i(\alpha_0+\beta_0)}=\lambda e^{i(\alpha_0+\beta_0)}$, so $e^{2\pi ib_0}=1$ and $F=\Theta_\lambda(e^{i\alpha},e^{i\beta})$.

If $a_0\in-\varrho+\ZZ$, reindex so that $a_0=-\varrho$. Writing $x=n-\varrho$, the coefficients take the form $d_n=c_*R^{n^2}S^n$ with $R=\lambda e^{i\alpha_0}$ and constants $c_*,S\ne0$, while the frequencies are $\nu_n=\Omega n^2-Nn$. The averages give $d_0=c_*=1$ and $d_{-1}=RS^{-1}=\lambda e^{i(\alpha_0+\beta_0)}$, so $S=e^{-i\beta_0}$. Hence $d_n=\lambda^{n^2}e^{in^2\alpha_0-in\beta_0}$ and
\[
F(t)=\sum_n\lambda^{n^2}e^{in^2\alpha_0-in\beta_0}e^{i(n^2\Omega-nN)t}=\sum_k\lambda^{k^2}e^{ik^2\alpha(t)+ik\beta(t)}=\Theta_\lambda(e^{i\alpha(t)},e^{i\beta(t)}),
\]
after the substitution $k=-n$. (The second case is the image of the first under the reflection $x\mapsto-\varrho-x$, which does not change the frequency $\Omega x(x+\varrho)$.) In both cases $F(t)=\Theta_\lambda(e^{i\alpha(t)},e^{i\beta(t)})$ for all $t\in\RR$.
\end{proof}

In the proof we use only the first two averages in \eqref{eq:threemoments}; the third is then automatic. We keep it in the construction because the formal verification (Section \ref{subsec:formal}) uses all three.

\begin{remark}\label{rem:zeroratio}
The irrationality of $\varrho=N/\Omega$ ensures that each pair of roots displayed in the proof of Proposition~\ref{prop:threemoments} contains at most one element of $a_0+\ZZ$. If $\varrho$ were an integer, this need not hold, and an average could see a sum of coefficients. If $\varrho=0$, which happens exactly when the $j$-th column of $B$ vanishes, the last two averages in \eqref{eq:threemoments} of the intended function $\Theta_\lambda$ itself are not $\lambda$: they see $d_1+d_{-1}$, and equal $\lambda(1+e^{\mp2i\beta_0})$. The hypothesis $\varrho\notin\QQ$ excludes such coincidences of frequencies. This is why the construction of Lemma~\ref{lem:template} makes every column of $B$ nonzero; then $\varrho\notin\QQ$ by Corollary~\ref{cor:irrational}, and the intended averages are those of Lemma~\ref{lem:graph}(c).
\end{remark}

\subsection{The polynomial system and its intended trajectories}\label{sec:source}

We now assemble the blocks of Sections~\ref{sec:clocks}--\ref{sec:jacobi} into one polynomial system $\dot x=X_e(x)$, together with a polynomial map $h_e$ whose coordinates are the ``tests'' of Section~\ref{subsec:overview2}. Along the intended trajectories of $\dot x=X_e(x)$ (Lemma~\ref{lem:graph}), one polynomial in the state, $\mathfrak D^{\mathrm{obs}}_e$, reproduces the function $\mathfrak D_e$ of Section~\ref{sec:arith}; the coordinates of $h_e$ are the real and imaginary parts of this polynomial, together with polynomials which encode the algebraic relations and the prescribed time averages used in Sections~\ref{sec:clocks}--\ref{sec:jacobi}. In Section~\ref{subsec:pass_tests} we show that along every recurrent solution which satisfies these relations and averages and does not grow too fast, $\mathfrak D^{\mathrm{obs}}_e$ reproduces $\mathfrak D_e$ along a linear flow on the torus (Theorem~\ref{thm:selector}), and in Section~\ref{subsec:intended} we describe the intended trajectories and prove the equivalence of this section (Corollary~\ref{cor:yes}). Throughout, $e\in\NN$ is the input of the reduction, and $A,B,c(e)$ are as in Lemma~\ref{lem:template}.

\subsubsection{\texorpdfstring{The polynomial system $\dot x=X_e(x)$ and the tests $h_e$}{The polynomial system and the tests}}\label{sec:the_source_system}
In this subsection we write down the system $\dot x=X_e(x)$ and the map $h_e$ explicitly. The coordinates of $h_e$ fall into three groups: the ``validity residuals'', which are nonnegative polynomials vanishing exactly where the algebraic relations satisfied by the intended trajectories hold; the real and imaginary parts of the ``moments'', which are polynomials whose time averages vanish exactly when the prescribed averages of Sections~\ref{sec:modular} and~\ref{sec:jacobi} hold; and the real and imaginary parts of the polynomial $\mathfrak D^{\mathrm{obs}}_e$, which equals $\mathfrak D_e$ along the intended trajectories. Table~\ref{tab:source} lists the coordinates of $x$, their values along the intended trajectories and the tests which involve them.

For an integer $a$ and $\zeta\in\CC$ put $\zeta^{[a]}=\zeta^a$ if $a\ge0$ and $\zeta^{[a]}=\bar\zeta^{-a}$ if $a<0$; on $|\zeta|=1$ this is $\zeta^a$, and for every $a$ it is a polynomial in $\re\zeta$ and $\im\zeta$, so that $V_j$ and $S_j$ below are polynomials even though $A$ and $B$ may have negative entries. In the phase-circle coordinates of \eqref{eq:controller} define the polynomials
\[
Q_j=\zeta_j^4,\quad\chi_l=\zeta_{q+l}^4,\quad V_j=\prod_l\chi_l^{[B_{lj}]},\quad S_j=\prod_l\chi_l^{[A_{lj}]},\quad W_j=\bar Q_jS_j,
\]
the frequencies $\omega_\ell=p_\ell r_\ell$, $\Omega_j=4\omega_j$, $N_j=4\sum_lB_{lj}\omega_{q+l}$, and $\nu_j=-\Omega_j+4\sum_lA_{lj}\omega_{q+l}$, which are polynomials in the controller coordinates. Along every solution of \eqref{eq:controller} we have $\dot Q_j=i\Omega_jQ_j$, $\dot V_j=iN_jV_j$ and $\dot W_j=i\nu_jW_j$, because $\omega_\ell$ is real and $\frac{d}{dt}\zeta_\ell^4=4i\omega_\ell\zeta_\ell^4$. The system $\dot x=X_e(x)$ has the following coordinates, for $1\le\ell\le q+r$ and $1\le j\le q$: the controllers $(r_\ell,s_\ell)\in\RR^2$ and phase circles $\zeta_\ell\in\CC$; the oscillator $R\in\CC$; a modular block $\mathbf m_j=(L_j,C_j,E_j,Y_j)\in\CC^4$; a Jacobi block $\mathbf J_j=(\psi_{2,j},\psi_{3,j},\psi_{4,j},H_j,T_{1,j},\dots,T_{4,j},\Delta_j,U_j)\in\CC^{10}$; an auxiliary coordinate $G_j\in\CC$, which will be a fourth root of $\lambda Q_jV_j^2$; and two reciprocals $G^+_j,G^-_j\in\CC$. Its equations are \eqref{eq:controller}, \eqref{eq:oscillator}, and
\begin{equation}\label{eq:source}
\begin{gathered}
\dot{\mathbf m}_j=i\Omega_j\,\cD(\mathbf m_j),\qquad\dot{\mathbf J}_j=i\Omega_j\,\mathcal J(\mathbf J_j)+\tfrac12(\Omega_j+N_j)\,\mathcal Z(\mathbf J_j),\qquad\dot G_j=\tfrac i4(\Omega_j+2N_j)G_j,\\
\dot G^+_j=i\eta\nu_jW_j(G^+_j)^2,\qquad\dot G^-_j=-i\eta\nu_j\bar W_j(G^-_j)^2.
\end{gathered}
\end{equation}
Separating real and imaginary parts, these equations form a polynomial system $\dot x=X_e(x)$ on $\RR^{n_s}$ with rational coefficients, for a number $n_s$ which does not depend on $e$; the input $e$ enters only through the coefficients $c(e)$ in \eqref{eq:oscillator}.
\begin{table}[t]
\centering
\def\cell#1{\begin{tabular}[t]{@{}l@{}}#1\end{tabular}}%
\caption{The coordinates of the system $\dot x=X_e(x)$, for $1\le\ell\le q+r$ and $1\le j\le q$, with their equations, their values at the point $\Gamma_e(\varphi)$ of \eqref{eq:graph}, which the intended trajectories $t\mapsto\Gamma_e(\varphi+\omega t)$ follow (Lemma~\ref{lem:graph}), and the tests of $h_e$ which involve them. The validity residuals are the squares of the real polynomials and the squared moduli of the complex polynomials listed here. The phase circles also enter the moments and some of the residuals, through $Q_j=\zeta_j^4$, $V_j$, $W_j$ and $\bar\zeta_j$. Besides these tests, $h_e$ contains the real and imaginary parts of $\mathfrak D^{\mathrm{obs}}_e=R\prod_jF_j\,(1-\eta^2)\,G_j^+G_j^-$, which equals $\mathfrak D_e(\varphi)$ at $\Gamma_e(\varphi)$. Figure~\ref{fig:circuit} shows how the groups of coordinates are connected.}
\label{tab:source}
\vspace{10px}
\footnotesize
\renewcommand{\arraystretch}{1.2}
\setlength{\tabcolsep}{4pt}
\begin{tabular}{@{}lllll@{}}
\hline
Coordinates & Real dim. & Value at $\Gamma_e(\varphi)$ & Validity residuals & Moments\\
\hline
\cell{controllers \eqref{eq:controller}\\ $(r_\ell,s_\ell)$} & $2(q+r)$ & $\bigl(p_\ell^{-1/2},(1-p_\ell^{-1})^{1/2}\bigr)$ & $r_\ell^2+s_\ell^2-1$ & ---\\
\cell{phase circles \eqref{eq:controller}\\ $\zeta_\ell$} & $2(q+r)$ & $e^{i\varphi_\ell}$ & $|\zeta_\ell|^2-1$ & ---\\
\cell{oscillator \eqref{eq:oscillator}\\ $R$} & $2$ & $\prod_l\chi_l^{-c_l(e)}$ & $|R|^2-1$ & ---\\
\cell{modular block \eqref{eq:source}\\ $(L_j,C_j,E_j,Y_j)$} & $8q$ & $(\vartheta_2^4,\vartheta_4^4,E_2,\vartheta_3)(\tau_j)$ & $Y_j^4-L_j-C_j$ & \cell{$Y_j-1$,\\ $Y_j\bar Q_j-2\lambda$}\\
\cell{Jacobi block \eqref{eq:source}\\ $(\psi_{2,j},\psi_{3,j},\psi_{4,j},H_j)$} & $8q$ & $\bigl(\vartheta_2^2,\vartheta_3^2,\vartheta_4^2,\tfrac1{12}E_2\bigr)(\tau_j)$ & \cell{$\psi_{2,j}^2-L_j$, $\psi_{3,j}-Y_j^2$,\\ $\psi_{4,j}^2-C_j$, $12H_j-E_j$} & \cell{$\psi_{4,j}-1$,\\ $\psi_{2,j}\bar\zeta_j^2-\frac14$}\\
\cell{Jacobi block \eqref{eq:source}\\ $(T_{1,j},\dots,T_{4,j},\Delta_j,U_j)$} & $12q$ & \cell{$T_{k,j}=\theta_k(w_j|\tau_j)$,\\ $\Delta_j=\frac1\pi\theta_1'(w_j|\tau_j)$,\\ $U_j=1/T_{1,j}$} & \cell{$R_{1,j}$, $R_{2,j}$ (see \eqref{eq:R12}),\\ $U_jT_{1,j}-1$} & \cell{$F_j-1$, $F_j\bar Q_j\bar V_j-\lambda$,\\ $F_j\bar Q_jV_j-\lambda$,\\ where $F_j=G_jT_{1,j}$}\\
\cell{auxiliary \eqref{eq:source}\\ $G_j$} & $2q$ & $e^{\pi iz_j+\pi i\tau_j/4}$ & $G_j^4-\lambda Q_jV_j^2$ & through $F_j$ (above)\\
\cell{reciprocals \eqref{eq:source}\\ $G_j^+,G_j^-$} & $4q$ & \cell{$(1-\eta W_j)^{-1}$,\\ $(1-\eta W_j^{-1})^{-1}$} & \cell{$(1-\eta W_j)G_j^+-1$,\\ $(1-\eta\bar W_j)G_j^--1$} & ---\\
\hline
total & \multicolumn{4}{@{}l}{$n_s=4(q+r)+2+34q=38q+4r+2$}\\
\hline
\end{tabular}
\end{table}

The \emph{validity residuals} are the following nonnegative polynomials: $(r_\ell^2+s_\ell^2-1)^2$, $(|\zeta_\ell|^2-1)^2$, $(|R|^2-1)^2$, and for every $j$ the squared moduli of
\begin{equation}\label{eq:residuals}
\begin{gathered}
Y_j^4-L_j-C_j,\quad \psi_{2,j}^2-L_j,\quad \psi_{3,j}-Y_j^2,\quad \psi_{4,j}^2-C_j,\quad12H_j-E_j,\quad R_{1,j},\quad R_{2,j},\quad U_jT_{1,j}-1,\\
G_j^4-\lambda Q_jV_j^2,\quad(1-\eta W_j)G^+_j-1,\quad(1-\eta\bar W_j)G^-_j-1,
\end{gathered}
\end{equation}
where $R_{1,j},R_{2,j}$ are the polynomials \eqref{eq:R12} in $\mathbf J_j$. With $F_j=G_jT_{1,j}$, the \emph{moments} are the complex polynomials
\begin{equation}\label{eq:moments}
\begin{gathered}
Y_j-1,\qquad Y_j\bar Q_j-2\lambda,\qquad \psi_{4,j}-1,\qquad \psi_{2,j}\bar\zeta_j^2-\tfrac14,\\
F_j-1,\qquad F_j\bar Q_j\bar V_j-\lambda,\qquad F_j\bar Q_jV_j-\lambda,
\end{gathered}
\end{equation}
and we put
\begin{equation}\label{eq:detobs}
\mathfrak D^{\mathrm{obs}}_e=R\prod_{j=1}^qF_j\,(1-\eta^2)G^+_jG^-_j.
\end{equation}
We let $h_e:\RR^{n_s}\to\RR^{n_h}$ be the polynomial map whose coordinates are all validity residuals, the real and imaginary parts of all moments, and the real and imaginary parts of $\mathfrak D^{\mathrm{obs}}_e$; here $n_h$ is the number of these coordinates. Following Section~\ref{subsec:overview2}, we think of each coordinate of $h_e$ as a ``test'' (see Section~\ref{subsec:pass_tests} for what it means for a solution to pass the tests). In Section~\ref{sec:compact}, $h_e$ becomes the velocity of an additional variable $z$, which keeps a ``running total'' of $h_e$. The number $n_h$ and the degree of $h_e$ do not depend on $e$; its coefficients do not depend on $e$ at all, since $e$ enters $\mathfrak D^{\mathrm{obs}}_e$ only through the dynamics of $R$. We keep the coordinates separate so that zero average of $h_e$ means zero average of every single coordinate.

\subsubsection{Trajectories which pass the tests}\label{subsec:pass_tests}
As in Section~\ref{sec:modular}, $\avg f=\lim_{T\to\infty}\frac1T\int_0^Tf(t)\,dt$, and ``$f$ has average zero'' means that this limit exists and equals zero. A solution $x$ is \emph{forward recurrent} if $x(t_k)\to x(0)$ for some sequence $t_k\to\infty$. As in Section~\ref{subsec:overview2}, we say that a solution $x:\RR\to\RR^{n_s}$ of $\dot x=X_e(x)$ \emph{passes the tests} if $\avg{h_e(x)}=0$, that is, every coordinate of $h_e$ has average zero along $x$, and, in addition, every validity residual vanishes at $x(t)$ for all $t\in\RR$. This additional condition is the technical condition mentioned in Section~\ref{subsec:overview2}, and it is condition (i) of Theorem~\ref{thm:selector} below; condition (iii) of that theorem then holds as well, since the real and imaginary parts of the moments are coordinates of $h_e$. Conditions (ii) and (iv) make precise what it means that $x$ is ``recurrent'' and does not ``grow too fast''. Theorem~\ref{thm:selector} shows that along every solution satisfying (i)--(iv), in particular along every forward recurrent solution which satisfies (iv) and passes the tests, the phases rotate with the frequencies $\pm\sqrt{p_\ell}$, and $\mathfrak D^{\mathrm{obs}}_e$ equals $\mathfrak D_e$ along this rotation up to a constant factor of modulus one. Its proof applies the results of Sections~\ref{sec:clocks}--\ref{sec:jacobi} block by block; the validity residuals and the moments provide exactly their hypotheses.

\begin{theorem}[Solutions with vanishing validity residuals and the prescribed averages]\label{thm:selector}
Let $x:\RR\to\RR^{n_s}$ be a solution of $\dot x=X_e(x)$ such that
\begin{enumerate}[label=(\roman*)]
\item every validity residual vanishes at $x(t)$ for all $t\in\RR$;
\item $x$ is forward recurrent;
\item every moment has average zero along $x$;
\item $\frac1T\int_0^T|x(t)|^\mu\,dt$ is bounded for $T\ge1$, for some $\mu\ge1$.
\end{enumerate}
Then there are $\epsilon\in\{-1,1\}^{q+r}$, $\varphi(0)\in\RR^{q+r}$ and $R_0\in\CC$ with $|R_0|=1$ such that, with $\omega=(\epsilon_\ell\sqrt{p_\ell})_\ell$ and $\varphi(t)=\varphi(0)+\omega t$, we have $\zeta_\ell(t)=e^{i\varphi_\ell(t)}$ and
\begin{equation}\label{eq:selected}
F_j(t)=\Theta_\lambda(Q_j(t),V_j(t)),\qquad G^\pm_j(t)=(1-\eta W_j(t)^{\pm1})^{-1},\qquad\mathfrak D^{\mathrm{obs}}_e(x(t))=R_0\,\mathfrak D_e(\varphi(t))
\end{equation}
for all $j$ and all $t\in\RR$, where $\mathfrak D_e$ is the function \eqref{eq:detector}.
\end{theorem}

\begin{proof}
Projections of a forward recurrent solution to subsystems are forward recurrent, and every component of $x$ has modulus at most $|x|$; so each block satisfies the hypotheses (ii) and (iv).

\emph{Clocks.} By (i) and Lemma~\ref{lem:controller}, every controller is at an equilibrium with $\omega_\ell=\epsilon_\ell\sqrt{p_\ell}$, and $\zeta_\ell(t)=e^{i\varphi_\ell(t)}$ for a continuous $\varphi_\ell(t)=\varphi_\ell(0)+\omega_\ell t$. Then $Q_j(t)=e^{i\alpha_j(t)}$ and $V_j(t)=e^{i\beta_j(t)}$ with $\alpha_j=4\varphi_j$, $\beta_j=4\sum_lB_{lj}\varphi_{q+l}$, so $\dot\alpha_j=\Omega_j$ and $\dot\beta_j=N_j$, and $\Omega_j\ne0$, $N_j/\Omega_j\notin\QQ$ by Corollary~\ref{cor:irrational}. By \eqref{eq:oscgraph}, $R(t)=R_0\prod_l\chi_l(t)^{-c_l(e)}$ with $|R_0|=|R(0)|=1$.

\emph{Modular block.} Fix $j$ and put $\tau_j(t)=\tau_\lambda+\alpha_j(t)/\pi$. Since the controllers are at equilibria, $\mathbf m_j$ solves the modular block \eqref{eq:modblock} with the constant frequency $\Omega_j\ne0$, and $\alpha_j(t)=\alpha_j(0)+\Omega_jt$ with $\alpha_j(0)\in\RR$; by (i), $Y_j^4=L_j+C_j$. The first two moments in \eqref{eq:moments} have average zero, that is, $\avg{Y_j}=1$ and $\avg{Y_je^{-i\alpha_j}}=2\lambda$, because $\bar Q_j=e^{-i\alpha_j}$. Together with (ii) and (iv) for $\mathbf m_j$, these are the hypotheses of Proposition~\ref{prop:modselect}, which gives $\mathbf m_j(t)=\mathbf m^{\mathrm{can}}(\tau_j(t))$. By the residuals, $\psi_{4,j}^2=C_j=\vartheta_4^4(\tau_j)$. Since $\vartheta_4$ has no zeros in $\HH$ by \ref{th:nozero}, $\psi_{4,j}(t)/\vartheta_4^2(\tau_j(t))$ is a continuous function of $t\in\RR$ with values $\pm1$, hence constant. The series $\vartheta_4^2(\tau_j(t))=\sum_{m,n\in\ZZ}(-1)^{m+n}\lambda^{m^2+n^2}e^{i(m^2+n^2)\alpha_j(t)}$ converges absolutely and uniformly in $t$, and every term with $(m,n)\ne(0,0)$ has the nonzero frequency $(m^2+n^2)\Omega_j$; averaging term by term, as in the proof of Lemma~\ref{lem:avg}, shows that $\vartheta_4^2(\tau_j(t))$ has average one. Since $\avg{\psi_{4,j}}=1$, the constant is $1$, and $\psi_{4,j}=\vartheta_4^2(\tau_j)$. Similarly $\psi_{2,j}=\pm\vartheta_2^2(\tau_j)$ with a constant sign. Now
\[
\vartheta_2^2(\tau)=\sum_{m,n\in\ZZ}e^{\pi i\tau((m+\frac12)^2+(n+\frac12)^2)},
\]
and $e^{\pi i\tau_j(t)\sigma}=\lambda^\sigma e^{4i\sigma\varphi_j(t)}$, where $\lambda^\sigma=e^{\pi i\sigma\tau_\lambda}>0$. The exponent $\sigma=(m+\frac12)^2+(n+\frac12)^2$ takes its smallest value $\frac12$ exactly for the four pairs $(m,n)\in\{0,-1\}^2$, and $\bar\zeta_j^2=e^{-2i\varphi_j}$. Hence $\vartheta_2^2(\tau_j)\bar\zeta_j^2$ is the sum of the constant $4\lambda^{1/2}=\frac14$ and of an absolutely and uniformly convergent series of terms $\lambda^\sigma e^{i(4\sigma-2)\varphi_j(t)}$ with $4\sigma-2=4(m^2+m+n^2+n)>0$, whose frequencies $(4\sigma-2)\omega_j$ are nonzero; so it has average $\frac14$. Since $\psi_{2,j}\bar\zeta_j^2=\pm\vartheta_2^2(\tau_j)\bar\zeta_j^2$ and $\avg{\psi_{2,j}\bar\zeta_j^2}=\frac14$, the sign is $+$, and $\psi_{2,j}=\vartheta_2^2(\tau_j)$. Finally $\psi_{3,j}=Y_j^2=\vartheta_3^2(\tau_j)$ and $H_j=E_j/12=E_2(\tau_j)/12$.

\emph{Jacobi block.} With $\alpha=\alpha_j$, $\beta=\beta_j$, $\tau=\tau_j$ and $w$ as in \eqref{eq:tauw}, the block $\mathbf J_j$ satisfies the hypotheses of Proposition~\ref{prop:jaccont}, by the previous step and (i). The coordinate $G_j$ satisfies $G_j^4=\lambda Q_jV_j^2=\lambda e^{i\alpha+2i\beta}$, so Proposition~\ref{prop:threemoments}, applied with the last three moments in \eqref{eq:moments}, gives $F_j(t)=\Theta_\lambda(e^{i\alpha_j(t)},e^{i\beta_j(t)})=\Theta_\lambda(Q_j(t),V_j(t))$.

\emph{Reciprocals and $\mathfrak D^{\mathrm{obs}}_e$.} Since $|W_j|=1$ and $\eta<1$, $1-\eta W_j^{\pm1}\ne0$, and the residuals give $G^\pm_j=(1-\eta W_j^{\pm1})^{-1}$; thus $(1-\eta^2)G^+_jG^-_j=P_\eta(W_j)=P_\eta(Q_j^{-1}S_j)$ by \eqref{eq:ThetaP}. Substituting everything into \eqref{eq:detobs} and comparing with \eqref{eq:detector} gives the last identity in \eqref{eq:selected}.
\end{proof}

\subsubsection{The intended trajectories}\label{subsec:intended}
For the converse direction, $e\notin K$, we need one explicit family of solutions, the ``intended trajectories'' of Section~\ref{subsec:overview2}, along which the system $\dot x=X_e(x)$ carries out the computation of $\mathfrak D_e$ for which it was designed. For each $\varphi$, the state $\Gamma_e(\varphi)$ defined below has every block at its intended value: the controllers at equilibria, the phase circles at $e^{i\varphi_\ell}$, and the theta constants and theta functions at the points $\tau_j$ and $(w_j|\tau_j)$ determined by $\varphi$. The intended trajectories are the solutions $t\mapsto\Gamma_e(\varphi+\omega t)$ of Lemma~\ref{lem:graph}(a) below. The image of $\Gamma_e$ is the graph of a function of the phase circles, since $\zeta_\ell=e^{i\varphi_\ell}$ determines $\varphi$ modulo $2\pi$. For $\varphi\in\RR^{q+r}$ put $\tau_j=\tau_\lambda+4\varphi_j/\pi$, $z_j=\frac2\pi\sum_lB_{lj}\varphi_{q+l}$ and $w_j=z_j+\frac{1+\tau_j}2$, and define $\Gamma_e(\varphi)\in\RR^{n_s}$ by
\begin{equation}\label{eq:graph}
\begin{gathered}
r_\ell=p_\ell^{-1/2},\quad s_\ell=(1-p_\ell^{-1})^{1/2},\quad\zeta_\ell=e^{i\varphi_\ell},\quad R=\prod_l\chi_l^{-c_l(e)},\quad\mathbf m_j=\mathbf m^{\mathrm{can}}(\tau_j),\\
\psi_{k,j}=\vartheta_k^2(\tau_j),\quad H_j=\tfrac1{12}E_2(\tau_j),\quad T_{k,j}=\theta_k(w_j|\tau_j),\quad\Delta_j=\tfrac1\pi\theta_1'(w_j|\tau_j),\quad U_j=1/T_{1,j},\\
G_j=e^{\pi iz_j+\pi i\tau_j/4}=r_0e^{i\varphi_j+2i\sum_lB_{lj}\varphi_{q+l}},\qquad G_j^\pm=(1-\eta W_j^{\pm1})^{-1}.
\end{gathered}
\end{equation}
By \ref{J:half}, $G_jT_{1,j}=\theta_3(z_j|\tau_j)=\Theta_\lambda(Q_j,V_j)$, which satisfies $|\Theta_\lambda-1|\le2\sum_{n\ge1}\lambda^{n^2}<1$; so $T_{1,j}\ne0$ and $U_j$ is defined.

\begin{lemma}\label{lem:graph}
The map $\Gamma_e$ is $2\pi$-periodic in each $\varphi_\ell$, so it defines a continuous map on $\TT^{q+r}$, and with $\omega_\ell=\sqrt{p_\ell}$:
\begin{enumerate}[label=(\alph*)]
\item $t\mapsto\Gamma_e(\varphi+\omega t)$ is a solution of $\dot x=X_e(x)$ for every $\varphi$;
\item every validity residual vanishes on the image of $\Gamma_e$;
\item every coordinate of $h_e\circ\Gamma_e$ has an absolutely convergent Fourier series on $\TT^{q+r}$; the constant terms of the moments are zero, and the constant term of $\mathfrak D^{\mathrm{obs}}_e\circ\Gamma_e=\mathfrak D_e$ is \eqref{eq:constterm}.
\end{enumerate}
\end{lemma}

\begin{proof}
Every coordinate is a series in the characters $e^{ik\cdot\varphi}$ with integer vectors $k$: for instance $e^{\pi i\tau_j(n+1/2)^2}=\lambda^{(n+1/2)^2}e^{i(4n^2+4n+1)\varphi_j}$, and $\theta_1(w_j|\tau_j)$ involves $e^{i(4n^2+8n+3)\varphi_j}$ and $e^{2i(2n+1)\sum_lB_{lj}\varphi_{q+l}}$. Hence $\Gamma_e$ is $2\pi$-periodic in each $\varphi_\ell$. (a) The controllers are at equilibria and $\dot\varphi=\omega$; the modular blocks solve \eqref{eq:modblock} by Lemma~\ref{lem:modcanon}; the Jacobi blocks are the case $a_0=b_0=0$, $C_0=1$ of Lemma~\ref{lem:bloch}; the remaining coordinates are explicit characters or rational functions of characters, and satisfy \eqref{eq:oscillator} and \eqref{eq:source} by direct differentiation. (b) The residuals \eqref{eq:residuals} vanish by \ref{th:jacobi}, \ref{J:quad} and the definitions. (c) Every coordinate is a finite product of characters, of absolutely convergent theta series in characters (for $\theta_1'$ the coefficients acquire a polynomial factor, which does not affect absolute convergence), and of the geometric series $(1-\eta W^{\pm1})^{-1}=\sum_{m\ge0}\eta^mW^{\pm m}$ and $1/\Theta_\lambda=\sum_{m\ge0}(1-\Theta_\lambda)^m$ (the latter for $U_j=G_j/\Theta_\lambda(Q_j,V_j)$). Products of absolutely convergent Fourier series converge absolutely, and the sum of the absolute values of the coefficients of a product is at most the product of these sums; for $1-\Theta_\lambda$ this sum is at most $2\sum_{n\ge1}\lambda^{n^2}<1$, so the second geometric series converges absolutely as well. For the moments: the constant term of $Y_j=\vartheta_3(\tau_j)=1+2\sum_{n\ge1}\lambda^{n^2}e^{4in^2\varphi_j}$ is $1$, and that of $Y_j\bar Q_j$ is $2\lambda$; the constant term of $\vartheta_4^2(\tau_j)$ is $1$, and that of $\vartheta_2^2(\tau_j)\bar\zeta_j^2$ is $\frac14$, as in the proof of Theorem~\ref{thm:selector}; and $F_j=\sum_n\lambda^{n^2}Q_j^{n^2}V_j^n$, where $Q_j^{n^2}V_j^n=e^{4in^2\varphi_j+4in\sum_lB_{lj}\varphi_{q+l}}$, so the constant terms of $F_j$, $F_j\bar Q_j\bar V_j$ and $F_j\bar Q_jV_j$ come from $n=0$, $n=1$ and $n=-1$ respectively (here $B_j\ne0$ is used), and equal $1$, $\lambda$ and $\lambda$. The last assertion holds because $\mathfrak D^{\mathrm{obs}}_e\circ\Gamma_e=\mathfrak D_e$.
\end{proof}

We can now prove the equivalence of this section: by Theorem~\ref{thm:selector}, a trajectory which passes the tests reveals the constant Fourier coefficient of $\mathfrak D_e$, and by Lemma~\ref{lem:graph} the intended trajectories pass the tests when this coefficient vanishes.

\begin{corollary}\label{cor:yes}
Let $e\in\NN$. Then $e\in K$ if and only if there is no solution $x:\RR\to\RR^{n_s}$ of $\dot x=X_e(x)$ which satisfies (ii) and (iv) of Theorem~\ref{thm:selector} and passes the tests.
\end{corollary}

\begin{proof}
Suppose first that $e\in K$, and that such a solution $x$ exists. Since $x$ passes the tests, it also satisfies (i) and (iii) of Theorem~\ref{thm:selector}, so by this theorem $\mathfrak D^{\mathrm{obs}}_e(x(t))=R_0\mathfrak D_e(\varphi(0)+\omega t)$ for all $t$, with $|R_0|=1$. The Fourier series of $\mathfrak D_e$ converges absolutely (Lemma~\ref{lem:constterm}) and $k\cdot\omega\ne0$ for $k\ne0$ (Corollary~\ref{cor:irrational}), so Lemma~\ref{lem:avg} gives $\avg{\mathfrak D^{\mathrm{obs}}_e(x)}=R_0(\hat{\mathfrak D}_e)_{(0,\dots,0)}$. Since $e\in K$ and $|R_0|=1$, this average is not zero by Lemma~\ref{lem:constterm}. On the other hand, the real and imaginary parts of $\mathfrak D^{\mathrm{obs}}_e$ are coordinates of $h_e$, so $\avg{\mathfrak D^{\mathrm{obs}}_e(x)}=0$ because $x$ passes the tests, a contradiction.

Conversely, suppose that $e\notin K$, let $\omega_\ell=\sqrt{p_\ell}$ as in Lemma~\ref{lem:graph}, and put $x(t)=\Gamma_e(\omega t)$ for $t\in\RR$. By Lemma~\ref{lem:graph}(a), $x$ is a solution of $\dot x=X_e(x)$ defined on $\RR$. Since $\Gamma_e$ is continuous on the compact torus $\TT^{q+r}$ (Lemma~\ref{lem:graph}), $x$ is bounded, which gives (iv). For (ii), let $\varphi_*\in\TT^{q+r}$ be a limit point of the sequence $(n\omega)_{n\in\NN}$ in $\TT^{q+r}$, which exists by compactness, and choose integers $n_1<n_2<\cdots$ with $n_{k+1}-n_k\ge k$ and $n_k\omega\to\varphi_*$ in $\TT^{q+r}$. Since the distance on $\TT^{q+r}$ is invariant under translations, the numbers $t_k=n_{k+1}-n_k\to\infty$ satisfy $t_k\omega=n_{k+1}\omega-n_k\omega\to\varphi_*-\varphi_*=0$ in $\TT^{q+r}$, and hence $x(t_k)=\Gamma_e(t_k\omega)\to\Gamma_e(0)=x(0)$ by continuity of $\Gamma_e$. Every validity residual vanishes along $x$ by Lemma~\ref{lem:graph}(b). Finally, let $g$ be a coordinate of $h_e$. By Lemma~\ref{lem:graph}(c), $g\circ\Gamma_e$ has an absolutely convergent Fourier series on $\TT^{q+r}$, and $k\cdot\omega\ne0$ for $k\ne0$ by Lemma~\ref{lem:radicals}, so by Lemma~\ref{lem:avg} the average $\avg{g(x)}$ exists and equals the constant Fourier coefficient of $g\circ\Gamma_e$. This coefficient is zero: for the validity residuals because $g\circ\Gamma_e$ vanishes identically by Lemma~\ref{lem:graph}(b); for the real and imaginary parts of the moments by Lemma~\ref{lem:graph}(c); and for the real and imaginary parts of $\mathfrak D^{\mathrm{obs}}_e$ because the constant Fourier coefficient of $\mathfrak D^{\mathrm{obs}}_e\circ\Gamma_e=\mathfrak D_e$ is zero by Lemma~\ref{lem:graph}(c) and Lemma~\ref{lem:constterm}, as $e\notin K$. (The constant Fourier coefficient of the real part of a function with an absolutely convergent Fourier series is the real part of its constant Fourier coefficient, and similarly for the imaginary part.) Hence $x$ passes the tests, and it satisfies (ii) and (iv). This proves the corollary.
\end{proof}

\section{\texorpdfstring{From trajectories to the sign of a quantity $\beta_e$}{From trajectories to the sign of a quantity beta\_e}}\label{sec:compact}

In this section we prove the third equivalence in the chain of Section~\ref{subsec:overview2}: we turn the alternative of Corollary~\ref{cor:yes}, whether or not some trajectory of $\dot x=X_e(x)$ which is recurrent and does not grow too fast passes the tests, into the sign of a quantity $\beta_e$ of the same kind as in Lemma~\ref{lem:gascrit} (Proposition~\ref{prop:signgap}); in Section~\ref{sec:single} we turn this sign into the GAS property of a homogeneous polynomial vector field (Theorem~\ref{thm:wrapper}). The system $\dot x=X_e(x)$ lives on $\RR^{n_s}$, which is not compact, and Corollary~\ref{cor:yes} concerns individual trajectories. We compactify $\RR^{n_s}$ to a sphere by a central projection with a time change (Section~\ref{subsec:central}). We also add a variable $z\in\RR^{n_h}$ with $\dot z=\kappa h_e(x)$, which keeps a ``running total'' of $h_e$: $z(t)-z(0)$ is $\kappa$ times the integral of $h_e$ along the solution. We compactify $\RR^{n_h}$ to a sphere by stereographic projection, which adds one point at infinity (Section~\ref{subsec:compactsystem}). If $h_e$ has a nonzero average along a solution, then $z$ tends to this point; accordingly, an (ergodic) invariant measure with a nonzero average of $h_e$ is concentrated on it. The function whose maximal invariant average $\beta_e$ we compute, which we call the ``reward'', is negative there. In Section~\ref{subsec:bounds} we prove bounds along the intended trajectories, which we use to choose two constants of the construction, and in Section~\ref{subsec:signgap} we prove Proposition~\ref{prop:signgap}. Table~\ref{tab:const9} lists the constants of this section and of Section~\ref{sec:single}, and which of them depend on $e$.

\begin{table}[!ht]
\centering
\caption{The constants of Sections~\ref{sec:compact} and~\ref{sec:single}. ``Fixed'' means independent of the input $e$; the constants which depend on $e$ are computed from $e$ in step~(2) of the proof of Theorem~\ref{thm:fixedpair}. The rationals $M_0$ and $K_0$ are fixed numbers which are only shown to exist: a program for $e\mapsto F_e$ contains them, and their numerical values are not needed (see after Lemma~\ref{lem:bounds}). Constants which appear inside a single proof are not listed.}
\label{tab:const9}
\vspace{10px}
\footnotesize
\renewcommand{\arraystretch}{1.15}
\begin{tabular}{@{}llll@{}}
\hline
Constant & Defined in & Depends on & Role\\
\hline
$\lambda=\frac1{256}$, $\eta=\frac12$ & \eqref{eq:lambda} & fixed & parameters of the factors of $\mathfrak D_e$\\
$n_s$, $n_h$ & Section~\ref{sec:the_source_system} & fixed & numbers of coordinates of $x$ and $h_e$\\
$M$ (even) & Section~\ref{subsec:compactsystem} & fixed & power of $w$ in the central compactification\\
$\rho_*=1+1/m_*$, $D_h$ & \eqref{eq:rhochoice} & fixed & weight of the norm $\|\cdot\|_{\rho_*}$; degree of $h_e$\\
$M_0$, $K_0$ (rational) & Lemma~\ref{lem:bounds} & fixed; exist & $|\Gamma_e|^2\le M_0$ on $\TT^{q+r}$\\
$K_e=K_0\rho_*^{4D_h|c(e)|_1}$ & Lemma~\ref{lem:bounds} & $e$ & $|\mathcal B_e|\le K_e$ for the primitive $\mathcal B_e$\\
$w^*=\frac1{1+M_0}$, $\varepsilon=\frac{w^*}4$ & \eqref{eq:kappaeps} & fixed & $w^2\ge w^*$ on intended trajectories; margin in $\cR_e$\\
$m_e$, $\kappa_e=\frac1{m_e}$ & \eqref{eq:kappaeps} & $e$ & gain of the running total $\dot z=\kappa_eh_e$\\
$D=M+3$, $L=3$, $E=d_*$ & Section~\ref{sec:fixed} & fixed & used in Theorem~\ref{thm:wrapper}, with $k=2$\\
$n_*=n_s+n_h+3$, $d_*=M+9$ & Section~\ref{sec:fixed} & fixed & dimension and degree of $F_e$\\
\hline
\end{tabular}
\end{table}

\subsection{Central compactification}\label{subsec:central}
The invariant-measure arguments of Section~\ref{subsec:erg} need a compact space, but the system $\dot x=X_e(x)$ lives on $\RR^{n_s}$. The central projection below identifies $\RR^n$ with each open hemisphere of a sphere and, after a time change, turns a polynomial vector field into a polynomial vector field tangent to the whole sphere; trajectories which escape to infinity, possibly in finite time, approach the ``equator'' instead (Figure~\ref{fig:compactification}(a)).
This is a variant of the classical Poincar\'e compactification of polynomial vector fields, which Poincar\'e introduced in the plane~\cite{Poincare1881} and which was extended to $\RR^n$ in~\cite{Velasco69,CimaLlibre}; see~\cite[Chapter~5]{DLA} for the planar case. Since we use a power $w^M$ with $M$ even and at least the degree of the vector field, rather than the power determined by the degree, we include the short proof of the properties we need.

Let $X$ be a polynomial vector field on $\RR^n$ and $M$ an even integer at least the degree of $X$. On $S^n=\{(y,w)\in\RR^n\times\RR:|y|^2+w^2=1\}$ define $\widehat X=(\widehat X_y,\widehat X_w)$ by
\begin{equation}\label{eq:central}
\widehat X_y=w^{M+1}\Bigl(X(y/w)-\bigl(y\cdot X(y/w)\bigr)y\Bigr),\qquad\widehat X_w=-w^{M+2}\,y\cdot X(y/w),
\end{equation}
expanded as polynomials. For a polynomial $g$ on $\RR^n$ of degree at most $M$ put $\hat g(y,w)=w^Mg(y/w)$.

\begin{lemma}\label{lem:central}
$\widehat X$ and $\hat g$ are polynomials with rational coefficients if $X$ and $g$ have rational coefficients, and $\widehat X$ is tangent to $S^n$. The sphere $S^n$, the equator $\{w=0\}$ and the two open hemispheres $\{\pm w>0\}$ are invariant, and $\widehat X$ generates a continuous flow on $S^n$. On each open hemisphere $\pm w>0$, the map $(y,w)\mapsto x=y/w$ is a diffeomorphism onto $\RR^n$ which sends $\widehat X$ to $(1+|x|^2)^{-M/2}X(x)$, and $\hat g=(1+|x|^2)^{-M/2}g(x)$ and $w^2=(1+|x|^2)^{-1}$ there.
\end{lemma}

\begin{proof}
A monomial of degree $d\le M$ in $y/w$ multiplied by $w^{M+1}$ or $w^{M+2}$ is a polynomial, and so is $w^M$ times it. On $S^n$, $\frac{d}{d\sigma}(|y|^2+w^2)=2w^{M+1}(y\cdot X)(1-|y|^2-w^2)=0$. The component $\widehat X_w$ is divisible by $w$. On a hemisphere, $x=y/w$ satisfies $\frac{dx}{d\sigma}=\widehat X_y/w-y\widehat X_w/w^2=w^MX(x)$, and $|y|^2+w^2=1$ gives $w^2=(1+|x|^2)^{-1}$; since $M$ is even, $w^M=(1+|x|^2)^{-M/2}$.
\end{proof}

Thus on each hemisphere the trajectories of $\widehat X$ are the trajectories of $X$, traversed with the time change $dt=w^M\,d\sigma$; trajectories of $X$ which escape to infinity in finite time become trajectories which approach the equator in infinite time.

\input{fig_compactification}

\subsection{\texorpdfstring{The system on $S^{n_s}\times S^{n_h}$ and the reward}{The system on the product of two spheres and the reward}}\label{subsec:compactsystem}
Let $X_e$ and $h_e$ be the vector field and the map of Section~\ref{sec:the_source_system}, and let $n_h$ be the number of coordinates of $h_e$. Fix an even integer $M\ge\max\{\deg X_e,\deg h_e\}$; it does not depend on $e$. For a rational gain $\kappa>0$ and a rational margin $\varepsilon>0$, to be chosen in \eqref{eq:kappaeps}, let $\cK=S^{n_s}\times S^{n_h}$, with points $(y,w,u,u_0)$, $|u|^2+u_0^2=1$, and define on $\cK$ the polynomial vector field
\begin{equation}\label{eq:compactsource}
(\dot y,\dot w)=\widehat X_e(y,w),\qquad\dot u=(1+u_0)\,\kappa\hat h_e-\kappa(u\cdot\hat h_e)\,u,\qquad\dot u_0=-(1+u_0)\,\kappa(u\cdot\hat h_e),
\end{equation}
with $\hat h_e=\hat h_e(y,w)$, together with the polynomial \emph{reward}
\begin{equation}\label{eq:reward}
\cR_e(y,w,u,u_0)=u_0+2w^2-1-\varepsilon.
\end{equation}
The vector field \eqref{eq:compactsource} is tangent to $\cK$, and $\cK$ is invariant: the equation for $(y,w)$ does not involve $(u,u_0)$, so $S^{n_s}$ is invariant by Lemma~\ref{lem:central}, and along solutions $V=|u|^2+u_0^2-1$ satisfies the linear equation $\dot V=-2\kappa(u\cdot\hat h_e)V$, so $V\equiv0$ if $V(0)=0$. Since $\cK$ is compact, \eqref{eq:compactsource} generates a continuous flow on $\cK$ (Section~\ref{subsec:ode}), which we denote by $\Psi_e$. The set where $(u,u_0)=(0,-1)$ is invariant, since $\dot u=0$ and $\dot u_0=0$ there. On $u_0\ne-1$ put $z=u/(1+u_0)\in\RR^{n_h}$ (the stereographic projection from the south pole $(0,-1)$; Figure~\ref{fig:compactification}(b)); the inverse map is $u=2z/(1+|z|^2)$, $u_0=(1-|z|^2)/(1+|z|^2)$, and a direct computation gives $\dot z=\kappa\hat h_e$. Hence, on $\{w\ne0,\ u_0\ne-1\}$, in the time $t$ of the original system, given by $dt=w^Md\sigma$, and the coordinates $x=y/w$ and $z$, the vector field \eqref{eq:compactsource} is the skew product
\begin{equation}\label{eq:skew}
\dot x=X_e(x),\qquad\dot z=\kappa\,h_e(x).
\end{equation}

Thus $z(t)-z(0)=\kappa\int_0^th_e(x(s))\,ds$ along a solution, that is, $z$ keeps a ``running total'' of $h_e$. The reward \eqref{eq:reward} is at most $-\varepsilon$ on the equator $w=0$, where $x$ is at infinity, because $u_0\le1$, and on the set $u_0=-1$, where $z$ is at infinity, because $w^2\le1$; where $z=0$ it equals $2w^2-\varepsilon$, which is positive if $x$ is bounded and $\varepsilon$ is small. Lemmas~\ref{lem:selectedorbit} and~\ref{lem:accumulator} concern an ergodic invariant measure $\nu_0$ of the flow of $\widehat X_e$ on $S^{n_s}$ which gives the equator measure zero: if $\int\hat h_e\,d\nu_0=0$, then $\nu_0$ yields a solution of $\dot x=X_e(x)$ which is forward recurrent, does not grow too fast and passes the tests; otherwise every invariant measure of $\Psi_e$ whose projection to $S^{n_s}$ is $\nu_0$ is concentrated on the set $u_0=-1$.

\begin{lemma}\label{lem:selectedorbit}
Let $\nu_0$ be an ergodic invariant probability measure of the flow of $\widehat X_e$ on $S^{n_s}$ with $\nu_0(\{w\ne0\})=1$ and $\int\hat h_e\,d\nu_0=0$. Then there is a solution $x:\RR\to\RR^{n_s}$ of $\dot x=X_e(x)$ which satisfies (i)--(iv) of Theorem~\ref{thm:selector} with $\mu=M$, and along which $\avg{h_e(x)}=0$; in particular, $x$ satisfies (ii) and (iv) and passes the tests (Section~\ref{subsec:pass_tests}).
\end{lemma}

\begin{proof}
Write $f=w^M$ and $a=\int f\,d\nu_0>0$. Each validity residual $v$ is a coordinate of $h_e$ and $v\ge0$, so $\hat v=fv(y/w)\ge0$ and $\int\hat v\,d\nu_0=0$ give $\hat v=0$ $\nu_0$-almost everywhere. By Fubini's theorem, for $\nu_0$-almost every $p$ the open set $\{\sigma:\hat v(\Phi_{\widehat X_e}^\sigma p)>0\}$ has Lebesgue measure zero, hence is empty. By \ref{E:birkhoff} and \ref{E:poincare} we may therefore choose $p_0$ with $w(p_0)\ne0$ such that: $p_0$ is forward recurrent; $\hat v(\Phi_{\widehat X_e}^\sigma p_0)=0$ for all $\sigma$ and all validity residuals $v$; $\frac1\sigma\int_0^\sigma f(\Phi_{\widehat X_e}^sp_0)ds\to a$ as $\sigma\to\pm\infty$; and $\frac1\sigma\int_0^\sigma\hat h_e(\Phi_{\widehat X_e}^sp_0)ds\to0$ as $\sigma\to+\infty$. The trajectory of $p_0$ stays in the open hemisphere containing $p_0$, which is invariant by Lemma~\ref{lem:central}, so $f(\Phi_{\widehat X_e}^\sigma p_0)>0$ for all $\sigma$. Hence $t(\sigma)=\int_0^\sigma f(\Phi_{\widehat X_e}^sp_0)\,ds$ is $C^1$ with $t'(\sigma)=f(\Phi_{\widehat X_e}^\sigma p_0)>0$, and $t(\sigma)/\sigma\to a>0$ as $\sigma\to\pm\infty$; so $t$ is an increasing bijection of $\RR$ onto $\RR$ with a $C^1$ inverse $\sigma(t)$. Define $x(t)=y(\Phi_{\widehat X_e}^{\sigma(t)}p_0)/w(\Phi_{\widehat X_e}^{\sigma(t)}p_0)$ for $t\in\RR$. By Lemma~\ref{lem:central}, $\frac{d}{d\sigma}(y/w)(\Phi_{\widehat X_e}^\sigma p_0)=f\,X_e(y/w)$ at $\Phi_{\widehat X_e}^\sigma p_0$, so the chain rule and $\sigma'(t)=1/f(\Phi_{\widehat X_e}^{\sigma(t)}p_0)$ give $\dot x(t)=X_e(x(t))$ for all $t\in\RR$. Property (i) holds because $v(x)=\hat v/f$. If $\Phi_{\widehat X_e}^{\sigma_k}p_0\to p_0$ with $\sigma_k\to\infty$, then $x(t(\sigma_k))\to x(0)$, which is (ii). By the change of variables $dt=f\,d\sigma$ and $\hat h_e=fh_e(x)$,
\[
\frac1{t(\sigma)}\int_0^{t(\sigma)}h_e(x)\,dt=\frac{\sigma}{t(\sigma)}\cdot\frac1\sigma\int_0^\sigma\hat h_e(\Phi_{\widehat X_e}^sp_0)\,ds\to0,\qquad\frac1{t(\sigma)}\int_0^{t(\sigma)}(1+|x|^2)^{M/2}dt=\frac{\sigma}{t(\sigma)}\to\frac1a,
\]
which gives (iii), $\avg{h_e(x)}=0$, and (iv).
\end{proof}

\begin{lemma}\label{lem:accumulator}
Let $\Phi$ be a continuous flow on a compact metric space $\cK_0$, $g:\cK_0\to\RR^m$ continuous, and consider on $\cK_0\times S^m$ the skew product over $\Phi$ given by $\dot u=(1+u_0)g-(u\cdot g)u$, $\dot u_0=-(1+u_0)(u\cdot g)$; that is, a point $(b,u,u_0)$ moves to $(\Phi^\sigma b,u(\sigma),u_0(\sigma))$, where $(u,u_0)$ solves these equations with $g=g(\Phi^\sigma b)$. Assume that this defines a continuous flow on $\cK_0\times S^m$ (in our application it is the flow $\Psi_e$ of \eqref{eq:compactsource}). If $\nu_0\in\cM(\Phi)$ is ergodic and $\int g\,d\nu_0\ne0$, then every invariant probability measure $\nu$ of the skew product whose projection to $\cK_0$ is $\nu_0$ satisfies $\nu(\{u_0=-1\})=1$.
\end{lemma}

\begin{proof}
The set $E=\{u_0\ne-1\}$ is invariant, since its complement $\cK_0\times\{(0,-1)\}$ is. Suppose $\nu(E)>0$, and let $\nu'=\nu(\cdot\cap E)/\nu(E)$, an invariant probability measure. Its projection to $\cK_0$ is at most $\nu_0/\nu(E)$, so the set $\mathcal G$ of points $b$ with $\frac1\sigma\int_0^\sigma g(\Phi^sb)\,ds\to\bar g:=\int g\,d\nu_0$, which has full $\nu_0$-measure by \ref{E:birkhoff}, has full measure for this projection. By \ref{E:poincare}, applied to $\nu'$, there is a forward recurrent point $p\in E$ of the skew product whose projection $b$ to $\cK_0$ lies in $\mathcal G$. Its trajectory stays in the invariant set $E$, and along it $z=u/(1+u_0)$ satisfies $\dot z=g(\Phi^\sigma b)$, by the computation before \eqref{eq:skew} with $\kappa\hat h_e$ replaced by $g$; so $z(\sigma)-z(0)=\int_0^\sigma g(\Phi^sb)\,ds$, so $|z(\sigma)|/\sigma\to|\bar g|>0$, hence $(u,u_0)(\sigma)\to(0,-1)$ as $\sigma\to\infty$, contradicting recurrence to $p\in E$.
\end{proof}

\subsection{Bounds along the intended trajectories}\label{subsec:bounds}
To choose $\kappa$ and $\varepsilon$ we need computable bounds on the map $\Gamma_e$ of Lemma~\ref{lem:graph}, which parametrizes the intended trajectories, and on an integral of $h_e$ along them. Let $d=q+r$. For $\rho_*\ge1$ and a function $f$ on $\TT^d$ with Fourier coefficients $\hat f_k$ put $\|f\|_{\rho_*}=\sum_k|\hat f_k|\rho_*^{|k|_1}\in[0,\infty]$. Then $\|f+g\|_{\rho_*}\le\|f\|_{\rho_*}+\|g\|_{\rho_*}$, $\|fg\|_{\rho_*}\le\|f\|_{\rho_*}\|g\|_{\rho_*}$, $\|\bar f\|_{\rho_*}=\|f\|_{\rho_*}$, $\|e^{ik\cdot\varphi}\|_{\rho_*}=\rho_*^{|k|_1}$, $\sup|f|\le\|f\|_1\le\|f\|_{\rho_*}$, and $\|1/f\|_{\rho_*}\le(1-\delta)^{-1}$ whenever $\|1-f\|_{\rho_*}\le\delta<1$, by the geometric series.

Fix a rational $\rho_*=1+1/m_*>1$, with $m_*$ the least positive integer such that
\begin{equation}\label{eq:rhochoice}
\lambda\rho_*^{4+4|B|_1}\le\tfrac18\qquad\text{and}\qquad\eta\rho_*^{4+4|A|_1}\le\tfrac34,
\end{equation}
where $|A|_1$ and $|B|_1$ are the sums of the absolute values of all entries of $A$ and $B$. Since $\lambda=\frac1{256}$ and $\eta=\frac12$ by \eqref{eq:lambda}, both inequalities hold strictly at $\rho_*=1$; as $1+1/m\to1$, the integer $m_*$ exists, and it is found by testing finitely many inequalities between rational numbers. Let $D_h$ be the degree of the map $h_e$. Neither $\rho_*$ nor $D_h$ depends on $e$.

\begin{lemma}\label{lem:bounds}
There are rational numbers $M_0\ge1$ and $K_0>0$, which do not depend on $e$, with the following properties. Put $K_e=K_0\rho_*^{4D_h|c(e)|_1}$. First, $|\Gamma_e(\varphi)|^2\le M_0$ for all $\varphi$. Second, the function
\begin{equation}\label{eq:primitive}
\mathcal B_e(\varphi)=\sum_{k\ne0}\frac{\hat h_k}{i\,k\cdot\omega}\,e^{ik\cdot\varphi},\qquad\text{where }h_e\circ\Gamma_e=\sum_k\hat h_ke^{ik\cdot\varphi},
\end{equation}
is continuous, takes values in $\RR^{n_h}$, satisfies $|\mathcal B_e|\le K_e$ for the Euclidean norm, and $\frac{d}{dt}\mathcal B_e(\varphi+\omega t)=h_e(\Gamma_e(\varphi+\omega t))-\hat h_0$. If $e\notin K$, then $\hat h_0=0$.
\end{lemma}

\begin{proof}
\emph{Step 1: the coordinates of $\Gamma_e$.} All coordinates of $\Gamma_e$ except $R$ are independent of $e$; we show that they have finite $\|\cdot\|_{\rho_*}$. The controllers are constants, and $\zeta_\ell$ and $G_j$ are constant multiples of characters. Since $e^{\pi i\tau_j\sigma}=\lambda^\sigma e^{4i\sigma\varphi_j}$, the theta constants and $E_2$ at $\tau_j$ are series $\sum_nc_n\lambda^{\sigma_n}e^{4i\sigma_n\varphi_j}$ in which $|c_n|$ grows at most polynomially and $\sigma_n\ge0$ grows at least linearly in $|n|$; as $\lambda\rho_*^4<1$, their norms are finite, and so are those of the coordinates of $\mathbf m_j$, of the $\psi_{k,j}$ and of $H_j$, which are products of them. Since $w_j=z_j+\frac{1+\tau_j}2$ and $\pi z_j=2\sum_lB_{lj}\varphi_{q+l}$, the $n$-th term of each of $T_{1,j},\dots,T_{4,j},\Delta_j$ has the form $c\,\lambda^{\sigma}e^{ik\cdot\varphi}$ with $|c|\le4(2|n|+1)$, $|k|_1\le4|\sigma|+2(2|n|+1)|B|_1$, and $\sigma$ a polynomial of degree two in $n$ with leading coefficient one; for instance $\sigma=(n+1)^2-\frac14$ for $T_{1,j}$, $T_{2,j}$ and $\Delta_j$, which is negative for $n=-1$. Only finitely many terms have $\sigma<0$, and for the others the norm of the term is at most $4(2|n|+1)(\lambda\rho_*^4)^\sigma\rho_*^{2(2|n|+1)|B|_1}$; since $\lambda\rho_*^4<1$ and $\sigma$ grows like $n^2$, the norms are again finite. Since $|n|\le n^2$, \eqref{eq:rhochoice} gives
\[
\|1-\Theta_\lambda(Q_j,V_j)\|_{\rho_*}\le\sum_{n\ne0}\lambda^{n^2}\rho_*^{4n^2+4|n||B|_1}\le2\sum_{n\ge1}8^{-n}=\tfrac27,
\]
so $\|1/\Theta_\lambda(Q_j,V_j)\|_{\rho_*}\le\frac75$, and $U_j=G_j/\Theta_\lambda(Q_j,V_j)$ (by \ref{J:half}) has finite norm. Finally $G^\pm_j=\sum_{m\ge0}\eta^mW_j^{\pm m}$ with $W_j=e^{-4i\varphi_j+4i\sum_lA_{lj}\varphi_{q+l}}$, so $\|G^\pm_j\|_{\rho_*}\le\sum_m(\eta\rho_*^{4+4|A|_1})^m\le\sum_m(\tfrac34)^m=4$. On the other hand $R=e^{-4i\sum_lc_l(e)\varphi_{q+l}}$ has $|R|=1$ and $\|R\|_{\rho_*}=\rho_*^{4|c(e)|_1}$. Since $\|\re f\|_{\rho_*}$, $\|\im f\|_{\rho_*}\le\|f\|_{\rho_*}$, there is a real $C_\Gamma\ge1$, independent of $e$, such that every real coordinate of $\Gamma_e$ has norm at most $C_\Gamma$, except the real and imaginary parts of $R$, which have norm at most $\rho_*^{4|c(e)|_1}$. As $\sup|f|\le\|f\|_{\rho_*}$ and $|R|=1$, we get $|\Gamma_e|^2\le(n_s-2)C_\Gamma^2+1\le n_sC_\Gamma^2$, so any rational $M_0\ge n_sC_\Gamma^2$ satisfies $|\Gamma_e|^2\le M_0$.

\emph{Step 2: the map $h_e$.} The coordinates $h_{e,i}$ of $h_e$ are polynomials of degree at most $D_h$ whose coefficients do not depend on $e$. Let $c_h$ be the sum of the absolute values of all their coefficients. By Step~1 and the properties of $\|\cdot\|_{\rho_*}$, every monomial of degree at most $D_h$ in the coordinates of $\Gamma_e$ has norm at most $C_\Gamma^{D_h}\rho_*^{4D_h|c(e)|_1}$, so
\[
\sum_i\|h_{e,i}\circ\Gamma_e\|_{\rho_*}\le c_hC_\Gamma^{D_h}\rho_*^{4D_h|c(e)|_1}.
\]

\emph{Step 3: the primitive.} Let $\hat h_{k,i}$ be the Fourier coefficients of $h_{e,i}\circ\Gamma_e$. By Lemma~\ref{lem:radicals} with $s=2^d-1$ and $P=p_d$, $|k\cdot\omega|^{-1}\le P^{s/2}|k|_1^s$ for $k\ne0$, and for $n\ge1$ we have $n^s\rho_*^{-n}\le(sm_*)^s$ (for $n<s$ because $n^s\rho_*^{-n}\le s^s$, and for $n\ge s$ because $\rho_*^n=(1+1/m_*)^n\ge\binom ns m_*^{-s}\ge(n/s)^sm_*^{-s}$). Hence, using that the Euclidean norm of a vector is at most the sum of the absolute values of its coordinates,
\[
|\mathcal B_e(\varphi)|\le\sum_i\sum_{k\ne0}\frac{|\hat h_{k,i}|}{|k\cdot\omega|}\le P^{s/2}(sm_*)^s\sum_i\sum_k|\hat h_{k,i}|\rho_*^{|k|_1}\le P^{s/2}(sm_*)^sc_hC_\Gamma^{D_h}\rho_*^{4D_h|c(e)|_1},
\]
and any rational $K_0\ge P^{s/2}(sm_*)^sc_hC_\Gamma^{D_h}$ works. In particular \eqref{eq:primitive} converges absolutely and uniformly, so $\mathcal B_e$ is continuous, and termwise differentiation, justified by the absolute convergence of $\sum_k|\hat h_k|$, gives the derivative along the flow. Since $h_e$ is real, $\hat h_{-k}=\overline{\hat h_k}$, so $\mathcal B_e$ is real-valued. If $e\notin K$, then $\hat h_0=0$ by Lemma~\ref{lem:graph}(b),(c) and Lemma~\ref{lem:constterm}.
\end{proof}

The numbers $\rho_*$, $D_h$, $M_0$ and $K_0$ are fixed constants of the construction. As for the matrices $A,B$ of Lemma~\ref{lem:template}, a computer program may contain $M_0$ and $K_0$, and we do not need their numerical values. Then we compute $K_e$ from $e$ by the formula in Lemma~\ref{lem:bounds}, without testing whether $e\in K$. Put
\begin{equation}\label{eq:kappaeps}
w^*=\frac1{1+M_0},\qquad\kappa_e=\frac1{m_e},\quad m_e\text{ the least positive integer with }m_e^2w^*\ge4K_e^2,\qquad\varepsilon=\frac{w^*}4,
\end{equation}
and use $\kappa=\kappa_e$ in \eqref{eq:compactsource} and this $\varepsilon$ in \eqref{eq:reward}. Of these numbers, only $m_e$ and $\kappa_e$ depend on $e$.

\subsection{\texorpdfstring{The sign of $\beta_e$}{The sign of beta\_e}}\label{subsec:signgap}
\input{fig_signgap}
We now show that the sign of the maximal invariant average of the reward decides whether $e\in K$ (Figure~\ref{fig:signgap}). If $e\in K$, then by Lemma~\ref{lem:selectedorbit} and Corollary~\ref{cor:yes}, every ergodic invariant measure which gives the equator measure zero has a nonzero average of $\hat h_e$, so by Lemma~\ref{lem:accumulator} it is concentrated on the set $u_0=-1$, where $z$ is at infinity; the reward is at most $-\varepsilon$ there and on the equator. If $e\notin K$, the intended trajectories, along which $z$ stays bounded, carry an invariant measure with positive average reward. In Section~\ref{sec:fixed} we combine this with Theorem~\ref{thm:wrapper} to prove Theorem~\ref{thm:fixedpair}.

\begin{proposition}[Sign of the maximal average of the reward]\label{prop:signgap}
Let $\beta_e=\beta(\Psi_e,\cR_e)$ be the maximal invariant average of the reward $\cR_e$ for the flow $\Psi_e$ of \eqref{eq:compactsource} on $\cK$. If $e\in K$, then $\beta_e\le-\varepsilon<0$. If $e\notin K$, then $\beta_e\ge\frac54w^*>0$.
\end{proposition}

\begin{proof}
We treat the cases $e\in K$ and $e\notin K$ in turn. Let $e\in K$. By \ref{E:decomp} it suffices to show $\int\cR_e\,d\nu\le-\varepsilon$ for every ergodic $\nu\in\cM(\Psi_e)$; fix such a $\nu$. The projection $\nu_0$ of $\nu$ to $S^{n_s}$ is an ergodic invariant measure for $\widehat X_e$ by \ref{E:factor}, and the equator $\{w=0\}$ is invariant, so $\nu_0(\{w=0\})\in\{0,1\}$. If it is $1$, then $w^2=0$ and $u_0\le1$ give $\cR_e\le-\varepsilon$ $\nu$-almost everywhere. Otherwise $\nu_0(\{w\ne0\})=1$. If $\int\hat h_e\,d\nu_0=0$, Lemma~\ref{lem:selectedorbit} gives a solution of $\dot x=X_e(x)$ which satisfies (ii) and (iv) of Theorem~\ref{thm:selector} and passes the tests; since $e\in K$, this contradicts Corollary~\ref{cor:yes}. Hence $\int\kappa_e\hat h_e\,d\nu_0\ne0$, and Lemma~\ref{lem:accumulator}, applied with $\cK_0=S^{n_s}$, the flow of $\widehat X_e$ and $g=\kappa_e\hat h_e$, gives $u_0=-1$ $\nu$-almost everywhere, where $\cR_e=2w^2-2-\varepsilon\le-\varepsilon$. Thus $\int\cR_e\,d\nu\le-\varepsilon$ for every ergodic $\nu\in\cM(\Psi_e)$, and $\beta_e\le-\varepsilon$.

Let $e\notin K$. Consider the set $\widehat\Gamma_e\subseteq\cK$ of points with $w>0$, $y/w=\Gamma_e(\varphi)$ and $u/(1+u_0)=\kappa_e\mathcal B_e(\varphi)$ for some $\varphi\in\TT^d$. It is the image of $\TT^d$ under the continuous map $\varphi\mapsto p_\varphi$, where $p_\varphi\in\cK$ has $(y,w)=(\Gamma_e(\varphi),1)(1+|\Gamma_e(\varphi)|^2)^{-1/2}$ and $(u,u_0)$ is the inverse stereographic image of $z=\kappa_e\mathcal B_e(\varphi)$; so it is compact. By Lemma~\ref{lem:bounds}, $w^2=(1+|\Gamma_e|^2)^{-1}\ge w^*$ on $\widehat\Gamma_e$. Fix $\varphi$. By Lemma~\ref{lem:graph}(a) and Lemma~\ref{lem:bounds} (with $\hat h_0=0$), $(x,z)(t)=(\Gamma_e(\varphi+\omega t),\kappa_e\mathcal B_e(\varphi+\omega t))$, $t\in\RR$, solves the skew product \eqref{eq:skew}. We undo the time change. The function $\sigma(t)=\int_0^t(1+|x(t')|^2)^{M/2}\,dt'$ satisfies $1\le\sigma'(t)\le(w^*)^{-M/2}$, so it is a $C^1$ bijection of $\RR$ onto $\RR$ with a $C^1$ inverse $t(\sigma)$. In the coordinates $(x,z)$ on $\{w>0,\ u_0\ne-1\}$, the vector field \eqref{eq:compactsource} is $(1+|x|^2)^{-M/2}\bigl(X_e(x),\kappa_eh_e(x)\bigr)$, by Lemma~\ref{lem:central} and the computation of $\dot z$ before \eqref{eq:skew}. By the chain rule, $\sigma\mapsto p_{\varphi+\omega t(\sigma)}$ is therefore a solution of \eqref{eq:compactsource}, defined for all $\sigma\in\RR$, which starts at $p_\varphi$; by uniqueness of solutions it is the trajectory of $p_\varphi$ under $\Psi_e$. Hence $\widehat\Gamma_e$ is invariant. On $\widehat\Gamma_e$, with $|z|=\kappa_e|\mathcal B_e|\le\kappa_eK_e$ and $\kappa_e^2K_e^2\le w^*/4$,
\[
u_0=\frac{1-|z|^2}{1+|z|^2}\ge1-2|z|^2\ge1-\frac{w^*}2,\qquad\text{so}\qquad\cR_e\ge1-\frac{w^*}2+2w^*-1-\frac{w^*}4=\frac54w^*.
\]
By \ref{E:compact} applied to the restriction of $\Psi_e$ to $\widehat\Gamma_e$, there is a probability measure on $\widehat\Gamma_e$ which is invariant under this restriction; extended by zero to $\cK$, it lies in $\cM(\Psi_e)$, and its reward average is at least $\frac54w^*$. Hence $\beta_e\ge\frac54w^*$.
\end{proof}

Together with Corollary~\ref{cor:yes}, Proposition~\ref{prop:signgap} shows that no trajectory of $\dot x=X_e(x)$ which satisfies (ii) and (iv) of Theorem~\ref{thm:selector} passes the tests if and only if $\beta_e<0$; this is the third equivalence in the chain of Section~\ref{subsec:overview2}. The set $\{u_0=-1\}$, where $z$ is at infinity, is invariant and carries negative reward also when $e\notin K$; what matters for $e\notin K$ is that some invariant measure has positive average reward.

\section{\texorpdfstring{From the sign of $\beta_e$ to a single $(n_*,d_*)$}{From the sign of beta\_e to a single (n*, d*)}}\label{sec:single}

In this section we prove the last equivalence in the chain of Section~\ref{subsec:overview2}: we turn the sign of the quantity $\beta_e=\beta(\Psi_e,\cR_e)$ of Section~\ref{sec:compact} into the GAS property of a homogeneous polynomial vector field $F_e$, and we check that the dimension and the degree of $F_e$ do not depend on $e$, which proves Theorem~\ref{thm:fixedpair}. In Section~\ref{sec:wrapper} we prove Theorem~\ref{thm:wrapper}, which converts a polynomial vector field $X$ tangent to a product $\cK$ of spheres, together with a polynomial $a$, into a homogeneous polynomial vector field whose origin is GAS exactly when the maximal invariant average \eqref{eq:beta} of $a$ for the flow of $X$ on $\cK$ is negative. In Section~\ref{sec:fixed} we apply it to the vector field \eqref{eq:compactsource}, whose flow is $\Psi_e$, and to the reward \eqref{eq:reward}, and we combine Theorem~\ref{thm:fixedpair} with Theorem~\ref{thm:transfer} to prove Theorem~\ref{thm:main}. The constants used in Section~\ref{sec:fixed} are listed in Table~\ref{tab:const9}.

\subsection{From a flow on a product of spheres to a homogeneous vector field}\label{sec:wrapper}

As explained in Section~\ref{subsec:overview2}, the idea is to reverse the polar decomposition of Section~\ref{subsec:polar}: if $Y$ is a homogeneous polynomial vector field of odd degree $E$ which is tangent to a sphere and $\mathcal A$ is a homogeneous polynomial of degree $E-1$, then on this sphere the vector field $F(v)=Y(v)+\mathcal A(v)\,v$ has angular vector field $Y$ and radial rate $\mathcal A$. Since $X$ lives on a product of spheres, we first embed two copies of $\cK$ into a single sphere by adding a coordinate $v_0$, and we extend the flow and the function from these copies to the whole sphere in such a way that all other invariant measures give average at most $-1$; then we make everything homogeneous. The idea of prescribing the angular flow and the radial rate of a homogeneous polynomial vector field is also used by Korda~\cite[Section~2]{Korda}, who constructs, from a polynomial vector field $P$ on $\RR^d$ and a polynomial $b$, a homogeneous polynomial vector field of odd degree whose angular vector field is a positive multiple of the Poincar\'e compactification of $P$ and whose radial rate on the sphere is given by $b$ \cite[(2.4), (2.5) and Remark~2.3]{Korda}, and who uses it for Lyapunov stability \cite[Proposition~2.4]{Korda}. In Theorem~\ref{thm:wrapper} we start instead from a flow on a product of spheres, and the additional work consists in placing two copies of it in one sphere so that the other invariant measures do not affect the sign of the maximal invariant average \eqref{eq:wrapperavg}.

\begin{theorem}[From a maximal invariant average to GAS]\label{thm:wrapper}
Let $k\ge1$, $n_1,\dots,n_k\ge1$, $\cK=S^{n_1-1}\times\dots\times S^{n_k-1}\subseteq\RR^{n_1}\times\dots\times\RR^{n_k}$, let $X=(X_1,\dots,X_k)$ be a polynomial vector field with rational coefficients which is tangent to $\cK$ at every point of $\cK$, and let $a$ be a polynomial with rational coefficients. Let $D\ge3$ be odd with $D\ge\deg X$ and $D-1\ge\deg a$, and put $L=k+1$ and $E=D+2L$. Then one can compute from $(X,a,D)$ a homogeneous polynomial vector field $F$ on $\RR^{1+n_1+\dots+n_k}$ of degree $E$ with rational coefficients such that
\begin{equation}\label{eq:wrapper}
0\text{ is GAS for }\dot v=F(v)\iff\beta(\Phi_X|_\cK,a)<0.
\end{equation}
\end{theorem}

\begin{proof}
\emph{Step 1 (a vector field on one sphere).} Write $v=(v_0,v_1,\dots,v_k)$ with $v_0\in\RR$ and $v_i\in\RR^{n_i}$, write $x_i=v_i$ for $1\le i\le k$ and $x=(x_1,\dots,x_k)$, and put $s_i=|v_i|^2$ and $|v|^2=\sum_{i=0}^ks_i$. Define the homogeneous polynomials
\[
\Pi_0=0,\qquad\Pi_i(v)=v_0^DX_i(x/v_0),\ i=1,\dots,k,\qquad\alpha(v)=v_0^{D-1}a(x/v_0),
\]
of degrees $D$ and $D-1$. Now, put
\[
T_i=s_i\Pi_i-v_i(v_i\cdot\Pi_i),\qquad\bar f(s)=\frac1L\sum_{j=0}^ks_j(1-s_j),\qquad N_i=(1-s_i-\bar f(s))v_i,\qquad\mathcal B=T+N,
\]
and $\mathcal H=\prod_{i=0}^ks_i$, $b=\mathcal H\alpha-(1-\mathcal H)$. We have $v_i\cdot T_i=0$, so $\frac d{dt}s_i=2v_i\cdot N_i=2s_i(1-s_i-\bar f)$ along $\dot v=\mathcal B(v)$, and $v\cdot\mathcal B(v)=\sum_is_i(1-s_i)-\bar f|v|^2$, which vanishes on the sphere $S_L=\{|v|^2=L\}$. Thus $\mathcal B$ is tangent to $S_L$.

\begin{figure}[t]
\centering
\begin{tikzpicture}[
  >={Stealth[length=2.2mm]},
  sheet/.style={line width=1.7pt, black},
  sheetback/.style={line width=0.9pt, black, dashed},
  minus/.style={line width=1.3pt, black!45},
  minusback/.style={line width=0.9pt, black!45, densely dotted},
  push/.style={->, line width=0.7pt, black},
  lab/.style={font=\footnotesize, align=right, anchor=east, inner sep=1pt},
  leader/.style={line width=0.3pt, black!60},
]
  \def\R{2.4}
  \draw[line width=0.6pt] (0,0) circle (\R);
  \draw[->, line width=0.4pt, black!60] (0,-2.95) -- (0,3.1)
    node[above, font=\footnotesize, black] {$v_0$};
  \draw[minusback] (\R,0) arc[start angle=0, end angle=180, x radius=\R, y radius=0.62];
  \draw[minus] (-\R,0) arc[start angle=180, end angle=360, x radius=\R, y radius=0.62];
  \foreach \cy in {1.64,-1.64} {
    \draw[sheetback] (1.70,\cy) arc[start angle=0, end angle=180, x radius=1.70, y radius=0.44];
    \draw[sheet] (-1.70,\cy) arc[start angle=180, end angle=360, x radius=1.70, y radius=0.44];
    \draw[sheet, -{Stealth[length=3.2mm, width=2.6mm]}]
      ([shift={(250:1.70 and 0.44)}]0,\cy)
      arc[start angle=250, end angle=292, x radius=1.70, y radius=0.44];
  }
  \foreach \py in {2.32,-2.32}
    \filldraw[fill=black!45, draw=black, line width=0.5pt] (0,\py) circle (2.6pt);
  \def\Rp{2.72}
  \draw[push] (73:\Rp) arc[start angle=73, end angle=50, radius=\Rp];
  \draw[push] (7:\Rp) arc[start angle=7, end angle=36, radius=\Rp];
  \draw[push] (-7:\Rp) arc[start angle=-7, end angle=-36, radius=\Rp];
  \draw[push] (-73:\Rp) arc[start angle=-73, end angle=-50, radius=\Rp];
  \node[font=\footnotesize, align=left, anchor=west] at (3.0,0)
    {term $N$ of $\mathcal B$\\ ($\mathcal B$ generates $\widehat{\Psi}_e$):\\
     moves points along\\ meridians, towards\\ $Z_{\{0,1\}}$ \;($Z$)};
  \draw[leader] (-1.72,1.64) -- (-3.0,1.64);
  \node[lab] at (-3.05,1.64)
    {$\{v_0=1\}\times \cK\subseteq Z_{\{0,1\}}$ \;($v_0=1$ part of $Z$)\\
     copy of $(X,a)$ via $y=x/v_0$ \;(copy of $(\Psi_e,\cR_e)$)};
  \draw[leader] (-\R,0) -- (-3.0,0);
  \node[lab] at (-3.05,0)
    {$Z_{\{1\}}$, the equator $v_0=0$\\
     $b=-1$ \;($\widehat{\cR}_e=-1$)};
  \draw[leader] (-1.72,-1.64) -- (-3.0,-1.64);
  \node[lab] at (-3.05,-1.64)
    {$\{v_0=-1\}\times \cK\subseteq Z_{\{0,1\}}$ \;($v_0=-1$ part of $Z$)\\
     copy of $(X,a)$ via $y=x/v_0$ \;(copy of $(\Psi_e,\cR_e)$)};
  \draw[leader] (-0.1,2.36) -- (-0.45,2.85);
  \node[lab] at (-0.5,2.9)
    {$Z_{\{0\}}$, the two poles: $b=-1$ \;($\widehat{\cR}_e=-1$)};
  \draw[leader] (-0.1,-2.36) -- (-0.45,-2.85);
  \node[lab] at (-0.5,-2.9)
    {$Z_{\{0\}}$: $b=-1$ \;($\widehat{\cR}_e=-1$)};
  \node[font=\footnotesize] at (0,-3.55)
    {$S_L=\{|v|^2=2\}\subseteq\RR^3$ \;($\sqrt{2}\,S^{2}\subseteq\RR^3$)};
\end{tikzpicture}
\caption{The construction in the proof of Theorem~\ref{thm:wrapper} in the
smallest case $k=1$, $n_1=2$, so that $\cK=S^1$ (a single circle in place of
$S^{n_s}\times S^{n_h}$ in Section~\ref{sec:single}). Adding the
coordinate $v_0$ gives the sphere $S_L=\{|v|^2=L\}\subseteq\RR^3$ with $L=2$
(the analogue of $\sqrt{3}\,S^{n_s+n_h+2}$). The two sheets
$\{v_0=\pm1\}\times \cK$, which form $Z_{\{0,1\}}$ (the set $Z$), are copies of
$\cK$: on each of them, the flow of $\mathcal B$ and the function $b$ are conjugate, via
$y=x/v_0$, to the flow of $X$ on $\cK$ and the function $a$ (so $\widehat{\Psi}_e$
and $\widehat{\cR}_e$ restrict to copies of $\Psi_e$ and $\cR_e$). The poles
$Z_{\{0\}}$ and the equator $Z_{\{1\}}$ are the other invariant sets $Z_I$, and
$b=-1$ on them ($\widehat{\cR}_e=-1$). Along the meridians, the term $N$ of $\mathcal B$
moves every other point towards the sheets, so every invariant measure of the
flow of $\mathcal B$ on $S_L$ is supported on the sheets, the poles and the equator,
which gives \eqref{eq:wrapperavg}. Thick black lines show the sheets, grey
lines and dots the sets on which $b=-1$, and dashed or dotted lines lie on the
far side of the sphere.}
\label{fig:wrapper}
\end{figure}
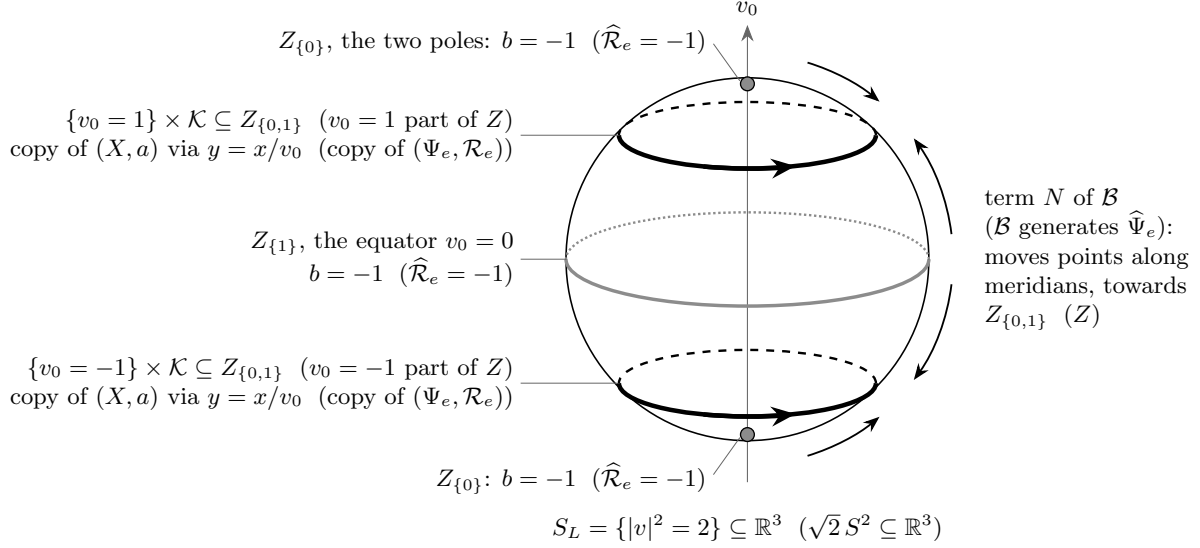

\emph{Step 2 (invariant measures on $S_L$).} Put $\mathcal L(s)=\sum_{i=0}^k(s_i-\frac12s_i^2)$. On $S_L$, using $\sum_is_i(1-s_i-\bar f)=0$,
\[
\frac d{dt}\mathcal L=\sum_i2s_i(1-s_i)(1-s_i-\bar f)=2\sum_is_i(1-s_i-\bar f)^2\ge0,
\]
with equality exactly when there is a nonempty $I\subseteq\{0,\dots,k\}$ with $s_i=L/|I|$ for $i\in I$ and $s_i=0$ for $i\notin I$. Each such set $Z_I\subseteq S_L$ is compact and invariant, since $\dot s=0$ on it. For $\nu\in\cM(\Phi_{\mathcal B}|_{S_L})$, \ref{E:deriv} gives $\int\dot{\mathcal L}\,d\nu=0$, so $\nu$ is supported on the finite union of the $Z_I$, and $\nu$ is a convex combination of invariant probability measures on the individual $Z_I$. If $I\ne\{0,\dots,k\}$, then $\mathcal H=0$ and $b=-1$ on $Z_I$. On $Z_{\{0,\dots,k\}}=\{v_0=\pm1\}\times\cK$ we have $\bar f=0$, so $\dot v_0=0$, and for $y=x/v_0\in\cK$ tangency gives $x_i\cdot\Pi_i=v_0^{D+1}y_i\cdot X_i(y)=0$, so $\dot x_i=\Pi_i=v_0^DX_i(y)$ and $\dot y=v_0^{D-1}X(y)=X(y)$, because $D-1$ is even. Likewise $b=\alpha=a(y)$ there. So on each of the two sheets $v_0=\pm1$ the flow of $\mathcal B$ with the function $b$ is conjugate to the flow of $X$ on $\cK$ with the function $a$. Consequently
\begin{equation}\label{eq:wrapperavg}
\beta(\Phi_{\mathcal B}|_{S_L},b)=\max\{\beta(\Phi_X|_\cK,a),-1\},
\end{equation}
the value $-1$ being attained, for instance, at the equilibria $(\pm\sqrt L,0,\dots,0)\in Z_{\{0\}}$. In particular $\beta(\Phi_{\mathcal B}|_{S_L},b)<0$ if and only if $\beta(\Phi_X|_\cK,a)<0$.

\emph{Step 3 (homogenization and the reverse polar decomposition).} The components of $T$ are homogeneous of degree $D+2$ and those of $N$ are sums of homogeneous pieces of degrees one, three and five; the pieces of $b$ have degrees zero, $2L$ and $2L+D-1$. Since $D\ge3$ is odd, all degrees in $\mathcal B$ are odd and at most $D+2\le E$, and all degrees in $b$ are even and at most $2L+D-1=E-1$. For a polynomial $G=\sum_\ell G_\ell$, with $G_\ell$ homogeneous of degree $\ell$, $\ell\le e$ and $\ell\equiv e\pmod2$, put $\mathrm{Hom}_e(G)=\sum_\ell(|v|^2/L)^{(e-\ell)/2}G_\ell$; it is homogeneous of degree $e$, has rational coefficients, and equals $G$ on $S_L$. Define
\[
Y=\mathrm{Hom}_E(\mathcal B),\qquad\mathcal A=\mathrm{Hom}_{E-1}(b),\qquad F(v)=Y(v)+\mathcal A(v)\,v.
\]
Then $F$ is homogeneous of degree $E$ with rational coefficients. The polynomial $v\cdot Y(v)$ is homogeneous and vanishes on $S_L$, hence vanishes identically. Apply Lemma~\ref{lem:gascrit} with $Q(v)=|v|^2/L$, so that $S_Q=S_L$: on $S_L$ we get $a_F=v\cdot F/L=\mathcal A=b$ and $Y_F=F-a_Fv=Y=\mathcal B$. Hence the origin is GAS for $F$ if and only if $\beta(\Phi_{\mathcal B}|_{S_L},b)<0$, which by \eqref{eq:wrapperavg} proves \eqref{eq:wrapper}. All steps are finite operations on rational polynomials.
\end{proof}

In the application in Section~\ref{sec:fixed} we have $k=2$, $\cK=S^{n_s}\times S^{n_h}$, $L=3$, $X$ the vector field \eqref{eq:compactsource} and $a=\cR_e$; in the notation of the overview in Section~\ref{subsec:overview-single}, the coordinates $(v_1,v_2,v_0)$ are $(v_s,v_h,v_0)$, the sphere $S_L$ is $\sqrt3\,S^{n_s+n_h+2}$, the set $Z_{\{0,1,2\}}=\{v_0=\pm1\}\times\cK$ is $Z$, and the flow $\Phi_{\mathcal B}|_{S_L}$ and the function $b$ are $\widehat\Psi_e$ and $\widehat\cR_e$.

\subsection{\texorpdfstring{Proofs of Theorems~\ref*{thm:fixedpair} and~\ref*{thm:main}}{Proofs of the existence of a single pair and of the main theorem}}\label{sec:fixed}

\begin{proof}[Proof of Theorem~\ref{thm:fixedpair}]
The reduction $e\mapsto F_e$ is the composition of the following finite steps. They use only rational arithmetic, the fixed integer matrices $A,B$ and vectors $c_0,c_1$ of Lemma~\ref{lem:template}, and the fixed numbers $\rho_*$, $D_h$, $M_0$, $K_0$ and $M$ of Section~\ref{sec:compact}.
\begin{enumerate}[label=(\arabic*)]
\item Compute $c(e)=c_0+ec_1$ and write down the vector field $X_e$ on $\RR^{n_s}$ and the map $h_e$ of Section~\ref{sec:the_source_system}. Only the coefficients of the oscillator equation \eqref{eq:oscillator} depend on $e$.
\item Compute the rational number $K_e=K_0\rho_*^{4D_h|c(e)|_1}$ of Lemma~\ref{lem:bounds}, and $\kappa_e$ by \eqref{eq:kappaeps}; the rationals $w^*$ and $\varepsilon$ of \eqref{eq:kappaeps} do not depend on $e$. We find the integer $m_e$ in \eqref{eq:kappaeps} by testing $m=1,2,\dots$; the search stops because $m^2w^*\to\infty$.
\item Write down the vector field \eqref{eq:compactsource} on $\cK=S^{n_s}\times S^{n_h}$ and the reward \eqref{eq:reward}.
\item Apply Theorem~\ref{thm:wrapper} with $X$ the vector field \eqref{eq:compactsource}, $a=\cR_e$, $k=2$, $n_1=n_s+1$, $n_2=n_h+1$ and $D=M+3$, and let $F_e$ be the resulting vector field.
\end{enumerate}
The numbers $n_s$, $n_h$ and $M$ do not depend on $e$ (Sections~\ref{sec:source} and~\ref{sec:compact}). The vector field \eqref{eq:compactsource} has rational coefficients by Lemma~\ref{lem:central}, since $\kappa_e$ is rational, and it is tangent to $\cK$ (Section~\ref{sec:compact}). Since $\deg X_e,\deg h_e\le M$, the terms of \eqref{eq:central} have degree at most $M+3$ and the remaining terms of \eqref{eq:compactsource} have degree at most $M+2$, so \eqref{eq:compactsource} has degree at most $M+3=D$. The polynomial $\cR_e$ of \eqref{eq:reward} has rational coefficients and degree two. Since $M$ is even, $D=M+3$ is odd and at least three, and $D-1=M+2\ge2$, so the hypotheses of Theorem~\ref{thm:wrapper} hold. Hence $F_e$ is a homogeneous polynomial vector field with rational coefficients on $\RR^{n_*}$, $n_*=1+(n_s+1)+(n_h+1)$, of the odd degree $d_*=D+2L=M+9$, where $L=k+1=3$. Neither number depends on $e$. By Proposition~\ref{prop:signgap} and Theorem~\ref{thm:wrapper},
\[
e\in K\iff\beta_e<0\iff0\text{ is GAS for }\dot x=F_e(x).
\]
Each of the steps (1)--(4) terminates for every $e\in\NN$, so the map $e\mapsto F_e$ is total and computable. Whether $e\in K$ is never used in computing $F_e$, only in the proof of the equivalences above. Hence the map is a many-one reduction from $K$ to $\HGAS(n_*,d_*)$. Since $K$ is many-one hard by \ref{C:K}, \ref{C:trans} shows that $\HGAS(n_*,d_*)$ is many-one hard, and it is r.e.\ by Proposition~\ref{prop:re}. This proves that $\HGAS(n_*,d_*)$ is many-one complete.
\end{proof}

The constants $n_*$ and $d_*$ are determined by the fixed matrices $A,B$ and the fixed construction, but we do not evaluate them numerically. As explained after Lemma~\ref{lem:template}, a program for $e\mapsto F_e$ contains these fixed matrices and vectors, and this is all that a many-one reduction requires.

\phantomsection\label{subsec:proofmain}%
We now combine Part~II with Part~I: Theorem~\ref{thm:fixedpair} gives one pair $(n_*,d_*)$, and Theorem~\ref{thm:transfer} transfers it to all large dimensions and all odd degrees.

\begin{proof}[Proof of Theorem~\ref{thm:main}]
Let $n_*$ and $d_*=2k_*+1$ be given by Theorem~\ref{thm:fixedpair}, so that $\HGAS(n_*,d_*)$ is many-one complete. By Theorem~\ref{thm:transfer}, the set $\HGAS(n,d)$ is many-one complete for every $n\ge N:=\binom{n_*+k_*-1}{k_*}$ and every odd $d\ge3$, and hence undecidable by Section~\ref{subsec:computability}.
\end{proof}

\paperpart{Part III. Consequences and discussion}
\section{Consequences and discussion}\label{sec:main}

Having finished the proof of Theorem~\ref{thm:main} at the end of Section~\ref{subsec:proofmain}, in this section we prove the consequences mentioned in Sections~\ref{sec:intro} and~\ref{sec:overview}. This includes Corollaries~\ref{cor:variable} and~\ref{cor:nonhomog} (Section~\ref{subsec:consequences}) and the results on Lyapunov stability, including Corollary~\ref{cor:ls-cubic} (Section~\ref{sec:transfer}). We end with open problems (Section~\ref{subsec:open}) and the description of a Lean formalization of our results (Section~\ref{subsec:formal}).

\subsection{Further consequences for GAS}\label{subsec:consequences}

Here we discuss two important consequences of Theorem~\ref{thm:main}. The first one concerns the decision problems in which the dimension, and possibly the degree, are part of the input.

\begin{corollary}[Variable dimension and degree]\label{cor:variable}
The sets $\HGAS$, $\HGAS(\cdot,3)$ and $\HGAS(\cdot,d)$, $d\ge3$ odd, are many-one complete. Their complements are not recursively enumerable.
\end{corollary}

\begin{proof}
Let $N$ be given by Theorem~\ref{thm:main} and let $d\ge3$ be odd, so that $\HGAS(N,d)$ is many-one complete. The map which sends every code of dimension $N$ and degree $d$ to itself, and every other code to a fixed invalid code, is total computable (Section~\ref{subsec:stability}), and it reduces $\HGAS(N,d)$ both to $\HGAS(\cdot,d)$ and to $\HGAS$. Hence these sets are many-one hard by \ref{C:trans}; the case $d=3$ gives $\HGAS(\cdot,3)$. They are r.e.\ by Proposition~\ref{prop:re}, so they are many-one complete, and their complements are not r.e.\ by Section~\ref{subsec:computability}.
\end{proof}

The second consequence concerns polynomial vector fields which are not homogeneous. A polynomial vector field $f=(f_1,\dots,f_n)$ has \emph{degree exactly $d$} if each of its coordinate functions $f_{i}$ has degree at most $d$ and at least one of them has degree $d$, after duplicate monomials have been combined. For integers $n,d\ge1$, let $\mathrm{PGAS}(n,d)$ (respectively, $\mathrm{PLAS}(n,d)$) be the set of codes corresponding to rational-coefficient polynomial vector fields $f:\RR^n\to\RR^n$ of degree exactly $d$ with $f(0)=0$, such that the origin of \eqref{eq:intro-ode} is GAS (respectively, locally asymptotically stable; see Section~\ref{subsec:stability}). Just like for $\HGAS$, we consider codes which do not describe such a vector field to be ``negative instances''.

\begin{corollary}[Polynomial vector fields which are not homogeneous]\label{cor:nonhomog}
There is an integer $N\ge1$ such that for every $n\ge N$ and \textit{every} $d\ge3$ (including even degrees), the sets $\mathrm{PGAS}(n,d)$ and $\mathrm{PLAS}(n,d)$ are many-one hard (and hence undecidable).
\end{corollary}

A homogeneous polynomial vector field of degree three has degree exactly $d$ only for $d=3$. For $d>3$ we therefore add two variables whose differential equations contain a term of degree $d$, and which do not affect the stability of the origin.

\begin{proof}
Let $N_0$ be given by Theorem~\ref{thm:main}, put $N=N_0+2$, and let $n\ge N$ and $d\ge3$.

\emph{Step 1.} First, we consider the polynomial vector field $g_{d} : \RR^{2} \to \RR^{2}$ defined by $g_d(w,z)=(-w^3,\,-z^3+w^d)$, which has degree exactly $d$. We claim that for this vector field, the origin is GAS. Indeed, $|w|$ does not increase along solutions, since $\dot w=-w^3$, and $w(t)\to0$ as $t\to\infty$. Put $V=w^2+z^2$. Where $|w|\le1$ we have $|w|^d\le|w|^3$, and Young's inequality $ab\le\frac14a^4+\frac34b^{4/3}$ with $a=|z|$ and $b=|w|^3$ gives $2|z||w|^d\le\frac12z^4+\frac32w^4$, so that
\[
\dot V=-2w^4-2z^4+2zw^d\le-\tfrac12\bigl(w^4+z^4\bigr)\le-\tfrac14V^2\qquad\text{where }|w|\le1.
\]
If $0<\varepsilon\le1$ and $V(0)<\varepsilon^2$, then $|w(t)|\le|w(0)|<1$ as long as the solution exists, so $V$ does not increase; hence the solution is bounded, it exists for all $t\ge0$ (Section~\ref{subsec:ode}), and $|(w,z)(t)|<\varepsilon$. Thus the origin is Lyapunov stable. For an arbitrary initial state, $|w(t)|\le|w(0)|$, and $|z|$ decreases whenever $|z|^3>|w(0)|^d$, so $|z(t)|\le\max\{|z(0)|,|w(0)|^{d/3}\}$ and the solution exists for all $t\ge0$. Since $w(t)\to0$, there is $t_1\ge0$ with $|w(t)|\le1$ for $t\ge t_1$, and then $\dot V\le-\frac14V^2$ gives $V(t)\to0$ as $t\to\infty$.

\emph{Step 2.} Next, we use $g_d$ to turn polynomial vector fields whose coordinates have degree at most three into polynomial vector fields of degree exactly $d$. Write the points of $\RR^n$ as $(x,w,z)$ with $x\in\RR^{n-2}$ and $w,z\in\RR$. For a polynomial vector field $F:\RR^{n-2}\to\RR^{n-2}$ with $F(0)=0$ whose coordinates have degree at most three, we define the polynomial vector field $f_F:\RR^n\to\RR^n$ by $f_F(x,w,z)=(F(x),g_d(w,z))$, which has degree exactly $d$ since $d\ge3$, and which satisfies $f_F(0)=0$. We claim that the origin is GAS, respectively locally asymptotically stable, respectively Lyapunov stable, for $f_F$ if and only if it is for $F$. Indeed, the system $\dot x=F(x)$, $(\dot w,\dot z)=g_d(w,z)$ is the product of two independent systems, and the origin of the second one is GAS by Step~1. If one of the three properties holds for $F$, then it holds for $f_F$, with the norm $|x|+|(w,z)|$, because the solutions of $f_F$ are pairs of solutions of the two factors. Conversely, the solutions of $f_F$ which start in the invariant subspace $\{w=z=0\}$ are the solutions of $F$, so each of the three properties for $f_F$ implies the same property for $F$.

Finally, consider the map which sends every valid code of dimension $n-2$ and degree three, describing a vector field $F$, to the code of $f_F$, and every other code to the code of $f_{\mathrm{id}}$, where $\mathrm{id}(x)=x$. Note that the origin of $f_{\mathrm{id}}$ is not Lyapunov stable, since its solutions with $w=z=0$ are $x(t)=e^tx(0)$. This map is total computable (Section~\ref{subsec:stability}). For a homogeneous polynomial vector field of degree three, local asymptotic stability and GAS are equivalent (Section~\ref{subsec:homog}), so by the above work this map reduces $\HGAS(n-2,3)$ both to $\mathrm{PGAS}(n,d)$ and to $\mathrm{PLAS}(n,d)$. Since $n-2\ge N_0$, the set $\HGAS(n-2,3)$ is many-one complete by Theorem~\ref{thm:main}, so both sets are many-one hard by \ref{C:trans}, and hence undecidable by Section~\ref{subsec:computability}.
\end{proof}

The same total computable map described in the above proof, with Corollary~\ref{cor:ls-cubic} in place of Theorem~\ref{thm:main}, shows that Lyapunov stability of polynomial vector fields of degree exactly $d$ is many-one hard as well, for every $d\ge3$ and every $n\ge N'+2$, where $N'$ is the integer of Corollary~\ref{cor:ls-cubic}: by Step~2 of the proof, the origin is Lyapunov stable for $f_F$ if and only if it is for $F$. We note that only many-one hardness can be concluded from Corollary~\ref{cor:nonhomog}; whether $\mathrm{PGAS}(n,d)$ and $\mathrm{PLAS}(n,d)$ are r.e.\ is an open problem (Section~\ref{subsec:open}).

\subsection{Lyapunov stability}\label{sec:transfer}

We now combine Theorem~\ref{thm:transfer} with the undecidability results of \cite{Korda,KordaSupp}. We use the statements of these results and, in two places, their proofs. In particular, in the proof of Corollary~\ref{cor:ls-cubic} we apply the construction of \cite[Sections~3--6]{Korda} to a particular register machine, and in Remark~\ref{rem:dim5} we use the reduction constructed in the proof of \cite[Theorem~3.1]{KordaSupp}.

\begin{proof}[Proof of Corollary~\ref{cor:ls-cubic}]
The construction of \cite[Sections~3--6]{Korda} applies to any register machine $\mathsf R$ in the sense of \cite[Section~3]{Korda}, that is, a program whose instructions increment a register, or test a register for zero and otherwise decrement it (this model is used because a step changes one of finitely many natural numbers by $\pm1$, which is easy to imitate by a vector field). The construction produces, computably in $m\in\NN$, a homogeneous polynomial vector field $F_m:\RR^{\tilde N}\to\RR^{\tilde N}$ with rational coefficients, of an odd degree $\tilde D=2\tilde k+1\ge3$ which does not depend on $m$, such that $\mathsf R$ halts on input $m$ if and only if the origin is Lyapunov stable for $F_m$ \cite[(1.1) and Section~6]{Korda}. Let $\mathsf R$ be the register machine obtained by applying the construction of \cite[Remark~3.2]{Korda} to the universal Turing machine $\mathsf U$ of \ref{C:K}. For $e\in\NN$ let $m(e)$ be the input word of $\mathsf U$ that represents $e$, read as a numeral in the base used there; then $e\mapsto m(e)$ is total computable, and $\mathsf R$ halts on $m(e)$ if and only if $\mathsf U$ halts on $e$. Hence the total computable map which sends $e$ to the code of $F_{m(e)}$ satisfies $e\in K$ if and only if this code lies in $\HLS(\tilde N,\tilde D)$, that is, $K\lem\HLS(\tilde N,\tilde D)$. Since $K$ is many-one complete by \ref{C:K}, the set $\HLS(\tilde N,\tilde D)$ is many-one hard by \ref{C:trans}, and Theorem~\ref{thm:transfer}, with $n_*=\tilde N$ and $k_*=\tilde k$, shows that $\HLS(n,d)$ is many-one hard for every $n\ge N=\binom{\tilde N+\tilde k-1}{\tilde k}$ and every odd $d\ge3$. Undecidability follows from Section~\ref{subsec:computability}.
\end{proof}

\begin{remark}\label{rem:contrast}
Interestingly, Proposition~\ref{prop:re} has no known analogue for Lyapunov stability. The reduction used in \cite{Korda} is carried out via a computable map $m\mapsto F_m$, which corresponds to ``$\mathsf R$ halts on $m$ $\iff$ the origin is Lyapunov stable for $F_m$'' \cite[(1.1)]{Korda}. When the halting set of $\mathsf R$ is many-one complete, as is the case for the register machine $\mathsf R$ in the proof of Corollary~\ref{cor:ls-cubic}, the Lyapunov stable codes form a many-one hard set in the dimension $\tilde N$ and the degree $\tilde D$ of that proof, so the codes corresponding to systems whose origin is \emph{not} Lyapunov stable cannot be enumerated (Section~\ref{subsec:computability} and Corollary~\ref{cor:ls-cubic}). For homogeneous polynomial vector fields, positive instances of GAS can therefore be certified, whereas negative instances of Lyapunov stability cannot. The converse theorem behind Proposition~\ref{prop:re} uses asymptotic stability and homogeneity in an essential way \cite{AEK}.
\end{remark}

\begin{remark}\label{rem:dim5}
Let $D_5=2k_5+1$ be the odd degree of \cite[Theorem~3.1]{KordaSupp}. By \cite[Theorem~5.1, Lemma~4.3(c)]{KordaSupp} and the proof of \cite[Theorem~3.1]{KordaSupp}, there is a computable map $m\mapsto F_m$ into rational-coefficient homogeneous polynomial vector fields on $\RR^5$ of degree $D_5$ such that the origin is Lyapunov stable for $F_m$ if and only if a fixed reversible two-counter machine $\mathsf V$ (a register machine with two registers whose steps can be undone uniquely \cite[Section~6]{KordaSupp}) does not halt from the configuration $(1,(2^m,0))$, and the set of $m$ for which $\mathsf V$ halts from this configuration is undecidable \cite[Theorem~6.2(a)]{KordaSupp}. Thus $m\mapsto F_m$ many-one reduces the complement of this set, which is undecidable as well, to $\HLS(5,D_5)$, so that $\HLS(5,D_5)$ is undecidable by \ref{C:trans}. Our Theorem~\ref{thm:transfer}, with $n_*=5$ and $k_*=k_5$, shows that $\HLS(n,d)$ is undecidable for every $n\ge\binom{5+k_5-1}{k_5}=\binom{k_5+4}{4}$ and every odd $d\ge3$.
\end{remark}

\begin{remark}\label{rem:bothsides}
The two reductions in \cite{Korda} and \cite{KordaSupp} ``run in opposite directions''. Specifically, in \cite{Korda} we have that halting corresponds to stability, while in \cite[Part~I]{KordaSupp} we have that halting corresponds to instability. If the halting set in \cite[Theorem~6.2(a)]{KordaSupp} is many-one complete, as one would generally expect from the universality of the machines on which it is based (but which is not stated there), then the argument of Remark~\ref{rem:dim5} together with the reduction $\HLS(5,D_5) \lem \HLS(n,d)$ constructed in the proof of Theorem \ref{thm:transfer} (Sections~\ref{subsec:proofcubic} and~\ref{sec:normalform}) shows that every complement of an r.e. set many-one reduces to $\HLS(n,d)$ for $n\ge\binom{k_5+4}{4}$ and odd $d\ge3$. For $n\ge\max\{N,\binom{k_5+4}{4}\}$, with $N$ as in Corollary~\ref{cor:ls-cubic}, the set $\HLS(n,d)$ would then be neither r.e.\ nor the complement of an r.e.\ set. That is, in contrast with GAS (Proposition~\ref{prop:re}), neither the stable nor the unstable homogeneous polynomial vector fields of degree three could be enumerated.
\end{remark}

Our ``single $(n_*,d_*)$ implies many $(n,d)$'' result, Theorem~\ref{thm:transfer}, concerns one fixed dimension and one fixed degree at a time. When the dimension is part of the input, the reduction to degree three gives the following more precise statement for Lyapunov stability: whatever the exact difficulty of Lyapunov stability for homogeneous polynomial vector fields of arbitrary dimension and degree, it is already attained in degree three. For GAS the corresponding statement follows directly from Corollary~\ref{cor:variable}, since any two many-one complete sets many-one reduce to each other. For Lyapunov stability there is no such shortcut, because $\HLS$ is not known to be recursively enumerable (Remark~\ref{rem:contrast}).

\begin{proposition}[Lyapunov stability with the dimension as part of the input]\label{prop:hls-variable}
The problems $\HLS$, $\HLS(\cdot,3)$ and, for each fixed odd $d\ge3$, $\HLS(\cdot,d)$ are pairwise many-one equivalent. Moreover, $\HLS(\cdot,d)\lem\HLS(\cdot,3)$ for every $d\ge1$.
\end{proposition}

The proof of Proposition \ref{prop:hls-variable}, which is given at the end of Section \ref{sec:transfer}, has to handle every valid input. For odd degrees $2k+1\ge3$ we will use Theorem~\ref{thm:cubicization}, now with a single algorithm for all pairs $(n,k)$ (see Proposition~\ref{prop:uniform} below). Linear inputs are decided directly (Lemma~\ref{lem:linear}), and inputs of even degree, which can be Lyapunov stable, are turned into inputs of odd degree with the same answer (Lemma~\ref{lem:even}).

\begin{lemma}\label{lem:even}
Let $f:\RR^n\to\RR^n$ be homogeneous of even degree $2k\ge2$, and define $\tilde f(x,w)=(w\,f(x),0)$ on $\RR^n\times\RR$. Then $\tilde f$ is homogeneous of degree $2k+1$, and the origin is Lyapunov stable for $\tilde f$ if and only if it is Lyapunov stable for $f$.
\end{lemma}

\begin{proof}
Along solutions of $\dot z=\tilde f(z)$, $z=(x,w)$, the coordinate $w(t)=w_0$ is constant and $\dot x=w_0f(x)$. If $w_0\ge0$, then $x(t)=\Phi_f^{w_0t}(x_0)$. If $w_0<0$, then $\zeta=-x$ satisfies $\dot\zeta=-w_0f(x)=|w_0|f(-\zeta)=|w_0|f(\zeta)$, because $f$ is even, so $x(t)=-\Phi_f^{|w_0|t}(-x_0)$.

If $f$ is Lyapunov stable, Lemma~\ref{lem:scaling} gives $C\ge1$ such that $\Phi_f^{t}(x)$ exists and $|\Phi_f^{t}(x)|\le C|x|$ for all $t\ge0$. By the formulas above, every solution of $\dot z=\tilde f(z)$ is global and satisfies $|x(t)|+|w(t)|\le C(|x_0|+|w_0|)$, so $\tilde f$ is Lyapunov stable by Lemma~\ref{lem:scaling}. Conversely, if $\tilde f$ is Lyapunov stable with constant $C$ for the norm $|x|+|w|$, then for $x_0\in\RR^n$ and $\varepsilon>0$ the solution from $(x_0,\varepsilon)$ is global and $|\Phi_f^{\varepsilon t}(x_0)|\le C(|x_0|+\varepsilon)$ for all $t\ge0$. Hence $\Phi_f^{s}(x_0)$ exists for all $s\ge0$, and letting $\varepsilon\to0$ gives $|\Phi_f^{s}(x_0)|\le C|x_0|$. Lemma~\ref{lem:scaling} shows that $f$ is Lyapunov stable.
\end{proof}

\begin{lemma}\label{lem:linear}
For $A\in\RR^{n\times n}$, the origin of $\dot x=Ax$ is Lyapunov stable if and only if there is a symmetric positive definite matrix $P$ with $A^TP+PA\preceq0$. Consequently Lyapunov stability of rational-coefficient linear vector fields is decidable, uniformly in $n$. GAS of rational linear vector fields is decidable as well, uniformly in $n$.
\end{lemma}

\begin{proof}
If such $P$ exists, $x^TPx$ is nonincreasing along solutions, so $|x(t)|^2\le c\,|x(0)|^2$ with $c=\lambda_{\max}(P)/\lambda_{\min}(P)$. Conversely, if the origin is Lyapunov stable, then every eigenvalue of $A$ has nonpositive real part, and every eigenvalue $\lambda$ on the imaginary axis, of algebraic multiplicity $q$, satisfies $\operatorname{rank}(A-\lambda I)=n-q$ \cite[Theorem~4.5]{Khalil}; that is, the Jordan blocks of the eigenvalues on the imaginary axis have size one. By the real Jordan form \cite[Theorem~3.4.1.5]{HornJohnson} there is an invertible $T$ with $T^{-1}AT=\mathrm{diag}(A_-,A_0)$, where every eigenvalue of $A_-$ has negative real part and $A_0=SJS^{-1}$ with $J$ block diagonal with blocks $0$ and $\bigl(\begin{smallmatrix}0&-\omega\\\omega&0\end{smallmatrix}\bigr)$, so that $J^T=-J$. Let $P_-$ be the positive definite solution of $A_-^TP_-+P_-A_-=-I$ \cite[Theorem~4.6]{Khalil}, and put $P=T^{-T}\mathrm{diag}(P_-,S^{-T}S^{-1})T^{-1}$. Then $A^TP+PA=T^{-T}\mathrm{diag}\bigl(-I,\,S^{-T}(J^T+J)S^{-1}\bigr)T^{-1}=T^{-T}\mathrm{diag}(-I,0)T^{-1}\preceq0$. For rational $A$, the sentence
\[
\exists P\ \ \Bigl(P=P^T\wedge\forall x\,\bigl(x\ne0\Rightarrow x^TPx>0\bigr)\wedge\forall x\ x^T(A^TP+PA)x\le0\Bigr)
\]
has rational data and is decidable by Theorem~\ref{thm:tarski}. The origin of $\dot x=Ax$ is GAS if and only if every eigenvalue of $A$ has negative real part \cite[Theorem~4.5]{Khalil}; writing $\det\bigl((\xi+i\zeta)I-A\bigr)=U(\xi,\zeta)+iV(\xi,\zeta)$ with $U,V\in\QQ[\xi,\zeta]$ computable from $A$, this is the sentence $\forall\xi\,\forall\zeta\,\bigl(U=0\wedge V=0\Rightarrow\xi<0\bigr)$, which is decidable by Theorem~\ref{thm:tarski}.
\end{proof}

Theorem~\ref{thm:cubicization} is stated for fixed $(n,k)$. In Proposition~\ref{prop:hls-variable} the dimension and the degree are part of the input, so we require one algorithm for all $(n,k)$.

\begin{proposition}[Uniformity of the reduction to degree three]\label{prop:uniform}
There is a single algorithm which, given $n\ge1$, $k\ge1$ and the rational coefficients of a homogeneous polynomial vector field $F:\RR^n\to\RR^n$ of degree $2k+1$, outputs the rational coefficients of a homogeneous polynomial vector field $H:\RR^m\to\RR^m$ of degree three, $m=\binom{n+k-1}{k}$, satisfying \eqref{eq:mainGAS} and \eqref{eq:mainLS}.
\end{proposition}

\begin{proof}
Every step in the proof of Theorem~\ref{thm:cubicization} is computable from $(n,k,F)$: the index set $A_k$, the form $Q$, the list of binomials and $W$ depend only on $(n,k)$; $G$ is obtained by expanding $D\nu_k(x)F(x)$; the rationals $\eta,c,C_G,B_G,L_G$ are found by searches over rationals in which every step is a decision of a sentence by Theorem~\ref{thm:tarski}, and these searches terminate by Lemmas~\ref{lem:gradineq}, \ref{lem:tangency} and~\ref{lem:lipschitz}; and $K$ is the least positive integer satisfying \eqref{eq:Kchoice}. None of these steps depends on $(n,k)$ other than through finite data computed from it.
\end{proof}

\begin{proof}[Proof of Proposition~\ref{prop:hls-variable}]
Let $e_\pm$ denote the codes of $x\mapsto\mp x^3$ on $\RR^1$; the origin is Lyapunov stable for $e_+$, and not for $e_-$, whose nonzero solutions blow up in finite time. We first show $\HLS\lem\HLS(\cdot,3)$. Let $e$ be a code. If $e$ is invalid, output $e_-$. Otherwise $e$ codes a triple $(n,d',f)$. If $d'=1$, decide Lyapunov stability by Lemma~\ref{lem:linear} and output $e_+$ or $e_-$. If $d'=2k$ is even, replace $(n,f)$ by $(n+1,\tilde f)$ from Lemma~\ref{lem:even}, which has degree $2k+1$ and the same answer. If the degree is now $2k+1\ge3$ and the dimension is $n'$, output the code of the vector field $H$ on $\RR^m$, $m=\binom{n'+k-1}{k}$, given by Proposition~\ref{prop:uniform}. This map is total and computable by Lemma~\ref{lem:linear} and Proposition~\ref{prop:uniform}, and it is a many-one reduction by Lemmas~\ref{lem:linear} and~\ref{lem:even} and Theorem~\ref{thm:cubicization}.

Next, let $d=3+2j$ be odd. The map which sends a valid code of a vector field $H$ of degree three on $\RR^m$ to the code of the vector field $\widehat H$ of Lemma~\ref{lem:padding} with $n=m$, and every other code to $e_-$, reduces $\HLS(\cdot,3)$ to $\HLS(\cdot,d)$ by Lemma~\ref{lem:padding}. Finally, for every $d\ge1$, the map which sends every code of degree $d$ to itself and every other code to $e_-$ reduces $\HLS(\cdot,d)$ to $\HLS$, since a code of degree $d$ lies in $\HLS(\cdot,d)$ exactly when it lies in $\HLS$, and $e_-$ lies in neither set. By transitivity \ref{C:trans}, $\HLS\lem\HLS(\cdot,3)\lem\HLS(\cdot,d)\lem\HLS$ for every odd $d\ge3$, which proves the first assertion, and $\HLS(\cdot,d)\lem\HLS\lem\HLS(\cdot,3)$ for every $d\ge1$.
\end{proof}

The relation of these results to the literature is discussed in Section~\ref{sec:intro}.

\subsection{Open problems}\label{subsec:open}
Degree three is the smallest possible degree for undecidability of GAS (Remark~\ref{rem:even} and Lemma~\ref{lem:linear}), but the bound $N$ on the dimension in Theorem~\ref{thm:main} is very large and is not computed. The least dimension in which GAS, or Lyapunov stability, of rational-coefficient homogeneous polynomial vector fields of degree three is undecidable is open. In particular, we do not know whether $\HGAS(3,3)$ or $\HLS(3,3)$ is decidable ($n=3$ is the first open case, since both of these sets are decidable for dimension $n \in \{1,2\}$; see \cite[Part III]{KordaSupp}). Also, for Lyapunov stability, Remark~\ref{rem:even} does not rule out the possibility of undecidability for even degrees; we do not know whether Lyapunov stability is decidable even for homogeneous polynomial vector fields of degree two. We note that the method of Part I cannot say anything about degree two: if $H$ has even degree, then $a_H(-u)=-a_H(u)$ and $Y_H(-u)=Y_H(u)$, so the antipodal map conjugates the angular flow of $H$ to its time reversal while changing the sign of the radial rate, and an antipodally symmetric invariant set, such as $\Sigma$, which attracts nearby angular trajectories would also repel them. We also do not know whether the sets $\mathrm{PGAS}(n,d)$ and $\mathrm{PLAS}(n,d)$ of Corollary~\ref{cor:nonhomog} are recursively enumerable; the proof of Proposition~\ref{prop:re} uses the converse theorem of \cite{AEK}, which is specific to homogeneous vector fields. It is interesting to note that if the halting set of \cite[Theorem~6.2(a)]{KordaSupp} (which is used in dimension five) is many-one complete, then $\HLS(n,d)$ is neither recursively enumerable nor the complement of a recursively enumerable set for large $n$ (Remark~\ref{rem:bothsides}); this would make Lyapunov stability strictly harder than GAS in the arithmetical hierarchy.

\subsection{Formal verification}\label{subsec:formal}

We have formalized the main results of this paper in the proof assistant Lean~4, using its mathematical library Mathlib \cite{Lean4,Mathlib}. The formalization is available at \url{https://github.com/j49liu/gas_decidability_lean/} and consists of two Lean libraries, each with its own root file. The library \texttt{HGAS} is the original formalization of Theorem~\ref{thm:main}, extended by the results of Part~I on Lyapunov stability. The library \texttt{Stability} uses \texttt{HGAS} without changing it, and adds a reduction from the halting problem to Lyapunov stability (described below), most of the consequences in Part~III, and a second proof of Theorem~\ref{thm:fixedpair}. As in Section~\ref{subsec:stability}, the decision problems are formalized as sets of codes of polynomial vector fields with rational coefficients. For the following results, the formalization proves the statement in this paper, or the equivalent form described, together with formal versions of the results on which their proofs depend.
\begin{itemize}
\item Theorems~\ref{thm:main} and~\ref{thm:fixedpair}, including formal versions of the classical results which Part~II cites: the case of the MRDP theorem needed for Lemma~\ref{lem:template}, the theta identities, the covering property used in Proposition~\ref{prop:modcomplete}, and the Dani--Smillie dichotomy for $\Gamma(2)$. The second proof of Theorem~\ref{thm:fixedpair}, in the library \texttt{Stability}, follows Sections~\ref{sec:orbits} and~\ref{sec:compact} more closely: it chooses $\kappa_e$ by the least-integer rule of \eqref{eq:kappaeps} and proves two bounds of the form of Proposition~\ref{prop:signgap}, with explicit constants of its own.
\item Lemma~\ref{lem:gascrit} for the Euclidean form, and versions of Lemmas~\ref{lem:scaling} and~\ref{lem:lyap} for the representation of solutions described below. The case of a general form $Q(y)=y^{\top}My$ follows by the linear change of coordinates $y\mapsto M^{1/2}y$, which maps $S_Q$ onto the unit sphere: the formalization proves that GAS and Lyapunov stability are invariant under linear changes of coordinates, and uses this for the form $Q$ of Section~\ref{sec:lifting}.
\item Theorem~\ref{thm:cubicization} for every fixed $n\ge2$ and $k\ge1$, with one computable map for both \eqref{eq:mainGAS} and \eqref{eq:mainLS}. Proposition~\ref{prop:shadow} is proved in an equivalent form, in which the bound on the trajectory and the bound on the radial integral are stated separately, for a norm equivalent to $\|\cdot\|_Q$. For $n=1$ the formalization proves the equivalences \eqref{eq:mainGAS} and \eqref{eq:mainLS}, and gives a computable map for GAS alone.
\item Lemma~\ref{lem:padding}, and the many-one reductions $\HGAS(n_*,d_*)\lem\HGAS(n,d)$ and $\HLS(n_*,d_*)\lem\HLS(n,d)$, for $n\ge N$, with $N$ as in Theorem~\ref{thm:transfer}, and odd $d\ge3$, from which Theorem~\ref{thm:transfer} follows (Section~\ref{subsec:computability}). The reduction for $\HLS$ is formalized for $n_*\ge2$. This is no restriction, since $\HLS(1,d)$ is decidable for every $d$ (Remark~\ref{rem:even}), and hence neither many-one hard nor undecidable.
\item Corollary~\ref{cor:ls-cubic}, without any hypothesis on the results of \cite{Korda}. The library \texttt{Stability} formalizes the reduction from the halting problem to Lyapunov stability constructed in \cite[Sections~3--6]{Korda}, following the written proof there, and applies it to a fixed register machine whose halting set is many-one complete. This machine has five registers and is obtained by compiling a universal machine from Mathlib, instead of being the machine $\mathsf R$ in the proof of Corollary~\ref{cor:ls-cubic}. The formalization does not reuse code from the Lean formalization of the main theorem of \cite{Korda} which is mentioned in \cite[Section~1.1]{Korda}.
\item Part of Corollary~\ref{cor:variable}: the set of codes of homogeneous polynomial vector fields whose origin is GAS, in which the dimension, an odd degree and the coefficients are all part of the input, is many-one complete. This set differs from $\HGAS$ only in conventions: codes of even degree are negative instances by definition, rather than by Remark~\ref{rem:even}, which is not formalized, and dimension zero is allowed. The statements for $\HGAS(\cdot,3)$ and $\HGAS(\cdot,d)$ are not formalized.
\item Corollary~\ref{cor:nonhomog}.
\end{itemize}
Remark~\ref{rem:dim5}, Propositions~\ref{prop:hls-variable} and~\ref{prop:uniform}, and Lemmas~\ref{lem:even} and~\ref{lem:linear} are not formalized. In particular, the formal versions of the map of Theorem~\ref{thm:cubicization} are computable for each fixed pair $(n,k)$, but not uniformly in $(n,k)$.

The proofs contain no unfinished steps (i.e., no occurrences of Lean's placeholder \texttt{sorry}), and the libraries declare no axioms of their own. Two scripts in the repository apply the command \texttt{\#print axioms}, which lists the axioms on which a proof depends, to 29 declarations of \texttt{HGAS}, among them the formal versions of Theorems~\ref{thm:main} and~\ref{thm:fixedpair}, and to 454 declarations of \texttt{Stability}, among them the formal versions of Corollaries~\ref{cor:ls-cubic} and~\ref{cor:nonhomog} and of the part of Corollary~\ref{cor:variable} described above. Each of these declarations depends only on the standard axioms \texttt{propext}, \texttt{Classical.choice} and \texttt{Quot.sound} of Lean.

The formalization follows the proofs in this paper, with the equivalent forms noted above and the following further differences. First, solutions of differential equations are represented by trajectories defined for all $t\ge0$, rather than by maximal solutions; accordingly, Lyapunov stability and local asymptotic stability include the requirement that such a trajectory exists from every initial state near the origin. Second, Proposition~\ref{prop:re} is formalized for odd degrees, uniformly in the dimension and the degree, through an exact finite grid certificate and a converse Lyapunov theorem for homogeneous vector fields proved in Lean, instead of \cite{AEK} and quantifier elimination. Third, the reduction to degree three uses, for each pair $(n,k)$, a coefficient $K$ of the penalty term in \eqref{eq:H} of the form $c_{n,k}(1+\|F\|_1)$, where $\|F\|_1$ is the sum of the absolute values of the coefficients of $F$ and $c_{n,k}$ is a natural number fixed once for each pair, instead of the searches by quantifier elimination. Fourth, the Diophantine system of Lemma~\ref{lem:template} is replaced by a system of the same kind whose right-hand side depends on $e$ in a different way, and whose numbers of equations and unknowns again do not depend on $e$. As in this paper, the integers $N$ of Theorem~\ref{thm:main} and Corollary~\ref{cor:ls-cubic} are given by expressions which are not evaluated. The file \texttt{FORMALIZATION\_LEDGER.md} in the repository records the detailed correspondence between the statements of this paper and their formal versions.

\section*{Acknowledgments and AI methodology statement}

The original results of this paper, including Theorem \ref{thm:main}, were generated and formalized in Lean using GPT-5.6 Sol Ultra (OpenAI) on July 26, 2026. While we were substantially restructuring the paper and revising the exposition, during which we also made use of various Claude (Anthropic) and OpenAI models, the preprint \cite{Korda} appeared. This new work addressed much of what we were doing, but for Lyapunov stability instead of GAS (see Section \ref{sec:intro} for a detailed discussion of our work in relation to \cite{Korda}). The supplement to this preprint also built on the work in \cite{Samardzija} to prove decidability of both Lyapunov stability and GAS for planar (i.e., dimension two) homogeneous polynomial vector fields \cite[Part~III]{KordaSupp}\footnote{The original draft of this paper also contained a proof of decidability of GAS in the planar homogeneous case, uniform in the degree of homogeneity. Following the treatment of the planar case in \cite[Part~III]{KordaSupp}, we omitted that proof from the present version to avoid duplication.}. We realized that our results (in particular, Theorem \ref{thm:transfer}) in Part~I could be extended from GAS to Lyapunov stability, and could subsequently be combined with the new results to prove undecidability of Lyapunov stability for homogeneous polynomial vector fields of degree three in a fixed dimension (Corollary~\ref{cor:ls-cubic}), among other results related to Lyapunov stability (Lemma~\ref{lem:tube}, Proposition~\ref{prop:shadow}, Proposition~\ref{prop:hls-variable}).

During the course of substantially reworking the manuscript into its current form, the authors have checked the statements, arguments, and references in the paper for correctness and accuracy. The formal statements corresponding to the results listed in Section~\ref{subsec:formal} are verified by the machine-checked Lean proofs and the reported axiom audit. Furthermore, the authors have spent considerable time checking manually that the formal statements of the headline results in the paper accurately correspond to the statements of these results as they are written in the paper. The authors take full responsibility for the content of this paper.

\bibliographystyle{alpha}
\bibliography{main_refs}

@article{ACH,
  author    = {M. J. Ablowitz and S. Chakravarty and R. G. Halburd},
  title     = {Integrable systems and reductions of the self-dual {Y}ang--{M}ills equations},
  journal   = {J. Math. Phys.},
  volume    = {44},
  number    = {8},
  pages     = {3147--3173},
  year      = {2003}
}

@book{Ahlfors,
  author    = {L. V. Ahlfors},
  title     = {Complex Analysis},
  edition   = {Third},
  publisher = {McGraw--Hill},
  address   = {New York},
  year      = {1979}
}

@inproceedings{Ahmadi12,
  author    = {A. A. Ahmadi},
  title     = {On the difficulty of deciding asymptotic stability of cubic homogeneous vector fields},
  booktitle = {2012 American Control Conference},
  pages     = {3334--3339},
  year      = {2012}
}

@article{AEK,
  author    = {A. A. Ahmadi and B. {El Khadir}},
  title     = {On algebraic proofs of stability for homogeneous vector fields},
  journal   = {IEEE Trans. Automat. Control},
  volume    = {65},
  number    = {1},
  pages     = {325--332},
  year      = {2020}
}

@book{Apostol,
  author    = {T. M. Apostol},
  title     = {Calculus, Volume {II}},
  edition   = {Second},
  publisher = {Wiley},
  address   = {New York},
  year      = {1969}
}

@incollection{Arnold76,
  author    = {V. I. Arnold},
  title     = {Problems of present day mathematics, {XVII} ({D}ynamical systems and differential equations)},
  booktitle = {Mathematical Developments Arising from {H}ilbert Problems},
  editor    = {F. E. Browder},
  series    = {Proc. Sympos. Pure Math.},
  volume    = {28},
  publisher = {Amer. Math. Soc.},
  address   = {Providence, RI},
  year      = {1976},
  pages     = {59}
}

@book{ArnoldMMCM,
  author    = {V. I. Arnold},
  title     = {Mathematical Methods of Classical Mechanics},
  edition   = {Second},
  series    = {Graduate Texts in Mathematics},
  volume    = {60},
  publisher = {Springer},
  address   = {New York},
  year      = {1989},
  note      = {Translated by K. Vogtmann and A. Weinstein}
}

@book{BPR,
  author    = {S. Basu and R. Pollack and M.-F. Roy},
  title     = {Algorithms in Real Algebraic Geometry},
  edition   = {Second},
  series    = {Algorithms and Computation in Mathematics},
  volume    = {10},
  publisher = {Springer},
  address   = {Berlin},
  year      = {2006}
}

@article{Baumgarte,
  author    = {J. Baumgarte},
  title     = {Stabilization of constraints and integrals of motion in dynamical systems},
  journal   = {Comput. Methods Appl. Mech. Engrg.},
  volume    = {1},
  number    = {1},
  pages     = {1--16},
  year      = {1972}
}

@article{Besicovitch,
  author    = {A. S. Besicovitch},
  title     = {On the linear independence of fractional powers of integers},
  journal   = {J. London Math. Soc.},
  volume    = {15},
  number    = {1},
  pages     = {3--6},
  year      = {1940}
}

@article{BlondelTsitsiklis2000,
  author    = {V. D. Blondel and J. N. Tsitsiklis},
  title     = {The boundedness of all products of a pair of matrices is undecidable},
  journal   = {Systems Control Lett.},
  volume    = {41},
  pages     = {135--140},
  year      = {2000}
}

@article{BBKT2001,
  author  = {V. D. Blondel and O. Bournez and P. Koiran and J. N. Tsitsiklis},
  title   = {The stability of saturated linear dynamical systems is undecidable},
  journal = {Journal of Computer and System Sciences},
  volume  = {62}, number = {3}, pages = {442--462}, year = {2001}}

@article{BGP17,
  author    = {O. Bournez and D. S. Gra{\c{c}}a and A. Pouly},
  title     = {Polynomial time corresponds to solutions of polynomial ordinary differential equations of polynomial length},
  journal   = {J. ACM},
  volume    = {64},
  number    = {6},
  pages     = {Art.~38},
  year      = {2017}
}

@article{Brezhnev,
  author    = {Y. V. Brezhnev},
  title     = {Non-canonical extension of $\theta$-functions and modular integrability of $\vartheta$-constants},
  journal   = {Proc. Roy. Soc. Edinburgh Sect.~A},
  volume    = {143},
  number    = {4},
  pages     = {689--738},
  year      = {2013}
}

@article{BLS,
  author    = {Y. V. Brezhnev and S. L. Lyakhovich and A. A. Sharapov},
  title     = {Dynamical systems defining {J}acobi's $\vartheta$-constants},
  journal   = {J. Math. Phys.},
  volume    = {52},
  number    = {11},
  pages     = {112704},
  year      = {2011}
}

@book{Brualdi,
  author    = {R. A. Brualdi},
  title     = {Introductory Combinatorics},
  edition   = {Fifth},
  publisher = {Pearson},
  address   = {Upper Saddle River, NJ},
  year      = {2010}
}

@article{BuserScarpellini,
  author    = {P. Buser and B. Scarpellini},
  title     = {Undecidable propositions by {ODE}'s},
  journal   = {Ann. Acad. Sci. Fenn. Math.},
  volume    = {32},
  number    = {2},
  pages     = {317--340},
  year      = {2007}
}

@article{BuserScarpellini16,
  author    = {P. Buser and B. Scarpellini},
  title     = {Undecidability through {F}ourier series},
  journal   = {Ann. Pure Appl. Logic},
  volume    = {167},
  number    = {7},
  pages     = {507--524},
  year      = {2016}
}

@inproceedings{BychkovPogudin,
  author    = {A. Bychkov and G. Pogudin},
  title     = {Optimal monomial quadratization for {ODE} systems},
  booktitle = {{IWOCA} 2021},
  editor    = {P. Flocchini and L. Moura},
  series    = {Lecture Notes in Comput. Sci.},
  volume    = {12757},
  publisher = {Springer},
  pages     = {122--136},
  year      = {2021}
}

@inproceedings{CaiPogudin,
  author    = {Y. Cai and G. Pogudin},
  title     = {Dissipative quadratizations of polynomial {ODE} systems},
  booktitle = {Tools and Algorithms for the Construction and Analysis of Systems ({TACAS} 2024), Part~{II}},
  editor    = {B. Finkbeiner and L. Kov{\'a}cs},
  series    = {Lecture Notes in Comput. Sci.},
  volume    = {14571},
  publisher = {Springer},
  pages     = {323--342},
  year      = {2024}
}

@article{CPSW,
  author    = {D. C. Carothers and G. E. Parker and J. S. Sochacki and P. G. Warne},
  title     = {Some properties of solutions to polynomial systems of differential equations},
  journal   = {Electron. J. Differential Equations},
  volume    = {2005},
  number    = {40},
  pages     = {1--17},
  year      = {2005}
}

@article{CimaLlibre,
  author    = {A. Cima and J. Llibre},
  title     = {Bounded polynomial vector fields},
  journal   = {Trans. Amer. Math. Soc.},
  volume    = {318},
  number    = {2},
  pages     = {557--579},
  year      = {1990}
}

@article{CEMR,
  author    = {L. Colarte-G{\'o}mez and J. Elias and R. M. Mir{\'o}-Roig},
  title     = {Sumsets and {V}eronese varieties},
  journal   = {Collect. Math.},
  volume    = {74},
  pages     = {353--374},
  year      = {2023}
}

@article{ColdingMinicozzi,
  author    = {T. H. Colding and W. P. {Minicozzi II}},
  title     = {{\L}ojasiewicz inequalities and applications},
  journal   = {Surv. Differ. Geom.},
  volume    = {19},
  pages     = {63--82},
  year      = {2014}
}

@book{ColoniusKliemann,
  author    = {F. Colonius and W. Kliemann},
  title     = {The Dynamics of Control},
  series    = {Systems \& Control: Foundations \& Applications},
  publisher = {Birkh{\"a}user},
  address   = {Boston, MA},
  year      = {2000}
}

@book{CFS,
  author    = {I. P. Cornfeld and S. V. Fomin and Ya. G. Sinai},
  title     = {Ergodic Theory},
  series    = {Grundlehren der mathematischen Wissenschaften},
  volume    = {245},
  publisher = {Springer},
  address   = {New York},
  year      = {1982}
}

@book{Sipser,
  author    = {M. Sipser},
  title     = {Introduction to the Theory of Computation},
  edition   = {Third},
  publisher = {Cengage Learning},
  address   = {Boston},
  year      = {2013}
}

@incollection{dCD93,
  author    = {N. C. A. da Costa and F. A. Doria},
  title     = {On {A}rnol'd's {H}ilbert symposium problems},
  booktitle = {Computational Logic and Proof Theory ({KGC} 1993)},
  editor    = {G. Gottlob and A. Leitsch and D. Mundici},
  series    = {Lecture Notes in Comput. Sci.},
  volume    = {713},
  publisher = {Springer},
  address   = {Berlin},
  pages     = {152--158},
  year      = {1993}
}

@article{DaniSmillie,
  author    = {S. G. Dani and J. Smillie},
  title     = {Uniform distribution of horocycle orbits for {F}uchsian groups},
  journal   = {Duke Math. J.},
  volume    = {51},
  number    = {1},
  pages     = {185--194},
  year      = {1984}
}

@article{DPR61,
  author    = {M. Davis and H. Putnam and J. Robinson},
  title     = {The decision problem for exponential {D}iophantine equations},
  journal   = {Ann. of Math. (2)},
  volume    = {74},
  pages     = {425--436},
  year      = {1961}
}

@book{DLA,
  author    = {F. Dumortier and J. Llibre and J. C. Art{\'e}s},
  title     = {Qualitative Theory of Planar Differential Systems},
  series    = {Universitext},
  publisher = {Springer},
  year      = {2006}
}

@book{EW,
  author    = {M. Einsiedler and T. Ward},
  title     = {Ergodic Theory with a view towards Number Theory},
  series    = {Graduate Texts in Mathematics},
  volume    = {259},
  publisher = {Springer},
  address   = {London},
  year      = {2011}
}

@article{Fenichel74,
  author    = {N. Fenichel},
  title     = {Asymptotic stability with rate conditions},
  journal   = {Indiana Univ. Math. J.},
  volume    = {23},
  number    = {12},
  pages     = {1109--1137},
  year      = {1974}
}

@article{GarayHofbauer,
  author    = {B. M. Garay and J. Hofbauer},
  title     = {Robust permanence for ecological differential equations, minimax, and discretizations},
  journal   = {SIAM J. Math. Anal.},
  volume    = {34},
  number    = {5},
  pages     = {1007--1039},
  year      = {2003}
}

@article{Velasco69,
  author    = {E. A. {Gonz{\'a}lez Velasco}},
  title     = {Generic properties of polynomial vector fields at infinity},
  journal   = {Trans. Amer. Math. Soc.},
  volume    = {143},
  pages     = {201--222},
  year      = {1969}
}

@article{GCB,
  author    = {D. S. Gra{\c{c}}a and M. L. Campagnolo and J. Buescu},
  title     = {Computability with polynomial differential equations},
  journal   = {Adv. in Appl. Math.},
  volume    = {40},
  number    = {3},
  pages     = {330--349},
  year      = {2008}
}

@article{gracca2008boundedness,
  author    = {D. S. Gra{\c{c}}a and J. Buescu and M. L. Campagnolo},
  title     = {Boundedness of the domain of definition is undecidable for polynomial {ODE}s},
  journal   = {Electron. Notes Theor. Comput. Sci.},
  volume    = {202},
  pages     = {49--57},
  year      = {2008}
}

@book{HardyWright,
  author    = {G. H. Hardy and E. M. Wright},
  title     = {An Introduction to the Theory of Numbers},
  edition   = {Sixth},
  publisher = {Oxford University Press},
  address   = {Oxford},
  year      = {2008},
  note      = {Revised by D. R. Heath-Brown and J. H. Silverman}
}

@book{Hatcher,
  author    = {A. Hatcher},
  title     = {Algebraic Topology},
  publisher = {Cambridge University Press},
  address   = {Cambridge},
  year      = {2002}
}

@book{HPS,
  author    = {M. W. Hirsch and C. C. Pugh and M. Shub},
  title     = {Invariant Manifolds},
  series    = {Lecture Notes in Math.},
  volume    = {583},
  publisher = {Springer},
  address   = {Berlin},
  year      = {1977}
}

@article{Hofbauer81,
  author    = {J. Hofbauer},
  title     = {A general cooperation theorem for hypercycles},
  journal   = {Monatsh. Math.},
  volume    = {91},
  number    = {3},
  pages     = {233--240},
  year      = {1981}
}

@book{HornJohnson,
  author    = {R. A. Horn and C. R. Johnson},
  title     = {Matrix Analysis},
  edition   = {Second},
  publisher = {Cambridge University Press},
  address   = {Cambridge},
  year      = {2013}
}

@incollection{HubbardMiller,
  author    = {J. H. Hubbard and R. L. Miller},
  title     = {Equidistribution of horocyclic flows on complete hyperbolic surfaces of finite area},
  booktitle = {In the Tradition of {A}hlfors--{B}ers, {V}},
  series    = {Contemp. Math.},
  volume    = {510},
  publisher = {Amer. Math. Soc.},
  address   = {Providence, RI},
  pages     = {129--150},
  year      = {2010}
}

@article{Jenkinson,
  author    = {O. Jenkinson},
  title     = {Ergodic optimization in dynamical systems},
  journal   = {Ergodic Theory Dynam. Systems},
  volume    = {39},
  number    = {10},
  pages     = {2593--2618},
  year      = {2019}
}

@article{JPS87,
  author    = {R. A. Johnson and K. J. Palmer and G. R. Sell},
  title     = {Ergodic properties of linear dynamical systems},
  journal   = {SIAM J. Math. Anal.},
  volume    = {18},
  number    = {1},
  pages     = {1--33},
  year      = {1987}
}

@article{Kerner,
  author    = {E. H. Kerner},
  title     = {Universal formats for nonlinear ordinary differential systems},
  journal   = {J. Math. Phys.},
  volume    = {22},
  number    = {7},
  pages     = {1366--1371},
  year      = {1981}
}

@book{Khalil,
  author    = {H. K. Khalil},
  title     = {Nonlinear Systems},
  edition   = {Third},
  publisher = {Prentice Hall},
  address   = {Upper Saddle River, NJ},
  year      = {2002}
}

@misc{Korda,
  author       = {M. Korda},
  title        = {{L}yapunov stability of polynomial vector fields is undecidable},
  howpublished = {arXiv:2609.22058v2},
  year         = {2026}
}

@misc{KordaSupp,
  author       = {M. Korda},
  title        = {Supplementary material for ``{L}yapunov stability of polynomial vector fields is undecidable'': dimension five, global asymptotic stability, and the planar case},
  howpublished = {\url{https://homepages.laas.fr/mkorda/arnold_supplement.pdf}},
  year         = {2026}
}

@book{Kuczma,
  author    = {M. Kuczma},
  title     = {An Introduction to the Theory of Functional Equations and Inequalities: {C}auchy's Equation and {J}ensen's Inequality},
  edition   = {Second},
  publisher = {Birkh{\"a}user},
  address   = {Basel},
  year      = {2009},
  note      = {Edited by A. Gil{\'a}nyi}
}

@book{Lawden,
  author    = {D. F. Lawden},
  title     = {Elliptic Functions and Applications},
  series    = {Applied Mathematical Sciences},
  volume    = {80},
  publisher = {Springer},
  address   = {New York},
  year      = {1989}
}

@inproceedings{Lean4,
  author    = {L. de Moura and S. Ullrich},
  title     = {The {L}ean~4 theorem prover and programming language},
  booktitle = {Automated Deduction -- {CADE} 28},
  series    = {Lecture Notes in Comput. Sci.},
  volume    = {12699},
  publisher = {Springer},
  pages     = {625--635},
  year      = {2021}
}

@book{Lee,
  author    = {J. M. Lee},
  title     = {Introduction to Smooth Manifolds},
  edition   = {Second},
  series    = {Graduate Texts in Mathematics},
  volume    = {218},
  publisher = {Springer},
  address   = {New York},
  year      = {2013}
}

@misc{liu2026contr,
  author       = {J. Liu and I. O. Shevchenko},
  title        = {Controllability of polynomial systems is undecidable and not recursively enumerable},
  howpublished = {arXiv:2610.26832},
  year         = {2026}
}

@incollection{Lojasiewicz63,
  author    = {S. {\L}ojasiewicz},
  title     = {Une propri{\'e}t{\'e} topologique des sous-ensembles analytiques r{\'e}els},
  booktitle = {Les {\'E}quations aux D{\'e}riv{\'e}es Partielles (Paris, 1962)},
  series    = {Colloques Internationaux du CNRS},
  volume    = {117},
  publisher = {{\'E}ditions du CNRS},
  address   = {Paris},
  pages     = {87--89},
  year      = {1963}
}

@inproceedings{Mathlib,
  author    = {{The Mathlib Community}},
  title     = {The {L}ean mathematical library},
  booktitle = {Proceedings of the 9th {ACM} {SIGPLAN} International Conference on Certified Programs and Proofs ({CPP} 2020)},
  pages     = {367--381},
  year      = {2020}
}

@article{Matiyasevich70,
  author    = {Y. V. Matiyasevich},
  title     = {Enumerable sets are {D}iophantine},
  journal   = {Soviet Math. Dokl.},
  volume    = {11},
  number    = {2},
  pages     = {354--357},
  year      = {1970}
}

@book{Matiyasevich93,
  author    = {Y. V. Matiyasevich},
  title     = {Hilbert's Tenth Problem},
  series    = {Foundations of Computing},
  publisher = {MIT Press},
  address   = {Cambridge, MA},
  year      = {1993}
}

@article{Mencinger,
  author    = {M. Mencinger},
  title     = {On quadratizations of homogeneous polynomial systems of {ODE}s},
  journal   = {Glas. Mat. Ser.~III},
  volume    = {50(70)},
  number    = {1},
  pages     = {163--182},
  year      = {2015}
}

@book{Meyer,
  author    = {C. D. Meyer},
  title     = {Matrix Analysis and Applied Linear Algebra},
  publisher = {SIAM},
  address   = {Philadelphia, PA},
  year      = {2000}
}

@article{Moore1990,
  author    = {C. Moore},
  title     = {Unpredictability and undecidability in dynamical systems},
  journal   = {Phys. Rev. Lett.},
  volume    = {64},
  number    = {20},
  pages     = {2354--2357},
  year      = {1990}
}

@article{Poincare1881,
  author    = {H. Poincar{\'e}},
  title     = {M{\'e}moire sur les courbes d{\'e}finies par une {\'e}quation diff{\'e}rentielle},
  journal   = {J. Math. Pures Appl. (3)},
  volume    = {7},
  pages     = {375--422},
  year      = {1881}
}

@article{Rosier,
  author    = {L. Rosier},
  title     = {Homogeneous {L}yapunov function for homogeneous continuous vector field},
  journal   = {Systems Control Lett.},
  volume    = {19},
  number    = {6},
  pages     = {467--473},
  year      = {1992}
}

@book{Rudin,
  author    = {W. Rudin},
  title     = {Principles of Mathematical Analysis},
  edition   = {Third},
  publisher = {McGraw--Hill},
  address   = {New York},
  year      = {1976}
}

@article{Samardzija,
  author    = {N. Samard{\v{z}}ija},
  title     = {Stability properties of autonomous homogeneous polynomial differential systems},
  journal   = {J. Differential Equations},
  volume    = {48},
  number    = {1},
  pages     = {60--70},
  year      = {1983}
}

@article{Schreiber98,
  author    = {S. J. Schreiber},
  title     = {On growth rates of subadditive functions for semiflows},
  journal   = {J. Differential Equations},
  volume    = {148},
  number    = {2},
  pages     = {334--350},
  year      = {1998}
}

@inproceedings{Sontag1988,
  author    = {E. D. Sontag},
  title     = {Some complexity questions regarding controllability},
  booktitle = {Proceedings of the 27th {IEEE} Conference on Decision and Control},
  volume    = {2},
  pages     = {1326--1329},
  year      = {1988}
}

@book{Sturmfels,
  author    = {B. Sturmfels},
  title     = {Gr{\"o}bner Bases and Convex Polytopes},
  series    = {University Lecture Series},
  volume    = {8},
  publisher = {Amer. Math. Soc.},
  address   = {Providence, RI},
  year      = {1996}
}

@book{Tarski,
  author    = {A. Tarski},
  title     = {A Decision Method for Elementary Algebra and Geometry},
  edition   = {Second},
  publisher = {University of California Press},
  address   = {Berkeley and Los Angeles},
  year      = {1951}
}

@book{Teschl,
  author    = {G. Teschl},
  title     = {Ordinary Differential Equations and Dynamical Systems},
  series    = {Graduate Studies in Mathematics},
  volume    = {140},
  publisher = {Amer. Math. Soc.},
  address   = {Providence, RI},
  year      = {2012}
}

@book{VO,
  author    = {M. Viana and K. Oliveira},
  title     = {Foundations of Ergodic Theory},
  series    = {Cambridge Studies in Advanced Mathematics},
  volume    = {151},
  publisher = {Cambridge University Press},
  address   = {Cambridge},
  year      = {2016}
}

@book{WW,
  author    = {E. T. Whittaker and G. N. Watson},
  title     = {A Course of Modern Analysis},
  edition   = {Fifth},
  publisher = {Cambridge University Press},
  address   = {Cambridge},
  year      = {2021},
  note      = {Edited by V. H. Moll}
}

@incollection{Zagier123,
  author    = {D. Zagier},
  title     = {Elliptic modular forms and their applications},
  booktitle = {The 1-2-3 of Modular Forms},
  editor    = {K. Ranestad},
  series    = {Universitext},
  publisher = {Springer},
  address   = {Berlin},
  pages     = {1--103},
  year      = {2008},
  note      = {By J. H. Bruinier, G. van der Geer, G. Harder and D. Zagier}
}

@book{Folland,
  author    = {G. B. Folland},
  title     = {Real Analysis: Modern Techniques and Their Applications},
  edition   = {Second},
  publisher = {Wiley},
  address   = {New York},
  year      = {1999}
}

@article{KoiranCosnardGarzon94,
  author  = {P. Koiran and M. Cosnard and M. Garzon},
  title   = {Computability with low-dimensional dynamical systems},
  journal = {Theoretical Computer Science},
  volume  = {132}, number = {1--2}, pages = {113--128}, year = {1994}}

@article{KoiranMoore99,
  author  = {P. Koiran and C. Moore},
  title   = {Closed-form analytic maps in one and two dimensions can simulate universal {Turing} machines},
  journal = {Theoretical Computer Science},
  volume  = {210}, number = {1}, pages = {217--223}, year = {1999}}

\end{document}